\documentclass[11pt,reqno]{amsart}

\pdfoutput=1

\usepackage{afterpage}

\usepackage[english]{babel}

\usepackage[lmargin=1.2in,rmargin=1.2in,tmargin=1in,bmargin=1in]{geometry}

\usepackage{amsmath}	
\usepackage{amssymb}
\usepackage{amsthm}
\usepackage{leftidx,tensor}
\usepackage{mathtools}
\usepackage{mathrsfs}

\usepackage{float}
\usepackage{graphicx}
\usepackage{pst-all}
\usepackage{color}
\usepackage{tikz}
\usepackage{xcolor}
\usepackage[all,cmtip]{xy}
\usetikzlibrary{calc,3d}
\usepackage{amsmath,amssymb,amsfonts}
\usepackage{pgfplots}
\pgfplotsset{compat=newest}

\usepackage{tikz-cd}
\usepackage{enumerate}
\usepackage{booktabs}

\input xy
\xyoption{all}

\usepackage{mathrsfs}
\usepackage{amssymb}

\usepackage{setspace}
\usepackage{subfig}
\usepackage{caption}
\usepackage{babel}

\usepackage{stackengine} 
\newcommand\oast{\stackMath\mathbin{\stackinset{c}{0ex}{c}{0ex}{\ast}{\bigcirc}}}

\newtheorem{theorem}{Theorem}[section]
\newtheorem*{theorem*}{Theorem}
\newtheorem{corollary}[theorem]{Corollary}
\newtheorem*{corollary*}{Corollary}
\newtheorem{lemma}[theorem]{Lemma}
\newtheorem*{lemma*}{Lemma}
\newtheorem{proposition}[theorem]{Proposition}
\newtheorem*{proposition*}{Proposition}

\newtheorem{assumption}[theorem]{Assumption}
\newtheorem*{assumption*}{Assumption}

\newtheorem{definition}[theorem]{Definition}
\newtheorem*{definition*}{Definition}
\newtheorem{remark}[theorem]{Remark}
\newtheorem{example}[theorem]{Example}

\newtheorem{innercustomassum}{Assumption}
\newenvironment{customassum}[1]
  {\renewcommand\theinnercustomassum{#1}\innercustomassum}
  {\endinnercustomassum}
  
\newtheorem{innercustomthm}{Theorem}
\newenvironment{customthm}[1]
  {\renewcommand\theinnercustomthm{#1}\innercustomthm}
  {\endinnercustomthm}

\DeclareMathOperator{\codim}{codim}
\DeclareMathOperator{\Lag}{Lag}
\DeclareMathOperator{\defe}{def}
\DeclareMathOperator{\Det}{Det}
\DeclareMathOperator{\Span}{span}
\DeclareMathOperator{\im}{im}
\DeclareMathOperator{\rank}{rank}

\newcommand{\Q}{\mathbb{Q}}
\newcommand{\C}{\mathbb{C}}
\newcommand{\Z}{\mathbb{Z}}
\newcommand{\R}{\mathbb{R}}

\title{Lagrangian Cobordism, Lefschetz Fibrations and Quantum Invariants}
\author{Berit Singer}
\thanks{The author was partially supported by the Swiss National Science Foundation (grant number 200021-156000).}
\date{\today}                                           

\begin{document}

\maketitle


\section{Abstract}

In this thesis we study Lagrangian cobordisms with the tools provided by Lagrangian quantum homology. In particular, we develop the theory for the setting of Lagrangian cobordisms or Lagrangians with cylindrical ends in a Lefschetz fibration, and put the different versions of the quantum homology groups into relation by a long exact sequence. We prove various practical relations of maps in this long exact sequence and we extract invariants that generalize the notion of discriminants to Lagrangian cobordisms in Lefschetz fibrations. We prove results on the relation of the discriminants of the ends of a cobordism and the cobordism itself.
We also give examples arising from Lagrangian spheres and relate the discriminant to open Gromov Witten invariants. We show that for some configurations of Lagrangian spheres the discriminant always vanishes. 

We study a set of examples that arise from Lefschetz pencils of complex quadric hypersurfaces of $\mathbb{C}P^{n+1}$.
These quadrics are symplectic manifolds endowed with real structures and their real part are the Lagrangians of interest.
Using the results established in this thesis, we compute the
discriminants of all these Lagrangians by reducing the calculation to
the previously established case of a real Lagrangian sphere in the quadric.

\tableofcontents



\section{Introduction}

The main idea of Lagrangian quantum homology is to define a so called pearl complex, which is based on counting pearly trajectories. These can schematically be thought of as Morse trajectories with finitely many points replaced by pseudo-holomorphic disks. 
The idea of the construction of the quantum homology first appeared in~\cite{Oh96}. It was then developed by Biran and Cornea in~\cite{BC07} and~\cite{BC09a} as a technique to allow effective computations in Lagrangian Floer homology as well as some other applications. 
For monotone, closed Lagrangians we can define on the chain level the PSS morphism, which induces an isomorphism between the quantum homology $QH(L)$ and $HF(L)$. (See~\cite{PSS96} and, for the Lagrangian case,~\cite{BaC06},~\cite{CL06} and~\cite{Alb08}.) In particular the quantum homologies are the endomorphism groups of the objects in the Fukaya category. 
Floer homology involves the count of perturbed holomorphic strips, which makes it generally very hard to calculate. Since quantum homology and its structures are defined by unperturbed pseudo-holomorphic curves, this is an advantage allowing the explicit computation of examples and identities.

In~\cite{BM15} Biran and Membrez found ways to express certain invariants (such as discriminants of Lagrangians) of the quantum homology via the quantum structure on the ambient manifold. In some cases these invariants are sufficient to understand the ring structures. Exploiting the structure on the ambient quantum homology they found the numerical values of the discriminants for many Lagrangian spheres, which would have been way more complicated to calculate directly.

The notion of embedded Lagrangian cobordisms was initially introduced by Arnold (\cite{Arn80a} and~\cite{Arn80b}) and studied by many people since. Biran and Cornea studied Floer theory of Lagrangian cobordism in~\cite{BC13} and outlined the construction of quantum homology for Lagrangian cobordisms. It was further developed in~\cite{Sin15}. 
They also extended the theory to Lagrangian cobordism in tame Lefschetz fibrations (see~\cite{BC15} and also section~\ref{sec:tameLFandCob}). 

The thesis below is concerned with extensions of the Lagrangian quantum homology theory to Lagrangian cobordisms in Lefschetz fibrations and applications to the theory of discriminants. 

\subsection{Quantum structures}

The first observation is that Lagrangian quantum homology can be extended to a setting for manifolds with cylindrical ends and in particular to Lagrangian cobordism in a Lefschetz fibration $\pi:E\to \mathbb{C}$. 
As we will see, we may assume that the Lefschetz fibration is trivial over a subset $\mathcal{W}\subset \mathbb{C}$. (Also called a tame Lefschetz fibration. See section~\ref{sec:tameLFandCob}.) By this we mean
$$E|_{\mathcal{W}}\cong \mathcal{W}\times M, \quad \Omega|_{\mathcal{W}}\cong\omega_{\mathbb{C}}\oplus \omega_M \text{ and } J|_{\mathcal{W}}\cong i\oplus J_M.$$
A Lagrangian cobordism is an embedded Lagrangian submanifold $V$ in $E$ such that outside of some compact subset $K\subset \mathcal{W}^c$ the projection of $V$ under $\pi$ is a union of rays of the type $\ell_i=(-\infty,a_i]\times \{i\}$, $i\in\{1,\cdots,r_-\}\subset \mathbb{N}$ and $\ell_j=[a_j,\infty)\times \{j\}$, $j\in\{1,\cdots,r_+\}\subset \mathbb{N}$, and $V$ looks like $\ell_i\times L_i$ and $\ell_j\times L_j$ over $\mathcal{W}$, where $L_{i}^{\pm}$ and $L_j^{\pm}$ are closed Lagrangian submanifols of $M$. We refer to $L_1^-,\ldots, L_{r_-}^-$ as the negative ends of $V$ and to $L_1^{+},\ldots,L_{r_+}^+$ as the positive ends, and we write $V:(L_1^-,\ldots, L_{r_-}^-) \to (L_1^{+},\ldots,L_{r_+}^+)$. Assume that $V$ is a monotone Lagrangian cobordism in $E$. 
Set $$\Lambda^+:=\mathbb{Z}_2[t] \text{ and }\Lambda:=\mathbb{Z}_2[t,t^{-1}].$$
We grade these rings such that $\deg(t)=-N_V$, where $N_V$ denotes the minimal Maslov number of $V$.
Similarly as for closed Lagrangians we can define a pearl complex of $\Lambda^+$ modules.
For a collection of connected components $S\subset \partial V$ this defines a version of quantum homology $Q^+H_*(V,S)$ which is a (possibly non-unital) ring and a two-sided algebra over the ambient quantum homology $Q^+H_*(E,\partial^v E)$, where $\partial^vE$ is the vertical part of the boundary of $E$ (see Section~\ref{sec:boundaryE} for details and Theorem~\ref{thm:quantum_structures_for_cobordism} for a precise statement.)
If we work with coefficients in $\Lambda$ we denote the resulting homology by $QH_*(V,S)$.
One can then relate the quantum homology of the cobordism to the quantum homology of its ends via a long exact sequence. (See Theorem~\ref{thm:longexact}).

\begin{customthm}{A}\label{thmA}
Let $S\subset \partial V$ be a union of connected components of $\partial V$.
There exists a long exact sequence
\begin{equation*}
\xymatrix{
 \dots \ar[r]^{\delta} & Q^+H_*(S) \ar[r]^{i_*} & Q^+H_*(V)\ar[r]^{j_*} & 
Q^+H_*(V,S)\ar[r]^{\delta} & Q^+H_{*-1}(S)\ar[r]^{i_*} & \dots,
}
\end{equation*}
which has the following properties:
\begin{enumerate}
\item Suppose that $S=\partial V$. Let $e_{(V,\partial V)}$ denote the unit of 
$Q^+H_*(V,\partial V)$. Furthermore, let $e_{L_i^-}$ and $e_{L_j^+}$ denote the units of $Q^+H_*(L_i^-)$ and $Q^+H_*(L_j^+)$ respectively.
Then 
$$\delta(e_{(V,\partial V)})= \oplus_{i} e_{L_i^-}\oplus_{j} e_{L_j^+}.$$
\item  The map $\delta$ is multiplicative with respect to the quantum product $*$, 
namely 
\begin{equation}
 \delta(x*y)=\delta(x)*\delta(y) \ \ \forall x,y \in Q^+H_*(V,S).
\end{equation}
\item  The product $*$ on $Q^+H_*(V)$ is trivial on the image of the map $i_*$.
In other words, for any two elements $x$ and $y$ in $Q^+H_*(S)$ we have that 
$$i_*(x)*i_*(y)=0.$$
\item The map $j_*$ is multiplicative with respect to the quantum product, namely
\begin{equation*}
 j_*(x*y)=j_*(x)*j_*(y)\ \ \forall x,y \in Q^+H_*(V).
\end{equation*}
Moreover, $j_*(Q^+H_*(V))\subset Q^+H_*(V,\partial V))$ is a two-sided ideal.
\item If the Lefschetz fibration is trivial (i.e. $(E,\Omega)=(\mathbb{C} \times M,\omega_{std}\oplus \omega)$) and the fibers are closed, there exists a 
ring isomorphism $\Phi: Q^+H_*(M) \rightarrow Q^+H_{*+2}(E,\partial E)$ and the following identities hold:
\begin{enumerate}
 \item[(i)] $i_*(a \star x)= \Phi(a) \star i_*(x)$, $\forall x\in Q^+H(\partial V)$ and $\forall a\in Q^+H(M)$.
 \item[(ii)] $j_*(a \star x)=a\star j_*(x)$, $\forall x\in Q^+H(V)$ and $\forall a\in 
Q^+H(E,\partial E)$.
 \item[(iii)] $\delta(a \star x)=\Phi^{-1}(a) \star \delta(x)$, $\forall x\in 
Q^+H(V,\partial V)$ and $\forall a\in Q^+H(E,\partial E)$.
\end{enumerate}
In other words, the maps in the long exact sequence are module maps over the ambient quantum homology rings 
$Q^+H_*(E,\partial E)$ and $Q^+H_*(M)$.
\begin{remark}
 Here the operation $*$ denotes the quantum product and $\star$ denotes the operation coming from the module structure of the various versions of the quantum homology of $V$.
\end{remark}

\item All statements remain true if we use $QH_*$ instead of $Q^+H_*$.
\item If $V$ is orientable and spin, we may work with $\Lambda_{\mathcal{R}}^+:={\mathcal{R}}[t]$ or
$\Lambda_{\mathcal{R}}:={\mathcal{R}}[t,t^{-1}]$, where ${\mathcal{R}}$ is any commutative unital ring and the analogous statements to $1.$-$6.$ above hold. 
\end{enumerate}
\end{customthm}

\subsection{Cobordisms and discriminants}

Let $(M,\omega)$ be a monotone symplectic $2n$-dimensional manifold and $L$ a closed, monotone, spin Lagrangian submanifold.
Suppose that the minimal Maslov number of $L$ divides $n$.
Let $QH_*(L;\mathbb{Z})$ denote the quantum homology with coefficients in $\mathbb{Z}$. (I.e. $QH(L;\Z)$ is obtained by setting $t=1$ at the chain level.)
By the duality properties of the quantum homology (see~\cite{BC07},~\cite{BC09a},~\cite{BC09b})
there exists a natural map 
$$\epsilon: QH_0(L;\Z)\to H_0(L;\mathbb{Z}),$$
which is surjective (and it is called the augmentation).
If the rank of $QH_0(L;\Z)$ is two, we have a short exact sequence
$$\xymatrix{0 \ar[r] & \ker(\epsilon)\ar[r] & QH_0(L;\Z) \ar[r]^{\epsilon} & H_0(L;\mathbb{Z}) 
\ar[r] & 0},$$
and the unit $e_L$ is a generator of $\ker(\epsilon)$. Choose $x\in QH_0(L;\Z)$ a lift of 
the class of a point in $H_0(L;\mathbb{Z})$. Then $QH_0(L;\Z)\cong 
\mathbb{Z}x\oplus \mathbb{Z}e_L$ is a quadratic algebra, which implies that we can write $QH_0(L;\Z)\cong\mathbb{Z}[T]/(f(T)),$ where $f(T)\in \mathbb{Z}[T]$ is a quadratic, monic polynomial. The discriminant $\Delta(f)$ of $f$ is an invariant of $QH_0(L;\Z)$ and moreover determines the ring structure of $QH_0(L;\Z)$ uniquely up to isomorphism. We write $\Delta_L=\Delta(f)$.

\begin{customassum}{$\mathcal{A}$}\label{assum:main_theorem}
 Let $V:L_0 \to (L_1,\cdots, L_r)$ be a connected, monotone, oriented and spin $(n+1)$-dimensional Lagrangian cobordism in a Lefschetz fibration. Suppose that the minimal Maslov number $N_V$ divides $n$ and that the ends of $V$ are closed Lagrangians with $\rank(QH_0(L_i;\Z))=2$ for all $i$.
\end{customassum}
Under these assumptions the ring structures of the ends of the cobordism are isomorphic (see Theorem~\ref{thm:equal_discr_rank2}). 

\begin{customthm}{B}\label{thmB}
Let $V$ be a cobordism as in~\ref{assum:main_theorem}.
Then,
$$\Delta_{L_i}=\Delta_{L_j} \text{ for all } i, j.$$
If in addition $r\geq 2$, then this number is a square.
\end{customthm}

A nice set of examples comes from considering the Lefschetz pencil of complex, quadric hypersurfaces in $\mathbb{C}P^{n+1}$.
More precisely, to the very ample line bundle $\mathscr{L}:= \mathcal{O}_{\mathbb{C}P^{n+1}}(2)$ we associate the real Lefschetz fibration $\pi:E\to \mathbb{C}P^1$, whose fibers are the hyperplane sections
$$\Sigma^{(\lambda)}:=\{[X_0:\ldots:X_{n+1}]\in \mathbb{C}P^{n+1}|\lambda(X_0,\ldots,X_{n+1})=0\},$$
where $[\lambda]\in \mathbb{P}(H^0(\mathscr{L}))$ is represented by a homogeneous quadratic polynomial $\lambda$. 
If $\lambda$ is a polynomial with real coefficients, i.e. $[\lambda]$ lies in the real part $\mathbb{P}_{\mathbb{R}}(H^0(\mathscr{L}))$, the hyperplane section $\Sigma^{(\lambda)}$ is endowed with a real structure and its real part is a Lagrangian submanifold. 
The real part of the discriminant locus $\Delta_{\mathbb{R}}(\mathscr{L})$ divides $\mathbb{P}_{\mathbb{R}}(H^0(\mathscr{L}))$ into $n/2+1$ chambers, which are in one-to-one correspondence with the signature of the matrix representing the quadratic polynomial. 
One of the chambers of $\mathbb{P}_{\mathbb{R}}(H^0(\mathscr{L}))\setminus\Delta_{\mathbb{R}}(\mathscr{L})$ contains the quadric whose real part is a Lagrangian sphere. Using the Lagrangian cubic equation Biran and Membrez showed that its discriminant is $(-1)^{\frac{n(n-1)}{2}+1}4$ (see~\cite{BM15}).

Given two real homogeneous, quadratic polynomials $\lambda_i=X_0^2+\ldots+X_i^2-X_{i+1}^2+\ldots +X_{n+1}^2$ and $\lambda_{i+1}=X_0^2+\ldots+X_{i+1}^2-X_{i+2}^2+\ldots +X_{n+1}^2$ we can define a Lefschetz pencil passing through them. This gives a description of a real Lefschetz fibration $E^{(i,i+1)}\to \mathbb{C}$ with one critical point. Its real part defines a monotone Lagrangian cobordism $V^{(i,i+1)}$ between the Lagrangians $\Sigma^{(\lambda_i)}_{\R}$ and $\Sigma^{(\lambda_{i+1})}_{\R}$. For $n=2 \mod4$ we prove that these cobordisms are orientable and spin for all $0\leq i\leq n$. Together with Theorem\ref{thmB} this implies the following result.

\begin{customthm}{C}\label{thmC}
Let $n=2 \mod 4$.
For any two real quadratic surfaces $\Sigma^{(\lambda)}$ and $\Sigma^{(\lambda')}$ in $\mathbb{P}_{\mathbb{R}}(H^0(\mathscr{L}))\setminus\Delta_{\mathbb{R}}(\mathscr{L})$ there exists a Lagragian cobordism between the real parts $\Sigma_{\R}^{(\lambda)}$ and $\Sigma_{\R}^{(\lambda')}$ that fulfils assumption~\ref{assum:main_theorem}.
In particular, $\Delta_{\Sigma^{(\lambda)}}=4$ for all $\lambda \in \mathbb{P}_{\mathbb{R}}(H^0(\mathscr{L}))\setminus\Delta_{\mathbb{R}}(\mathscr{L})$.
\end{customthm}

A more precise statement is given in Theorem~\ref{thm:discs_real_parts}.

\subsection{Strategy for the proof of the quantum structures and of Theorem~\ref{thmA}}

Most of the ingredients of the construction of the quantum homologies $Q^+H(V,S)$, $Q^+H(V)$ and also $QH(E,\partial^{v}E)$ are similar to the case of closed Lagrangians submanifolds. The main two differences are that now we work with Lagrangians that have boundaries (or cylindrical ends), and more significantly, that we are in a non-compact setting. The latter requires a way to ensure that the spaces of holomorphic disks with boundary on $V$ are compact. Our approach to solving this problem, following~\cite{BC13}, is to use the open mapping theorem.
As we will see in section~\ref{sec:preliminaries} we may assume that we can find a set $\mathcal{W}$ 
containing the cylindrical ends of $V$, a symplectic structure $\Omega$ on $E$ 
and an almost complex structure $J$ such that over $\mathcal{W}$ we have 
$$E|_{\mathcal{W}}\cong \mathcal{W}\times M, \quad \Omega|_{\mathcal{W}}\cong\omega_{\mathbb{C}}\oplus \omega_M \text{ and } J|_{\mathcal{W}}\cong i\oplus J_M.$$
Moreover, $\pi:E\to \mathbb{C}$ is $(J,i)$-holomorphic. A pseudo-holomorphic disk with boundary on the cylindrical part of $V$ is thus forced to be constant under the projection $\pi$ as a consequence of the open mapping theorem. 

This description allows describe the compactifications of the moduli spaces of pearly trajectories in an analogous way to the closed setting treated in~\cite{BC07} and~\cite{BC09a} and hence, also their proof can be adopted with small modifications.
An algebraic count of the boundary of the compactification of the one-dimensional moduli space of pearly trajectories gives by definition the square of the differential of the pearl complex. It is then not hard to see that this number vanishes, hence we obtain a chain complex. Its homology groups are the quantum homology groups of $V$.
The various other structures are proved in similar ways.

The long exact sequence of Theorem~\ref{thmA} is the long exact sequence in homology coming from a short exact sequence that looks roughly like
 \begin{equation}
\xymatrix{
0 \ar[r] & C^+_k(U;f|_U,J) \ar[r]^-j \ar[d]^{d_U} & 
C^+_k((V,S);f,J) \ar[r]^-{\delta} \ar[d]^{d_{(V,S)}} & 
C^+_{k-1}(S;f|_S,J) \ar[d]^{d_S} \ar[r] & 
0\\
0 \ar[r] & C^+_{k-1}(U;f|_U,J) \ar[r]^-{j} & 
C^+_{k-1}((V,S);f,J) \ar[r]^-{\delta} & 
C^+_{k-2}(S;f|_{S},J) \ar[r] & 0. \\
}
\end{equation}
Here $C^+_*(U,f|_U,J)$ is the subcomplex of $C^+_*((V,S),f,J)$ obtained by restricting to an open set $U\subset V$, such that $V\setminus U$ is contained in the cylindrical ends of $V$. The idea is that $-\nabla f$ points inwards along the whole boundary $\partial U$ such that $C_*(U,f|_U,J)$ actually computes $QH_*(V)$.
The complex $C^+_*(S,f|_S,J)$ is just the restriction to $S$, which is homotopy equivalent to $V\setminus U$.
An outline of the construction of this long exact sequence has already been given in~\cite{BC13}.

\subsection{Strategy of the proof of Theorem~\ref{thmB} and Theorem~\ref{thmC}}

Fix elements $p_i\in QH_0(L_i;\Z)$ that are lifts of a point in $H_0(L_i;\mathbb{Z})$ under the augmentation map. Using the long exact sequence from Theorem~\ref{thmA} we show that for any $i\neq j$ there exists an element $\alpha_{ij}\in QH_1(V,\partial V)$ with the property that $\delta(\alpha_{ij})=p_i-p_j$, where $\delta$ is the connecting homomorphism of the long exact sequence and moreover that the $\alpha_{0i}$ together with the unit form a basis of $ QH_1(V,\partial V)$. For any given end $L_i$ of the cobordisms $V$, we then consider the maps in the long exact sequence of Theorem~\ref{thmA} that have image in the quantum homology ring $QH_0(L_i)$. Exploiting the multiplicativity of the map $\delta$ and properties of other maps in the long exact sequence, we are able to show that for each $i\in\{1,\cdots, r\}$ the ring $QH_0(L_i)$ is isomorphic to the subring of $QH_1(V,\partial V)$ spanned by $\alpha_{0i}$ and the unit $e_{(V,\partial V)}$. Utilizing again the identities presented in Theorem~\ref{thmA} one finds that in the basis $\{e_{(V,\partial V)},\alpha_{01},\ldots,\alpha_{0r}\}$ of $QH_1(V,\partial V)$ most quantum products vanish and the discriminants of $QH_0(L_i)$ can be extracted from the quantum product structure of $QH_1(V,\partial V)$. Finally, since $\delta(\alpha_{0i})=p_0-p_i$ and $\delta(e_{(V,\partial V)})=-e_{L_0}\oplus_{i=1,\ldots, r}e_{L_i}$ we can directly compare the quantum product structure of $QH_0(L_0)$ and $QH_0(L_i)$ and show that $\Delta_{L_0}=\Delta_{L_i}$.

The real parts $\Sigma^{(\lambda_i)}_{\R}$ and $\Sigma^{(\lambda_{i+1})}_{\R}$ of the real quadrics are Lagrangians. They are the fixed point locus of anti-symplectic involutions on the real quadrics. These Lagrangians can be interpreted as the total space of some Serre fibrations. With this Serre fibration the singular homology groups of the Lagrangians can be expressed as the singular homology groups with local coefficients in the singular homology of the fibers and moreover we determine orientability of the Lagrangians.  Another fibration together with the Gysin sequence allows us to determine the spinnability of the Lagrangians.
The fixed point locus of the anti-symplectic involution on the total space of the Lefschetz fibration $E^{(i,i+1)}\to \C$ gives a monotone Lagrangian cobordism $V^{(i,i+1)}$ between $\Sigma^{(\lambda_{i+1})}_{\R}$ and $\Sigma^{(\lambda_i)}_{\R}$. The homotopy type of the cobordisms $V^{(i,i+1)}$ relative to the subspace that is given by its intersection with the thimble is equal to the homotopy type of the singular fibre relative to the vanishing cycle. Relative Stiefel-Whitney classes, their naturality properties and the long exact sequence of $\Sigma^{(\lambda_i)}_{\R}$, respectively $\Sigma^{(\lambda_{i+1})}_{\R}$ and their vanishing cycles are needed to prove spinnability of $V^{(i,i+1)}$. Since smooth quadrics are homogeneous manifolds, their tangent bundle is nef. Reflecting holomorphic disks by the anti-holomorphic involution we get a holomorphic bundle over $\C P^1$ and we can split the latter into the direct sum of holomorphic line bundles. Since the tangent bundle of the quadrics are nef the line bundles, which are the summands of the splitting, are also nef. From this the regularity of the standard complex structure follows. 
Finally, regularity of the standard complex structure and arguments based on a spectral sequence enable us to compute the quantum homologies of the Lagrangians  $\Sigma^{(\lambda_i)}_{\R}$ and $\Sigma^{(\lambda_{i+1})}_{\R}$. The quantum homology groups turn out to be wide.
This shows that the necessary assumptions for Theorem~\ref{thmB} hold and it implies Theorem~\ref{thmC}. 

\subsection{Organization of the thesis}

In the first section of the thesis we review the definition of Lefschetz fibrations, the basic setting and concepts we are working with and the quantum homology for closed Lagrangians as well as some of its properties. We also recall the definition of discriminants as given by~\cite{BM15} and summarize some of their important results.
In the second section we build up the theory of quantum homology for Lagrangians with cylindrical ends in Lefschetz fibrations and prove their properties. In particular, we prove Theorem~\ref{thmA}.
The third section is dedicated to developing the more general notion of discriminants for cobordisms. It contains some results on the relation of the discriminants of the ends of a cobordism and the cobordism itself. Also the proof of Theorem~\ref{thmB} can be found in this section.
The fourth section proves Theorem~\ref{thmC}.
We also added an appendix on filtered chain complexes and spectral sequences, a brief appendix on spin structures and an appendix on projective duality, discriminants and hyperdeterminants.


\section{Preliminaries}\label{sec:preliminaries}

In this section we outline the most important concepts used in this thesis.
The notations and definitions presented in section 1.1 to 1.6 are mainly borrowed 
from Biran and Cornea in~\cite{BC15},~\cite{BC17}, ~\cite{BC07} and~\cite{BC09a}.

\subsection{Lefschetz fibrations}\label{sec:LF}

One can find several versions of the notion of a Lefschetz fibration in the literature. 
We will stick to the version given by Biran and Cornea (\cite{BC15}) and 
which is also similar to the setup in~\cite{Sei08} and~\cite{Sei03}.

\begin{definition}[\cite{BC15}]\label{def:LefschetzFibration}
 A Lefschetz fibration with compact fibers consists of the following data:
\begin{enumerate}[i.]
  \item A symplectic manifold $(E,\Omega)$ without boundary, endowed with a compatible 
almost complex structure $J_E$.
  \item A Riemann surface $(S,j)$ (not necessarily compact, typically we choose 
$S=\mathbb{C}$).
  \item A proper $(J_E, j)$- holomorphic map $\pi: E\to S$. In particular, the regular (or smooth) fibers of 
$E$ are closed manifolds.
  \item Assume that $\pi$ has a finite number of critical points. Moreover, every 
critical value of $\pi$ corresponds to exactly one critical point of $\pi$. 
We denote the set of critical points by $Crit(\pi)$ and the critical values by 
$Critv(\pi)$. Sometimes the critical points of $\pi$ are also called singularities of $E$.
  \item For every critical point $p\in Crit(\pi)$ there exists a local $J_E$-holomorphic 
chart around $p$ and a $j$-holomorphic chart around $\pi(p)$ with respect to which $\pi$ 
is a holomorphic Morse function. (I.e. every critical point is an ordinary double point.) 
\end{enumerate}
Let $E_z:=\pi^{-1}(z)$ be the fiber over $z$. Fix a point $z_0\in S\setminus Critv(\pi)$. 
The symplectic manifold $(M,\omega_M):=(\pi^{-1}(z_0),{\Omega_E}|_M)$ is referred to as 
the fiber of the Lefschetz fibration. 
For a subset $T\subset S$ and a set $A\subset E$ let us denote 
$A|_T:=A\cap{\pi^{-1}(T)}$. 
\end{definition}

If the fibers are not compact we can adjust the preceding definition. The changes that we 
need to make are the following. First, the condition $iii.$ is changed. Namely, the map 
$\pi$ can no longer be proper as to allow non-compact fibers.
Second, we now need to assume explicitly that $\pi$ is a smooth locally trivial fibration 
away from its critical points and values and that we have triviality at infinity.

\begin{definition}\label{def:contact_type_boundary,symplectization}
 The boundary $\partial M$ of a symplectic manifold $(M,\omega)$ is said to be of contact 
type, if there exists a conformal symplectic vector field $X$, i.e. 
$\mathcal{L}_X\omega=d(\iota_X\omega)=\omega$, defined near $\partial M$ and everywhere 
transverse to $\partial M$.
In this case $(\partial M,ker(\iota_X\omega))$ is a contact manifold.
The symplectization of a contact manifold $(N,\alpha)$ is the symplectic manifold 
$(N\times \mathbb{R},d(e^{t}\alpha))$. 
A (non-compact) symplectic manifold $(M,\omega)$ is said to be convex at infinity, if 
there exists a compact subset $M^0\subset M$ with boundary of 
contact type and such that $M$ looks like a symplectization of $\partial M^0$ outside of 
$M^0$.
\end{definition}

\begin{definition}[\cite{BC15}]\label{def:LefschetzFibration_nc}
Let $(M,\omega_M)$ be a (non-compact) symplectic manifold which is convex at infinity.
A Lefschetz fibration with generic fiber 
$(M,\omega_M)=(E_{z_0},\Omega|_{E_{z_0}}):=(\pi^{-1}(z_0),\Omega|_{\pi^{-1}(z_0)})$ 
with $z_0\in S\setminus Critv(\pi)$ consists of the following data:
\begin{enumerate}[i.]
  \item A symplectic manifold $(E,\Omega)$ without boundary, endowed with a compatible 
almost complex structure $J_E$.
  \item A Riemann surface $(S,j)$.
  \item A $(J_E, j)$- holomorphic map $\pi: E\to S$.
  \item Assume that $\pi$ has a finite number of critical points. Moreover, every 
critical value of $\pi$ corresponds to exactly one critical point of $\pi$.
  \item For every critical point $p\in Crit(\pi)$ there exists a local $J_E$-holomorphic 
chart around $p$ and a $j$-holomorphic chart around $\pi(p)$ with respect to which $\pi$ 
is a holomorphic Morse function. (I.e. every critical point is an ordinary double point.) 
  \item The map $\pi:E\setminus \pi^{-1}(Crit(\pi)) \to S\setminus Critv(\pi)$ is a 
smooth locally trivial fibration.
\end{enumerate}
And additionally
\begin{customassum}{$\mathbf{T^{\infty}}$} [Triviality at infinity (\cite{BC15})]\label{assum:T_infty}
There exists a subset $E^0\subset E$ with the properties:
\begin{enumerate}[(1)]
 \item For every compact subset $K\subset S$ the set $E^0\cap \pi^{-1}(K)$ is also 
compact. I.e. $\pi|_{E^0}:E^0\to S$ is a proper map.
  \item Set $E^{\infty}:=E\setminus E^0$ and $E^{\infty}_{z_0}:=E^{\infty}\cap 
\pi^{-1}(z_0)$, where $z_0\subset S\setminus Critv(\pi)$ is a fixed base-point. 
Then there exists a trivialization $\phi:S\times E^{\infty}_{z_0}\to E^{\infty}$ of 
$\pi|_{E^{\infty}}:E^{\infty}\to S$ such that 
$$\phi^*\Omega_E=\omega_S\oplus \omega_M|_{E^{\infty}_{z_0}}, \text{ and } 
\phi^*J_E=j\oplus J_M$$
where $\omega_S$ is a positive (with respect to $j$) symplectic form on $S$ and $J_M$ is 
a fixed almost complex structure on $M=\pi^{-1}(z_0)$, compatible with $\omega_M$.
  \item $E^0\cap M=M^0$ and thus also $E^{\infty}\cap M 
\cong (\partial M^0,d(e^t\iota_X\omega))$. 
Moreover, this holds for any other regular 
fiber, i.e. if $z\in S\setminus Crit(\pi)$ then $(E^0\cap E_z,\Omega|_{E^0\cap E_z})\cong 
(M^0,\omega_M|_{M_0})$ and $(E^{\infty}\cap E_z,\Omega|_{E^{\infty}\cap E_z})\cong (\partial 
M^0,d(e^t\iota_X\omega))$.
\end{enumerate}
\end{customassum}
\end{definition}

\begin{remark}
 Definition~\ref{def:LefschetzFibration} is in fact a special case of 
definition~\ref{def:LefschetzFibration_nc} in the sense that a Lefschetz fibration with 
compact fiber is a Lefschetz fibration as in~\ref{def:LefschetzFibration_nc} if we take 
$E^0=E$ and $E^{\infty}=\emptyset$.
\end{remark}

Throughout this thesis, unless otherwise stated, by a Lefschetz fibration we mean one 
with compact fiber and satisfying Definition~\ref{def:LefschetzFibration} or one with 
non-compact fiber satisfying the more general Definition~\ref{def:LefschetzFibration_nc}.
From now on we also assume that all fibers of a Lefschetz fibration have positive 
dimension.

\begin{remark}
Compactness of the fiber $(M,\omega)$ or convexity at infinity together 
with~\ref{assum:T_infty} is assumed in order 
to make the fiber amenable to methods such as Gromov compactness for $J$-holomorphic 
curves. In particular these assumptions ensure that the pearly trees used in the 
definition of the quantum homology and its relations live inside a compact subset.
\end{remark}

\subsection{\texorpdfstring{$\Omega_E$}{Lg}-orthogonal connection and 
parallel transport}

The symplectic form $\Omega_E$ naturally defines a connection $\Gamma=\Gamma(\Omega_E)$ 
on $E\setminus Crit(\pi)$. This connection is most simply described by its 
associated horizontal distribution $\mathcal{H}\subset TE$, which is defined to be the 
$\Omega_E$-orthogonal complement to the tangent spaces of the fibers.
Let $T^v_xE$ be the vertical tangent space of $T_xE$ for some $x\in E\setminus Crit(\pi)$. We define the horizontal part 
of $T_xE$ by 
$$\mathcal{H}_x=\{u\in T_x(E)|\Omega_E(\xi, u)=0 \forall \xi \in T^v_xE\}.$$
For a path $\gamma:[a,b] \to  S\setminus Crit(\pi)$, let $\Pi_{\gamma}: E_{\gamma(a)} 
\to E_{\gamma(b)}$ denote the parallel transport along $\gamma$ with respect to the 
connection $\Gamma$.
Due to the assumption~\ref{assum:T_infty} the parallel transport becomes the 
identity at 
infinity with respect to the trivialization. In particular the parallel transport is also 
defined for the case of a Lefschetz fibration with non-compact fibers.

Here we summarize some facts about the parallel transport in symplectic Lefschetz 
fibrations. For more details we refer 
to~\cite{McS17},~\cite{McS12} 
and~\cite{Sei08}.
The parallel transport is a symplectomorphism between the fibers endowed with the 
symplectic structures induced by $\Omega_E$.
For a loop $\gamma$ with $\gamma(a)=\gamma(b)=z$ the parallel transport 
$\Pi_{\gamma}:E_z\to E_z$ is called the holonomy of $\Gamma$ along $\gamma$. 
If $\gamma$ is contractible within $E\setminus Crit{\pi}$ one can show that 
$\Pi_{\gamma}$ is a Hamiltonian diffeomorphism.

The image of a Lagrangian $L$ under the parallel transport along $\gamma:[a,b]\to 
E\setminus Critv(\pi)$ is $L_t:=\Pi_{\gamma_{[a,t]}}(L)\subset E_{\gamma(t)}$ for 
$t\in[a,b]$. The union
$$\gamma L:=\bigcup_{t\in [a,b]}L_t$$ is a Lagrangian submanifold of $(E,\Omega_E)$ 
called the \emph{trail} of $L$ along $\gamma$.

\subsection{Lagrangians with cylindrical ends}
Let $\pi: E \to \mathbb{C}$ be a Lefschetz fibration and $\mathcal{U}\subset \mathbb{C}$ 
an open subset containing $Critv(\pi)$. We introduce the following useful terminology:
A horizontal ray $\ell\in \mathbb{C}$ is a half-line of the form $(-\infty, 
-a_{\ell}]\times \{b_{\ell}\} $ or $[a_{\ell},\infty)\times 
\{b_{\ell}\}$ with $a_{\ell}>0$, $b_{\ell}\in \mathbb{R}$.
The imaginary coordinate $b_{\ell}$ is also called the height of $\ell$.

\begin{definition}[\cite{BC15}]\label{def:cylindrical_ends}
 A Lagrangian submanifold (without boundary) $V\subset (E,\Omega_E)$ is said to have 
\emph{cylindrical ends} outside of $\mathcal{U}$ if 
\begin{enumerate}[i.]
 \item For every $R>0$, the subset $V\cap \pi^{-1}([-R,R]\times \mathbb{R})$ is compact
  \item $\pi(V)\cap \mathcal{U}$ is bounded
  \item $\pi(V)\setminus \mathcal{U}$ consists of a finite union of horizontal rays, 
$\ell_i\subset \mathbb{C}$, $i=1,\ldots,r$. For every $i$ we have 
$V|_{\ell_i}=\ell_iL_i$, the trail of some $L_i$ along $\ell_i$. Here $L_i\subset 
E_{\sigma_i}$ is some Lagrangian and $\sigma_i$ is the starting point of $\ell_i$ (i.e. $\{-a_{l_i}\}\times \{b_{l_i}\}$ or $\{a_{l_i}\}\times \{b_{l_i}\}$).
\end{enumerate}
If all rays $\ell_i$ have positive height, we call $V$ a \emph{cobordism} in $E$.
\end{definition}

\begin{remark}
 Condition $ii.$ in the above definition has the purpose that the set $\mathcal{U}$ does 
not contain any of the ends of $V$.
\end{remark}

\subsection{Tame Lefschetz fibrations and cobordisms}\label{sec:tameLFandCob}

Definition~\ref{def:cylindrical_ends} is almost a generalization of the notion of a 
Lagrangian cobordism in a trivial Lefschetz fibration, i.e. in $E=M\times \mathbb{C}$. 
However, there is some imprecision to this notation since we did not fix a trivialization 
of the Lefschetz fibration $\pi:E \to \mathbb{C}$ at infinity. Such a notion is necessary, 
since only then the ends of a cobordism are well-defined. 
Therefore, we introduce the notion of tame Lefschetz fibrations.

\begin{definition}[\cite{BC15}]\label{def:tameLF}
 Let $\pi:E\to \mathbb{C}$ be a Lefschetz fibration as in 
Definition~\ref{def:LefschetzFibration} or~\ref{def:LefschetzFibration_nc}. 
Let $U\subset \mathbb{C}$ be a closed subset, $z_0\in \mathbb{C}\setminus U$ a base 
point and $(M,\omega_M)$ the fiber over $z_0$. We say that the Lefschetz fibration is 
tame outside of $U$ if there exists a trivialization
$$\psi_{E,\mathbb{C}\setminus U}:(\mathbb{C}\setminus U) \times M \to 
E_{\mathbb{C}\setminus U} ,$$
such that $\psi^*_{E,\mathbb{C}\setminus U}(\Omega_E)=c\cdot \omega_{\mathbb{C}}\oplus 
\omega_M$ and for some $c>0$. We call $(M,\omega_M)$ the generic fiber of $\pi$. 
\end{definition}

Notice that this definition implies that $Critv(\pi)\subset U$. 
Let $\mathcal{W}:=\mathbb{C}\setminus U$. We also say that $\pi$ is tame over 
$\mathcal{W}$. 
Given a tame Lefschetz fibration, we fix the data $(U_E:=U, z_0,\psi_{E,\mathbb{C}\setminus U})$.
Moreover, we assume that there exists $a_U>0$ such that $U$ is disjoint from
$$Q^-_U:=(-\infty, -a_U]\times [0,\infty) \text{ and } Q^+_U:=[a_U,\infty)\times [0, 
\infty).$$

Now we can give a precise definition of a cobordism relation in a tame Lefschetz 
fibration.

\begin{definition}[\cite{BC15}]\label{def:cobordism}
 Fix a Lefschetz fibration that is tame outside of $U\subset \mathbb{C}$ with fiber 
$(M,\omega)$ over $z_0\in \mathbb{C}\setminus U$. Let $(L_i^-)_{1\leq i\leq k_-}$ and 
$(L^-_j)_{1\leq j\leq k_+}$ be two families of closed Lagrangian submanifolds of $M$. We 
say that these two families are Lagrangian cobordant in $E$, if there exists a Lagrangian 
submanifold $V\subset E$ with the properties
\begin{enumerate}[i.]
  \item There is a compact set $K\subset E$ so that $V\cap U\subset V\cap K $ and $ 
V\setminus K \subset \pi^{-1}(Q^-_U\cup Q^+_U)$
  \item $V\cap \pi^{-1}(Q^-_U)=\coprod _i((-\infty,-a_U]\times \{i\})\times L_i^-$
  \item $V\cap \pi^{-1}(Q^+_U)=\coprod _j([a_U, \infty)\times \{j\})\times L_j^+$,
\end{enumerate}
where the formulas at $ii.$ and $iii.$ are written with respect to the trivialization of 
the fibration over the complement of $U$.
\end{definition}

Clearly, the manifold $V$ in the definition is a Lagrangian cobordism in the sense of 
Definition~\ref{def:cylindrical_ends}. Because of the tameness condition its ends are 
well-defined, and we can therefore say that $V$ is a cobordism between the family $L_i^-$ 
and $L_j^+$. We write $V:L_i^- \to L_j^+$.

The next result, taken from~\cite{BC15} ensures that we can always pass from a general Lefschetz fibration to a 
tame one.

\begin{proposition}[\cite{BC15}]\label{prop:tameLF}
 Let $\pi:E \to \mathbb{C}$ be a Lefschetz fibration and let $\mathcal{N}\subset 
\mathbb{C}$ be an open subset that contains all the critical values of $\pi$ and has the 
shape depicted in figure~\ref{fig:tameLF}. Let $\mathcal{W}\subset \mathbb{C}$ be 
another open subset as depicted in figure ~\ref{fig:tameLF} with 
$\overline{\mathcal{W}}\cap\overline{\mathcal{N}}=\emptyset$ and 
$dist(\overline{\mathcal{W}},\overline{\mathcal{N}})>0$. 
Then there exists a symplectic structure $\Omega'=\Omega'_{E,\mathcal{N},\mathcal{W}}$ 
on $E$ and a trivialization $\phi:\mathcal{W}\times M \to E|_{\mathcal{W}}$ with the 
following properties:
\begin{enumerate}[(1)]
 \item On $\mathcal{W}\times M$ we have $\phi^*\Omega'=c\cdot\omega_{\mathbb{C}}\oplus 
\omega_M$ for some $c>0$.
  \item $\Omega'$ coincides with $\Omega_E$ on all the fibers of $E$
  \item $\Omega'=\Omega_E$ on $\pi^{-1}(\mathcal{N})$.
  \item There exists an $\Omega'$-compatible almost complex structure $J'_E$ on $E$ which 
coincides with $J_E$ on $\pi^{-1}(\mathcal{N})$ and such that the projection $\pi:E \to 
\mathbb{C}$ is $(J'_E,i)$-holomorphic.
\end{enumerate}
In particular, the Lefschetz fibration $\pi:E\to \mathbb{C}$ is tame over 
$\mathcal{W}$ when endowed with 
the symplectic structure $\Omega'$.
\end{proposition}

\begin{figure}[H]
 \centering
\includegraphics[scale=0.15]{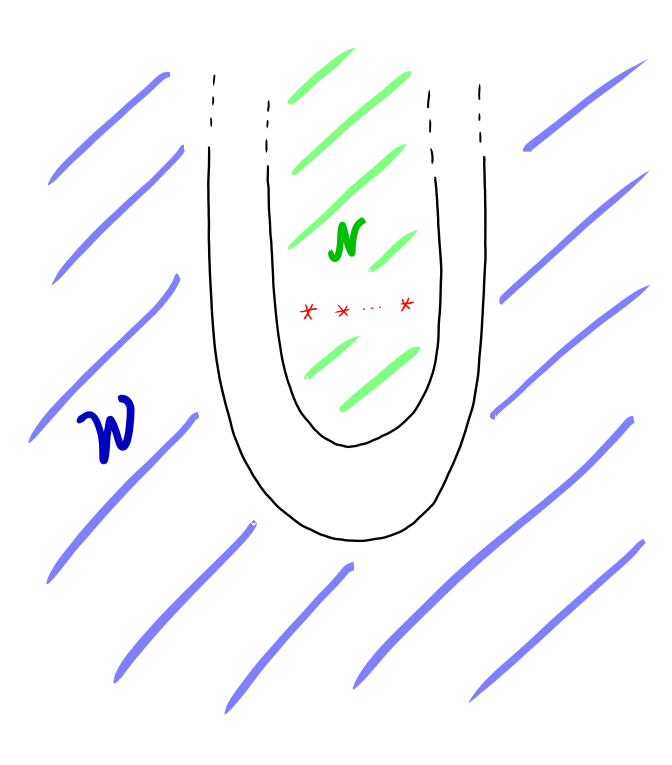}
\caption{We may assume that the Lefschetz fibration is tame over the set 
$\mathcal{W}$.}\label{fig:tameLF}
\end{figure}

\begin{remark}\label{rmk:cobordism_in_tame_Lefschetz}
 Using the above proposition it is easy to pass from a cobordism (or a Lagrangian with 
cylindrical ends) in a general Lefschetz 
fibration to a cobordism (or Lagrangian with 
cylindrical ends) in a tame Lefschetz fibration, and hence a cobordism with well-defined ends. (See also figure~\ref{fig:cylindrical_ends}.)
For this we simply need to perform the parallel transport corresponding to the 
connection $\Omega'$ of the corresponding Lagrangians along the horizontal 
rays. 
More precisely, let $\mathcal{N}\subset \mathbb{C}$ be a subset as in 
Proposition~\ref{prop:tameLF}. Suppose $\ell_i$ are the horizontal rays corresponding to 
the ends of $V$ and $L_i\subset E_{\sigma_i}$ are the corresponding Lagrangians over the 
starting points of the rays.
Performing parallel transport of the $L_i$ along $\ell_i$ with respect to the connection 
induces by $\Omega'$ yields a new Lagrangian submanifold $V'\subset (E,\Omega')$. 
One can easily check that 
\begin{enumerate}[(1)]
 \item $V'$ coincides with $V$ over $\mathcal{N}$
  \item $V'$ has cylindrical ends outside of $\mathcal{N}$
  \item The ends of $V'$ have the form described in $ii.$ and $iii.$ in 
Definition~\ref{def:cobordism}
  \item $V|_{\mathcal{N}}$ is compact and in particular there exists a compact set 
$K\subset \mathbb{C}$ such that $V|_K$ is compact and $V|_{K^c}$ is cylindrical.
  \item One can always find an almost complex structure $J$ compatible with 
$\Omega'$ and such that $J=i\oplus J_M$ over $\pi^{-1}(\mathcal{W})$.
\end{enumerate}
\end{remark}

\begin{figure}[H]
 \centering
\includegraphics[scale=0.15]{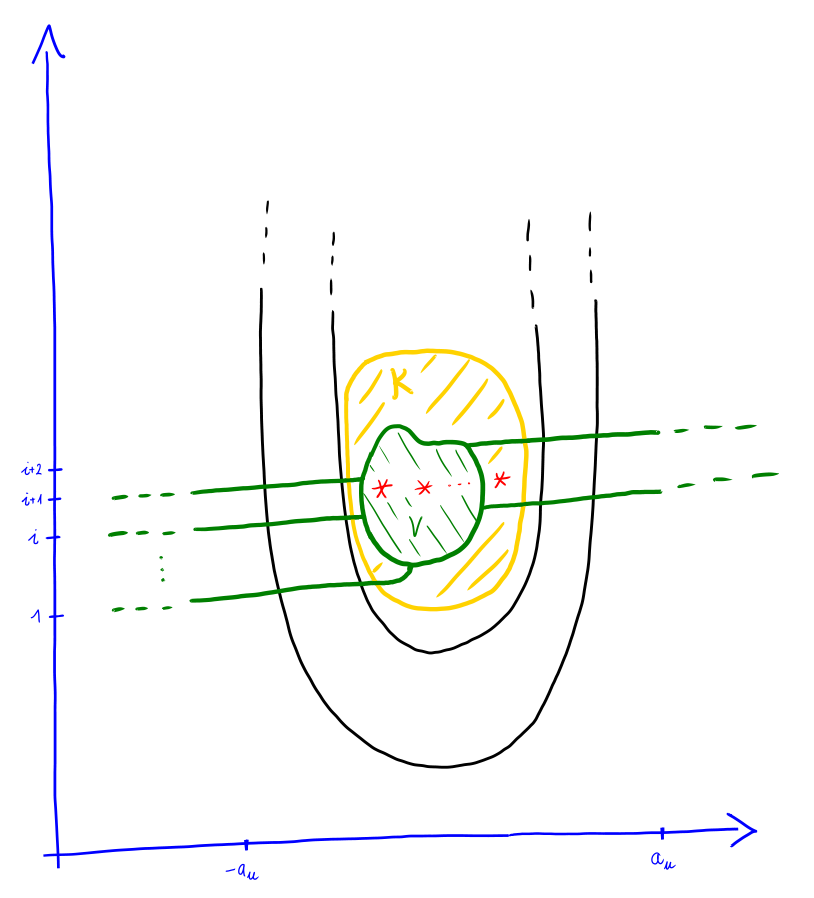}
\caption{Lagrangian cobordism in $E$ with well-defined ends.}\label{fig:cylindrical_ends}
\end{figure}

\subsection{Real Lefschetz fibrations}

By a real Lefschetz fibration we mean a symplectic Lefschetz fibration $\pi:E \to 
\mathbb{C}$ endowed with an anti-symplectic involution, i.e. a map $c_E:E\to E$ such that 
$c_E\circ c_E=id_E$, $c_E^*\Omega=-\Omega$ and $\pi\circ c_E=c_{\mathbb{C}}\circ \pi$, where $c_{\mathbb{C}}$ is the standard complex conjugation. 

Lagrangian cobordism are naturally related to the notion of real Lefschetz fibrations as 
the following Proposition from~\cite{BC15} shows.
\begin{proposition}[\cite{BC15}]\label{prop:real_part_is_cobordism}
 Let $V:=Fix(c_E)$ denote the fixed point locus of $c_E$. Under the above assumptions $V$ 
is a Lagrangian cobordism with at most one positive and at most one negative end (but 
possibly no ends at all). 
Its projection is of the form $\bigcup_{j\in\mathcal{S}} \overline{I_j}$, where 
$\mathcal{S}$ is a subset of the set of connected components of $\mathbb{R}\setminus 
Critv(\pi)$, $I_j$ stands for the path connected component corresponding to $j$ and 
$\overline{I_j}$ is its closure. In particular, $\partial V\subset Critv(\pi)\cap 
\mathbb{R}$.

Moreover, for every $z\in \mathbb{R}\setminus Critv(\pi)$ the part of $V$ lying over $z$, 
$V_z:=E_z\cap V$, coincides with the fixed point locus of the anti-symplectic involution 
$c_E|_{E_z}$, hence is either empty or a smooth Lagrangian submanifold of $E_z$ (possibly 
disconnected). In particular, the Lagrangians corresponding to the ends of $V$ (if they 
exist) are real with respect to the restriction of $c_E$ to the regular fibers over the 
real axis at $\pm \infty$.

If $(E,\Omega)$ is a monotone symplectic manifold then $V$ is a monotone Lagrangian 
submanifold of $E$. Further, denote by $c^{min}_1(E)$ the minimal Chern number on 
spherical classes in $E$ and by $N_V$ the minimal Maslov number of $V$. If $c^{min}_1(E)$ 
is odd then $c^{min}_1(E)|N_V$, and if $c^{min}_1(E)$ is even then 
$\frac{1}{2}c_1^{min}(E)|N_V$.

If $dim_{\mathbb{C}}M\geq 2$ and $(M,\omega)$ is monotone then $(E,\Omega)$ is monotone 
too and $c^{min}_1(E)=c^{min}_1(M)$, hence $V$ is a monotone Lagrangian cobordism.
\end{proposition}

In order to prove this proposition the following Lemma is useful.
\begin{lemma}\label{lem:crit_pt_on_lagrangian}
 Let $(M,\omega)$ be a symplectic manifold with compatible almost complex structure 
$J_M$, $N$ a smooth manifold endowed with an almost complex structure $J_N$ and $f:M\to N$ 
a $(J_M,J_N)$-holomorphic map. For any Lagrangian submanifold $L\subset M$ a point $x\in 
L$ is a critical point of $f$ if and only if it is a critical point of $f|_L$.
\end{lemma}

\begin{proof}
 One direction is obvious. Let us prove that for $x$ to be a critical point of $f$ it 
suffices that $Df|_{TL}=0$.
Suppose that for all $v\in T_xL$ we have $D_xf(v)=0$. Notice that $J_M|_{TL}:TL\to TN$ induces an 
isomorphism, where $NL$ denotes the normal bundle of $L$. 
But since
$$D_xf(J_Mv)=J_ND_xf(v)=0$$
we see that $Df$ also vanishes on the normal bundle of $L$.
\end{proof}

\begin{proof}[Proof of Proposition~\ref{prop:real_part_is_cobordism}]
 We show first that the fixed point locus of an anti-symplectic involution is Lagrangian. 
Notice that for every $v\in T(Fix(c_E))$ we have $Dc_Ev=v$. 
Thus, for every $v,w\in T(Fix(c_E))$ we have
$$\begin{array}{ccc}
   \Omega_E(v,w) &=& \Omega_E(Dc_Ev, Dc_E w)\\
  &=& c_E^*\Omega_E(v,w)\\
  &=& -\Omega_E(v,w),
  \end{array}$$
which implies $\Omega_E(v,w)=0$.

We now show that $V$ is a cobordism and prove the other statements about the projection 
$\pi(V)$. 
By Lemma~\ref{lem:crit_pt_on_lagrangian} we know that 
$D\pi_x|_{T_xV}:T_xV\to \mathbb{R}$ vanishes if and only if $x\in Crit(\pi)$. 
Hence, $\pi(V)\setminus Critv(\pi)$ is an open subset of $\mathbb{R}$ and all 
values in $\pi(V)\setminus Critv(\pi)$ are regular values of $\pi|V:V\to \mathbb{R}$. 
By construction $V\subset E$ is a closed subset. Therefore a connected component 
$I\subset \mathbb{R}\setminus Critv(\pi)$ with $\pi(V)\cap I\neq \emptyset$ must be 
contained in $\pi(V)$, i.e. $I\subset \pi(V)$. 
Since $V$ is Lagrangian it is also invariant with respect to parallel transport along any 
interval $I\subset \pi(V)\setminus Critv(\pi)$. 

It is clear from the definition that $Fix(c_E|_{E_z})=V_z$.

We prove monotonicity of $V$. It follows by a reflection argument and from the 
spherical monotonicity of $(E,\Omega)$.
Let $u:(D,\partial D)\to (E,V)$ be a holomorphic disk. Then also the disk 
$\overline{u}:=c_E\circ u \circ c_D$ is holomorphic and $v:=u+\overline{u}$ represents a 
holomorphic sphere in $E$. Suppose that $u$ represents the homology class $B\in 
H_2(E,V)$, $\overline{u}$ represents the homology class $B'\in H_2(E,V)$ (with possibly 
$B=B'$) and $v$ represents the class $A\in H_2(E)$. 
Now, $\Omega_E(B)=\Omega_E(B')$ and $\Omega_E(A)=\Omega_E(B)+\Omega_E(B')=2\Omega_E(B)$. 
Since $(E,\Omega)$ is monotone with monotonicity constant $\eta$ we have 
$\eta c_1(A)=\Omega_E(A)=2\Omega_E(B)$. 
Moreover, $\mu(u)=\mu(\overline{u})$ and 
$c_1(v)=\frac{1}{2}\mu(u)+\frac{1}{2}\mu(\overline{u})=\mu(u)$. 
Hence, we can see that $\omega(B)=\frac{1}{2}\eta c_1(A)=\frac{1}{2}\eta \mu(B)$ and the 
monotonicity constant $\tau$ of $V$ satisfies $2\tau=\eta$.
Assume now that $u:(D,\partial D)\to (M,L)$ is a holomorphic disk with minimal Maslov 
number $N_V$. 
Then the sphere $v=u+\overline{u}$ is holomorphic and has first Chern number equal to 
$k\cdot c_1^{min}$, with some $k\in \mathbb{Z}$ and $c^{min}_1$ is the minimal Chern 
number.
Then $k\cdot c_1^{min}=c_1(v)=\mu(u)=N_V$ and hence $c_1^{min}|N_V$. Clearly, if 
$c_1^{min}$ is even, then $\frac{1}{2}c_1^{min}|c_1^{min}|N_V$.

Let us prove the last statement that relates the spherical monotonicity of $M$ with that 
of $E$. Let $E_{z_0}\subset E$ be a smooth fiber of the Lefschetz fibration and let $F$ 
be a singular fiber of the Lefschetz fibration. It is know that $F$ is 
homotopy equivalent to $E_{z_0}/C$, where 
$C$ is a vanishing cycle in $E_{z_0}$. With the assumption that $dim_{\mathbb{C}}M\geq 2$ 
we know that $C\cong S^k$ with $k\geq2$. In particular, $C$ is simply-connected and hence 
by the Blakers–Massey theorem for homotopy groups  we know that the quotient map 
$q:E_{z_0}\to E_{z_0}/C$ induces a surjective group homomorphism $\pi_2(E_{z_0},C)\to 
\pi_2(E_{z_0}/C)\cong \pi_2(F)$. Moreover, $\pi_2(E_{z_0})\to \pi_2(E_{z_0},C)\to 
\pi_1(C)=0$ is exact and so $\pi_2(E_{z_0})\to \pi_2(E_{z_0},C)$ is surjective and hence also $\pi_2(E_{z_0})\to \pi_2(F)$. 
Let $v:S^2\to E$ be a holomorphic sphere.
Since $\pi$ is holomorphic the open mapping theorem implies that $\pi\circ v$ must be 
constant equal to $z\in \mathbb{C}$. In other words $im(v)\subset E_{z}$, where $z\in 
\mathbb{C}$ is either a regular or a critical value.
\end{proof}

\subsection{Quantum homology}

\subsubsection{Monotonicity}

For connected, closed, monotone Lagrangian submanifolds $L\subset (M^{2n},\omega)$ where 
$(M,\omega)$ is a tame monotone symplectic manifold, quantum homology has been studied by 
Biran and Cornea and in~\cite{BC07} the algebraic properties are 
described in detail. We will briefly recall the main algebraic structures.
There is a homomorphisms
$$\omega:\pi_2(M,L)\to \mathbb{R}: A\mapsto \int_{D^2}u^*\omega,$$ 
where $u:(D^2,\partial D^2)\to (M,L)$ represents the class $A\in \pi_2(M,L)$ and which is 
independent of a choice of a representative $u$ of $A$.
Another homomorphism is the Maslov index: $$\mu:\pi_2(M,L)\to \mathbb{Z}.$$
The monotonicity assumption means that 
$$\omega(A)>0 \text{ iff } \mu(A)>0, \quad \forall A\in \pi_2(M,L).$$
Clearly, this is equivalent to the existence of a constant $\tau>0$ such that 
$$\omega(A)=\tau \mu(A), \quad \forall A\in \pi_2(M,L).$$
The constant $\tau$ is called the monotonicity constant of $L\subset(M,\omega)$.
The minimal Maslov number is defined to be the integer
$$N_L:=\min\{\mu(A)>0|A\in \pi_2(M,L)\}.$$
Monotone Lagrangians are usually assumed to have $N_L\geq2$. 
Fix the notation
$$\overline{\mu}=\frac{\mu}{N_L}:\pi_2(M,L)\to \mathbb{Z},$$
which is possible since the Maslov index always comes in multiples of $N_L$. 
It is convenient to work with the coefficient ring of Laurent polynomials 
$\Lambda_{\mathcal{R}}:=\mathcal{R}[t,t^{-1}]$ and 
$\Lambda_{\mathcal{R}}^+:=\mathcal{R}[t]$ for some ring $\mathcal{R}$, where 
$|t|:=-N_L$. Often we choose $\mathcal{R}$ such that $L$ is 
$\mathcal{R}$-orientable.
\begin{remark}
 The commonly used coefficient ring in Floer homology is the Novikov ring over 
$\pi_2(M,L)$. Since we assume monotonicity we can simplify this and use the rings 
$\Lambda_{\mathcal{R}}$ and $\Lambda_{\mathcal{R}}^+$ defined above. 
\end{remark}
In particular, if the Lagrangian is oriented and spin with a fixed spin structure we can 
take $\mathcal{R}$ to be $\mathbb{Z}$. It is sometimes necessary or helpful to work with coefficients 
$\mathcal{R}=\mathbb{Z}_2$ (if the Lagrangian is not orientable or not spin for example), 
$\mathcal{R}=\mathbb{Q}$ or 
$\mathcal{R}=\mathbb{C}$ (see section~\ref{sec:disc_lagrangians}).
The grading induced on $\Lambda_{\mathcal{R}}$ and $\Lambda^+_{\mathcal{R}}$ by setting 
$|t|=-N_L$ also induces natural decreasing filtrations called the degree filtrations
$$\mathcal{F}^k\Lambda_{\mathcal{R}}=\{P\in 
\mathcal{R}[t,t^{-1}]|P(t)=a_kt^k+a_{k+1}t^{k+1}+\ldots\}$$
on $\Lambda_{\mathcal{R}}$ as well as on $\Lambda_{\mathcal{R}}^+$.
This degree filtration induces a natural degree filtration on modules over $\Lambda_{\mathcal{R}}$ and $\Lambda_{\mathcal{R}}^+$.

\subsubsection{The ambient quantum homology}

If the Lagrangian $L\subset M$ is monotone then $M$ is spherically monotone in the sense that there exists a constant $\nu>0$ such that
$$\omega(\lambda)=\nu c_1(\lambda)\quad \forall \lambda\in \pi_2(M),$$
where $c_1\in H^2(E)$ is the Chern class of the tangent bundle $TM$. 
This constant is related to the monotonicity constant by $\nu=2\tau$. 
Define the minimal Chern number of $M$:
$$C_M:=\min\{c_1(\lambda)>0|\lambda\in \pi_2(M)\}.$$
Let $h$ be a Morse function on $M$ and define $\Gamma:=\mathbb{Z}_2[s,s^{-1}]$, where $|s|:=-2C_M$.
If $C_M=\infty$ take $\Gamma:=\mathbb{Z}_2$.
We define $QH(M):=H(M,\mathbb{Z}_2)\otimes \Gamma$.
Since $N_L|2C_M$ we can also work with the coefficient ring $\Lambda:=\mathbb{Z}_2[t,t^{-1}]$, where $|t|=-N_L$.
Define the Morse complex 
$$C(M,h,\rho_M;\Lambda)=\mathbb{Z}\langle Crit(h)\rangle\otimes \Lambda.$$
Then the embedding $\Gamma\hookrightarrow \Lambda:s\mapsto t^{2C_M/N_L}$ induces an isomorphism 
 $$QH(M;\Lambda):=QH(M)\otimes_{\Gamma}\Lambda \cong H_*(C(M,h,\rho_M;\Lambda)).$$

%
Let
\begin{equation*}
 \begin{array}{ccc}
  * : QH_{k}(M)\otimes QH_{l}(M) \ \rightarrow \ QH_{k+l-2n}(M),
 \end{array}
\end{equation*}
denote the \emph{quantum intersection product}.
The unit with respect to this product is the fundamental class $[M]\subset QH_{2n}(M)$.
For more details on this subject the reader is referred to~\cite{McS12}.

\subsubsection{Quantum structures}

The following theorem is the main result in~\cite{BC07} and it summarizes the most 
important algebraic structures of Lagrangian quantum homology.
Let $\Lambda:=\mathbb{Z}_2[t,t^{-1}]$ and $\Lambda^+:=\mathbb{Z}_2[t]$, where $|t|=-N_L$.

\begin{theorem}[\cite{BC07} and~\cite{BC09a}]\label{thm:quantum_structures_closed_Lagrangians}
Under the assumptions and conventions described above and for a generic choice of the 
triple $(f,\rho,J)$ there exists a chain complex
$$C^+(L;f,\rho,J)=(\mathbb{Z}_2\langle Crit(f)\rangle \otimes \Lambda^+,d)$$ 
with the following properties:
\begin{enumerate}[(i)]
 \item The homology of the chain complex is independent of the choices of $(f, \rho, J)$. 
It will be denoted by $Q^+H_*(L)$. There exists a canonical (degree preserving) 
augmentation $$\epsilon_L:Q^+H_*(L)\to \Lambda^+,$$ which is a $\Lambda^+$-module map.
\item The homology $Q^+H(L)$ has the structure of a two-sided algebra with unit over the 
quantum homology of $M$, $Q^+H(M)$. More specifically, for every $i, j, k$ there exist 
$\Lambda^+$-bilinear maps
$$Q^+H_i(L)\otimes Q^+H_j(L)\to Q^+H_{i+j-n}(L); x\otimes y \mapsto x 
*y$$
$$Q^+H_k(M)\otimes Q^+H_i(L)\to Q^+H_{k+i-2n}(L); a\otimes x \mapsto a\star x.$$
The first endows $Q^+H(L)$ with the structure of a ring with unit. This ring is in 
general not commutative. The second map endows $Q^+H(L)$ with the structure of a module 
over the quantum homology ring $Q^+H(M)$. Moreover, when viewing these two structures 
together, the ring $Q^+H(L)$ becomes a two-sided algebra over the ring $Q^+H(M)$. The 
unit $[M]$ of $Q^+H(M)$ has degree $2n=dim M$ and the unit of $Q^+H(L)$ has degree $n=dim 
L$.
\item There exists a  map
$$i_L:Q^+H_*(L)\to Q^+H_*(M),$$
which is a $Q^+H_*(M)$-module morphism and which extends the inclusion in singular 
homology. This map is determined by the relation
$$<h^*, i_L(x)>=\epsilon_L(h\star x)$$
for $x\in Q^+H(L)$, $h\in Q^+H(M)$, with $(-)^*$ the Poincar\'e duality and $<-,->$ the 
Kronecker pairing.
\item The differential $d$ respects the degree filtration
$$\mathcal{F}_pC_i:=\bigoplus_{j\geq p} C_{i+jN_L}t^j$$
and all the structures above 
are compatible with the resulting spectral sequence.
\item The homology of the complex:
$$C(L;f,\rho,J)=C^+(L;f,\rho, j)\otimes_{\Lambda^+} \Lambda$$
is denoted by $QH_*(L)$ and all the points above remain true if using $QH(-)$ instead of 
$Q^+H(-)$. The map $Q^+H(L)\to QH(L)$ induced in homology by the change of coefficients 
above is canonical. Moreover, there is an isomorphism
$$QH_*(L)\to HF_*(L)$$
which is canonical up to a shift in grading.
\end{enumerate}
\begin{remark}
By a two-sided algebra $A$ over a ring $R$ we mean that $A$ as a left-module has the 
structure of an algebra over the ring $R$ and the right-action defined by $ar:=ra$ of $R$ 
on $A$ also endows $A$ with the structure of an $R$-module. 
In particular this means that $ r(ab)=(ra)b=a(rb)$ for every $a,b\in A$ and every $r\in 
R$.
\end{remark}
\begin{remark}
 If $L$ is orientable and spin we may as well choose the ground ring $\mathcal{R}$ to be $\mathbb{Z}$, $\mathbb{Q}$ or $\mathbb{C}$ and work with the ring of Laurent polynomials or polynomials
 $$\Lambda_{\mathcal{R}}:=\mathcal{R}[t,t^{-1}], \quad \Lambda_{\mathcal{R}}^+:=\mathcal{R}[t].$$
\end{remark}
\end{theorem}

\subsection{Numerical invariants of the quantum homology}

Numerical invariants of the quantum homology have been studied by Biran and Cornea 
in~\cite{BC09a} and by Biran and Membrez in~\cite{BM15}. We summarize in this subsection some of their definitions and results.

\paragraph{Discriminants of quadratic algebras}
A quadratic algebra over the integers $\mathbb{Z}$ is a commutative unital ring such that 
$\mathbb{Z}$ embeds as a subring of $A$, $i:\mathbb{Z}\to A$ and $A/\mathbb{Z}\cong 
\mathbb{Z}$. In particular the underlying additive abelian group of $A$ is free and of 
rank two. Let $p\in A/\mathbb{Z}$ be a generator such that $A/\mathbb{Z}=\mathbb{Z}p$.
Then there exists a short exact sequence
$$\xymatrix{ 0  \ar[r] & \mathbb{Z} \ar[r]^i & A \ar[r]^{\epsilon} & \mathbb{Z}p 
\ar[r] & 0},$$
where $i$ is the ring embedding and $\epsilon$ is the projection.
Given a lift $x\in A$ of $p$, i.e. $\epsilon(x)=p$, one can write $A\cong 
\mathbb{Z}\oplus \mathbb{Z}x$. 
Hence, there exist $\sigma(x,p)$ and $\tau(x,p)$ such that 
$$x^2=\sigma(x,p)x + \tau(x,p).$$
Clearly, the integers $\sigma(x,p)$ and $\tau(x,p)$ depend on $x$ and $p$.
\begin{definition}\label{def:discriminant_general}
 We define the discriminant of $A$ to be the expression
$$\Delta_A:=\sigma(x,p)^2+4\tau(x,p)\in \mathbb{Z}.$$ 
\end{definition}
One can check that the definition of $\Delta_A$
is independent of $p$ and $x$ and it is therefore an invariant of the isomorphism type of 
$A$. In fact, the isomorphism type of $A$ is determined by $\Delta_A$.

\subsubsection{Discriminants of Lagrangians}

Let $L$ be a Lagrangian submanifold of $(M^{2n},\omega)$.
Let $QH_*(L;\Z)$ denote the quantum homology of $L$ with coefficients in 
$\mathbb{Z}$. I.e. it is obtained from $QH_*(L;\Z[t,t^{-1}])$ by setting $t=1$.

\begin{assumption}[\cite{BM15}]\label{assum:discriminant_closed}
 \hspace{2em}
 \begin{enumerate}
  \item $L$ is closed (i.e. compact, without boundary) and monotone with minimal Maslov 
number $N_L$ satisfying $N_L|n$. 
\item $L$ is oriented and spinnable.
\item $QH_0(L;\Z)$ has rank two.
 \end{enumerate}
\end{assumption}

By part (i) in Theorem~\ref{thm:quantum_structures_closed_Lagrangians} there exists an 
augmentation
$$\epsilon:QH_0(L;\Z)\to H_0(L;\mathbb{Z})\cong \mathbb{Z}.$$
By the duality properties of the quantum homology (see~\cite{BC07},~\cite{BC09a},~\cite{BC09b})
the augmentation is surjective if $QH_0(L;\Z)$ is not narrow.
Moreover, this augmentation map fits into a short exact sequence
$$\xymatrix{0 \ar[r] & \ker(\epsilon)\ar[r] & QH_0(L;\Z) \ar[r]^{\epsilon} & H_0(L,\mathbb{Z}) 
\ar[r] & 0},$$
where $e_L$ is a generator of $\ker(\epsilon)$. Choose $x\in QH_0(L;\Z)$ a lift of 
the class of a point in $H_0(L,\mathbb{Z})$. Then $QH_0(L;\Z)\cong 
\mathbb{Z}x\oplus \mathbb{Z}e_L$. 
In particular, $QH_0(L;\Z)$ is a quadratic algebra, and we can define
\begin{definition}\label{def:discriminant_old}
 $$\Delta_L:=\Delta_{QH_0(L;\Z)}.$$
\end{definition}
(See~\cite{BM15} and chapter~\ref{sec:disc_lagrangians}.)

One result of~\cite{BM15} is the so called Lagrangian cubic equation.

\begin{theorem}[The Lagrangian cubic equation~\cite{BM15}] Let $L\subset M$ be a Lagrangian submanifold satisfying assumption~\ref{assum:discriminant_closed} and with non-vanishing Euler characteristic $\chi$. Then there exist unique constants $\sigma_L\in \frac{1}{\chi^2}\mathbb{Z}$, $\tau_L\in \frac{1}{\chi^3}\mathbb{Z}$ such that the following equation holds in the ambient quantum homology ring $QH(M;\mathbb{Q}[t])$:
 $$[L]^{*3}-\epsilon \chi \sigma_L [L]^{*2}q^{n/2}-\chi^2\tau_L[L]q^n=0.$$
 Moreover, $\Delta_L=\sigma_L^2+4\tau_L.$
\end{theorem}

If the Lagrangian $L$ is a sphere the equation simplifies.

\begin{corollary}[\cite{BM15}]
 If $L$ is a Lagrangian sphere (or more generally a Lagrangian admitting a symplectic diffeomorphism $\phi:M\to M$ with $\phi_*([L])=-[L]$) then $\sigma_L=0$ and the Lagrangian cubic equation reads
 $$[L]^{*3}-\chi^2\tau_L[L]q^n=0$$
 and hence $\Delta_L=4\tau_L$.
\end{corollary}

Another result from~\cite{BM15}, which will be sharpened in chapter~\ref{sec:disc_lagrangians} is the following.

\begin{theorem}[\cite{BM15}]
 Let $L_1, \ldots, L_r\subset M$ be monotone Lagrangian submanifolds, each satisfying 
assumption~\ref{assum:discriminant_closed}. Let $V^{n+1}\subset \mathbb{C}\times M$ be 
a connected monotone Lagrangian cobordism in the trivial Lefschetz fibration 
$\mathbb{C}\times M$. Suppose that the ends of $V$ correspond to $L_1,\ldots, L_r$ and 
that $V$ admits a spin structure. Let $N_V$ denote the minimal Maslov number of $V$ and 
assume that
\begin{enumerate}
 \item $H_{jN_V}(V,\partial V)=0 \quad \forall j$.
\item $H_{1+jN_V}(V)=0 \quad \forall j$.
\end{enumerate}
Then $\Delta_{L_1}=\ldots =\Delta_{L_r}$. Moreover, if $r\geq 3$ then $\Delta_{L_i}$ is a 
perfect square for every $i$.
\end{theorem}


\section{Quantum homology of Lagrangian cobordisms}\label{sec:quantum_homology}

Let $(M,\omega)$ be a tame connected symplectic manifold.
The theory of Lagrangian quantum homology was developed in~\cite{BC07} and~\cite{BC09a} for connected and closed 
(i.e. compact and without boundary) Lagrangian submanifolds $L\subset 
(M,\omega)$.
In this section we extend this definition to Lagrangians with 
cylindrical ends or Lagrangian cobordisms inside symplectic Lefschetz fibrations. 
This setting is also similar to our discussion in~\cite{Sin15}, where we consider the case of trivial Lefschetz fibrations.
The main issue here is the non-compactness of the Lagrangians. 
Namely, that pearly trajectories or trees could possibly escape towards an end at infinity.
To ensure that this does not happen is crucial for the compactness of the 
underlying moduli spaces of pearly trees that are used to define the quantum structures. 
The main tools to deal with this problem are presented in the first section. Using these 
    tools we can establish the well-definedness and existence of the 
various quantum structures by adjusting the methods in~\cite{BC07} and~\cite{BC09a}.

\subsection{Tools to deal with the cylindrical ends}\label{sec:tools}

In this section we describe the main ingredients in extending the notion of Lagrangian 
quantum homology to Lagrangians $V$ with cylindrical ends in a Lefschetz fibration 
$\pi:E\to \mathbb{C}$.
As in the case of closed Lagrangians, we define a chain complex, where the differential is 
given by counting elements in the moduli spaces of pearly trajectories. 

For a set $U\subset \mathbb{C}$ we write $E|_U$ for the restriction of $E$ to $\pi^{-1}(U)$.
Similarly, if $V$ is a Lagrangian with cylindrical ends inside $E$ we write $V|_U:=\pi^{-1}(U)\cap V$.
Fix also the notations $\Omega|_U:= \Omega|_{\pi^{-1}(U)}$ and $J|_{U}:=J|_{\pi^{-1}(U)}$ for the restrictions.
Away from the boundary of the Lagrangian cobordism transversality and compactness of the 
moduli spaces follow in the same way as for closed Lagrangians. 
Proposition~\ref{prop:tameLF} ensures that we can assume the existence of a set $\mathcal{W}$ 
containing the cylindrical ends, a symplectic structure $\Omega$ on $E$ 
and an almost complex structure $J$ such that over $\mathcal{W}$ we have 
$$E|_{\mathcal{W}}= \mathcal{W}\times M, \quad\Omega|_{\mathcal{W}}=\omega_{\mathbb{C}}\oplus \omega_M \text{ and } J|_{\mathcal{W}}=i\oplus J_M.$$
Moreover, $\pi:E\to \mathbb{C}$ is $(J,i)$-holomorphic.

A pseudo-holomorphic disk with boundary mapping to the cylindrical part of $V$
will turn out to be constant under the projection $\pi:E \rightarrow 
\mathbb{C}$, hence its image lies in one fiber of $\pi$.
This is a consequence of the open mapping theorem and it is the content of the following 
lemma from~\cite{BC13}.
\begin{lemma}[\cite{BC13}]\label{lem:curve_openmapping}
Let $u: \Sigma \rightarrow E$ be a $J$-holomorphic curve, where $\Sigma$ 
is either $S^2$ or the unit disk $D$ with $u(\partial D) \subset V$.
Then, either $\pi \circ u$ is constant or its image is contained in 
$\pi^{-1}(\mathcal{W}^c)$, i.e. $\pi \circ u(\Sigma) \subset \mathcal{W}^c$.
\end{lemma}

\begin{proof}
Notice that there exists a compact set $K\subset \mathbb{C}$ such that $V$ is cylindrical 
outside of $K$. (This is point (i) of Definition~\ref{def:cobordism}.) We may assume that $K\subset \mathcal{W}^c$. (See figure~\ref{fig:cylindrical_ends}.)
Note that $\pi\circ u$ is bounded since $\Sigma$ is compact.
Suppose now that $\pi \circ u(\Sigma) \nsubseteq K$ and that $\pi \circ u$ is not constant.
The set $\mathbb{C}\setminus ( K \cup \pi(V))$ is a union of unbounded, connected, open 
subsets of $\mathbb{C}$. 
Let $W$ be one of these connected open subsets.
Notice that $\pi\circ u(\Sigma)\cap W=\pi\circ u(int(\Sigma))\cap W=\overline{\pi\circ 
u(\Sigma)}\cap W$.
By the open mapping theorem, the image $\pi\circ u(int(\Sigma))\cap W$ of the open 
connected set $int(\Sigma)$ under the holomorphic map $\pi \circ u$ must be open in $W$.

On the other hand $\overline{\pi\circ u(\Sigma)}\cap W=\pi\circ u(\Sigma)\cap W$ is 
closed in $W$ and $W$ is connected. 
Hence, $\pi\circ u(\Sigma)\cap W=W$, which contradicts the unboundedness of $W$.
\end{proof}
Over the set $\mathcal{W}$ we have to consider the moduli spaces of $(i\oplus 
J)$-holomorphic disks, which are constant in the $\mathbb{R}^2$-factor.
A computation shows the next statement.
\begin{lemma}\label{lem:automatic_transversality}
Let $c:\Sigma \rightarrow \mathbb{C}: z \mapsto p$ be a constant map, where $\Sigma$ is 
either the unit disk $D$ or the sphere $S^2$. 
Then the linearization $D_c$ of the $\overline{\partial}$-operator at the curve $c$ is 
surjective. Suppose $u:\Sigma\rightarrow \mathbb{C}\times M$ is a $i\oplus J$ 
pseudo-holomorphic curve, which is constant in the $\mathbb{C}$ factor. I.e. $u=(p,u')$, 
where $u'$ is a $J$-holomorphic curve in the fiber $M$ and $p\in \mathbb{C}$.
Then $D_u=D_{(p,u')}$ is onto if and only if $D_{u'}$ is onto, and furthermore
$$ind(D_u)=ind(D_{u'})+1.$$
\end{lemma}

\subsection{The boundary of $E$}\label{sec:boundaryE}
By Remark~\ref{rmk:cobordism_in_tame_Lefschetz} there exists a compact set $K\subset 
\mathbb{C}$ such that $V|_{K^c}$ is cylindrical and moreover such that all critical 
values of $\pi:E\to \mathbb{C}$ are contained inside $K$.
Let $\mathcal{D}\subset \mathbb{C}$ be a disk such that $K\subset\mathcal{D}$.

\begin{remark}
 We may assume that $\mathcal{D}$ is large enough such that there exists an arc $\alpha\subset \partial \mathcal{D}$ with $\pi(V)\cap \partial \mathcal{D}\subset \alpha$ and $E|_{\pi^{-1}(\alpha)}\to \alpha$ is a trivial fibration. This is the case if $\alpha$ lies inside the set $\mathcal{W}$.
\end{remark}

We restrict $E$ to the disk $\mathcal{D}$ and by abuse of notation continue to denote it by $E$.

\begin{definition}\label{def:horizontal_and_vertical_boundary}
We define the vertical part of the boundary of $E$ by 
$\partial^vE:=\pi^{-1}(\partial \mathcal{D})$. If the fibers of $E$ are compact, 
$\partial^vE$ is closed.
If the generic fiber $(M,\omega)$ is compact without boundary, $E$ has only vertical 
boundary components.
If the generic fiber $(M,\omega)$ is convex at infinity, we define 
$\partial^hE:=\partial E^0\cap \pi^{-1}(int(\mathcal{D}))$, where $E^0\subset E$ is as in assumption~\ref{assum:T_infty} of section~\ref{sec:LF}.
\end{definition}

We need to specify the right assumptions on the Morse functions on a Lagrangian with cylindrical ends.
One main ingredient for quantum homology is to investigate how the negative gradient 
trajectories break. This gives a description of the boundary of some underlying 
moduli spaces, which is then used to describe the quantum structures.
We therefore assume that the negative gradient of a Morse function on
$V$ is transverse to the boundary of $V$. 

For the module structure of $Q^+H_*(E,\partial^v E)$ on $Q^+H_*(V,\partial V)$ we work with 
Morse functions on $E$, such that their negative gradients are transverse to the 
boundary of $E$. However, since we would like to compute the quantum homology relative to 
the part of the boundary $\partial^vE$, we would like such a Morse function to point 
outwards along $\partial^vE$ and inwards along $\partial^hE$ if the fibers are convex at 
infinity. (Compare with definition~\ref{def:horizontal_and_vertical_boundary}.)

\subsection{Quantum structures for Lagrangians with cylindrical ends}

Let $\pi:E\to \mathbb{C}$ be a Lefschetz fibration that is tame over an open 
set $\mathcal{W}\subset \mathbb{C}$ as in figure~\ref{fig:cylindrical_ends}, endowed with a symplectic structure $\Omega$ and with fiber 
$(M,\omega)$.
Let $\mathcal{J}_{\mathcal{W}}$ denote the space of $\Omega$-compatible almost complex structures with the property
$$J|_{\mathcal{W}}=i\oplus J_M, \quad \forall J\in \mathcal{J}_{\mathcal{W}}.$$
Suppose that $V\subset E$ is a monotone Lagrangian cobordism between families $(L_i^-)_{1\leq i\leq 
r_-}$ and $(L_j^+)_{1\leq j \leq r_+}$. (See Definition~\ref{def:cobordism}.)
Denote $H^D_2(E,V)\subset H_2(E,V)$ the image of the Hurewicz homomorphism $\pi_2(E,V)\to H_2(E,V)$.
Monotonicity of the Lagrangian cobordism means the following.
The homomorphism
$$\Omega:H^D_2(E,V)\to \mathbb{R}: \lambda \mapsto \int_{\lambda} \Omega$$
and the Maslov index
$$\mu:H^D_2(E,V)\to \mathbb{Z}$$
satisfy  
$$\Omega(\lambda)>0 \quad \iff \quad \mu(\lambda)>0, \quad \forall \lambda \in H^D_2(E,V).$$
This is equivalent to the existence of a constant $\tau>0$ such that 
$$\Omega(\lambda)=\tau \mu(\lambda),\quad \forall \lambda\in H^D_2(E,V).$$
The constant $\tau$ is referred to as the monotonicity constant of $V\subset (E,\Omega)$.
We also define the integer
$$N_V:=\min\{\mu(\lambda)>0|\lambda \in H^D_2(E,V)\},$$
called the minimal Maslov number of $V$.
We assume throughout the paper that $N_V\geq2$. 
The Maslov numbers come in multiples of $N_V$, and we fix the notation
$$\overline{\mu}:=\frac{\mu}{N_V};H^D_2(E,V)\to \mathbb{Z}.$$
Let 
$$\Lambda^+:=\mathbb{Z}_2[t] \text{ and }\Lambda:=\mathbb{Z}_2[t,t^{-1}].$$
We grade these rings so that $\deg(t)=-N_V$.

Fix a constant $R\in \mathbb{R}$ such that $V$ is cylindrical outside of $[-R,R]\times 
\mathbb{R}\subset\mathcal{W}^c\subset  \mathbb{C}$. Recall from~\ref{def:cobordism} that
\begin{equation*}
\begin{array}{lll}
 V\cap \pi^{-1}((-\infty,-R]\times\mathbb{R}) &=&\coprod _i((-\infty,-R]\times \{i\})\times L_i^-\\
 V\cap \pi^{-1}([R,\infty)\times \mathbb{R})\quad&=&\coprod _j([R, \infty)\times \{j\})\times L_j^+.
 \end{array}
\end{equation*}
Consider $V_R:=V|_{[-R,R]\times \mathbb{R}}$ as well as $E_R:=E|_{[-R,R]\times \mathbb{R}}$.
For subsets $I_- \subset \{1, \cdots, r_- \}$ and $J_+ \subset \{ 1, \cdots, r_+\}$.
Let $S_{R,I_-,I_+}$ be the union
\begin{equation*}
 S_{R,I_-,J_+}:= (\coprod_{i \in I_-} \{(-R,i)\}\times L_i^-) \cup (\coprod_{j\in J_+} 
\{(R,j)\}\times L_j^+)\subset \partial V_R
\end{equation*}
of the boundary components of $V_R$ corresponding to $I_-$ and $J_+$.
Let $f:V_R\rightarrow \mathbb{R}$ be a Morse function and $\rho$ a 
Riemannian metric on $V$. Denote by $\nabla f$ the gradient vector field of $f$ with respect to $\rho$.
We assume that $-\nabla f$ points outwards along $S_R$ and inwards along $\partial V_R \setminus S_{R,I_-,J_+}$ and in particular it is transverse to $\partial V_R$.

\begin{remark}
 For simplicity, we will mostly omit the subscript $(I_+,J_-)$ and $R$ and write $S$, $V$ and $E$ instead, if there is no ambiguity.
 As a Lagrangian with cylindrical ends $V$ has no boundary, but the corresponding cobordism $V_{R}$ does. 
 We use these two descriptions exchangeably. We will switch between $V$ and the cobordism $V_{R}$, and we write $\partial V$ instead of $\partial V_R$.
\end{remark}

\begin{definition}\label{def:MF_respecting_exit_region}
 We call such a function $f$ a Morse function respecting the exit region $S$.
\end{definition}
Let $\phi_t$ denote the flow of $-\nabla f$ for $-\infty\leq t\leq \infty$.
For a critical point $x\in Crit(f)$ we denote by $W^u_x(f)$ and $W^s_x(f)$ the unstable and stable submanifolds of the flow $\phi_t$.
For the Morse index of $x$ we write $|x|$.

\begin{theorem}\label{thm:quantum_structures_for_cobordism}
Under the assumptions and conventions described above and for a generic choice of the 
triple $(f,\rho,J)$ there exists a chain complex
$$C^+((V,S);f,\rho,J)=(\mathbb{Z}_2\langle Crit(f) \rangle \otimes \Lambda^+,d)$$ 
with the following properties:
\begin{enumerate}[(i)]
 \item The homology of the chain complex is independent of the choices of $(f, \rho, J)$. 
It will be denoted by $Q^+H_*(V,S)$. If $S=\emptyset$ there exists a canonical (degree 
preserving) augmentation $$\epsilon_V:Q^+H_*(V)\to \Lambda^+,$$ which is a 
$\Lambda^+$-module map.
\item For every $i, j, 
k$ there exist $\Lambda^+$-bilinear maps
$$Q^+H_i(V,S)\otimes Q^+H_j(V,S)\to Q^+H_{i+j-(n+1)}(V,S): x\otimes y \mapsto x 
*y$$
$$Q^+H_k(E,\partial^vE)\otimes Q^+H_i(V,S)\to Q^+H_{k+i-(2n+2)}(V,S): a\otimes x 
\mapsto a\star x.$$
The first endows $Q^+H(V,S)$ with the structure of a (possibly non-unital) ring. 
This ring is in general not commutative. It is unital if and only if $S=\partial 
V$ and in this case we denote its unit by $e_{(V,\partial V)} \in Q^+H_{n+1}(V,\partial V)$, and it corresponds to the fundamental class of $V$ under the natural map $QH_{n+1}^+(V,\partial V)\to H_{n+1}(V,\partial V)$. The second map endows $Q^+H(V,S)$ with the structure of a module over the quantum 
homology ring $Q^+H(E,\partial^vE)$. Moreover, when viewing these two structures 
together, the ring $Q^+H(V,S)$ becomes a two-sided algebra over the ring 
$Q^+H(E,\partial^v E)$. If the fibers of $E$ are closed then $\partial^h E=\emptyset$ and $Q^+H(E,\partial E)$ is 
unital with unit $e_{(E,\partial E)} \in Q^+H_{2n+2}(E,\partial E)$.
\item There exists a map
$$i_{(V,S)}^{(E,\partial^vE)}:Q^+H_*(V,S)\to Q^+H_*(E,\partial^v E),$$
which is a $Q^+H_*(E,\partial^vE)$-module morphisms and which extend the inclusion in 
singular homology.
There also exists a map
$$i_{V}^{E}:Q^+H_*(V)\to Q^+H_*(E),$$
which extend the inclusions in singular homology. However, the later is not a $Q^+H_*(E,\partial^vE)$-module morphism. If $S=\emptyset$ this map is determined by the relation
$$<h^*, i_{V}^E(x)>=\epsilon_{V}(h\star x)$$
for $x\in Q^+H(V)$, $h\in Q^+H(E,\partial^v E)$, with $(-)^*$ the Poincar\'e duality and $<-,->$ the 
Kronecker pairing.
\item The differential $d$ respects the degree filtration and all the structures above 
are compatible with the resulting spectral sequences.
\item The homology of the complex:
$$C((V,S);f,\rho,J)=C^+((V,S);f,\rho, j)\otimes_{\Lambda^+} \Lambda$$
is denoted by $QH_*(V,S)$ and all the points above remain true if using $QH(-)$ instead 
of $Q^+H(-)$. (Except for the correspondence of $e_{(V,\partial V)}\in Q^+H_{n+1}(V,\partial V)$ and $[V,\partial V]\in H_{n+1}(V,\partial V)$.) The map $Q^+H(V,S)\to QH(V,S)$ induced in homology by the change of 
coefficients above is canonical. Moreover, there is an isomorphism
$$QH_*(V,S)\to HF_*((V,S),(V,S))$$
which is canonical up to a shift in grading. (See~\cite{BC13} for the definition of Floer homology for cobordisms.)
\item If $V$ is orientable and spin, we may work with $\Lambda_{\mathcal{R}}^+:={\mathcal{R}}[t]$ or
$\Lambda_{\mathcal{R}}:={\mathcal{R}}[t,t^{-1}]$, where ${\mathcal{R}}$ is any commutative unital ring and the analogous statements hold.
\end{enumerate}
\end{theorem}

\begin{theorem}\label{thm:longexact}
Let $S\subset \partial V$ be a union of connected components of $\partial 
V$.
There exists a long exact sequence
\begin{equation*}
\xymatrix{
 \dots \ar[r]^{\delta} & Q^+H_*(S) \ar[r]^{i_*} & Q^+H_*(V)\ar[r]^{j_*} & 
Q^+H_*(V,S)\ar[r]^{\delta} & Q^+H_{*-1}(S)\ar[r]^{i_*} & \dots,
}
\end{equation*}
which has the following properties:
\begin{enumerate}
\item Suppose that $S=\partial V$. Let $e_{(V,\partial V)}$ denote the unit of 
$Q^+H_*(V,\partial V)$. Furthermore, let $e_{L_i^-}$ and $e_{L_j^+}$ denote the units of $Q^+H_*(L_i^-)$ and $Q^+H_*(L_j^+)$ respectively.
Then 
$$\delta(e_{(V,\partial V)})= \oplus_{i} e_{L_i^-}\oplus_{j} e_{L_j^+}.$$
\item  The map $\delta$ is multiplicative with respect to the quantum product $*$, 
namely 
\begin{equation}
 \delta(x*y)=\delta(x)*\delta(y) \ \ \forall x,y \in Q^+H_*(V,S).
\end{equation}
\item  The product $*$ on $Q^+H_*(V)$ is trivial on the image of the map $i_*$.
In other words, for any two elements $x$ and $y$ in $Q^+H_*(S)$ we have that 
$$i_*(x)*i_*(y)=0.$$
\item The map $j_*$ is multiplicative with respect to the quantum product, namely
\begin{equation*}
 j_*(x*y)=j_*(x)*j_*(y)\ \ \forall x,y \in Q^+H_*(V).
\end{equation*}
Moreover, $j_*(Q^+H_*(V))\subset Q^+H_*(V,\partial V))$ is a two-sided ideal.
\item If the Lefschetz fibration is trivial (i.e. $(E,\Omega)=(\mathbb{C} \times M,\omega_{std}\oplus \omega)$) and the fibers are closed, there exists a 
ring isomorphism $\Phi: Q^+H_*(M) \rightarrow Q^+H_{*+2}(E,\partial E)$ and the following identities hold:
\begin{enumerate}
 \item[(i)] $i_*(a \star x)= \Phi(a) \star i_*(x)$, $\forall x\in Q^+H_*(\partial V)$ and $\forall a\in Q^+H_*(M)$.
 \item[(ii)] $j_*(a\star x)=a\star j_*(x)$,  $\forall x\in Q^+H_*(V)$ and $\forall a\in 
Q^+H_*(E,\partial E)$.
 \item[(iii)] $\delta(a\star x)=\Phi^{-1}(a)\star \delta(x)$, $\forall x\in 
Q^+H_*(V,\partial V)$ and $\forall a\in Q^+H_*(E,\partial E)$.
\end{enumerate}
In other words, the maps in the long exact sequence are module maps over the ambient quantum homology rings 
$Q^+H_*(E,\partial E)$ and $Q^+H_*(M)$.
\item All statements remain true if we use $QH_*(-)$ instead of $Q^+H_*(-)$.
\item If $V$ is orientable and spin, we may work with $\Lambda_{\mathcal{R}}^+:={\mathcal{R}}[t]$ or
$\Lambda_{\mathcal{R}}:={\mathcal{R}}[t,t^{-1}]$, where ${\mathcal{R}}$ is any commutative unital ring and the analogous statements to $1.-6.$ above hold. 
\end{enumerate}
\end{theorem}

We will now go over the construction of $Q^+H_*(-)$ and the products and module structures and prove Theorems~\ref{thm:quantum_structures_for_cobordism} and~\ref{thm:longexact}.

\subsubsection{Pearly trajectories}\label{sec:pearly_trajectories}

Denote by $D$ the unit disk in $\mathbb{C}$. Let $f$ be a Morse function respecting the exit region $S$.
To define the pearl complex for Lagrangian cobordism we introduce the following moduli 
spaces.
\begin{definition}\label{def:P_prl}
Let $x$ and $y$ be two critical points of $f$ in $V$.
Let $l \geq 1$ and let $\lambda$ be a non-zero class in $H^D_2(E,V)\subset 
H_2(E,V)$.
Consider the space of all sequences $(u_1, \dots, u_l)$ that satisfy:
\begin{itemize}
  \item[(i)] For every $i$, $u_i:(D, \partial D) \rightarrow (E,V)$is a 
non-constant $J$-holomorphic disk.
  \item[(ii)] $(u_1(-1))\in W^u_x$.
  \item[(iii)] $(u_l(1))\in W^s_y$.
  \item[(iv)] For every $i$ there exists $0 < t_i < \infty$ such that 
$\phi_{t_i}(u_i(1))=u_{i+1}(-1)$.
  \item[(v)] The sum of the classes of the $u_i$ equals $\lambda$, i.e. 
$\sum\limits_{i=1}^{l} [u_i]=\lambda \in H^D_2(E,V)$.
\end{itemize}
Let $\mathcal{P}_{prl}(x,y,\lambda;f,\rho,J)$ be the space of all such 
sequences, modulo the following equivalence relation.
Two elements $(u_1, \dots, u_l)$ and $(u'_1, \dots, u'_k)$ are equivalent if $l=k$ and for 
every $i$ there exists an automorphism $\sigma_i\in Aut(D)$ with $\sigma_i(-1)=-1$, 
$\sigma_i(1)=1$ and $u_i'=u_i\circ \sigma$.
We call elements in $\mathcal{P}_{prl}(x,y,\lambda;f,\rho,J)$ pearly 
trajectories connecting $x$ and $y$ of class $\lambda$.
If $\lambda=0$ set $\mathcal{P}_{prl}(x,y,0;f,\rho,J)$ to be the space of 
unparametrized trajectories of $\phi$ connecting $x$ and $y$. These are just the usual Morse trajectories.
We sometimes denote $\mathcal{D}:=(f,\rho,J)$.
\end{definition}
The virtual dimension of $\mathcal{P}_{prl}(x,y,\lambda;f,\rho,J)$ is 
$$\delta_{prl}(x,y,\lambda):= |x|-|y|+\mu(\lambda)-1.$$
If $\lambda=0$ then $\mathcal{P}_{prl}(x,y,0,f,\rho,J)$ are the usual Morse trajectories and $\delta_{prl}(x,y,0,f,\rho,J)=|x|-|y|$.
Lemma~\ref{lem:curve_openmapping} and~\ref{lem:automatic_transversality} together with 
the methods in~\cite{BC07} ensure that if $\delta_{prl}(x,y,\lambda)\leq1$, the 
space $\mathcal{P}_{prl}(x,y,\lambda;f,\rho,J)$ is a smooth manifold of 
the dimension equal to its virtual dimension for a generic triple 
$(f,\rho,J)$. See also chapter~\ref{sec:main_floer_techniques}.

\subsubsection{The pearl complex}
Put
$$C^+_i((V,S);f,\rho,J):=\mathbb{Z}_2(\langle Crit(f) 
\rangle \otimes \Lambda^+)_i.$$
Abbreviate $\overline{\mu}:=\frac{\mu}{N_V}$ and define the differential for $x\in Crit(f)$ by
\begin{equation}\label{def:differential}
 d(x):=\smashoperator{\sum\limits_{\begin{subarray}{c} y,\lambda \\  
\delta_{prl}(x,y,\lambda)=0 \end{subarray}}}
\#_{\mathbb{Z}_2} \mathcal{P}_{prl} 
(x,y,\lambda;\mathcal{D}) t^{\overline{\mu}(\lambda)}.
\end{equation}

The homology of this chain complex is independent of the choice of 
the data $\mathcal{D}=(f,\rho,J)$.

\begin{proposition}\label{prop:LQHforLC}
If $\mathcal{D}_S:=(f,\rho,J)$ is generic, then the pearl complex
\begin{equation}
 C^+((V,S);\mathcal{D}_S):=(\mathbb{Z}_2\langle Crit(f) \rangle \otimes 
\Lambda^+,d)
\end{equation}
is a well-defined chain complex and its homology is independent of the choices 
of $\mathcal{D}_S$.
The homology is called the quantum homology of $(V,S)$ and is denoted by $Q^+H_*(V,S)$.
\end{proposition}
The proof of this proposition is postponed to section~\ref{sec:proof_chain_complex}. One essential 
ingredient is the following lemma.
\begin{lemma}\label{lem:traj_away_from_boundary}
 Let $\mathbf{u}\in \mathcal{P}_{prl}(x,y,\lambda;f,\rho,J)$ be a pearly 
trajectory connecting two critical points of $f$.
Then $\mathbf{u}$ does not reach the boundary $\partial V$.
\end{lemma}

\begin{remark}
 If we regard $V$ as a Lagrangian with cylindrical ends, this means that the pearly trajectories do not escape towards infinity along the cylindrical ends. 
\end{remark}

\begin{proof}[Proof of Lemma~\ref{lem:traj_away_from_boundary}]
By definition $-\nabla f$ is transverse to the boundary $\partial 
V$. In particular the critical points $x$ and $y$ lie in the interior of 
$V$. For a fixed boundary component $A$ of $\partial V$ the negative gradient $-\nabla f$ 
points either always out along $A$ or always in along $A$.
Moreover, by Lemma~\ref{lem:curve_openmapping} any $J$-holomorphic disk $u$ near the 
boundary must be constant under the projection $\pi$, hence it must lie in one fiber.
Therefore, no pearly trajectory between two critical points can reach the boundary.
\end{proof}

\subsection{Main Morse and Floer theoretic techniques}\label{sec:main_floer_techniques}

This section follows~\cite{BC09a} and the same issues are diskussed in depth in~\cite{BC07}.
The scheme of most of the proofs of 
theorem~\ref{thm:quantum_structures_closed_Lagrangians} and
theorem~\ref{thm:quantum_structures_for_cobordism} follow standard arguments form Morse 
and Floer theory with the following main ingredients:
\begin{enumerate}[(i)]
 \item transversality
 \item compactness and
 \item gluing
\end{enumerate}
of some specific moduli space. (For example $\mathcal{P}_{prl}$ or other moduli spaces 
modeled on planar trees.)
Transversality shows that for generic choices of some of the defining data, these moduli 
spaces are smooth manifolds if their virtual dimension is at most $1$ and moreover, their 
dimension is equal to the virtual dimension. These manifolds are typically not 
compact. Compactness and gluing are used to describe their compactification, which 
still has the structure of a $1$-dimensional manifold and whose 
boundary can then be expressed in terms of compact $0$-dimensional manifold of the same 
type, i.e. a finite even number of points in these moduli spaces. 
The boundary descriptions of these manifolds yield expressions, which can be interpreted as the chain level quantum structures. 
For example the boundary description of the compactification of the $1$-dimensional $\mathcal{P}_{prl}$ proves that $d$ is a 
differential.

A very elaborate consideration of these methods for the case of a closed, monotone 
Lagrangian is given in~\cite{BC07}. We give a sketch of them below.

\subsubsection{Transversality}\label{sec:transversality}

We defined $\mathcal{P}_{prl}$ in sections~\ref{sec:pearly_trajectories} and we are 
going to introduce the moduli spaces $\mathcal{P}_{prod}$, $\mathcal{P}_{mod}$ and 
$\mathcal{P}_{inc}$, which are moduli spaces modeled over planar trees.
Let $\mathcal{P}$ denote any type of these moduli spaces and $\mathcal{D}$ the Morse data 
that is needed to define them. Let $\lambda\in H_2(E,V)$ and denote by $I$ the tuple 
consisting of the critical points and the class $\lambda$. We have seen or will see that:
\begin{enumerate}
 \item If $\mathcal{P}=\mathcal{P}_{prl}$, then $\mathcal{D}=(f,\rho)$, 
$I=(x,y,\lambda)$, where $f:V\to \mathbb{R}$ is a Morse function respecting the exit 
region $S$, $\rho$ is a Riemannian metric on $V$ and $x,y\in Crif(f)$.
 \item If $\mathcal{P}=\mathcal{P}_{prod}$, then $\mathcal{D}=(f,f',f'',\rho)$, 
$I=(x,y,z;\lambda)$, where $f,f'$ and $f''$ are Morse functions on $V$ respecting the 
exit region $S$, $x$, $y$ and $z$ are critical points of $f,f'$ and $f''$ respectively and
$\rho$ is a Riemannian metric on $L$.
 \item If $\mathcal{P}=\mathcal{P}_{mod}$, then 
$\mathcal{D}=(h,\rho_E,f,\rho_V)$, $I=(a,x,y;\lambda)$, where $h$ is a Morse function on 
$E$ respecting the exit region $\partial^vE$. (sometimes we also require the negative gradient to point inwards along the whole boundary $\partial E$. See section~\ref{sec:inclusion}.) Furthermore, $f$ is a Morse function on 
$V$ respecting the exit region $S$, $\rho_V$ is a Riemannian metric on $V$, $\rho_E$ 
is a Riemannian metric on $E$, $a\in Crit(h)$ and $x,y\in Crit(f)$.
 \item If $\mathcal{P}=\mathcal{P}_{inc}$, then $\mathcal{D}=(h,\rho_E, f, \rho_V)$, 
$I=(x,a;\lambda)$, where $h$ is a Morse function on $E$ respecting the exit region 
$\partial^vE$, $f$ is a Morse function on 
$V$ respecting the exit region $S$, $\rho_V$ is a Riemannian metric on $V$, $\rho_E$ 
is a Riemannian metric on $E$, $a\in Crit(h)$ and $x\in Crit(f)$.
\end{enumerate}

Let $\delta_{\mathcal{P}}(I)$ denote the virtual dimension of the moduli space 
$\mathcal{P}(I,\mathcal{D},J)$. We need some 
genericity assumptions on the Morse data $\mathcal{D}$ to establish the transversality results.

\begin{customassum}{G}[Genericity of the Morse data. Compare~\cite{BC07} and~\cite{BC09a})]\label{assum:genericity}
$\text{   }$
\begin{enumerate}
 \item When $\mathcal{P}=\mathcal{P}_{prl}$ assume that $(f,\rho)$ is Morse-Smale.
 \item When $\mathcal{P}=\mathcal{P}_{prod}$ assume that $(f,\rho)$, $(f',\rho)$ and $(f'',\rho)$ are Morse-Smale and for all critical points $p,p'$ 
and $p''$ of $f,f'$ and $f''$ respectively, the triple intersection
$$W^u_p;(f)\cap W^u_{p'}(f')\cap W^s_{p''}(f'')$$ is transverse.
 \item When $\mathcal{P}=\mathcal{P}_{mod}$ or $\mathcal{P}=\mathcal{P}_{inc}$ assume 
that both $(f,\rho_V)$ and $(h,\rho_E)$ are Morse-Smale and
 \begin{enumerate}[(i)]
    \item $h$ is proper, bounded from below and $|Crit(h)|< \infty$. (Since $E$ may be non-compact.)
    \item $Crit(h)\cap V=\emptyset$.
    \item For all $a\in Crit(h)$ the stable and unstable manifolds $W^s_a(h)$ and 
$W^u_a(h)$ are transverse to $V$.
    \item For all $a\in Crit(h)$ and all $x,y\in Crit(f)$, $W^u_a(h)$ is transverse to 
$W^u_x(f)$ and $W^s_y(f)$. 
\end{enumerate}
\end{enumerate}
\end{customassum}

From Morse theory it is well-known that if $\mathcal{D}$ is generic in the usual sense 
then also Assumption~\ref{assum:genericity} holds.
The indispensable transversality results for the proof of the quantum structures are 
stated below.

\begin{proposition}[Transversality]\label{prop:transversality_general}
 Let $\mathcal{P}$ and $\mathcal{D}$ be of one type listed above and assume that 
$\mathcal{D}$ fulfills~\ref{assum:genericity}. 
If $\delta_{P}(I)=1$, $N_V=2$ then assume additionally that $\mathcal{P}\neq 
\mathcal{P}_{mod}$. Then there exists a second category subset $\mathcal{J}_{reg}\subset 
\mathcal{J}$ such that for all $J\in \mathcal{J}_{reg}$ the following holds.
If $\delta_{\mathcal{P}}(I)\leq 1$ then $\mathcal{P}(I,\mathcal{D},J)$ is either empty or 
a smooth manifold of dimension $\delta_{\mathcal{P}}$. If $\delta_{\mathcal{P}}=0$ then 
this $0$-dimensional manifold is compact, i.e. consists of a finite number of points. 
\end{proposition}

\begin{remark}[See also~\cite{BC07} and~\cite{BC09a}]
 \begin{enumerate}
  \item The case $\mathcal{P}=\mathcal{P}_{mod}$, $\delta_{P}=1$, $N_V=2$ needs a special 
consideration. In fact, in this case one needs to work with Hamiltonian perturbations. 
Roughly speaking, this is due to the fact that elements in $\mathcal{P}_{mod}$ involve 
holomorphic disks with interior marked points and reduction to simple disks (as it is the 
subject the next paragraph) might increase the degree of freedom of the marked point by 
one.
 \item For the proof of associativity, linearity and the other properties of the quantum 
structures, we have to introduce new types of moduli spaces. Their description is however 
very similar to the moduli spaces considered above and transversality statements are 
comparable.
 \end{enumerate}
\end{remark}

\paragraph{Reduction to simple disks}\label{sec:reduction_to_simple_disks}

To prove that the moduli spaces involving holomophic curves are smooth 
manifolds, we need to verify that certain evaluation maps of the moduli space of 
holomorphic curves are transverse to certain submanifolds in the target manifold. 
By standard arguments (\cite{McS12}) this is possible if we restrict to simple (or 
somewhere injective) curves. It is well-known that for a generic almost complex 
structure $J$ the space of simple $J$-holomorphic curves in a fixed class is a smooth 
manifold with dimension equal to the virtual dimension. Simple curves can also be 
arranged to be transverse to any given submanifold.
It is therefore sufficient to show that (at least for dimension $\leq 1$)
$$\mathcal{P}=\mathcal{P}^*,$$
where by $\mathcal{P}^*$ we denote the corresponding moduli space that involves only 
simple curves.
Reducing to simple $J$-holomorphic disks is fairly easy in the monotone case.

In the case of closed curves, i.e. holomorphic maps $u:\Sigma\to M$, where $\Sigma$ is a 
closed Riemann surface, reduction to simple curves is done as follows. Suppose the 
curve $u$ is not simple, then it factors through a simple curve $u':\Sigma' \to M$ in the 
sense that $u=u'\circ \phi$, where $\phi:\Sigma\to \Sigma$ is branched covering and 
$u':\Sigma'\to M$ is a simple curve. Now we can simply replace $u$ by $u'$. 

Now suppose $u:(D,\partial D)\to (M,L)$ is a holomorphic disk and assume it is not 
simple. Unlike for closed curves, $u$ does not necessarily factor through a simple curve. 
However,~\cite{Laz00}, ~\cite{Laz11} and ~\cite{KY00} showed that one can decompose $D$ 
into subdomains $\mathcal{D}_i$, such that each $u|_{\mathcal{D}_i}$ factors through 
simple disks $v_i:(D,\partial D)\to (M,L)$ via a branched covering $\mathcal{D}_i\to D$ 
of some degree $m_i$. Moreover, $[u]=\sum_{i=1,\ldots, r} m_i [v_i]$. 

Now the procedure is roughly speaking the following. Suppose we are given an element of a 
moduli space $\mathcal{P}(\mathcal{D},I,J)$ in the homology class $\lambda$, which has 
dimension at most $1$. Suppose furthermore that this element involves a non-simple disk 
$u$. Suppose that $u$ has only two marked points on the boundary.
By the diskussion above we can find simple disks $v_{i_1}, \cdots, v_{i_r}$ and replace 
$u$ by these disks. We end up with an element in $\mathcal{P}^*(\mathcal{D},I',J)$ and in 
a new homology class $\lambda'$. By monotonicity the dimension of the moduli space must 
have decreased by at least two, since 
$$\mu(\lambda')\leq \mu(\lambda)-N_L\leq \mu(\lambda)-2.$$
But $\delta_{\mathcal{P}}((\mathcal{D},I,J))\leq 1$ and hence 
$\delta_{\mathcal{P}}((\mathcal{D},I',J))\leq -1$. It follows by transversality that 
$\mathcal{P}(\mathcal{D},I,J)=\emptyset$, which contradicts our assumption. Thus, all 
elements in $\mathcal{P}(\mathcal{D},I,J)$ involve only simple disks. Similar arguments 
work for simple spheres and for disks that have more than two marked points and possibly 
also interior marked points. The arguments need to be carried out more carefully, 
considering all combinatorial possibilities on how the marked points are distributed 
among the disks $v_i$'s. 

The last complication that arises on the way to proving transversality is the following. 
We must also show that the disks in an object of $\mathcal{P}$ of dimension less or 
equal to $1$ are absolutely distinct. 
As above if the disks are not absolutely distinct we can omit some suitable choice 
of disks, and we obtain a new element. Using monotonicity again we can show that the 
moduli space containing the new element has negative dimension and thus is empty.

\subsubsection{Compactness and gluing}

Assume the almost complex structure and the Morse functions are as in 
section~\ref{sec:tools} and~\ref{sec:pearly_trajectories}. Then compactification and 
gluing can be carried out.

We diskuss the case that $\mathcal{P}=\mathcal{P}_{prl}$. (Compare~\cite{BC07} and~\cite{BC09a}.)
The compactness and gluing for 
the cases of other moduli spaces are very similar. 
Let $\mathbf{u_k}:=(u_{1,k},\ldots, u_{l,k})$ be a sequence in 
$\mathcal{P}_{prl}(x,y;\lambda;f;\rho,J)$ that does not have a convergent subsequence in 
that space. 
We may restrict to a subsequence of $\mathbf{u_k}$ and describe its limit.
There are five possibilities:
\begin{enumerate}
 \item[(P1)] One of the gradient trajectories of $f$ breaks at a new critical point $z$, 
such that $\mathbf{u_k}$ converges to a concatenation of two elements 
$u'\in\mathcal{P}_{prl}(x,z;\lambda';f;\rho,J)$ and $u''\in 
\mathcal{P}_{prl}(z,y;\lambda'';f;\rho,J)$, where $\lambda'+\lambda''=\lambda$.
  \item[(P2)] One of the gradient trajectories of $f$ shrinks to a point $p$. Assume 
that this trajectory is the one located between $u_i$ and $u_{i+1}$. In other words, 
$\mathbf{u}_k$ converges to $(\mathbf{u}',\mathbf{u}'')$, where 
$\mathbf{u}'=(u_1',\cdots,u_{l'}')\in\mathcal{P}_{prl}(x,p;\lambda';f;\rho,J)$ 
and $\mathbf{u}''=(u_1'',\cdots,u_{l''}'')\in\mathcal{P}_{prl}(p,y;\lambda'';f;\rho,J)$ 
with $\lambda'+\lambda''=\lambda$, $l,l''\geq1$, $l'+l''=l$ and 
$p=u_{l'}'(1)=u_{1}''(-1)$. Notice however that $p$ is not a critical point in general.
Let us denote this space by $\mathcal{P}_{prl,(P-2)}(x,y;(\lambda',\lambda'');I,J)$.
A computation of the virtual dimensions shows that 
$$\delta_{prl,(P-2)}(x,y;\lambda',\lambda'';I;J)=\delta_{prl}(x,y;\lambda;I;J)-1.$$
  \item[(P3)]
A $J$-holomorphic disk bubbles off at a boundary point of some holomorphic disks.
In other words $\mathbf{u_k}$ converges to a sequence, where one $u_{i,k}$ converges to a 
reducible $J$-holomorphic curve consisting of two $J$-holomophic disks $u_{i,\infty}$ 
and $u_{i,\infty}'$ attached to each other at a boundary point.
We distinguish two types of bubbling:
\begin{enumerate}
  \item[(P3a)] The two disks both have two marked points and the gradient trajectory 
arrives at the first disk at the point $-1$ leaves the second disk from the point $1$.
The limit is a pair of sequences $(\mathbf{u}_{k,\infty}',\mathbf{u}_{k,\infty}'')$ 
of total length $l+1$. 
The description of this space is formally the same as the one in $(P2)$. But we 
distinguish them, since they arise as the limit of different sequences. Let us denote 
this space by $\mathcal{P}_{prl,(P3a)}(x,y;(\lambda',\lambda'');I,J)$.
  \item[(P3b)] Only one disk $u_{i,\infty}'$ has two marked points where the 
gradient trajectories arrive and leave. The second disk $u _{i,\infty}''$ is 
attached to the first one at a boundary point not equal to $\pm1$. Hence, in this case we 
can remove $u_{i,\infty}''$ from the limit to obtain a new pearly trajectory $v$ 
connecting $x$ and $y$. However, again by the monotonicity argument the virtual dimension 
of the moduli space containing $v$ decreases by at least two:
$$\delta_{prl}(x,y;\lambda-[u_{i,\infty}''])\leq\delta_{prl}(x,y;\lambda)-2.$$
\end{enumerate}
  \item[(P4)] Bubbling of a $J$-holomorphic sphere occurs at one of the disks $u_{i,k}$ 
either at an interior point or a boundary point. Let $\tilde{\lambda}$ denote the class 
of the sphere. By monotonicity:
$$\delta_{prl}(x,y;\lambda-\tilde{\lambda})=\delta_{prl}(x,y;\lambda)-2.$$
  \item[(P5)] Any combination of these possibilities and repetitions can occur.
\end{enumerate}

We outline the main steps for the proof of Proposition~\ref{prop:transversality_general} 
for the moduli spaces $\mathcal{P}_{prl}$.

\begin{lemma}
 If $\delta_{prl}(x,y;A;I;J)=0$ then $\mathcal{P}_{prl}(x,y;\lambda;I;J)$ is compact for 
generic $J$. 
\end{lemma}

\begin{proof}[Sketch of the Proof]
 A transversality argument similar to the one in section~\ref{sec:transversality} shows 
that for $\delta_{prl}(x,y;\lambda;I;J)\leq 1$ the moduli spaces describing case $(P2)$ 
or $(P3a)$ are smooth manifold of dimension equal to their virtual dimension. 
Now we notice that none of the possibilities $(P1)$ to $(P5)$ can occur. Indeed, each case
involves a smooth manifold, whose dimension is less than $\delta_{prl}(x,y,;\lambda;I;J)$ and hence negative.
This manifold must therefore be empty. 
\end{proof}

\begin{lemma}
 If $\delta_{prl}(x,y;A;I;J)=1$ then $\mathcal{P}_{prl}(x,y;\lambda;I;J)$ can be 
compactified in to a manifold with boundary:
\begin{equation}\label{eq:boundary_1dim_prl_moduli}
 \begin{array}{ccc}
  \partial \overline{\mathcal{P}_{prl}}(x,y;\lambda;I;J) &=& \bigcup\limits_{
	\begin{subarray}{c} 
	  z\in Crit(f), \lambda'+\lambda''=\lambda\\
	  \delta_{prl}(x,z;\lambda')=0\\
	  \delta_{prl}(z,y;\lambda''=0 
    \end{subarray}} 
\mathcal{P}_{prl}(x,z;\lambda';I,J)\times \mathcal{P}_{prl}(z,y;\lambda'';I;J) \\ 
&& \coprod\bigcup\limits_{\lambda'+\lambda''=\lambda} 
\mathcal{P}_{prl,(P2)}(x,y;\lambda',\lambda'';I;J) \\
 & & \quad \coprod \bigcup\limits_{\lambda'+\lambda''=\lambda}\mathcal{P}_{prl,(P3a)}(x,y;\lambda',
\lambda'';I;J)
 \end{array}
\end{equation}
\end{lemma}

\begin{proof}[Sketch of the Proof]
 Possibilities $(P3b)$, $(P4)$ and $(P5)$ cannot occur, since they give rise to 
elements in moduli spaces of negative dimension by transversality. Hence, such elements 
cannot occur. This proves one inclusion $\partial \overline{\mathcal{P}_{prl}}\supset 
(\text{RHS of ~\ref{eq:boundary_1dim_prl_moduli}})$.

The other inclusion is a consequence of the gluing procedure. Details can be found 
in~\cite{BC07}. Another important feature is the surjectivity of the gluing map, which 
ensures that each element on the RHS of~\ref{eq:boundary_1dim_prl_moduli} corresponds to 
a unique end of $\overline{\mathcal{P}_{prl}}$.
A detailed elaboration of the gluing procedure for pseudo-holomorphic curves is to be found in~\cite{McS12}. 
Gluing of pseudo-holomorphic disks is developed in~\cite{FOOO09} and further elaborated in~\cite{BC07}.
\end{proof}

\subsection{Filtration of the pearl complex and the spectral sequence}

Let $CM_*:=(CM((V,S);f),\partial_0)$ denote the Morse complex of $(V,S)$ associated to 
$f$. 
The pearl complex $C_i:=C_i((V,S);\mathcal{D}_S)$ has an increasing degree filtration:
$$\mathcal{F}_pC_i=\bigoplus_{j\geq-p}C_{i+jN_V}t^j.$$
This is a bounded and increasing filtration and due to~\ref{thm:spectral_seq_from_filtered_complex} it 
gives rise to a spectral sequence, $\{E^r_{p,q},d_r\}_{r\geq0}$ that converges to 
$QH_*(V,S)=H_*(C_i)$. 
More precisely, we have 
\begin{enumerate}
 \item $E^0_{p,q}=CM_{p+q-pN_V}t^{-p},d_0=\partial_0$.
 \item $E^1_{p,q}=H_{p+q-pN_V}((V,S);\mathbb{Z}_2)t^{-p}$.
 \item The sequence collapses after $\leq [\frac{n+2}{N_V}]+1$ steps.
\end{enumerate}

In particular, the filtration on $C_*$ induces a filtration on the quantum homology
$QH_*(V,S)=H_*(C_*(V,S;\mathcal{D}_S),d)$, where 
$$\mathcal{F}_pH_i=\bigoplus_{j\geq -p}H_{i+jN_V}((V,S);\mathbb{Z}_2)t^{-p}$$
and there is an isomorphism

$$QH_i\cong \bigoplus_{p+q=i}E^{\infty}_{p,q}.$$

\begin{remark}
Notice that the spectral sequence above is multiplicative. To see this, let $f'$ be a 
small perturbation of the Morse function $f$, such that $f'$ is still a Morse function 
respecting the exit region $S$. 
In section~\ref{sec:product} we will define the quantum product $*$. It is easy to see 
that this respects the filtrations in the sense that
$$*:\mathcal{F}_{p_1}C_{i_1}(f)\otimes 
\mathcal{F}_{p_2}C_{i_2}(f')\to 
\mathcal{F}_{p_1+p_2}C_{i_1+i_2-n}(f).$$
If the perturbation $f'$ is close enough to $f$ then the chain complexes $C((V,S),(f,\rho,J))$ and 
$C((V,S),(f',\rho,J))$ are isomorphic by a base preserving isomorphism (see 
section~\ref{sec:invariance} below).
Thus, we can identify the two resulting spectral sequences and we denote them both 
by $\{E^r_{p,q},d_r\}$. Hence, this induces a product $*_r$ on $E^r_{p,q}$, which shows 
that the spectral sequence is multiplicative.
Let $*_{\infty}$ denotes the product induced on $E^{\infty}_{p,q}$, then this also 
defines a product on $QH_i(V,S)\cong \bigoplus_{p+q=i}E^{\infty}_{p,q}$. 
Notice however that $*_{\infty}$ and the quantum product $*$ (as defined 
in~\ref{sec:product} below) are in general not isomorphic.
\end{remark}

\subsection{Proof of Proposition~\ref{prop:LQHforLC}}\label{sec:proof_chain_complex}

\subsubsection{$d\circ d=0$}\label{sec:d^2=0}
By Lemma~\ref{lem:traj_away_from_boundary} and Lemma~\ref{lem:automatic_transversality} 
we can use the methods of~\cite{BC07} to show that the moduli spaces of 
pearly trajectories of virtual dimension at most $1$ form smooth manifolds and that they 
are compact if their dimension is zero.

Let $\mathcal{P}_{prl}$ be a one dimensional moduli space of pearly trajectories on a 
Lagrangian cobordism $V\subset E$. By Lemma~\ref{lem:traj_away_from_boundary_general} its compactification has the form described in~(\ref{eq:boundary_1dim_prl_moduli}).
Let $x$ and $y$ be two critical points of the Morse function $f$.
From the definition of the differential we compute:
\begin{equation*}
\begin{split}
 \langle d\circ d(x),y\rangle = \sum\limits_{
  \begin{subarray}{c} z\in 
    Crit(f),\lambda',\lambda''\\ 
    \delta_{prl}(x,z,\lambda')=0\\
    \delta_{prl}(z,y,\lambda'')=0
  \end{subarray}}
\#_{\Z_2} \mathcal{P}_{prl}(x,z;\lambda';I;J)\#_{\Z_2} 
\mathcal{P}_{prl}(z,y;\lambda'';I;J)t^{\overline{\mu}(\lambda)+\overline{\mu}(\lambda')}\\
= \left( \sum\limits_{
  \begin{subarray}{c}
    z\in Crit(f),\lambda',\lambda''\\
    \mu(\lambda'+\lambda'')=2-|x|+|y| \\
    \delta_{prl}(x,z,\lambda')=0
  \end{subarray}}
\#_{\Z_2} \mathcal{P}_{prl}(x,z;\lambda';I;J)\#_{\Z_2}\mathcal{P}_{prl}(z,y;\lambda'';I;J) \right) 
t^{(2-|x|+|y|)/N_V},
\end{split}
\end{equation*}
where the last equality holds, since $\delta_{prl}(x,z;\lambda)=\delta(z,y;\lambda')=0$ 
if and only if $\mu(\lambda)+\mu(\lambda')=2-|x|+|y|$ and $\delta_{prl}(x,z,\lambda)=0$.
In particular $\mu(\lambda)+\mu(\lambda')$ is constant.

Now we use the description~(\ref{eq:boundary_1dim_prl_moduli}) to show that the above sum 
vanishes. The boundary of a compact $1$-dimensional manifold consists of an 
even number of points. (Moreover, if this manifold is oriented it induces orientations on 
its boundary and the boundary points cancel each other when counted with signs.)
Hence:
\begin{equation*}
 \begin{array}{ccc}
  0 &=& \# \mathcal{P}_{prl}(x,y,\lambda;I;J)\\
 &=& \sum\limits_{
\begin{subarray}{c}
    z\in Crit(f),\lambda'+\lambda''=\lambda\\
    \delta_{prl}(x,z;\lambda')=0\\
    \delta_{prl}(z,y;\lambda'')=0 
\end{subarray}}
\#\mathcal{P}_{prl}(x,z;\lambda';I;J)\#\mathcal{P}_{prl}(z,y;\lambda'';I;J)\\
  &+& \sum\limits_{
\begin{subarray}{c}
\lambda'+\lambda''=\lambda
\end{subarray}}
\# \mathcal{P}_{prl,(P2)}(x,y;\lambda';\lambda'';I;J)
+\#\mathcal{P}_{prl,(P3a)}(x,y;\lambda';\lambda'';I;J)\\
&=& \sum\limits_{
\begin{subarray}{c}
    z\in Crit(f),\lambda'+\lambda''=\lambda\\
    \delta_{prl}(x,z;\lambda')=0\\
    \delta_{prl}(z,y;\lambda'')=0 
\end{subarray}}
\#\mathcal{P}_{prl}(x,z;\lambda';I;J)\#\mathcal{P}_{prl}(z,y;\lambda'';I;J),
 \end{array}
\end{equation*}
where the last equality holds, since we saw that the two spaces 
$\mathcal{P}_{prl,(P2)}(x,y;\lambda';\lambda'';I;J)$ and 
$\mathcal{P}_{prl,(P3a)}(x,y;\lambda';\lambda'';I;J)$ are identical.

\subsubsection{Invariance}\label{sec:invariance}

In this section we define chain morphisms
$$\phi:C^+((V,S),\mathcal{D}_S)\to C^+((V,S), \mathcal{D}_S')$$ between the chain complexes 
associated to two distinct generic data $\mathcal{D}_S=(f,\rho,J)$ and 
$\mathcal{D}_S'=(f', \rho',J')$.
The aim is then to show that these morphisms induce canonical isomorphisms in homology.
I.e. given two such morphisms $\phi,\phi':C^+((V,S),\mathcal{D}_S)\to C^+((V,S), \mathcal{D}_S')$ 
we need to prove that they are chain homotopic.
By Theorem~\ref{thm:moduli_as_in_closed_setting} the proof follows~\cite{BC07} and we only give a sketch.
The main idea is to use a Morse cobordism $(F,G)$ relating $(f,\rho)$ and 
$(f',\rho')$ and a homotopy of almost complex structures between $J$ and $J'$. A careful 
diskussion of the construction and properties of Morse cobordisms is given in~\cite{CR03}.

Fix a smooth family of almost complex structures $\overline{J}_t$, $t\in[0,1]$ connecting 
$J$ and $J'$. 
For every $t$, we may choose $J_t=i\oplus J_M$ outside of $\pi^{-1}(U)$, where $J_M$ is 
an almost complex structure on the fiber $(M,\omega)$.
Then the homotopy $J_t$ is constant equal $i\oplus J_M$ over $\pi^{-1}(\mathcal{W})$, where the fibration is tame.
Let $\overline{J} := \{J_t\}_{t\in[0,1]}$.
Let $G$ be a Riemannian metric on $V\times [0,1]$ so that $G|_{V\times 
\{0\}}=\rho$ and $G_{V\times \{1\}}=\rho'$. We also need a Morse function $H:V\times 
[0,1]\to \mathbb{R}$ with $H_{V\times \{0\}}(x)=f(x)+c$ and $H_{V\times 
\{1\}}(x)=f'(x)$, where $c$ is some constant and such that $-\nabla H$ points out along 
$S\times [0,1]$. Moreover, we can find such a 
Morse-Smale pair $(H,G)$ with the additional properties that 
$$Crit_i(H)=Crit_{i-1}(f)\times \{0\} \cup Crit_i(f')\times \{1\}$$ and $\frac{\partial 
H}{\partial t}(x,t)=0$ for $t=0,1$, $\frac{\partial H}{\partial t}(x,t)< 0$ for $t\in 
(0,1)$. Such a Morse function is easy to construct and more details are given 
in~\cite{CR03}. The pair $(F,G)$ is a special case of a Morse cobordism.

The comparison moduli space is now constructed in a very similar way as the moduli space 
of pearly trajectories (~\ref{def:P_prl}). 
Let $x\in Crit(f)$ and $y'\in Crit(f')$ be two critical points. We consider the space of 
pearly trajectories connecting $x$ and $y'$ but with the following modifications compared 
to~\ref{def:P_prl}.
Namely,
\begin{enumerate}
 \item The holomorphic disks map to the fibers $(E,V)\times {t_i}$, $t_i\in[0,1]$, i.e. $(D,\partial D)\to (E,V)\times \{t_i\}\subset (E,V)\times 
[0,1]$ and they are $J_{t_i}$-holomorphic,
\item the incidence relations of point $(ii),(iii)$ and $(iv)$ take place in $(E,V)\times 
[0,1]$, where we consider the flow of $-\nabla_G F$. 
\end{enumerate}
The space of all such pearly trajectories after dividing by the reparametrization 
group is what we call the comparison moduli space 
$\mathcal{P}_{comp}(x,y',\lambda;\overline{J},H,G)$. 
Similar to case of $\mathcal{P}_{prl}$ one can show that for 
$|x|-|y'|+\mu(\lambda)\leq 1$ the set $\mathcal{P}_{comp}$ is a manifold of dimension 
$\delta_{comp}=|x|-|y'|+\mu(\lambda)$, it is compact if $\delta=0$ and void if $\delta\leq -1$.
We define the chain morphism 
$$\phi^{\overline{J},F,G}:C^+((V,S),\mathcal{D}_S)\to C^+((V,S), \mathcal{D}_S')$$
by 
\begin{equation*}
\phi^{\overline{J},F,G}(x):=
\smashoperator{ \sum\limits_{\begin{subarray}{c} y'\in 
Crit(f'), \lambda \\ 
|x|-|y'|+\mu(\lambda)=0\end{subarray}}} \#_{\Z_2} \mathcal{P}_{comp}(x,y',\lambda;\overline{J},H,G)y' 
t^{\overline{\mu}(\lambda)}.
\end{equation*}
By the same methods as before we can describe the boundary of the compactification of 
$\mathcal{P}_{comp}$ for $\delta=1$ and the resulting relation proves that 
$\phi^{\overline{J},F,G}$ is a chain morphism. 

The morphism $\phi=\phi^{\overline{J},F,G}$ respects the degree filtration of 
the two pearl complexes $C^+((V,S);\mathcal{D})$ and $C^+((V,S);\mathcal{D}')$.
Since the same is true for the differential of the pearl complex, $\phi$ is a morphism 
between the spectral sequences associated to the filtration of $C^+((V,S);\mathcal{D})$ and 
$C^+((V,S);\mathcal{D}')$. The zeroth page $E^0_{p,q}$ is the complex itself with the 
differential the Morse differential. From Morse theory we already know that this induces 
an isomorphism in homology. Hence, $\phi$ induces an isomorphism on the whole 
spectral sequence and thus an isomorphism of the quantum homologies.

To show that this isomorphism is canonical we use an argument
similar to the one above. Suppose 
$\phi^{\overline{J},F,G}$ and $\tilde{\phi}^{\tilde{\overline{J}},\tilde{F},\tilde{G}}$ 
are morphisms between the complexes $C^+((V,S),\mathcal{D}_S)$ and $C^+((V,S), \mathcal{D}_S')$. 
Now we can construct a suitable homotopy between $(\overline{J},F,G)$ and 
$(\tilde{\overline{J}},\tilde{F},\tilde{G})$ and use it to define moduli spaces of 
appropriate pearly trajectories. A count of such pearly trajectories then defines a chain 
homotopy. 
For precision see~\cite{BC07}.

\subsection{Behaviour of pearly trees near \texorpdfstring{$\partial V$}{Lg}}

In this section we adapt the previous ideas to a more general setting and show that 
similar results hold for different kinds of moduli spaces modeled on planar trees.
These methods was introduces in~\cite{BC07} and similar constructions 
were used already in~\cite{McS12} and in~\cite{CL06}.
The remaining quantum structures can be defined by such moduli spaces. 
\paragraph{Moduli spaces modeled on planar trees}

If the following we assume that a Morse function $f$ on the cobordism $V$ is such that 
$-\nabla f$ is transverse to $\partial V$. Similarly a Morse function $g$ on $E$ is 
assumed to have $-\nabla g$ transverse to $\partial E$, namely pointing inwards along the 
horizontal boundary $\partial^hE$ (if it is not empty) and pointing outwards along the 
vertical part $\partial^vE$.
By a moduli spaces modeled on planar tree we mean the following. It consists of 
configurations modeled on planar trees with oriented edges. 
The edges correspond to negative gradient trajectories of Morse functions on either the 
Lagrangian or the ambient 
manifold and with the orientation corresponding to the direction of the negative gradient 
flow. We mark the edges with the corresponding function. The vertices of valence one 
correspond to critical points of the Morse function used along the unique 
edge ending there. Vertices of valence two ore more correspond to $J$-holomorphic spheres 
or $J$-holomorphic disk with boundary in the Lagrangian and with some marked points 
(situated on the boundary or in the interior) equal to the valency of 
that vertex. These marked points correspond to the end points of the edges ending at 
the vertex. In particular, interior marked points correspond to end points of edges 
labeled by Morse functions on the ambient manifold and marked points on the boundary 
of the disk correspond to Morse functions on the Lagrangian. For transversality reasons 
it is important to assume that if a disks has 
more that one entry point, then the corresponding edges ending there must be marked by 
pairwise distinct Morse functions in generic position. The vertices of the tree are 
marked with the corresponding Maslov indices of the $J$-holomorphic curves. 
We also allow constant $J$-holomorphic curves as long as they are stable in the sense 
that the valence of the vertex is at least three for spheres and the valence is at 
least two for disks, with one marked point on the boundary and one marked point in the 
interior. 

The number of entries and exits as well as the valence of the vertices and the 
position of their marked points depends on the structures or relations that the moduli 
spaces are describing. For example, differentials or morphisms are described by trees 
with 
one entry and one exit. An operation has two entries, associativity requires three 
entries and structures involving the ambient quantum homology requires some internal 
marked points and some edges labeled with Morse functions on the ambient manifold.
The trees have typically one exit. We will refer to elements of such moduli spaces as 
\emph{pearly trees}.

\begin{lemma}~\label{lem:traj_away_from_boundary_general}
Let $\mathcal{P}$ be a moduli space modeled over planar trees of some type described 
above.
Assume that all the edges corresponding to either flow lines of Morse functions on
$V$ respecting the exit region $S\subset \partial V$ or to Morse functions on $E$, such 
that its gradient points out along $\partial^v E$ and in along $\partial^h E$.
Suppose that the almost complex structure is $J\cong j\oplus J_M$ near the cylindrical 
ends of $V$. (See~\ref{rmk:cobordism_in_tame_Lefschetz}.)
Then, the elements of the moduli space $\mathcal{P}$ cannot reach the boundary $\partial 
V$ nor the boundary $\partial E$. In particular, these elements cannot escape to infinity along $E^{\infty}$.
\end{lemma}

\begin{proof}
The proof is similar to the proof of~\ref{lem:traj_away_from_boundary}. 
\end{proof}

Lemma~\ref{lem:traj_away_from_boundary_general} together with 
Lemma~\ref{prop:transversality_general} imply:

\begin{theorem}\label{thm:moduli_as_in_closed_setting}
 Let $\mathcal{P}$ be a moduli space modeled over planar trees of some type described 
above and let 
$\mathcal{P}^c$ be the analogous moduli space on planar trees but for the setting 
described in~\cite{BC07}, i.e. for closed Lagrangians. 
Suppose we know that if the virtual dimension of $\mathcal{P}^c$ is at most $1$, it is a 
smooth manifold of dimension equal to its virtual dimension and it is compact if 
$0$-dimensional.
Then the same holds for $\mathcal{P}$ and the compactification of a one dimensional such 
moduli space has the same description as the compactification
of the analogous moduli space $\mathcal{P}^c$.
\end{theorem}

\subsection{The augmentation}\label{sec:augmentation}

A useful feature called the augmentation map comes from a chain morphism 
$$\epsilon_{V,f}:C^+(V;f)\to \Lambda^+,$$
where $\Lambda^+$ is viewed as a chain complex with trivial differential.
The following theorem was taken from~\cite{Hir94}.

\begin{theorem}\label{thm:diffeo_to_Dn}
 Let $f:M\to \mathbb{R}$ be a Morse function on a compact connected $n$-manifold $M$ and 
such that $-\nabla f$ is transverse to $\partial M$. Suppose that $f$ has exactly three 
critical points of index $0,0$ and $1$. Then $M$ is diffeomorphic to a $n$-dimensional 
disk $D^n$. 
\end{theorem}

\begin{lemma}\label{lem:single_min}
There exists a Morse function $f$ on $V$, with $-\nabla f$ pointing inwards 
along the boundary $\partial V$ and with a single minimum, which we call $x_0$.
\end{lemma}

\begin{proof}
Let $f:V\to \mathbb{R}$ be a Morse function on with finitely many critical points and pointing inwards along $\partial V$ and suppose 
$f$ has more than one local minimum. 
Since we know that $H_0(V)\cong \mathbb{Z}_2$ there must be two local minima $m_0$ and 
$m_1$ and a critical point $x$ of index $1$ such that the Morse differential of $x$ is 
$m_0-m_1$. By theorem~\ref{thm:diffeo_to_Dn} there exists an embedded $n$-dimensional ball 
containing only the critical points $x$, $m_0$ and $m_1$. 
We can replace $x$, $m_0$ and $m_1$ by a single critical point of index $0$ to get 
a new Morse function on $V$ with one local minimum less. 
The lemma is proved by iteration of this process.
\end{proof}

\begin{proposition}\label{prop:min_not_boundary}
 Let $f$ be a Morse function on $V$ with $-\nabla f$ pointing inwards along 
the boundary $\partial V$ and with a single minimum $x_0$.
Then the minimum $x_0$ is not a boundary of the pearl complex of $V$.
\end{proposition}

\begin{proof}
Consider moduli spaces $\mathcal{P}_{prl} 
(x,x_0,\lambda)$ of dimension zero.
Suppose first that $\mu(\lambda)\neq 0$ and let $x\in Crit(f)\setminus \{ x_0 \}$.
Then 
$$dim \mathcal{P}_{prl}(x, x_0; \lambda)=|x|-|x_0|+\mu(\lambda)-1\geq 1-0+2-1=2>0.$$
If $\lambda \neq 0$ then $\mathcal{P}_{prl}(x, x_0, \lambda)$ is not zero 
dimensional. 
Therefore we may assume that $\lambda = 0$ and in particular the index of $x$ is $1$. 
The Morse homology of $V$ is non trivial in degree zero and since $x_0$ is the only 
critical point of index zero it is not a boundary of the Morse differential.
\end{proof}
\begin{proposition}\label{prop:augmentation}
There exists a canonical, degree preserving augmentation 
\begin{equation*}
 \epsilon_{V}:Q^+H_*(V) \rightarrow \Lambda^+,
\end{equation*}
which is a $\Lambda^+$-module map.
\end{proposition}
\begin{proof}
For a choice of data $\mathcal{D}=(f,\rho,J)$ we define 
$\epsilon_{V,\mathcal{D}}$ on the chain level as follows.
\begin{equation}\label{eq:augm}
 \begin{array}{cccc}
  \epsilon_{V,\mathcal{D}}: & Crit(f) & \rightarrow & \Lambda\\
              & x 	&\mapsto & \begin{cases}
				  1 \text{, if }  |x| =0\\
				0 \text{, else }\\
                  	         \end{cases},
 \end{array}
\end{equation}
and extend it linearly over $C^+(V;\mathcal{D})$.
Assume now that $f$ is a Morse function with a single minimum, which exists by Lemma~\ref{lem:single_min}.
Then $\epsilon_{V,\mathcal{D}}$ is a chain map, as the following calculation shows.
Clearly $\epsilon_{V,\mathcal{D}}(d(x_0))=0=d(\epsilon_{V,\mathcal{D}}(x_0))$.
If $x\neq x_0$ then $d\circ \epsilon_{V,\mathcal{D}}(x)=0$ and
\begin{equation*}
\begin{array}{ccl}
  \epsilon_{V,\mathcal{D}} \circ d(x) & = & \epsilon_{V,\mathcal{D}} \left( 
\sum\limits_{\begin{subarray}{c} y\\
\delta_{prl}=0\\
                                                                        \end{subarray}} 
\#_{\Z_2} \mathcal{P}_{prl}(x,y,0)y + 
\sum\limits_{\begin{subarray}{c} y,\lambda\neq0\\
\delta_{prl}=0\\
                                                                        \end{subarray}}
(-1)^{|y|} \#_{\Z_2} \mathcal{P}_{prl}(x,y,\lambda)y 
t^{\overline{\mu}(\lambda)} \right) \\
			& = & \epsilon_{V,\mathcal{D}} \left(  
 
\#_{\Z_2} \mathcal{P}_{prl}(x,x_0,0)x_0 + 
\sum\limits_{\begin{subarray}{c} \lambda\neq0\\
                                                                         
\delta_{prl}=0\\
                                                                        \end{subarray}}
(-1)^{|x_0|} \#_{\Z_2}v \mathcal{P}_{prl}(x,x_0,\lambda)x_0 
t^{\overline{\mu}(\lambda)} \right) \\
			& = & 0,\\
\end{array}
\end{equation*}
where the second equality holds since $f$ has the unique minimum $x_0$ and 
$\epsilon_{V,\mathcal{D}}(y)=0$ whenever $y\neq x_0$.
The last equality holds by a similar argument as in the previous lemma. We saw that $\mathcal{P}_{prl}(x,x_0,\lambda)$ is either void or $\lambda=0$.
Moreover, the Morse differential of a critical point of index one is an even number of critical points of index zero and $x_0$ is the only minimum.

Let $\mathcal{D'}$ be another generic data.
We verify $\epsilon_{V,\mathcal{D'}}\circ 
\phi^{\mathcal{D'}}_{\mathcal{D}} = \epsilon_{V,\mathcal{D}} $,
where $\phi^{\mathcal{D'}}_{\mathcal{D}}$ is the comparison morphism between the chain 
complexes.
A dimension calculation shows that the comparison moduli space 
$\mathcal{P}_{comp}(x,y',\lambda)$, with $|y'|=0$ has dimension zero if and only if 
$|x|=-\mu(\lambda)=0$. 
In other words, there are no non-empty moduli spaces of dimension zero connecting a 
critical point $x\neq x_0$ of $f$ to a critical point $y'$ of $f'$ of 
index zero, and for $x\neq x_0$ 
$$\epsilon_{V,\mathcal{D'}}\circ 
\phi^{\mathcal{D'}}_{\mathcal{D}}(x)=0=\epsilon_{V,\mathcal{D}}(x).$$
If $x=x_0$ and $|y'|=0$, then $\mu(\lambda)$ must be zero, and
\begin{equation}
\begin{array}{ccc}
  \epsilon_{V,\mathcal{D'}}\circ \phi^{\mathcal{D'}}_{\mathcal{D}}(x_0) &=&
\epsilon_{V,\mathcal{D'}} \left( \sum\limits_{\begin{subarray}{c} |y'|=0\\
\delta_{comp}=0 \end{subarray}} \# \mathcal{P}_{comp}(x_0,y',0)y'\right)\\
 &=& \sum\limits_{\begin{subarray}{c} |y'|=0\\
\delta_{comp}=0 \end{subarray}} \# \mathcal{P}_{comp}(x_0,y',0).\\
\end{array}
\end{equation}
Recall that the moduli spaces were defined using a Morse cobordism $(F,G)$ between 
$(f,\rho)$ and $(f',\rho')$. 
Notice that $\sum\limits_{\begin{subarray}{c}|y'|=0 \\ \delta_{comp}=0 \end{subarray}} 
\# \mathcal{P}_{comp}(x_0,y',0)y'$ 
is the Morse part of the differential for the Morse function $F$ of the critical point 
$x_0$ of $F$.
As a critical point of $F$, $x_0$ has index one and it lies on the boundary.
Standard Morse theoretic arguments show that the Morse differential of $x_0$ is exactly one 
critical point of $F$ of index zero. 
Therefore, $\sum\limits_{\begin{subarray}{c} |y'|=0\\ \delta_{comp}=0 
\end{subarray}} 
\#\mathcal{P}(x_0,y',0)=1$, which proves the statement.
\end{proof}

\subsection{The quantum product}\label{sec:product}

\begin{proposition}\label{prop:product_and_unit}
 There exists a $\Lambda$-bilinear map
\begin{equation*}
\begin{array}{ccc}
  *:Q^+H_i(V,S)\otimes Q^+H_j(V,S) & \rightarrow & Q^+H_{i+j-(n+1)}(V,S) \\
  \alpha \otimes \beta & \mapsto & \alpha * \beta,\\
\end{array}
\end{equation*}
which endows $Q^+H_*(V,S)$ with the structure of a (possibly non-unital) ring.
Moreover, if $S=\partial V$ the product turns $Q^+H_*(V,\partial V)$ into a ring with 
a unit.
\end{proposition}
The underlying moduli spaces are:
\begin{definition}\label{def:P_prod}
Fix three Morse functions $f$, $f'$ and $f''$ on $V$ respecting 
the exit region $S$, a Riemannian metric $\rho$ and an almost complex 
structure $J\in \mathcal{J}_{\mathcal{W}}$.
Let $x$, $y$ and $z$ be critical points of $f$, $f'$ and 
$f''$, respectively. 
Let $\lambda$ be a class in $H^D_2(E,V)\subset H_2(E,V)$.
Consider the space of all tuples $(\mathbf{u}, \mathbf{u'}, \mathbf{u''}, v)$ that 
satisfy:
\begin{itemize}
  \item[(i)] $v:(D, \partial D) \rightarrow (E,V)$ is a 
$J$-holomorphic disk, possibly constant.
$v$ is also called the core of $(\mathbf{u}, \mathbf{u'}, \mathbf{u''}, v)$.
  \item[(ii)] Denote $\tilde{x}=v(e^{-2\pi i/3})$, $\tilde{y}=v(e^{2\pi i/3})$ and 
$\tilde{z}=v(1)$. The points $\tilde{x}$, $\tilde{y}$ and $\tilde{z}$ are not critical 
points. Then:
	      $\mathbf{u}\in 
\mathcal{P}_{prl}(x,\tilde{x},\mathbf{\xi};J,\rho,f)$,
	      $\mathbf{u'}\in 
\mathcal{P}_{prl}(y,\tilde{y},\mathbf{\xi'};J,\rho,f')$,
	      $\mathbf{u''}\in 
\mathcal{P}_{prl}(\tilde{z},z,\mathbf{\xi''};J,\rho,f'')$,
      for some $\mathbf{\xi},\mathbf{\xi'}, \mathbf{\xi''} \in H^D_2(E,V)$.
  \item[(iii)] $\mathbf{\xi}+ \mathbf{\xi'}+ \mathbf{\xi''}+[v]=\mathbf{\lambda}$.
\end{itemize}
Denote by 
$\mathcal{P}_{prod}(x,y,z,\lambda;f,f',f'',\rho,J)$ the space of all such sequences.
An element in this space is called a figure-$Y$ pearly trajectory from $x$ and $y$ to $z$.
In other words
$\mathcal{P}_{prod}(x,y,z,\lambda;f,f',f'',\rho,J)$ is a moduli space modeled over a 
planar tree with two entries $x$ and $y$, one exit $z$ and one vertex of valence three 
with three marked points on the boundary.
\end{definition}
The virtual dimension of 
$\mathcal{P}_{prod}(x,y,z,\lambda;f,f',f'',\rho,J
)$ is 
$$\delta_{prod}(x,y,z,\mathbf{\lambda}):=|x|+|y|-|z|+\mu(\lambda)-(n+1).$$

Suppose that Assumption~\ref{assum:genericity} 2. holds. 
By Theorem~\ref{thm:moduli_as_in_closed_setting} for generic $J$ and if the virtual dimension
$\delta_{prod}(x,y,z,\mathbf{\lambda})\leq1$, the moduli spaces 
$\mathcal{P}_{prod}(x,y,z,\lambda)$ form smooth manifolds of dimension equal to their 
virtual dimension, and they are compact if $\delta_{prod}(x,y,z,\mathbf{\lambda})=0$. 
The quantum product on the chain level is defined by counting the elements in the zero 
dimensional moduli spaces of figure-$Y$ pearly trajectories.
More precisely, for every $x\in Crit(f)$, $y \in Crit(f')$ we set
\begin{equation}\label{eq:product}
 x * y := \sum\limits_{ \begin{subarray}{c}
                         z,  \lambda \\ \delta_{prod}=0\\
                         \end{subarray}
} \# \mathcal{P}_{prod}(x, y, z, 
\mathbf{\lambda};f,f',f'',\rho, J) z
t^{\overline{\mu}(\lambda)},
\end{equation}
where the sum runs over all $z\in Crit(f'')$ and $\mathbf{\lambda} \in 
H^D_2(E,V)$, such that $\delta_{prod}(x,y,z,\mathbf{\lambda})=0$.

\begin{lemma}\label{lem:prod}
For a generic choice of data, the operation in~(\ref{eq:product}) is well-defined and it 
is a chain map.
It induces a product in homology, which is independent of the choices made in the 
construction.
\end{lemma}

\begin{proof}
 By Theorem~\ref{thm:moduli_as_in_closed_setting} we can refer to~\cite{BC07} for details.
\end{proof}

\begin{lemma}\label{lem:assoc}
 The product
\begin{equation*}
 *:Q^+H_i(V,S)\otimes Q^+H_k(V,S)\rightarrow Q^+H_{i+k-(n+1)}(V,S)
\end{equation*}
is associative.
\end{lemma}
\begin{proof}
We want to define a moduli spaces $\mathcal{P}_{\tau}$ modeled over trees $\tau$ 
of some specific type and show that
$$\xi:C^+((V,S);f)\otimes C^+((V,S);f''^+((V,S);f'')\to 
C^+((V,S);f)$$
defined by
$$\xi(x\otimes y\otimes z):=\sum\limits_{ \begin{subarray}{c}
                         w,  \lambda \\ \delta_{\tau}=0\\
                         \end{subarray}
} \# \mathcal{P}_{\tau}(x, y, z, w) w
t^{\overline{\mu}(\lambda)},$$
is a chain homotopy $\xi:((-*-)*-)\simeq (-*(-*-))$.
A tree of type $\tau$ have three entries, labeled by critical points of $f$, $f'$ and 
$f''$ respectively. The tree has one exit, labeled by a critical point of $f$. It is easy 
to see that the tree must have either two vertices of index three, that correspond to 
disks with three boundary marked points, or one vertex, corresponding to a disk with four 
boundary marked points. The vertices are all labeled by elements of $H_2(E,V)$. 
See~\cite{BC07} for details.
\end{proof}

For the existence of a unit $-\nabla f$ has to point outwards along the boundary $\partial V$.
\begin{lemma}\label{lem:single_max}
There exists a Morse function $f$ on $V$ with a single maximum $m$ and such that
$-\nabla f$ points out along $\partial V$.
\end{lemma}
\begin{proof}
We can use the Morse function $-f$, where $f$ was defined in~\ref{lem:single_min}.
\end{proof}

\begin{lemma}\label{lem:unit_rel}
There exists a canonical element $e_{(V,\partial V)}\in Q^+H_{n+1}(V,\partial V)$, 
which is a unit with respect to the quantum product, i.e. $e_{(V,\partial V)}*x=x$ for 
every $x\in 
Q^+H_*(V,\partial V)$.
\end{lemma}

\begin{proof}
Let $f$ be a Morse function with a single maximum $m$. Then $m$ represents the unit in $Q^+H_*(V,\partial V)$.
Any non-void moduli space $\mathcal{P}_{prod}(m,x,\lambda)$ with $\mu(\lambda)>0$ has dimension 
$$|m|-|x|+\mu(\lambda)-1=(n+1)-|x|+\mu(\lambda)-1 \geq (n+1)-n+2-1=2.$$
Thus, we may assume $\lambda=0$, which corresponds to computing the Morse differential. 
Since $H_{n+1}(V,\partial V)\cong \mathbb{Z}_2$ the critical point $m$ must be a cycle. 
By a similar dimension argument it is easy to see that $m$ is not a boundary either.
(This is true if we work with $\Lambda^+$. For coefficients in $\Lambda$ the element represented by $m$ might be zero, in which case the quantum homology is narrow.)

Choose $(f, f',f'',\rho)$ in generic position.
We may assume $f'=f''$.
On the chain level
\begin{equation*}
\begin{array}{cccc}
  m * y & = & \sum\limits_{\begin{subarray}{c}
			z,\lambda\\
			\delta_{prod}=0\\
			 \end{subarray}
} \# \mathcal{P}_{prod}(m,y,z,\lambda)zt^{\overline{\mu}(\lambda)}.\\
\end{array}
\end{equation*}
Since $f'=f''$, we see that the figure-$Y$ pearly trajectory gives rise 
to a pearly trajectory from $y$ to $z$.
The identity 
$$0=\delta_{prod}(m,y,z,\lambda)=|m|+|y|-|z|+\mu(\lambda)-(n+1)$$ 
implies
$$0=|y|-|z|+\mu(\lambda),$$ 
since $|m|=n+1$.
Assume that $|y|\neq|z|$ and $\mu(\lambda)\neq 0$, then
\begin{equation*}
 \delta_{prl}(y,z,\lambda)=|y|-|z|+\mu(\lambda)-1<0,
\end{equation*}
which is a contradiction.
We conclude that $\mu(\lambda)=0$ and $|y|=|z|$. 
Thus, $y=z$ and since $m$ is the only maximum of $f$ there exists a unique flow line of 
$-\nabla f$ starting at $m$ and going through $y=z$. 
Therefore
\begin{equation*}
\begin{array}{cccccc}
  m * y & = & \smashoperator{\sum\limits_{\begin{subarray}{c}
			z,\lambda\\
			\delta_{prod}=0\\
			 \end{subarray}}
} \# \mathcal{P}_{prod}(m,y,z,\lambda)zt^{\overline{\mu}(\lambda)}
	& = & y.
\end{array}
\end{equation*}

It is left to show that the definition of the unit on the chain level is canonical.
Choose another Morse function $f'$ and suppose $m'_1, \cdots, m'_k$ are all the index $n+1$ critical points of $f'$. Let $(F,G)$ be a Morse cobordism between $(f,\rho)$ and $(f',\rho')$.
Then $m$ is the unique critical point of $F$ of index $n+2$.
For a critical point $y$ of $f'$ the dimension of the comparison moduli space $\mathcal{P}_{comp}(m,y',\lambda)$ is 
$$|m|-|y'|+\mu(\lambda)=(n+1)-|y'|+\mu(\lambda).$$
This is only zero if $\mu(\lambda)=0$ and $|y'|=n+1$.
Moreover, Morse theoretic arguments show that for all $i=1,\cdots, k$ there exists exactly one gradient trajectory of $-\nabla F$ from $m$ to $m'_i$.
In other words, we have $\phi_{\mathcal{D}}^{\mathcal{D}'}(m)=m'_1+\cdots +m'_k$.
\end{proof}

\subsection{The ambient quantum homology}\label{sec:ambient_QH}

If the Lagrangian $V\subset E$ is monotone then $E$ is spherically monotone in the sense that there exists a constant $\nu>0$ such that
$$\Omega(\lambda)=\nu c_1(\lambda),\quad \forall \lambda\in \pi_2(E),$$
where $c_1\in H_2(E)$ is the Chern class of the tangent bundle $TE$. 
This constant is related to the monotonicity constant by $\nu=2\tau$. 
Define the minimal Chern number of $E$:
$$C_E:=\min\{c_1(\lambda)>0|\lambda\in \pi_2(E)\}.$$
Let $h$ be a Morse function on $E$ such that $-\nabla h$ points out along $\partial^vE$ and in along $\partial^hE$.
Let $\Gamma:=\mathbb{Z}_2[s,s^{-1}]$, where $|s|:=-2C_M$ and set $Q^+H(E,\partial^vE):=H(E,\partial^vE)\otimes \Gamma$.
(If $C_E=\infty$ take $\Gamma:=\mathbb{Z}_2$.)
Since $N_V|2C_E$ we can also work with the coefficient ring $\Lambda:=\mathbb{Z}_2[t,t^{-1}]$, where $|t|=-N_V$.
Define the Morse complex 
$$C^+((E,\partial^vE),h,\rho_E;\Lambda)=\mathbb{Z}\langle Crit(h)\rangle\otimes \Lambda.$$
Then the embedding $\Gamma\hookrightarrow \Lambda:s\mapsto t^{2C_E/N_V}$ induces an isomorphism 
 $$Q^+H(E,\partial^vE;\Lambda):=Q^+H(E,\partial^vE)\otimes_{\Gamma}\Lambda \cong H_*(C^+((E,\partial^vE),h,\rho_E;\Lambda)).$$

\begin{proposition}
There exists a bilinear map 
\begin{equation*}
 \oast:Q^+H_i(E,\partial^vE)\otimes Q^+H_j(E,\partial^vE)\to Q^+H_{i+j-(2n+2)}(E,\partial^vE)
\end{equation*}
that turns $Q^+H_*(E,\partial^v E)$ in to a (possibly non-unital) ring. If the 
generic fiber $(M,w)$ is compact without boundary (or equivalently $\partial^hE=\emptyset$ and $\partial E=\partial^vE$) then this ring is unital and the unit $e_{(E,\partial E)}$ 
corresponds to the fundamental class $[E,\partial E]\in H_*(E,\partial E)$.
\end{proposition}

Let $g$, $g'$ and $g''$ be Morse functions on $(E, \partial^v E)$, and their negative 
gradient should point outwards along $\partial^v E$ and inwards along $\partial^h E$. 
The quantum product can be defined by counting the elements of a moduli spaces modeled 
over trees with two entries and one exit point and one vertex of valence three 
corresponding to a $J$-holomorphic sphere.
The entry points correspond to two critical points $x$ and $y$ of $g$ and $g'$ 
respectively and the exit point is a critical point $z$ of $g''$.
The edges of the tree correspond to the flow lines of the negative gradient of the corresponding Morse functions.
Compare this with Definition~\ref{def:P_prod}.
We denote these moduli spaces by 
$\mathcal{M}_{prod}(x,y,z,\lambda;g,g',g'',J,\rho)$.
Let $x\in Crit(g)$ and $y\in Crit(g')$ be two critical points.
Then
\begin{equation}\label{eq:qprod_E}
 x \oast y:=\sum\limits_{z,\lambda} 
\# \mathcal{M}_{prod}(x,y,z,\lambda;g,g',g'' ) 
z t^{\overline{\mu}(\lambda)},
\end{equation}
where the sum runs over all $z\in Crit(g)$ and $\lambda$, such that 
$\mathcal{M}_{prod}(x,y,z,\lambda)$ is $0$-dimensional.

\begin{lemma}
 For a generic choice of data, the operations~(\ref{eq:qprod_E})
are well-defined and their linear extensions define chain maps. 
They induce a product on $Q^+H_*(E,\partial^v E)$, which is 
independent of the choices made in the constructions.
\end{lemma}

\begin{proof}
 By Lemma~\ref{lem:traj_away_from_boundary_general} and Theorem~\ref{thm:moduli_as_in_closed_setting} the proof is equivalent to the 
proof given in~\cite{McS12}.
\end{proof}

\subsection{The module structure}\label{sec:module_structure}

We want to endow the quantum homology $Q^+H(V,S)$ with the structure of a 
module over the ring $Q^+H_*(E,\partial^vE)$.

\begin{proposition}
 There exists a bilinear map
\begin{equation*}
 \star:Q^+H_i(E,\partial^v E) \otimes Q^+H_j(V,S) 
\rightarrow 
Q^+H_{i+j-(2n+2)}(V,S),
\end{equation*}
which endows $Q^+H_*(V,S)$ with the structure of a two-sided algebra over the 
(possibly non-unital) ring $Q^+H_*(E,\partial^v E)$.
\end{proposition}

\subsubsection{The module structure of \texorpdfstring{$Q^+H_*(V,\partial V)$}{Lg} over 
\texorpdfstring{$Q^+H_*(E,\partial^v E)$}{Lg}}

We can define a module structure of the quantum homology $Q^+H_*(V,S)$ 
over the ring $Q^+H_*(E,\partial^v E)$.
\begin{definition}\label{def:Pmod}
Let $f:V\rightarrow \mathbb{R}$ be a Morse function on $V$ respecting the exit 
region $S$ and $h:E\to \mathbb{R}$ a Morse function on $E$ respecting the exit region $\partial^vE$.
Furthermore, let $\rho_E$, $\rho_V$ be Riemannian metrics on $E$, $V$ respectively and $J\in 
\mathcal{J}_{\mathcal{W}}$ an almost complex structure.
Let $\phi_t^h$ and $\phi_t^f$ denote the flow for the negative gradient of $h$ and $f$ 
respectively. 
Let $x,y$ be two critical points of $f$ and let $a$ be a 
critical point of $h$.
Let $\lambda \in H_2^D(E,V)$ be a class, possibly zero.
Consider the space of all sequences $(u_1,\dots u_l;k)$ of every possible length $l\geq 
1$,where
\begin{enumerate}
 \item $1\leq k\leq l$.
 \item $u_i:(D,\partial D)\rightarrow (E,V)$ is a $J$-holomorphic disk 
for every $1\leq i\leq l$, which is assumed to be non-constant, except possibly if $i=k$.
 \item $u_1(-1)\in W^u_x(f)$.
 \item For every $1\leq i \leq l-1$ there exists $0 <t_i<\infty$ such that 
$\phi^{f}_{t_i}(u_i(1))=u_{i+1}(-1)$.
 \item $u_l(1)\in W^s_y(f)$.
 \item $u_k(0)\in W^u_a(h)$.
 \item $[u_1]+ \dots +[u_l]=\lambda$.
\end{enumerate}
Two elements $(u_1,\dots u_l;k)$ and $(u'_1,\dots u'_{l'};k')$ in this space are viewed 
as equivalent if $l=l'$, $k=k'$ and for every $i\neq k$ there exists an automorphism 
$\sigma_i\in Aut(D)$ fixing $-1$ and $1$ such that $\sigma_i\circ u_i=u'_i$.
The space of all such elements $(u_1,\dots u_l;k)$ modulo the above equivalence relation 
is denoted 
$\mathcal{P}_{mod}(a,x,y,\lambda;h,\rho_{E},f,\rho_{V},J)$.
It is modeled over planar trees with two entries $x$ and $a$, one exit $y$ and one vertex 
of valence three with two marked points on the boundary and one marked point in the 
interior. 
\end{definition}
The virtual dimension is $\delta_{mod}(a,x,y,\lambda):=|a|+|x|-|y|+\mu(\lambda)-(2n+2)$.
On the chain level for $a\in Crit(g)$ and $x\in Crit(f)$ we define
\begin{equation}\label{eq:Pmod}
 a\star x:=\smashoperator{\sum\limits_{ \begin{subarray}{c} y,\lambda\\ 
\delta_{mod}(a,x,y,\lambda)=0\end{subarray}}}
\#_{\Z_2} \mathcal{P}_{mod}(a,x,y,\lambda;g,\rho_{E},f,
\rho_{V},
J) yt^{\overline{\mu}(\lambda)},
\end{equation}
where the sum is taken over the critical points $y\in Crit(f)$, such that the 
corresponding moduli spaces have dimension zero.
 
\begin{lemma}\label{prop:modulestructureV}
 For a generic choice of the data the map in~(\ref{eq:Pmod}) is well-defined and a chain map.
It induces a map in homology
\begin{equation*}
 \star: Q^+H_{i}(E,\partial^v E)\otimes Q^+H_{j}(V,S) \rightarrow Q^+H_{i+j-2n-2}(V,S),
\end{equation*}
which is independent of the choice of the data.
\end{lemma}

\begin{proof}[Proof of Lemma~\ref{prop:modulestructureV}]
Applying the Lemma~\ref{lem:traj_away_from_boundary} we see that the proof 
given in~\cite{BC07} adapts to our setting.
\end{proof}

The following properties show that this product gives $Q^+H_*(V,S)$ the structure of a 
two-sided algebra over $Q^+H_*(E,\partial^hE)$.

\begin{lemma}
 For every $a,b \in Q^+H_*(E,\partial^v E)$ and for every $x\in Q^+H_*(V,S)$ the 
following properties are fulfilled.
\begin{itemize}
 \item[(i)] $(a\oast b)\star x=a\star (b\star x)$
 \item[(ii)] $e_{(E,\partial E)} \star x=x$ if the fibers of $E$ are closed.
\end{itemize}
 For any $a\in Q^+H_*(E,\partial^v E)$ and any $x,y\in Q^+H_*(V,\partial 
V)$ we have:
\begin{itemize}
 \item[(iii)] $a\star (x*y)=(a\star x)*y$
and
 \item[(iv)] $a\star (x*y)=x*(a\star y)$.
\end{itemize}
\end{lemma}

\begin{proof}
The proof of point $(i)$, $(iii)$ and $(iv)$ are very similar to the proof of the 
associativity of the product, and we omit it. However, we note that once again due to 
Lemma~\ref{lem:traj_away_from_boundary_general} a proof for the non-compact setting 
(\cite{BC07}) directly generalizes to this case.
For point $(ii)$ notice that if $a\in Crit_{2n+2}$ represents the unit then $\delta(a, 
x,y,\lambda)=2n+2+|x|-|y|+\mu(\lambda)$ is zero if and only if $|x|=|y|$ and $\lambda=0$. 
Therefore, the only pearly trees that contribute to $[E,\partial E]\star x$ are the classical 
Morse theoretic ones and in fact we can see that $x=y$, which proves the 
statement.
\end{proof}

\subsection{The inclusion}\label{sec:inclusion}

\begin{proposition}\label{prop:inclusion}
 There exists a $Q^+H_*(E,\partial^vE)$- linear inclusion map
\begin{equation*}
 i=i_{(V,S)}^{(E,\partial^hE)}:Q^+H_*(V,S) \to Q^+H_*(E,\partial^v E),
\end{equation*}
which extends the inclusion in singular homology.
\end{proposition}

The necessary moduli spaces are:
\begin{definition}\label{P_inc}
 Fix Morse functions $h:E\to \mathbb{R}$ respecting the exit region 
$\partial^vE$ and $f:V\to \mathbb{R}$ respecting the exit region $S$. Furthermore, 
let $\rho_E$, $\rho_V$ be Riemannian metrics on $E$, $V$ respectively and $J\in 
\mathcal{J}_{\mathcal{W}}$ an almost complex structure.
Let $\phi_t^h$ and $\phi_t^f$ denote the flow for the negative gradient of $h$ and $f$ 
respectively. 
For $x\in Crit(f)$, $a\in Crit(h)$ and $\lambda\in H^D_2(E,V)$ consider the space of all 
sequences $(u_1,\ldots,u_l)$ of every possible length $l\geq 1$ such that:
\begin{enumerate}
 \item $u_i:(D,\partial D)\to (E,V)$ is a $J$-holomorphic disk for $1\leq i \leq l$. 
All disks but possibly the $l$-th one are non-constant.
 \item $u_1(-1)\in W^u_x(f)$.
 \item For all $1\leq i\leq l-1$ there exists $0<t_i<\infty$ such that 
$\phi^f_{t_i}(u_i(1))=u_{i+1}(-1)$.
 \item $u_l(0)\in W^s_a(h)$.
 \item $[u_1]+\ldots +[u_l]=\lambda$.
\end{enumerate}
Two elements $(u_1,\ldots,u_l)$ and $(u_1',\ldots,u_{l'}')$ are considered equivalent if 
$l=l'$ and for all $1\leq i\leq l-1$ there exists $\sigma_i\in Aut(D)$ with 
$\sigma_i(-1)=-1$, $\sigma_i(1)=1$ and such that $u_i\circ \sigma_i=u_i'$.
After modding out by this equivalence relation we obtain the moduli space $\mathcal{P}_{inc}(x,a;\lambda;(h,\rho_E,f,\rho_V,J))$. 
\end{definition}
The virtual dimension of $\mathcal{P}_{inc}(x,a;A;(h,\rho_E,f,\rho_V,J))$ is 
$\delta_{inc}(x,a;\lambda)=|x|-|a|+\mu(\lambda)$.
Put
\begin{equation*}
\begin{array}{ccccc}
 i:& C^+((V,S);f,\rho,J)& 
\longrightarrow &
C^+((E,\partial^v E);h,\rho,J)\\
		& x & \longmapsto & \smashoperator{\sum\limits_{\begin{subarray}{c}
						a,\lambda \\
 						\delta_{inc}(x,a,\lambda)=0\\
						\end{subarray}}}		
\#_{\Z_2}\mathcal{P}_{inc}(x,a,\lambda;f,h) a t^{\overline{\mu}(\lambda)}\\
\end{array}.
\end{equation*}
\begin{proof}[Proof of Proposition~\ref{prop:inclusion}]
By Theorem~\ref{thm:moduli_as_in_closed_setting}, the proof given 
in~\cite{BC07} still works for the non-compact case.
To show $Q^+H_*(E,\partial^vE)$-linearity, the goal is to construct a chain homotopy 
\begin{equation}\label{eq:chain_homotopy_linearity_inclusion}
 \xi: C^+((V,S);f,\rho,J)\otimes C^+((E,\partial^vE);h,\rho_E;J)\to 
C^+((E,\partial^vE);h;\rho_E;J)
\end{equation}
between $i(-\star -)$ and $-\star i(-)$. For this we define a new moduli space 
$\mathcal{P}_{\mathcal{T}}(a,x,b)$ modeled over trees with two entry points, one labeled by 
$-\nabla f$ and the other labeled by $-\nabla h$, and one exit, labeled by $-\nabla h$.
For generic choices of the data the moduli space fulfils the 
necessary regularity conditions and a boundary description of the one-dimensional version 
of this moduli space yields the chain 
homotopy~\ref{eq:chain_homotopy_linearity_inclusion}.
As for the moduli space $\mathcal{P}_{mod}$ we need to perform some 
Hamiltonian perturbations to show the identity.

We briefly sketch the definition of the moduli space $\mathcal{P}_{\mathcal{T}}(a,x,b)$.
The trees describing the elements in $\mathcal{P}_{\mathcal{T}}(a,x,b)$ have only 
vertices of valence two, except one vertex with valence three. The vertices of valence 
two correspond to either a $J$-holomorphic disks with two marked points on the boundary, 
one entry and one exit of a flow line of $-\nabla f$, or it corresponds to a 
$J$-holomorphic disk with one marked point on the boundary (entry point of $-\nabla f$) 
and one marked point in the interior (exit point of $-\nabla h$).
The vertex of valence three may correspond to a $J$-holomorphic disk or to a 
$J$-holomorphic sphere. There are three possibilities for vertices of valence 
three. Firstly, a disk with two entry points on the boundary (one for $-\nabla f$ and one 
for $-\nabla h$) and one exit point on the boundary (also labeled by $-\nabla h$). 
Secondly, a disk with one entry point on the boundary (labeled by $-\nabla f$), an entry 
point in the interior of the disk (corresponding to a flow line of $-\nabla h$) and an 
exit point for $-\nabla h$ also in the interior of the disk.
As always, trivial disks and spheres are allowed as long as they are stable.
\end{proof}

\begin{remark}
 There also exists an inclusion
$$i_{V}^{E}:Q^+H_*(V,S)\to Q^+H_*(E),$$
and the proof is the same as above. Similarly, we can show that there exists a module 
action
$$\star:Q^+H_*(E)\otimes Q^+H_*(V,S)\to Q^+H_*(V,S),$$
and $i_V$ is a $Q^+H_*(E)$-module morphism.
\end{remark}

\begin{lemma}
Let $[V,\partial V]\in H_{n+1}(V,\partial V)$ be the fundamental class and by abuse of notation denote its image under the classical inclusion in $H_{n+1}(E,\partial^hE)$ also by $[V,\partial V]$.
Then
$$i_{(V,\partial V)}^{(E,\partial^v E)}(e_{(V,\partial V)})=[V,\partial V].$$
\end{lemma}

\begin{proof}
Let $m$ be the unique maximum of a Morse function $f$ with $-\nabla f$ pointing outwards along 
$\partial V$.
Consider an element in $\mathcal{P}_{inc}(m,a;\lambda;h,f,J)$ with 
$\delta_{inc}(m,a;\lambda)=|m|-|a|_h+\mu(\lambda)=0$ and assume that $\mu(\lambda)\neq 0$. 
Let $u_k$ be the disk corresponding to the last vertex in the tree. Reversing the arrow on 
the flow line from the last disk to the critical point $a$ and adding the unique flow line 
from $u_k(1)$ to the maximum $m$ yields an element in 
$\mathcal{P}_{mod}(m,a,m;\lambda;-h,f,J)$. As a critical point of $-h$ the index of $a$ is $|a|_{-h}=2n-2-|a|_h$. 
The latter moduli space then has dimension 
$\delta_{mod}(m,a,m;\lambda)=|m|+|a|_{-h}-|m|+\lambda-(2n+2)=-|a|_h+\mu(\lambda)$, which is equal to $-n-1$, since we 
assumed that $\delta_{inc}(m,a;\lambda)=0$. This is a contradiction and hence $\lambda$ must be zero.
\end{proof}

\begin{lemma}
 If $S=\emptyset$, then the inclusion map from Proposition~\ref{prop:inclusion} satisfies 
the property
\begin{equation}
 \langle h^*,i_{V}^E(x)\rangle=\epsilon_V(h\star x),
\end{equation}
for every $x\in Q^+H_*(V)$, $h\in Q^+H_*(E,\partial^v E)$ and where $\langle,\rangle$ denotes the 
Kronecker pairing. 
\end{lemma}

\begin{proof}
Let $x_0$ denote the minimum of $f$. (We may assume it to be unique by 
Lemma~\ref{lem:single_min}).
There exists a bijection between the moduli spaces 
$$b:\mathcal{P}_{inc}(x,a;f,h)\to \mathcal{P}_{mod}(a,x,x_o;f,-h).$$
Given a tree $T$ describing an element in $\mathcal{P}_{inc}(x,a)$ the tree describing 
the corresponding element in $\mathcal{P}_{mod}(a,x,x_0)$ is obtained by reversing 
the direction of the last edge in $T$ and then adding an edge to the last vertex of 
valence two (making it a vertex of valence three). This last edge ends at a vertex 
labeled by $x_0$. This description makes sense, because $x_0$ is the unique minimum and thus the point $u_k(1)$ of the last disk 
belongs to the unstable manifold of $x_0$. Also, if reversing the direction of the gradient 
flow of $-\nabla h$ corresponds to choosing the Morse function $-h$ instead of $h$.
Now by the definition of $\epsilon_V$ and the Kronecker pairing, the Lemma follows.
\end{proof}

\subsection{Duality}\label{sec:duality}

We start this section by introducing some notation from~\cite{BC07}. Let $(\mathcal{C},\partial)$ be a chain complex over $\Lambda^+$.
Let $C^{\odot}:=Hom_{\Lambda^+}((\mathcal{C};f,\rho,J),\Lambda^+)$ such that the dual of 
$x$ has index $|x^*|:=-|x|$ and the differential is given by
\begin{equation}\label{eq:adjointdiff}
 \langle \partial^*g,x\rangle:= \langle g,\partial x\rangle.
\end{equation}

We denote by $s^jC$ the $j$'th suspension of the complex $C$, i.e. $(s^jC)_k:=C_{k-j}$.
\begin{remark}\label{rmk:cohom} 
 Let $C^*$ denote the cochain complex, which is dual to $C$, namely it is defined by $C^k:= 
Hom_{\mathbb{Z}_2}(C_k,\mathbb{Z}_2)\otimes \Lambda^+$, where the grading is given by 
$|x^*|=|x|$ and the differential is the adjoint of the differential of $C$.
Notice that the chain complex $C^{\odot}$ is the same as $C^*$, only with opposite signs 
in the grading. 
Therefore, we have isomorphisms of the following form
\begin{equation*}
 H_k(s^{(n+1)}C^{\odot})\cong H_{k-(n+1)}(C^{\odot})\cong H^{(n+1)-k}(C^*).
\end{equation*}
In particular, we define $Q^+H^{(n+1)-k}(V,S)$ to be the $ k$'th cohomology of the cochain 
complex $C^+((V,S);f,\rho,J,\Lambda^+)^*$.
$$Q^+H^{(n+1)-k}(V,S):=H^k(C^+((V,S);f,\rho,J,\Lambda^+)^*).$$
\end{remark}

\begin{proposition}
Let $f'$ and $f$ be two Morse functions respecting the exit region $S$ 
and in generic position.
 There exist a degree preserving chain morphism
\begin{equation*}
\eta:C^+((V,S);f',\rho,J)\rightarrow 
s^{(n+1)}(C^+(V,\partial V \setminus S);-f,\rho,J))^{\odot},
\end{equation*}
which descends to an isomorphism in homology.
By remark~\ref{rmk:cohom} this induces an isomorphism
\begin{equation}
 \eta:Q^+H_k(V,S)\rightarrow Q^+H^{(n+1)-k}(V,\partial V \setminus S).
\end{equation}
The corresponding (degree $-(n+1)$) bilinear map
\begin{equation}
\begin{array}{ccc}
 \tilde{\eta}:Q^+H_k(V,S) \otimes Q^+H_{n+1-k}(V,\partial V \setminus S) & \rightarrow & \Lambda^+\\
x\otimes y & \mapsto & [\eta(x)(y)]
\end{array}
\end{equation}
coincides with $\epsilon_V (x *y)$.
Here, $\Lambda^+$ denotes the chain complex, with $\Lambda^+$ in degree zero, and zero otherwise 
and with trivial differential.
In particular, $\tilde{\eta}(x\otimes y)=0$ if $|x|+|y|\not= n+1$.
\end{proposition}

\begin{proof}
As always, by Theorem~\ref{thm:moduli_as_in_closed_setting} this proof is a generalization of the proof given in~\cite{BC07}.
Let $f$ be a Morse function respecting the exit region $S$, then $-f$ is 
a Morse function respecting the exit region $\partial V \setminus S$. 
We have a basis preserving isomorphism $\iota$ between 
$C^+((V,S);f,\rho,J,\Lambda^+)$ and $s^{(n+1)} C^+((V,\partial V \setminus 
S);-f,\rho,J,\Lambda^+)^{\odot}$ defined as follows.
It takes a critical point $x$ of $f$ of index $k$ and sends it to the same critical point, now 
seen as a critical point of $-f$.
As a generator in $s^{(n+1)} C^+((V,\partial V \setminus 
S);-f,\rho,J,\Lambda^+)^{\odot}$, the critical point $x$ has index 
$-(n+1-k)+(n+1)=k$. Moreover, $\iota$ is a chain morphism.
We then compose $\iota$ with the comparison morphism 
$\phi:C^+((V, S);f',\rho,J,\Lambda^+)\rightarrow C^+((V, S);f,\rho,J,\Lambda^+)$, 
which induces a canonical isomorphism in homology.
The composition $\eta:=\iota \circ 
\phi:C^+((V,S);f',\rho,J,\Lambda^+)\rightarrow s^{(n+1)} C^+((V,\partial V \setminus 
S);-f,\rho,J,\Lambda^+)^{\odot}$ induces an isomorphism
$Q^+H_k(V,S)\rightarrow Q^+H^{(n+1)-k}(V,\partial V \setminus S)$. This proves the first part 
of the theorem.

We prove the identity $\tilde{\eta}(x\otimes y)=\epsilon_V(x*y)$.
Instead of working with the comparison chain morphism $\phi:C^+((V,S);f)\to C^+((V,\partial V\setminus S);-f)$ that we consider above, it is enough to take any comparison morphism $\phi:C^+((V,S);f')\to C^+((V,\partial V\setminus S);f)$, where $f'$ is in generic position to $f$.
In fact, we could even work with any comparison chain morphism $\psi: C^+((V,S);f')\to C^+((V,\partial V\setminus S);f)$ and we will in the following define $\psi$ in a particularly useful way. For this we introduce yet another moduli space.
It is modeled on linear trees similar to the pearly trajectories, except 
that one edge corresponds to a marked point instead of a pseudo-holomorphic disk.
The linear tree connects a critical point $x$ of the Morse function $f'$ respecting 
the exit region $S$ to a critical point $y$ of a Morse function $g$ also 
respecting the exit region $S$.
All edges between $x$ and the marked point are labeled by the negative gradient flow 
lines of $f'$ and all edges between the marked point and $y$ are labeled by the 
negative gradient flow lines of $f$.
We call these moduli spaces $\mathcal{P}^{!}(x,y;f',g)$ and their virtual dimension is $|x|-|y|+\mu(\lambda)$. 
Counting the zero dimensional moduli 
spaces $\mathcal{P}^{!}(x,y;f',-f)$ gives defines the chain map $\psi$, which is chain homotopic to $\phi$.
Moreover, $<\psi(x),y>=<\phi(x),y>=\iota\circ \phi(x)(y)$.

Let $\tilde{f''}$ be a Morse function respecting the exit region $\emptyset$. 
In particular, we may assume that $\tilde{f''}$ has a unique minimum $m$ inside $V$.
Consider the moduli spaces of the type $\mathcal{P}_{prod}(x,y,m;f',-f,f'')$.
The zero-dimensional moduli spaces of this sort compute $\epsilon_V(x*y)$.
In dimension zero, the moduli spaces 
$\mathcal{P}_{prod}(x,y,m;f',-f,f'')$ and 
$\mathcal{P}^{!}(x,y;f',f)$ are in bijection. 
Indeed, if $dim \mathcal{P}_{prod}(x,y,m)=0$, the central disk with valence three is 
constant and there is a unique flow line of $-\nabla f''$ from this point to $m$.
\end{proof}

\subsection{Proof of Theorem~\ref{thm:longexact}}

In this section we prove the existence of a long exact sequence as in~\ref{thm:longexact}.
For $\mathbb{Z}_2[t]$ coefficients this was proved by~\cite{BC13}.
Recall that $R\in \mathbb{R}$ is a constant such that $V$ is cylindrical outside of $[-R,R]\times\mathbb{R}\subset \mathcal{W}^c\subset \mathbb{C}$.
Let $S$ be the union of some of the ends of $V_R:=V|_{\pi^{-1}([-R,R]\times\mathbb{R})}$, i.e. 
$$S= (\coprod_{i \in I_-} \{(-R,a_i^-)\}\times L_i^-) \cup (\coprod_{j\in J_+} 
\{(R,a_j^+)\}\times L_j^+),$$
where $I_-\subset \{1,\dots k_-\}$ and $J_+\subset \{1,\dots k_+\}$.

\begin{proposition}[\cite{BC13}]\label{prop:long_exact_sequence_z2}
 Assume that $\mathcal{R}=\mathbb{Z}_2$ and hence $\Lambda^+=\mathbb{Z}_2[t]$.
There exists a long exact sequence
\begin{equation*}
\xymatrix{
 \dots \ar[r]^{\delta} & Q^+H_*(S) \ar[r]^{i_*} & Q^+H_*(V)\ar[r]^{j_*}
& Q^+H_*(V,S)\ar[r]^{\delta} & Q^+H_{*-1}(S)\ar[r]^{i_*} & \dots
}
\end{equation*}
More generally, if $A$ and $B$ are collections of connected components of $\partial V$, such that $A\cap 
B=\emptyset$, then there exists a long exact sequence
\begin{equation*}
\xymatrix{
 \dots \ar[r]^{\delta} & Q^+H_*(A) \ar[r]^{i_*} & Q^+H_*(V,B)\ar[r]^{j_*}
& Q^+H_*(V,A)\ar[r]^{\delta} & Q^+H_{*-1}(A)\ar[r]^{i_*} &\dots
}
\end{equation*}
\end{proposition}

We define a special class of Morse functions on the cobordisms, which is used to prove the existence of the long exact sequence in~\ref{thm:longexact}.

\begin{definition}[see~\cite{BC13}]\label{def:f_adapted_S}
Take a small extension of $V_R$, namely 
$V_{R+\epsilon}:=V\cap(\pi^{-1}[-R-\epsilon,R+\epsilon])$. 
Denote by 
$$S_{R+\epsilon}:=(\coprod_{i \in I_-} \{(-R - \epsilon,a_i^-)\}\times L_i^-) \cup  
(\coprod_{j\in J_+} \{(R + \epsilon, a_j^+)\}\times L_j^+)$$ 
the corresponding union of connected components of $\partial V_{R+\epsilon}$.
Let $f: V_{R+\epsilon} \rightarrow \mathbb{R}$ be a Morse function on  
$V_{R+\epsilon}$, such that $-\nabla f$ is transverse to the boundary $\partial 
V_{R+\epsilon}$ and points out along $S_{R+\epsilon}$ and in along $\partial 
V_{R+\epsilon} \setminus S_{R+\epsilon}$.
Moreover, we want $f$ to fulfil the following properties.
\begin{equation*}
        \begin{array}{lll}
        f(t,a_j^+,p)=f_j^+(p)+\sigma_j^+(t) & 
\sigma_j^+:[R,R +\epsilon]\rightarrow \mathbb{R}, & p\in M, j=1,\dots, k_+, \\
	f(t,a_i^-,p)=f_i^-(p)+\sigma_i^-(t) & \sigma_i^-:[-R-\epsilon,-R 
]\rightarrow
\mathbb{R}, & p\in M, i=1,\dots, k_-, \\
       \end{array},
\end{equation*}
where $f_j^+:L_j^+\rightarrow \mathbb{R}$ and $ f_i^-:L_i^- \rightarrow \mathbb{R}$ are 
Morse functions on $L_j^+$ and $L_i^-$ respectively.
The functions $\sigma_j^+$ and $\sigma_i^-$ are also Morse, each with a unique critical 
point and satisfying
\begin{enumerate}
  \item [(i)] $\sigma_j^+(t)$ is a non-constant linear function for $t \in [R + 
\frac{3\epsilon}{4},R + \epsilon]$ and $\sigma_j^+(t)$ is decreasing in this interval 
if 
$j\in J_+$ and increasing if $j\in  \{1,\dots,k_+\}\setminus J_+$.
Furthermore $\sigma_j^+(t)$ has a unique critical point at $R + \frac{\epsilon}{2}$ of 
index $1$ if $i\in J_+$ and of index $0$ if $j\in \{1,\dots,k_+\}\setminus J_+$
  \item [(ii)] $\sigma_i^-(t)$ is a non-constant linear function for $t \in [-R 
-\epsilon, -R - \frac{3\epsilon}{4}]$ and $\sigma_i^-(t)$ is increasing in this interval 
if $i\in I_-$ and decreasing if $i\in  \{1,\dots,k_-\}\setminus I_-$.
Furthermore, $\sigma_i^-(t)$ has a unique critical point at $-R - \frac{\epsilon}{2}$ of 
index $1$ if $i\in I_-$ and of index $0$ if $i\in \{1,\dots,k_-\}\setminus I_-$
\end{enumerate}
We call such a function on $V_{R+\epsilon}$ a Morse function adapted to the exit region 
$S_{R+\epsilon}$. By abuse of notation we also say that $f$ is a \emph{Morse function on $V$ adapted to the exit region $S$}.
Let $N$ be the neighbourhood of $S_{R+\epsilon}$ given by
\begin{equation*}
 N:=\coprod_{i \in I_-} [-R -\epsilon,  -R -3\epsilon/8]\times \{a^-_i\} \times 
L_i^-\coprod_{j\in J_+} [R+3\epsilon/8,  R +\epsilon]\times \{a^+_j\} \times L_j^+.
\end{equation*}
Then the set $U$ is defined by $U:=V_{R+\epsilon}\setminus N$.
\end{definition}
\begin{figure}[H]
 \centering
\includegraphics{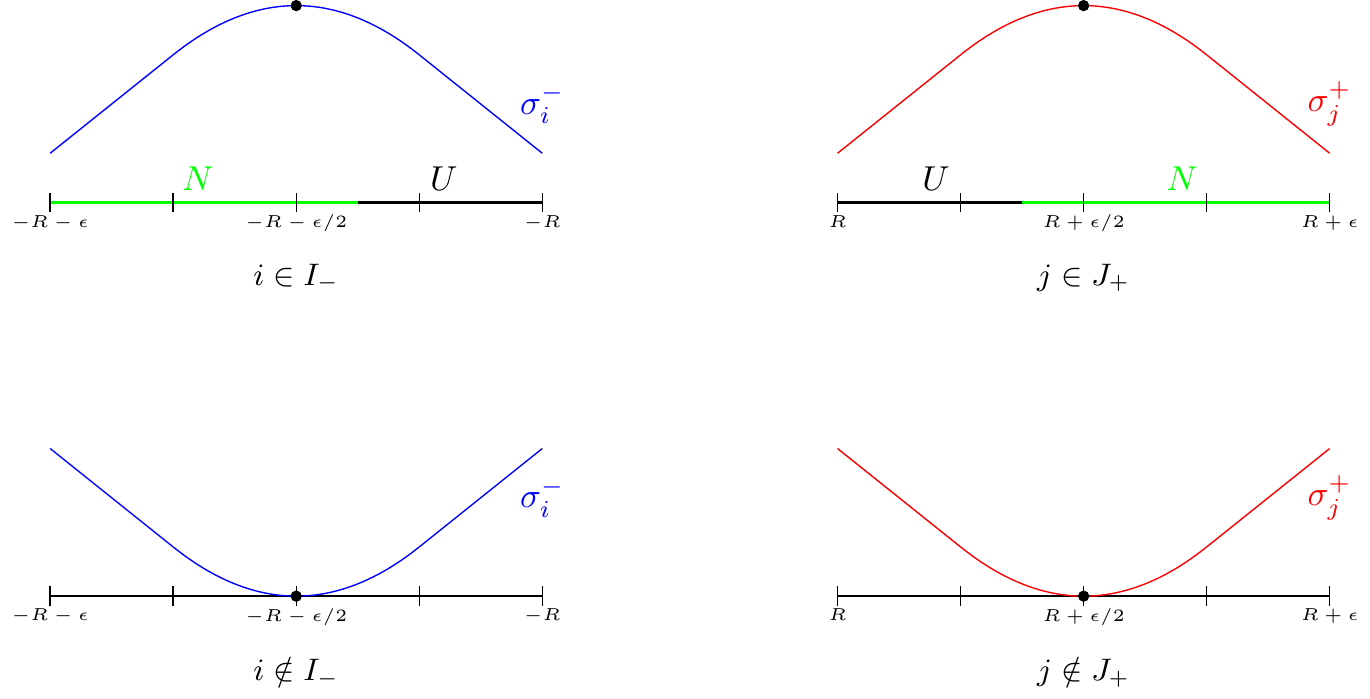}
\caption{The functions $\sigma_i^-$ and $\sigma_j^+$.}\label{fig:MFadaptedS}
 \end{figure}
 
If we take coefficients in $\Lambda^+:=\mathbb{Z}_2[t]$ this Morse function implies the existence of the following short exact sequence of chain complexes, which is used to prove~\ref{prop:long_exact_sequence_z2}. (See~\cite{BC13})

 \begin{equation}\label{eq:short_exact_chain_Z}
  \scalebox{0.9}{
\xymatrix{
0 \ar[r] & C^+_k(U;f|_U,J) \ar[r]^-j \ar[d]^{d_U} & 
C^+_k((V_{R+\epsilon},S_{R+\epsilon});f,J) \ar[r]^-{\delta} \ar[d]^{d_{(V_{R+\epsilon},S_{R+\epsilon})}} & 
C^+_{k-1}(S_{R+\epsilon/2};f|_{S_{R+\epsilon/2}},J) \ar[d]^{d_{S_{R+\epsilon/2}}} \ar[r] & 
0\\
0 \ar[r] & C^+_{k-1}(U;f|_U,J) \ar[r]^-{j} & 
C^+_{k-1}((V_{R+\epsilon},S_{R+\epsilon});f,J) \ar[r]^-{\delta} & 
C^+_{k-2}(S_{R+\epsilon/2};f|_{S_{R+\epsilon/2}},J) \ar[r] & 0. \\
}}
\end{equation}
\begin{remark}
 Notice that the connecting homomorphism in homology coming from the short exact sequence 
of chain complexes~(\ref{eq:short_exact_chain_Z}) gives us the inclusion map $i_*:Q^+H_*(S)  
\rightarrow Q^+H(V)$.
In singular homology the connecting homomorphism is the one between the homologies 
$H_*(V,S)$ and $H_{*-1}(S)$ and it is often denoted by $\delta$.
However, in the case of quantum homology the map $\delta$ is not the connecting 
homomorphism.
\end{remark}

In order to generalize this proof to arbitrary rings we need to bear in mind the orientations of the moduli spaces. 
The map $j:C^+_*(U;f|_U,J) \hookrightarrow C^+_*((V,S);f,J)$ is just an inclusion of subcomplexes and hence the orientations of the moduli spaces behave well. For the map $\delta$ this is not obvious and shall be investigated below. 

\paragraph{Orientations of the moduli spaces on the boundary of 
\texorpdfstring{$V$}{V}}\label{sec:orientation}

Working with rings of characteristic other than two the moduli spaces need to be endowed with orientations. 
Recall that the orientations of the moduli spaces of pseudo-holomorphic disks are defined 
using a fixed spin structure on the Lagrangian submanifold. For a precise definition we refer to~\cite{McS12}.
Let $\mathcal{M}(\lambda;J;(E,V))$ be the moduli spaces of 
$J_E$-holomorphic disks of class $\lambda$ in $(E,V)$ and endow it with the orientation induced by the spin structure on $V$.
Given a spin structure on the Lagrangian cobordism $V$ there is a canonical way to define 
a spin structure on its boundary components.
For a description the reader is referred to~\cite{LM89}.
Let $\mathcal{M}(\lambda;J;(M,L))$ be the moduli space of $J_M$-holomorphic disks of class $\lambda$ in 
$(M,L)$ and endow $\mathcal{M}(\lambda;J;(M,L))$ with the orientation coming from the spin structure on $L$.
The space $\mathcal{M}(\lambda;J;(M,L))$ is also oriented as a submanifold of $\mathcal{M}(\lambda;J;(E,V))$.
To establish the short exact sequence of chain complexes in Proposition~\ref{prop:long_exact_sequence} we have to 
show that these two orientations match.
\begin{lemma}\label{lem:orientation_match}
 The orientation of $\mathcal{M}(\lambda;J;(E,V))$ restricted to the set 
of $J$-holomorphic curves contained in $(M,L)$ is the same as the orientation of 
$\mathcal{M}(\lambda;J;(M,L))$ coming from the spin structure on $L$.
\end{lemma}
In order to prove this lemma we need to briefly recall the construction of the 
orientation 
of $\mathcal{M}(\lambda;J;(E,V))$.
\begin{proposition}[\cite{FOOO09}]\label{prop:indexbundle}
 Let $E$ be a complex vector bundle over a disk $D^2$. Let $F$ be a totally real 
subbundle 
of $E|_{\partial D^2}$ over $\partial D^2$. 
We denote by $\overline{\partial}_{(E,F)}$ the Dolbeault operator on $D^2$ with 
coefficients in $(E,F)$,
\begin{equation}
 \overline{\partial}_{(E,F)}:W^{(1,p)}(D^2,\partial D^2;E,F)\rightarrow L^{p}(D^2;E).
\end{equation}
Assume $F$ is trivial and take a trivialization of $F$ over $\partial D^2$.
Then the trivialization induces a canonical orientation of the index bundle $Ker 
\overline{\partial}_{(E,F)}-Coker \overline{\partial}_{(E,F)}$.
\end{proposition}
If a Lagrangian submanifold $L\subset M$ is oriented then $(u|_{\partial D^2})^*TL$ is 
trivial.
Indeed, $TL$ is trivial over the $1$-skeleton of $V$ and by a cellular approximation 
argument we may assume that $TL$ is trivial over the image of $u|_{\partial D^2}$. 
Thus, we can apply proposition~\ref{prop:indexbundle} to the case 
$(E,F)=(u^*TM,(u|_{\partial D^2})^*TL)$.
This gives us a pointwise orientation of the index bundle of the Dolbeault operator 
$\overline{\partial}_{(E,F)}$.
Since the orientation of the index bundle of $\overline{\partial}_{(E,F)}$ depends only 
on the trivialization of $(u|_{\partial D^2})^*TL$, it therefore depends only on the spin 
structure on $L$.
In~\cite{FOOO09} it is shown that with a choice of a spin 
structure on $L$ the pointwise
orientation from the proposition~\ref{prop:indexbundle}
can be extended in a unique way to give a consistent orientation of the index bundle.
The orientation of the index bundle then induces an orientation of the determinant bundle 
of the linearized operator
$D\overline{\partial}_u$ of the $J$-holomorphic curve equation.
The determinant bundle of a Fredholm operator is defined as
\begin{equation*}
 det(D \overline{\partial}_u):=det(Coker D\overline{\partial}_u)^*\otimes det( Ker D 
\overline{\partial}_u).
\end{equation*}

\begin{proof}[Proof of Lemma~\ref{lem:orientation_match}]
By the construction of the orientation of the moduli spaces, 
(which is described in~\cite{FOOO09},) it suffices to show that the orientations of 
the determinant line bundles of the linearized operators coincide.
Let $\tilde{u}\in \mathcal{M}(\lambda;J;(E,V))$ be a 
$J$-holomorphic disk that is mapped to $\mathcal{W}\times M$. (Compare with section~\ref{sec:tools}.) 
In particular, we may assume that $J=i\oplus J$ and $\pi(\tilde{u})$ is a constant.
Then the linearization of the operator $D \overline{\partial}_{\tilde{u}}$ splits into $D
\overline{\partial}_{u_{\mathbb{C}}}\oplus D\overline{\partial}_{u}$, where $u\in 
\mathcal{M}(\lambda;J;(M,L))$ and $u_{\mathbb{C}}$ is a constant curve in $\mathbb{C}$.
We compute the determinant line bundle:
\begin{equation*}
\begin{split}
 &det \left( D \overline{\partial}_{\tilde{u}}\right) = det\left(Coker 
D\overline{\partial}_{\tilde{u}}\right)^*\otimes det\left( 
KerD\overline{\partial}_{\tilde{u}}\right)\\
 &= det\left(Coker D\overline{\partial}_{u_{\mathbb{C}}}\oplus Coker 
D\overline{\partial}_u\right)^*\otimes det\left(Ker 
D\overline{\partial}_{u_{\mathbb{C}}}\oplus Ker D\overline{\partial}_u\right)\\
 &= det\left( Coker D\overline{\partial}_u\right)^*\otimes det\left(Coker 
D\overline{\partial}_{u_{\mathbb{C}}}\right)^* \otimes
 det\left(Ker D\overline{\partial}_{u_{\mathbb{C}}}\right)\otimes det\left( Ker 
D\overline{\partial}_u\right)\\
 &= det\left(Coker D\overline{\partial}_{u_{\mathbb{C}}}\right)^* \otimes
 det\left(Ker D\overline{\partial}_{u_{\mathbb{C}}}\right)\otimes det\left( Coker 
D\overline{\partial}_u\right)^*\otimes det\left( Ker D\overline{\partial}_u\right),\\
\end{split}
\end{equation*}
where the last step follows, since $D\overline{\partial}_u$ is surjective and therefore 
the cokernel is trivial.
Hence, we have the identity
\begin{equation*}
 det \left(D 
\overline{\partial}_{\tilde{u}}\right)=det\left(D\overline{\partial}_{u_{\mathbb{C}}}
\right)\otimes det\left(D \overline{\partial}_u\right).
\end{equation*}
Restricting the trivialization of $(\tilde{u}|_{\partial D^2})^*TV$ to a trivialization 
of $(u|_{\partial D^2})^*TL$, implies that an orientation of the determinant bundle of $D 
\overline{\partial}_{\tilde{u}}$ induces an orientation of $D \overline{\partial}_u$.
But this trivialization is exactly the trivialization induced by the spin structure on 
$L$. 
Therefore, the two orientations are equivalent.
\end{proof}

Let $e_1:U\rightarrow X$ and $e_2: V\rightarrow X$ be smooth maps between oriented 
manifolds, that are transverse.
We denote by $U\sideset{{}_{e_1}}{{}_{e_2}}{\mathop{\times}}V:=\{(u,v)\subset U\times 
V|e_1(u)=e_2(v)\}$ the fiber product of $U$ and $V$ with the orientation induced by the 
orientations of $U$, $V$ and $X$.

\begin{remark}
Let $e_1:U\rightarrow X$ and $e_2: V\rightarrow X$ be smooth maps and $X$ a submanifold 
in 
$Y$. 
Denote by $i:X\rightarrow Y$ the inclusion and write $e_1':U\rightarrow Y$ and $e_2':V 
\rightarrow Y$ for the compositions $i\circ e_1$ and $i\circ e_2$ respectively.
Then:
\begin{equation}\label{eq:orientationFiberProduct}
 U\sideset{{}_{e_1}}{{}_{e_2}}{\mathop{\times}}V=
U\sideset{{}_{e_1'}}{{}_{e_2'}}{\mathop{\times}}V.
\end{equation}
\end{remark}

\begin{corollary}\label{cor:orientationsmoduli}
There are two ways to orient moduli spaces modeled over planar trees in $\partial V$.
One is induced by the spin structure on $V$ and the other one is induced by the spin 
structure on $L$.
These two orientations are the same.
\end{corollary}

\paragraph{The long exact sequence}

\begin{proposition}\label{prop:long_exact_sequence}
Let $V$ be an orientable and spin cobordism. 
Let $\mathcal{R}$ be a commutative unital ring not necessarily with characteristic two. We work with $\Lambda_{\mathcal{R}}^+=\mathcal{R}[t]$ coefficients.
There exists a long exact sequence in homology
\begin{equation*}
\xymatrix{
 \dots \ar[r]^-{\delta} & Q^+H_*(S) \ar[r]^-{i_*} & Q^+H(V) \ar[r]^-{j_*}
& Q^+H_*(V,S) \ar[r]^-{\delta} & Q^+H_{*-1}(S)\ar[r]^-{i_*} &\dots
}
\end{equation*}
\end{proposition}

\begin{proof}
 Let $f:V\to \mathbb{R}$ be a Morse function adapted to the exit region $S$.
 It is easy to see that $C^+_*(U,f|_U,J)$ is a subcomplex of $C^+_*((V,S;f,J)$ and hence the inclusion is a chain map. 
 The map $\delta$ is given by restriction and Lemma~\ref{lem:orientation_match} implies that it is also a chain map.
\end{proof}

\subsubsection{Properties of the long exact sequence and the 
quantum structures}\label{sec:properties_les}

We define a Morse function, which is a small modification of 
Definition~\ref{def:f_adapted_S}.
The functions $\sigma_j^+$ and $\sigma_i^-$ at the cylindrical ends of the cobordism 
corresponding to the set $S_{R+\epsilon}$ are defined with two instead of one critical 
points.

\begin{definition}[\cite{BC13}]\label{def:special_MF}
Let $f$ be a Morse function on $V_{R+\epsilon}$ as in 
Definition~\ref{def:f_adapted_S}, with the only difference that $\sigma_j^+$ and 
$\sigma_i^-$ satisfy
\begin{enumerate}
  \item [(i)] $\sigma_j^+(t)$ is a non-constant linear function for $t \in [R + 
\frac{3\epsilon}{4},R + \epsilon]$.
$\sigma_j^+(t)$ is decreasing in this interval if $j\in J_+$ and increasing if $j\in 
\{1,\dots,k_+\}\setminus J_+$.
Furthermore, if $j\in J_+$ then $\sigma_j^+(t)$ has a critical point $t_{j,1}^+:=R + 
\frac{\epsilon}{2}$ of index $1$, and a critical point $t_{j,0}^+:=R + 
\frac{\epsilon}{4}$ of index $0$.
If $j\in\{1,\dots,k_+\}\setminus J_+$ then $\sigma_j^+$ has a exactly one critical point 
$R + \frac{\epsilon}{2}$, which is of index zero.
  \item [(ii)] $\sigma_i^-(t)$ is a non-constant linear function for $t \in [-R 
-\epsilon, -R -\frac{3\epsilon}{4}]$.
$\sigma_i^-(t)$ is increasing in this interval if $i\in I_-$ and decreasing if $i\in 
\{1\dots,k_-\}\setminus I_-$.
Furthermore, if $i\in I_-$ then $\sigma_i^-(t)$ has a critical point $t_{i,1}^-:=-R - 
\frac{\epsilon}{2}$ of index $1$ and a critical point $t_{i,0}^-:=-R 
-\frac{\epsilon}{4}$ of index $0$. 
If $i\in \{1\dots,k_-\}\setminus I_-$ then $\sigma_i$ has a exactly one critical point 
$-R - \frac{\epsilon}{2}$ of index $0$.
\end{enumerate}
Now we define the sets $N$ and $U$ similarly as in Definition~\ref{def:f_adapted_S}.
\end{definition}

Throughout the rest of this section we will use the following notation.
On a cylindrical end $[-R -\epsilon, -R]\times \{a_i^-\}\times L_i^-$ corresponding to a component of
$S$ the critical points are of the form $(t_{i,1}^+,a_i^-,p)$ or $(t_{i,0}^-,a_i^-,p)$, where $p$ is a 
critical point of $f_i^-$ and $t^-_{i,1}=-R- \epsilon/2$, $t^-_{i,0}=-R- \epsilon/4$.
Similarly on a cylindrical end $[R, R +\epsilon]\times \{a_j^+\}\times L_j^+$ corresponding to a component of $S$ the 
critical points are of the form $(t_{j,1}^+,a_j^+,p)$ or $(t_{j,0}^+,a_j^+,p)$, where $p$ 
is a critical point of $f_j^+$ and $t^+_{j, 1}=R+\epsilon/2$ and $t^+_{j,0}=R+\epsilon/4$.
If we do not want to specify the connected component of $S$, we will sometimes write 
$(t_{\iota,0},a_\iota,p)$ and $(t_{\iota,1},a_{\iota},p)$, as well as $f_{\iota}$ and 
$L_{\iota}$ for $\iota\in I_- \cup J_+$.
The functions $\sigma_i^-$ and $\sigma_j^+$ are illustrated in figure~\ref{fig:specialMF}.
Notice that $S_{R+\epsilon/4}\subset U$.
\begin{figure}[H]
 \centering
\includegraphics[scale=0.9]{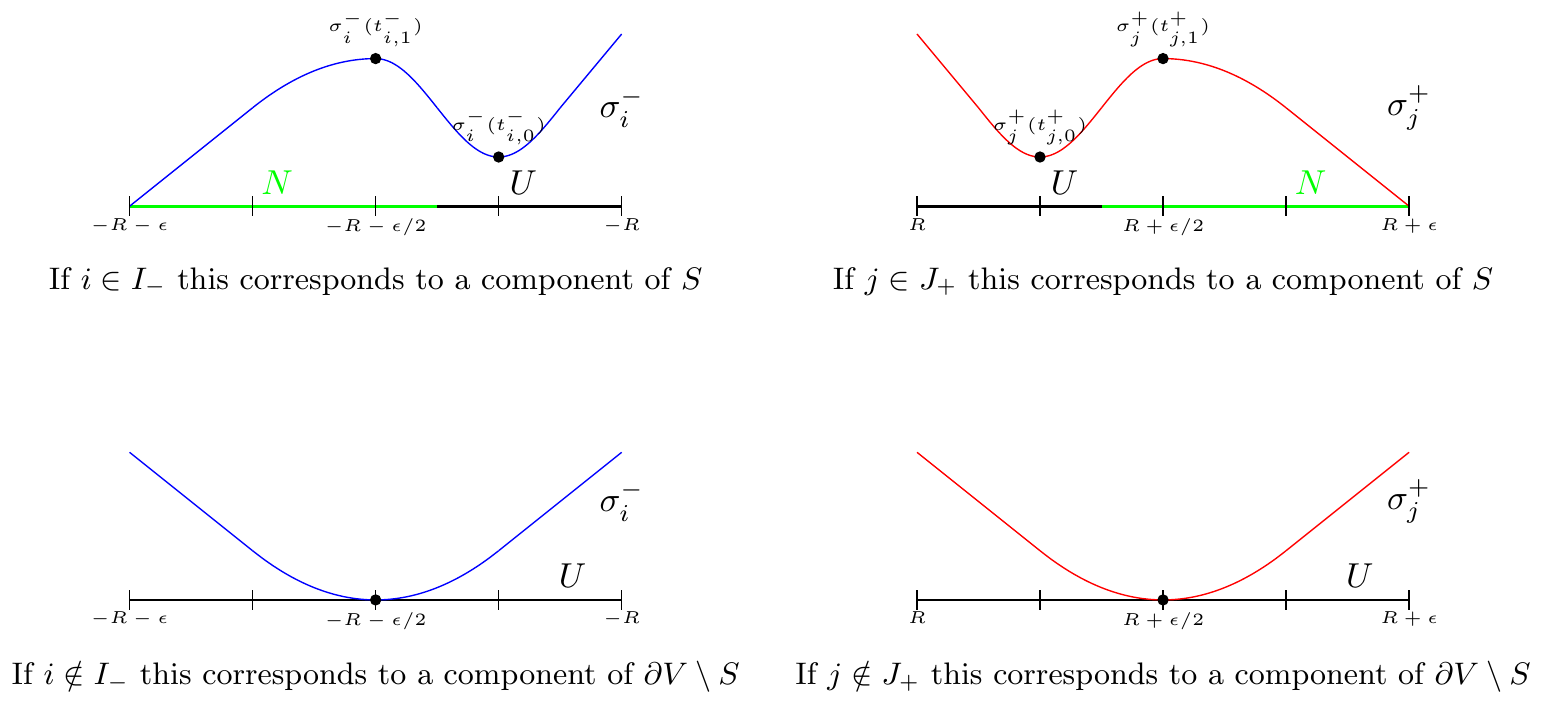}
\caption{The functions $\sigma_i^-$ and $\sigma_j^+$.}\label{fig:specialMF}
 \end{figure}
Consider the two sets $Crit(f)\cap S_{R+\epsilon/2}$ and $Crit(f)\cap 
S_{R+\epsilon/4}$, where $f$ is a function as given in 
Definition~\ref{def:special_MF}.
Then there is a natural identification between these two sets with a shift in degree by 
one and given by $(t_{\iota,1},a_{\iota},p) \mapsto (t_{{\iota},0},a_{\iota},p)$.
Note that the index of $(t_{{\iota},1},a_{\iota},p)$ is $|p|+1$ and the index of 
$(t_{{\iota},0},a_{\iota},p)$ is $|p|$.

\begin{lemma}\label{lem:sum_of_units}
Suppose that $S=\partial V$. Then $\delta(e_{(V,\partial V)})= \oplus_{i} -e_{L_i^-} 
\oplus_{j} e_{L_j^+}$.
\end{lemma}
\begin{proof}
We choose a Morse function $f$ as in Definition~\ref{def:special_MF} with 
$S=\partial V$.
Then $-\nabla f$ is still transverse to $V_{R+\epsilon}$ and points out along 
$\partial V_{R+\epsilon}$.
Suppose $f$ has a unique maximum $m$ inside the set $V|_{[-R,R]}$, which is 
possible by Proposition~\ref{lem:single_max}.
In addition, assume that each $f_i^{-}$ and $f_j^{+}$ are Morse functions on $L_i^{-}$ and 
$L_j^{-}$, both with unique maxima $m_i^{-}$ and $m_j^{+}$ respectively.
Then
$$\tilde{m}_{i,1}^{-}:=(t_{i,1}^-,a_i^-,m_i^{-}) \text{ and } 
\tilde{m}_{j,1}^{+}:=(t_{j,1}^+,a_j^+,m_j^{+})$$
 are also maxima of $f$.
We want to show that 
$$m+(-\sum\limits_i \tilde{m}_{i,1}^{-} +\sum\limits_j \tilde{m}_{j,1}^{+})$$
 is a cycle, and
$$e=[m]+[-\sum\limits_i \tilde{m}_{i,1}^{-} + \sum\limits_j \tilde{m}_{j,1}^{+}]\in Q^+H_*(V,\partial V).$$
Fix a critical point $\tilde{m}_{\iota,1}=(t_{\iota,1},a_{\iota},m_{\iota})\in\{ 
\tilde{m}_{i,1}^{-} , \tilde{m}_{j,1}^{+}\}_{i,j}$ and compute $d(\tilde{m}_{\iota,1})$.
By the choice of the $\sigma_j^+$ and $\sigma_i^-$ we know that any pearly trajectory 
from a critical point in $Crit(f)\cap S_{R+\epsilon/2}$ cannot get into the set 
$V|_{[-R, R]\times \mathbb{R}}$.
Hence, it has to end at a point in $Crit(f)\cap ({S_{R+\epsilon/2}}\cup 
S_{R+\epsilon/4})$.
Moreover, $m_{\iota}$ is a cycle in the complex 
$C^+(S_{R+\epsilon/2};f|_{S_{R+\epsilon/2}},J)$, as we saw in the proof of 
Lemma~\ref{lem:unit_rel}.
Therefore it suffices to consider pearly trajectories in $0$-dimensional moduli spaces 
from $\tilde{m}_{\iota,1}\in Crit(f)\cap S_{R+\epsilon/2}$ to a critical point 
$x:=(t_{\iota,0},a_{\iota},p)\in Crit(f)\cap S_{R+\epsilon/4}$.
And we have 
$$\delta_{prl}(\tilde{m}_{\iota},x,\lambda;f)=|\tilde{m}_{\iota}
|-|x|+\mu(\lambda)-1=0.$$
Outside of the set $V|_{[-R,R]}$ the Morse function $f$ is the sum of $f_k$ 
and $\sigma_k$, and thus we can project to the last factor to get a pearly trajectory of 
$f_{\iota}$ in $L_{\iota}$, going from $m_{\iota}$ to $p$.
Assume $p\neq m_{\iota}$.
Computing the dimension of moduli space of the projected pearly trajectory, we get 
$$\delta_{prl}(m_{\iota},p,\lambda;f_{\iota})=|m_{\iota}|-|p|+\mu(\lambda)-1=(|\tilde{m}_{
\iota}|-1)-|x|+\mu(\lambda)-1=-1,$$
 which is a contradiction.
Thus, $d(\tilde{m}_{{\iota},1})=\tilde{m}_{{\iota},0}$ and therefore 
$$d(-\sum\limits_i \tilde{m}_{i,1}^{-} + \sum\limits_j \tilde{m}_{j,1}^{+})= 
-\sum\limits_i 
\tilde{m}_{i,0}^{-} +\sum\limits_j \tilde{m}_{j,0}^{+}.$$
It is left to show that $d(m)=\sum\limits_i {\tilde{m}_{i,0}^{-}} - \sum\limits_j 
{\tilde{m}_{j,0}^{+}}$.
A computation shows that for $\mu(\lambda)\neq0$ 
$$\delta_{prl}(m,z,\lambda;f)=(n+1)-|z|+\mu(\lambda)-1\geq (n+1)-n+2>0.$$
Thus, $d(m)$ is equal to the Morse part of the differential, and hence $|z|=n$.
At each critical point of index $n$ two trajectories of the negative gradient of 
$f$ end.
If $z$ is one of the critical points of index $n$ on $S_{R+\epsilon/4}$, then one negative 
gradient is coming from $\tilde{m}_{\iota}$ for some $\tilde{m}_{\iota}\in 
\{\tilde{m}_{i,1}^{-},\tilde{m}_{j,1}^{+}\}_{i,j}$, and the other one therefore has to 
come from $m$, because it is the only maximum within $V|_{[-R, R]\times \mathbb{R}}$ 
and $-\nabla f$ is transverse to $\partial V|_{[-R, R]\times \mathbb{R}}$ and 
pointing outwards.
Clearly the moduli spaces $\mathcal{P}_{prl}(m,\tilde{m}_{{\iota},0};0)$ and 
$\mathcal{P}_{prl}(\tilde{m}_{\iota},\tilde{m}_{{\iota},0};0)$ must have opposite 
orientations.
If $z$ is contained in $V|_{[-R, R]\times \mathbb{R}}$, then there are two flow lines 
from $m$ to $z$, counted with opposite signs.
This proves that 
$$d(m)=\sum\limits_i \tilde{m}_{i,0}^{-} - \sum\limits_j \tilde{m}_{j,0}^{+}= 
-d(-\sum\limits_i \tilde{m}_{i,1}^{-} +\sum\limits_j \tilde{m}_{j,1}^{+})$$
 and thus 
$$e_{(V,\partial V)}=[m]+[-\sum\limits_i \tilde{m}_{i,1}^{-} + \sum\limits_j 
\tilde{m}_{j,1}^{+}].$$
Now it is easy to conclude
$$\delta(e_{(V,\partial V)})=\delta([m+(-\sum\limits_i \tilde{m}_{i,1}^{-} + 
\sum\limits_j \tilde{m}_{j,1}^{+})])=\sum\limits_i [-m_i^{-}] + \sum\limits_j [m_j^{+}]= 
\oplus_{i} -e_{L_i^-}\oplus_{j} e_{L_j^+}.$$
\end{proof}

\begin{lemma}\label{lem:delta_multiplicative}
 The map $\delta$ in Proposition~\ref{prop:long_exact_sequence} is multiplicative with 
respect to the quantum product $*$.
In fact, the chain map $\delta$ satisfies $\delta(x*y)=\delta(x)*\delta(y)$.
\end{lemma}
\begin{proof}
 Recall that the long exact sequence in Proposition~\ref{prop:long_exact_sequence} comes from the 
short exact sequence
\begin{equation*}
\scalebox{0.9}{
\xymatrix{
 0 \ar[r] & C^+(U;f|_U,J)\ar[r]^-{j} &
C^+((V_{R+\epsilon},S_{R+\epsilon});f,J)\ar[r]^-{\delta} &
C^+(S_{R+\epsilon/2};f|_{S_{R+\epsilon/2}},J)\ar[r] & 0. }
}
\end{equation*}
The map $\delta:C^+((V_{R+\epsilon},S_{R+\epsilon});f,J)\rightarrow 
C^+(S_{R+\epsilon/2};f|_{S_{R+\epsilon/2}},J)$ is given by restricting the critical 
points of $f$ to the critical points in $f|_{S_{R+\epsilon/2}}$.
Recall that we defined the product $x * y:= \smashoperator{\sum\limits_{\delta_{prod}(x,y,z,\lambda)=0}} \#_{Z_2} 
\mathcal{P}_{prod}(x,y,z,\lambda)  zt^{\overline{\mu}(\lambda)}$.
Note that for any point $z\in Crit(f)\cap S_{R+\epsilon/2}$ the space 
$\mathcal{P}_{prod}(x,y,z,\lambda)$ is non-empty if and only if $x,y\in 
Crit(f)\cap S_{R+\epsilon/2}$.
Indeed, the $J$-holomorphic curves which are not completely contained in $U$ have 
to be constant under $\pi$ by Lemma~\ref{lem:curve_openmapping}.
Moreover, a flow line of $-\nabla f$ can not go from $U$ to $S_{R+\epsilon/2}$ by the 
construction of $f$.
By Lemma~\ref{cor:orientationsmoduli} the signs of 
$\mathcal{P}_{prod}(x,y,z,\lambda;f)$ and 
$\mathcal{P}_{prod}(\delta(x),\delta(y),\delta(z),\lambda;f|_{S_{R+\epsilon/2}}))$ 
coincide and we have $\delta(x* y)= \delta(x)*\delta(y)$.
\end{proof}

\begin{lemma}\label{lem:prod_and_i}
 The product $*$ on $Q^+H_*(V)$ is trivial on the image of the inclusion $i_*$.
In other words, for any two elements $a$ and $b$ in $Q^+H_*(S)$ we have that 
$i_*(a)*i_*(b)=0$.
\end{lemma}
\begin{proof}
 Recall that the map $i$ is the connecting homomorphism of the following short exact 
sequence
\begin{equation*}
\scalebox{0.9}{
\xymatrix{
0 \ar[r] & C^+_k(U;f|_U,J) \ar[r]^-j \ar[d]^{d_U} & 
C^+_k((V_{R+\epsilon},S_{R+\epsilon});f,J) \ar[r]^{\delta} \ar[d]^{d_{(V_{R+\epsilon},S_{R+\epsilon})}} & 
C^+_{k-1}(S_{R+\epsilon/2};f|_{S_{R+\epsilon/2}},J) 
\ar[d]^{d_{S_{R+\epsilon/2}}} \ar[r] & 0\\
0 \ar[r] & C^+_{k-1}(U;f|_U,J) \ar[r]^-{j} & 
C^+_{k-1}((V_{R+\epsilon},S_{R+\epsilon});f,J) \ar[r]^{\delta} & 
C^+_{k-2}(S_{R+\epsilon/2};f|_{S_{R+\epsilon/2}},J) \ar[r] & 0 \\
}}
\end{equation*}
We choose a Morse function $f$ as in Definition~\ref{def:special_MF}.
In particular $-\nabla f$ is transverse to $\partial V|_{[-R, R]\times 
\mathbb{R}}$ and points out along $\partial V|_{[-R, R]\times \mathbb{R}}$.
Let $p\in C^+_{k-1}(S_{R+\epsilon/2};f|_{S_{R+\epsilon/2}},J)$ be a cycle.
We consider the critical point $x_{{\iota},1}:=(t_{{\iota},1},a_{\iota},p)\in 
Crit(f)\cap S_{R+\epsilon/2}$.
As in the proof of Lemma~\ref{lem:sum_of_units} we can show that 
$d(x_{{\iota},1})=x_{{\iota},0}$, where $x_{{\iota},0}:=(t_{{\iota},0},a_{\iota},p)\in 
Crit (f)\cap S_{\frac{\epsilon}{4}}$.
Moreover, $x_{{\iota},0}$ is a cycle for the complex $C^+((V,S);f,J)$.
Indeed, the critical point $p$ is a cycle of $f_k$ and $t_{{\iota},0}$ is a minimum of 
the function $\sigma_{\iota}$.
By the definition of the connecting homomorphism it follows that $i(p)=x_{{\iota},0}$.

It is left to show that the product on the chain level for $C^+(U;f|_U,J)$ 
is zero for any two critical points in $S_{R+\epsilon/4}$. The product is defined using 
three Morse functions $f$, $f'$ and $f''$ on $U$ in generic position.
In particular the stable and unstable manifolds of $f$ and $f'$ have to be 
transverse.
We may choose $f=f''$.
Assume that $f$ is a function as in Definition~\ref{def:special_MF}.
To define $f'$ we first fix $f_{\iota}^{\pm}$ such that 
${f'}_{\iota}^{\pm}$ is transverse to $f_{\iota}^{\pm}$ for all $\iota$. 
We also have to perturb ${\sigma'}^{\pm}_{\iota}$, namely assume that they are as in~\ref{def:special_MF} except that 
now the unique minimum is perturbed to lie at $R +\epsilon/8$ or $-R-\epsilon/8$ respectively.

Let $x\in Crit(f|_U)\cap S_{R+\epsilon/4}$ and $y\in Crit(\tilde{f'}|_U)\cap S_{R+\epsilon/8}$.
Any trajectory of $-\nabla f'$ starting at a critical point in $S_{R+\epsilon/8}$ 
stays in $S_{R+\epsilon/8}$, by the choice of $\sigma_i'$ and $\sigma_j'$. Similarly, the trajectory of
$-\nabla f$ stays in $S_{R+\epsilon/4}$.
By Lemma~\ref{lem:curve_openmapping} all the $J$-holomophic disks involved in a 
figure-$Y$ pearly trajectory have to map to a constant point under $\pi$.
Thus, the moduli space of figure-$Y$ pearly trajectories starting at $x$ and $y$ is empty 
and therefore their product is zero.
\end{proof}

\begin{lemma}\label{lem:j_and_prod}
We have the following relation
\begin{equation*}
 j_*(x*y)=j_*(x)*j_*(y).
\end{equation*}
Moreover, $j_*(Q^+H_*(V))\subset Q^+H_*(V,S)$ is a two-sided ideal.
\end{lemma}
\begin{proof}
 The map $j:C^+(U;f|_U, J)\rightarrow C^+((V,S);f,J)$ on the 
chain level is given by the inclusion. 
Any figure-$Y$ pearly trajectory in $C^+(U,f|_U, J)$ is also a trajectory in 
$C^+((V,S);f,J)$. 
By the construction of the function $f$ in Definition~\ref{def:f_adapted_S}, there 
exists no figure-$Y$ pearly trajectory starting at two points in $Crit(f|_U, )$ and ending at a point in $Crit(f)\setminus 
Crit(f|_U, J)$.

The second part of the lemma follows from the fact that on the chain level $j(x)*y$ as well as $y*j(x)$ must lie inside $C^+_*(U,f|_U,J)$ for all $y\in C^+_*((V,S),f,J)$ and for all $x\in im(j)=C^+_*(U,f|_U,J)$.
\end{proof}

\subsection{Trivial Lefschetz fibrations}\label{sec:QH_trivial_LF}

For the case of a trivial Lefschetz fibration, i.e. a fibration with no critical points, 
we have $E\cong \mathbb{C}\times M$. We also diskussed the theory for this case in~\cite{Sin15}.
In this case we can describe the quantum 
homology of $E$ in terms of the quantum homology of $M$ and the module structure of 
$Q^+H_*(V,\partial V)$ over $Q^+H_*(E,\partial^v E)$ becomes simpler.
These observations are listed below.

There are two natural ways to define the ambient manifold.
Let $c$ be big enough such that $V\subset \mathbb{R}\times (-c,c)$. Assume further that $R>0$ is such that $V$ is cylindrical outside of $[-R,R]\times \mathbb{R}$. Set 
$$H:=[-R-\epsilon,R+\epsilon]\times (-c,c)\subset\mathbb{R}^2$$
 and 
$$\square':=[-R-\epsilon,R+\epsilon]\times [-c,c] \subset\mathbb{R}^2.$$
Let $\square$ denote the set obtained from $\square'$ by smoothening 
its boundary. Define $\tilde{M}_H:=H\times M$ and $\tilde{M}_{\square}:=\square\times M$.
The aim of this section is to prove the following result
\begin{proposition}
 There exists a bilinear map
\begin{equation*}
 \star:Q^+H_i(\tilde{M}_{\square},\partial \tilde{M_{\square}}) \otimes Q^+H_j(V,\partial V) 
\rightarrow 
Q^+H_{i+j-(2n+2)}(V,\partial V),
\end{equation*}
which endows $Q^+H_*(V,\partial V)$ with the structure of a two-sided algebra over the 
unital quantum homology ring 
$Q^+H_*(\tilde{M}_{\square},\partial \tilde{M_{\square}})$.
\end{proposition}

\subsubsection{The ambient quantum homology}\label{sec:ambient_QH_rel}

Before we start with the construction of the module structure, we give the definition of 
the ambient quantum homology and explain their structures and relations.
Choose a symplectic form $\omega_{\tilde{M}_H}=(\omega_{\mathbb{C}}\oplus 
\omega_M)|_H$ and $\omega_{\tilde{M}_{\square}}=(\omega_{\mathbb{C}}\oplus \omega_M)|_{\square}$ 
respectively.
Let $J$ be an $\omega$-compatible almost complex structure on $\tilde{M}_H$ and 
$\tilde{M}_{\square}$ respectively and assume that the projection $\pi$ is $(J,i)$-holomorphic 
and that $V$ is cylindrical outside of $[-R,R]\times \mathbb{R}\times M$.
Let $h:\tilde{M}_H\rightarrow \mathbb{R}$ be a Morse function on $\tilde{M}$ such 
that $-\nabla h$ points out along the boundary $\partial \tilde{M}_H$ and 
in at $\partial \overline{\tilde{M}_H}\setminus \partial \tilde{M}_H$.
Similarly, let $g:\tilde{M}_{\square}\rightarrow \mathbb{R}$ be a Morse function on 
$\tilde{M}$ such that $-\nabla g$ points out at the boundary $\partial 
\tilde{M}_{\square}$.
For instance, we can take $h$ to be of the form $\tau_H+ g_M$, where $g_M$ is a Morse function on $M$ and $\tau_H$ is a Morse function on $\mathbb{R}^2$ 
with a single critical point of index $1$ inside $T$, and such that $-\nabla \tau_H$ 
points out of $\partial T$ and in of $\partial\overline{T}\setminus\partial T$.
In the same way we could choose $g$ to be of the form $\tau_{\square} + g_M$, where $g_M$ is a Morse function on $M$ and $\tau_{\square}$ is a Morse function on $\mathbb{C}$ with a single critical point of index $2$ inside $R$ such that $-\nabla 
\tau_{\square}$ points out along $\partial R$.

As we will see below the product on $Q^+H_*(\tilde{M}_H,\partial\tilde{M}_H)$ turns out to be zero. 
The product on $Q^+H_*(\tilde{M}_{\square},\partial \tilde{M}_{\square})$ can be described using the 
homology $Q^+H_{*-2}(M)$.
\begin{lemma}\label{prop:Kunneth}
 The quantum homology $Q^+H_*(\tilde{M}_{\square},\partial \tilde{M}_{\square})$ is isomorphic as a ring to 
the quantum homology of $M$ by a shift in degree.
More precisely,
\begin{equation*}
 \begin{array}{cccc}
Q^+H_*(\tilde{M}_{\square},\partial \tilde{M}_{\square})&\cong &Q^+H_{*-2}(M).\\
 \end{array}
\end{equation*}
The quantum homology $Q^+H_*(\tilde{M}_{\square},\partial \tilde{M}_{\square})$ is a unital ring.
The quantum product on $Q^+H_*(\tilde{M}_H,\partial\tilde{M}_H)$ is trivial.
Additively it is isomorphic to $Q^+H_{*-1}(M)$.
\end{lemma}

\begin{proof}[Proof]
The quantum homologies $Q^+H_*(\tilde{M}_H,\partial 
\tilde{M}_H)$ and $Q^+H_*(\tilde{M}_{\square},\partial\tilde{M}_{\square})$ are additively the same as the 
singular homologies tensored with $\Lambda^+$, i.e.  
$H_*(\tilde{M}_H,\partial \tilde{M}_H)\otimes \Lambda^+$ and 
$H_*(\tilde{M}_{\square},\partial\tilde{M}_{\square})\otimes \Lambda^+$ respectively.
By the K\"unneth isomorphism 
$$H_i(\tilde{M}_H,\partial\tilde{M}_H)\cong H_{i-1}(M)\otimes_{\mathcal{R}} 
H_1(T,\partial T)\cong H_{i-1}(M)\otimes_{\mathcal{R}} {\mathcal{R}} \cong H_{i-1}(M)$$ 
and $$H_i(\tilde{M}_{\square},\partial\tilde{M}_{\square})\cong H_{i-2}(M)\otimes_{\mathcal{R}} 
H_2(R,\partial R) 
\cong H_{i-2}(M)\otimes_{\mathcal{R}} {\mathcal{R}}
\cong H_{i-2}(M).$$
Therefore, the isomorphism in Lemma~\ref{prop:Kunneth} can be described as follows.
\begin{equation*}
 \begin{array}{ccc}
\Phi_{\square}:Q^+H_{*-2}(M) & \xrightarrow{\cong} & Q^+H_*(\tilde{M}_{\square},\partial \tilde{M}_{\square})\\
b & \mapsto & \square \times b\\
\end{array},
\end{equation*}
where $b$ is a homology class in $Q^+H_{*-2}(M)$ and $\square$ represents the generator of 
$H_2(R,\partial R)$. 
In particular, for every $\square \times b$ and $\square\times b'$ elements of 
$Q^+H_*(\tilde{M}_{\square},\partial \tilde{M}_{\square})$ we have that
\begin{equation}\label{Phi_square}
 \square\times(b*b')=\Phi_{\square}(b * b')= \Phi_{\square}(b) * \Phi_{\square}(b') = (\square \times b)* (\square\times b').
\end{equation}
For $\tilde{M}_H$ we have
\begin{equation*}
\begin{array}{ccc}
\Phi_H:H_{*-1}(M) & \xrightarrow{\cong} & H_*(\tilde{M}_H,\partial \tilde{M}_H)\\
a & \mapsto
 & I\times a\\
\end{array},
\end{equation*}
where $a$ denotes some homology class in $Q^+H_{*-1}(M)$, and $I$ denotes the generator of 
$H_*(T,\partial T)$, which is represented by the 
interval $[-R-\epsilon,R+\epsilon]\times \{ 0 \} \subset T$.

It is left to show that the first isomorphism $\Phi_{\square}$ 
respects the quantum product.
Choose the almost complex structure on $\tilde{M}_{\square}$ to be split, i.e. $J=i\oplus 
J_M$ for some almost complex structure $J_M$ on $M$.
Consider an element in the moduli space 
$\mathcal{M}_{prod}(x,y,z,\lambda;g,g',g'' )$.
The projection of the $J$-holomorphic sphere, corresponding to the interior vertex 
of the tree, is constant. 
Hence, this sphere is completely contained in one of the fibers of $\pi:T\times M 
\rightarrow T$, say $\pi^{-1}(t)$.
We may assume that $g=g''$ and that $g=g_M+ \tau$ and 
$g'=g_M'+ \tau'$. Here $g_M$ and $g_M'$ are Morse functions on $M$ and $\tau$ and 
$\tau'$ Morse functions on $T$, each with a unique maximum $\xi \in Crit(\tau)$ and 
$\xi'\in Crit(\tau')$ respectively.
In particular $\pi(x)=\pi(z)=\xi$.
Conclude that the part of the tree, corresponding to the negative gradient flow of 
$g$, also has to be completely contained in $\pi^{-1}(\xi)$.
We may assume that $\xi$ is not a critical point of $\tau'$, then the gradient flow line 
of $-\nabla g'$ projects under $\pi$ to the unique negative gradient flow of 
$\tau'$ connecting $\xi'$ to $\xi$.
Let $x=(\xi, a)$, $y=(\xi', b)$ and $z=(\xi, c)$. 
The projection $\pi$ induces a bijection between $\mathcal{M}_{prod}(x,y,z,\lambda)$ and 
$\mathcal{M}(a,b,c,\lambda)$.
This proves that the identity~(\ref{Phi_square}) holds.

For the proof of the second part of the proposition, consider two functions $h$ 
and $h'$ in general position.
Since the single critical point $\xi'$ of $\tau'$ has index one, we may choose the 
critical point $\xi$ of $\tau$ such that it is not contained in the unstable manifold of 
$\xi'$.
By this choice of $\tau$ and $\tau'$, the moduli space 
$\mathcal{M}_{prod}(x,y,z,\lambda;h,h',h)$ is empty. 
\end{proof}

It is easy to see that the inclusion $H_*(\tilde{M}_H,\partial\tilde{M}_H)\hookrightarrow 
H_*(\tilde{M}_{\square},\partial \tilde{M}_{\square})$ in singular homology is trivial.
The relation between the quantum homologies $Q^+H_*(\tilde{M}_H,\partial \tilde{M}_H)$ and 
$Q^+H_*(\tilde{M}_{\square},\partial \tilde{M}_{\square})$ is given by the following two corollaries.
\begin{corollary}
The inclusion map
\begin{equation*}
\begin{array}{ccc}
 Q^+H_*(\tilde{M}_H,\partial \tilde{M}_H) & \rightarrow & Q^+H_*(\tilde{M}_{\square},\partial 
\tilde{M}_{\square})\\
\end{array}
\end{equation*}
is trivial.
\end{corollary}
\begin{proof}
Clearly $I$ is zero inside $H_*(R,\partial R)$ which proves the corollary.
\end{proof}
\begin{lemma}
There exists a bilinear map (which we also denote by $\oast$ for simplicity)
\begin{equation}\label{eq:incl_ambient}
 \oast:Q^+H_*(\tilde{M}_H,\partial \tilde{M}_H)\otimes Q^+H_*(\tilde{M}_{\square},\partial 
\tilde{M}_{\square}) 
\rightarrow Q^+H_*(\tilde{M}_H,\partial \tilde{M}_H),
\end{equation}
which endows $Q^+H_*(\tilde{M}_H,\partial \tilde{M}_H)$ with the structure of a module over $Q^+H_*(\tilde{M}_{\square},\partial \tilde{M}_{\square})$.
Moreover, this product can be described as follows.
For every $(I\times a)\in Q^+H_*(\tilde{M}_H,\partial \tilde{M}_H)$ and $(\square\times b)\in 
Q^+H_*(\tilde{M}_{\square},\partial \tilde{M}_{\square})$ we have
\begin{equation}
 (I \times a) \oast (\square\times b)=(I\times (a\oast b)).
\end{equation}
\end{lemma}
\begin{proof}
Consider the moduli spaces $\mathcal{M}_{prod}(x,y,z,\lambda;h, 
g,h')$, which are defined in a similar way as the moduli spaces 
$\mathcal{M}_{prod}$ in the beginning of the section  with the only difference that the 
edges corresponding to the flow lines of $-\nabla h$, $-\nabla g$ and 
$-\nabla h'$ for Morse functions $h$ and $h'$ on $\tilde{M}_H$ 
and $g$ on $\tilde{M}_{\square}$.
We may assume that $h=h'$.
By the same argument as in the proofs earlier in this thesis we see that the flow line 
corresponding to $-\nabla h$ and the $J$-holomorphic sphere in the core all 
map to the same point under the projection $\pi$.
In other words, they are all contained in one fiber.
Suppose $h=\tau_K+ h_M$, $h=\tau_{\square} + h_M'$ and $x=(\xi, a)$, $y=(\xi', b)$ 
and $z=(\xi, c)$, where $\xi$ is the critical point of $\tau_K$ and $\xi'$ the critical 
point of $\tau_{\square}$.
There exists a unique flow line of 
$\tau_{\square}$ from $\pi(y)=\xi$ to the fiber containing the core and the flow lines of 
$-\nabla h$.
Hence, there exists a bijection between the elements in 
$\mathcal{M}_{prod}(x,y,z,\lambda;h,g,h')$ and the elements in 
the moduli space $\mathcal{M}_{prod}(a,b,c,\lambda;h_M,g_M',h_M)$, which proves the lemma.
\end{proof}

\subsubsection{The inclusion}

\begin{lemma}\label{lem:inc_delta_commute}
 The diagram
\begin{equation*}
\xymatrix{
  C^+_i(V, \partial V) \ar[r]^(0.45){i_{(V, \partial V)}} \ar[d]^{\delta}& 
C^+_{i-1}(\tilde{M}_H, \partial \tilde{M}_H) \ar[d]^{\delta}\\
 C^+_i(\partial V) \ar[r]^{i_{\partial V}} & C^+_{i-1}(\partial \tilde{M}_H)\\
}
\end{equation*}
is commutative, where $i_{\partial V}:=\oplus i_{L_i^+}\oplus i_{L_j^-}$.
\end{lemma}

\begin{proof}[Proof of lemma~\ref{lem:inc_delta_commute}]
We choose a Morse function $h_M$ on $M$ and a special Morse function $\tau$ on 
$\mathbb{R}^2$ with the property that $h:=\tau+ h_M$ fulfills the necessary 
requirements to define the quantum homology of $\tilde{M}_H$ (see 
section~\ref{sec:module_structure}).
Suppose the function $\tau$ has two critical 
points at $(-R - \epsilon/2,0)$ and $(R + \epsilon/2,0)$ of index $1$ and a unique critical point of index $2$ at $(0,0)$.
We may assume that the unstable manifolds of $(-R - \epsilon/2,0)$ and $(R + 
\epsilon/2,0)$ are vertical, such that all together the negative gradient of $\tau$ points 
out along the boundary of $T$ and in along its complement.
Let $f$ be a Morse function on $V$ adapted to the exit region $\partial V$. (See Defintion~\ref{def:f_adapted_S}.)
The map $\delta: C^+_i(\tilde{M}_H, \partial \tilde{M}_H) \rightarrow C^+_{i-1}(\partial 
\tilde{M}_H)$ is given by restricting of the generators of $h$ to the 
generators in the fiber of $(-R - \epsilon,0)$ and $(R + \epsilon,0)$.
For any $x\in Crit(f)\setminus Ker(\delta)$ and any $a\in Crit(h)\setminus 
Ker(\delta)$ the map $\delta$ induces a bijection between the elements in the space 
$\mathcal{P}_{inc}(\delta(x),\delta(a),\lambda;f,h)$ and 
$\mathcal{P}_{inc}(x,a,\lambda;f|_{S_{R+\epsilon/2}},h|_{(-R - 
\epsilon/2,0)\cap (R + \epsilon/2,0)})$.
\end{proof}

\subsubsection{Module structures and the long exact sequence}

Once again consider the long exact sequence
\begin{equation*}
\xymatrix{
 \dots \ar[r]^-{i_*} & Q^+H_*(V)\ar[r]^-{j_*} & 
Q^+H_*(V,\partial V) \ar[r]^-{\delta} & Q^+H_{*-1}(\partial V) \ar[r]^-{i_*} &\dots
}
\end{equation*}
Notice that each of the quantum homologies in this sequence is a module over a version of the ambient 
quantum homology.
We examine the compatibility of the quantum products with these module 
structures.
Consider the sequence
\begin{equation*}
\xymatrix{
 \dots \ar[r]^-{\Phi}& Q^+H_{*+2}(\tilde{M}_{\square},\partial \tilde{M}_{\square}) 
\ar[r]^-{ id} & Q^+H_{*+2}(\tilde{M}_{\square},\partial \tilde{M}_{\square}) \ar[r]^-{\Phi^{-1}} & 
Q^+H_*(M) \ar[r] & \cdots, \\
}
\end{equation*}
where $\Phi$ is the ring isomorphism $\Phi:Q^+H_*(M) \rightarrow 
Q^+H_{*+2}(\tilde{M}_{\square},\partial \tilde{M}_{\square}) : b \mapsto R\times b$.
\begin{lemma}
 The module structures of $Q^+H_*(\partial V)$,  $Q^+H_*(V)$ and $Q^+H_*(V,\partial V)$ over 
the ring $Q_*^+H(M)$ and $Q^+H_*(\tilde{M}_{\square},\partial \tilde{M}_{\square})$ respectively fulfil the following identities:
\begin{enumerate}
 \item[(i)] $i_*(a \star x)= \Phi(a) \star i_*(x)$, $\forall x\in Q^+H_*(\partial V)$ and $\forall a\in Q^+H_*(M)$.
 \item[(ii)] $j_*(a\star x)=a\star j_*(x),$ $\forall x\in Q^+H_*(V)$ and every $\forall a\in 
Q^+H_*(\tilde{M}_{\square},\partial \tilde{M}_{\square})$.
 \item[(iii)] $\delta(a\star x)=\Phi^{-1}(a)\star \delta(x)$, $\forall x\in 
Q^+H_*(V,\partial V)$ and $\forall a\in Q^+H_*(\tilde{M}_{\square},\partial \tilde{M}_{\square})$.
\end{enumerate}
i.e. the maps in the long exact sequence are module maps.
\end{lemma}

\begin{proof}
\begin{enumerate}
 \item[(i)]We choose a Morse function $f$ on $V$ as in 
Definition~\ref{def:special_MF} and a Morse function $g:=\tau_{\square} + g_M$ on $\tilde{M}$ as in Section~\ref{sec:module_structure}.
Let $\xi$ denote the unique critical point of $\tau_{\square}$. 
Let $(\tau_{\iota,1},a_{\iota},x)$ and $(\tau_{\kappa,1},a_{\kappa},y)\in Crit(f|_{\partial V_{R+\epsilon/2}})$ and $a\in Crit(g_M)$. 
We know that $i(y)=(\tau_{\kappa,0},a_{\kappa},y)\in Crit(f|_U)$ and $i(x)=(\tau_{\iota,0},a_{\iota},x)$. (See Lemma~\ref{lem:prod_and_i}.)
Any pearly trajectory in $V$ starting at the critical point in $Crit(f)\cap \partial V_{R+\epsilon/4}$ ends at a point in $Crit(f)\cap \partial V_{R+\epsilon/4}$ and thus there exists a bijection between $\mathcal{P}_{mod}(a,x,y;g_M,f|_{\partial V_{R+\epsilon/2}})$ and $\mathcal{P}_{mod}((\xi,a),(\tau_{\iota,0},a_{\iota},x),(\tau_{\kappa,0},a_{\kappa},y); g,f)$.
 \item[(ii)] 
For $x\in im(j)$ the moduli space 
$\mathcal{P}_{mod}(a,x,y,\lambda;g,f)$ is empty if $y \notin 
im(j)$.
 \item[(iii)]
Recall that the map $\delta$ on the chain level, is defined by restricting 
$Crit(f)$ to $Crit(f|_{S_{R+\epsilon/2}})$.
The expression $\delta(a\star x)$ is defined by counting the elements in the moduli 
spaces of the type $\mathcal{P}_{mod}(a,x,y,\lambda;g,f)$ with $x\in 
Crit(f)$, $a\in  Crit(g)$ and 
$y=(t_{{\iota},1},a_{\iota},q)\in Crit(f)\cap S_{R+\epsilon/2}$.
As before, the critical point $x$ has to lie in $S_{R+\epsilon/2}$, since otherwise there 
exists no pearly trajectory from $x$ to $y$ by the choice of $f$. 
Notice also that this space is empty whenever $y$ is not contained in the same connected 
component of the boundary as $x$.
Therefore, we can write $x:=(t_{{\iota},1},a_{\iota},p)$, where $p\in 
Crit(f|_{S_{R+\epsilon/2}})$.
Suppose $a=(\xi,a_M)$, then $\Phi^{-1}(a)\star \delta(x)$ we count elements in the moduli space of the 
form $\mathcal{P}_{mod}(a_M,p,q,\lambda;g,f|_{S_{R+\epsilon/2}})$.
Again, this space is empty if $p$ and $q$ are not contained in the same connected component of 
$\partial V_{R+\epsilon/2}$.
Clearly there exists a unique negative gradient flow line of $\tau_{\square}$ from $\pi(a)$ to 
$\xi=\pi(y)=(t_{{\iota},1},a_{\iota})$.
This shows that there is a bijection between the moduli space 
$\mathcal{P}_{mod}(a,x,y,\lambda;g,f)$ and 
$\mathcal{P}_{mod}(a_M,p,q,\lambda;g_M,f|_{\partial V_{R+\epsilon/2}})$, which proves the last part of the lemma.
\end{enumerate}
\end{proof}


\section{Discriminants and the quantum homology ring}\label{sec:disc_lagrangians}

One can define discriminants for rings in general. The discriminants of the quantum homology rings are invariants of the Lagrangian cobordism or of a closed Lagrangian. 
In the following we are interested in $QH_1(V,\partial V;\mathbb{Q})$, respectively $QH_0(L;\mathbb{Q})$.

\subsection{The ring $QH_1(V,\partial V;\mathbb{Q})$}

Assume that $V$ is a Lagrangian cobordism in a Lefschetz fibration fulfiling the assumptions of section~\ref{sec:quantum_homology}.
Moreover, we have:
\begin{assumption}[\cite{BM15}]\label{assum:for_discriminant_cobordism}
 \hspace{2em}
\begin{enumerate}
 \item $V$ is a connected, monotone, Lagrangian cobordism with $N_V |n$, where $n=dim(V)-1$.
 \item $V$ is oriented, spin and its ends are endowed with the spin structure induced by the spin structure on $V$.
\end{enumerate}
\end{assumption}
Set $\nu:=\frac{n}{N_V}$.
In particular, $et^{\nu}\in QH_1(V,\partial V;\mathbb{Q}[t])$, where $e$ denotes the unit in 
$QH_{n+1}(V,\partial V;\mathbb{Q}[t])$.
Moreover, for any elements $\alpha$, $\beta\in QH_1(V,\partial V;\mathbb{Q}[t])$ the product
$\alpha*\beta$ lies in $QH_{1-n}(V,\partial V;\mathbb{Q}[t])$, which is isomorphic to $QH_1(V,\partial V;\mathbb{Q}[t])t^{\nu}$.
In particular the quantum product defines a ring structure on $QH_1(V,\partial V;\mathbb{Q})$, which is obtained by setting $t=1$ on the chain level and inducing a grading $\mod N_V$.
We say that  $\mathcal{B}:=\{\gamma_0, \gamma_1,\ldots,\gamma_d\}$ is a $\mathbb{Z}$-basis of 
$QH_1(V,\partial V;\mathbb{Q})$ if it is maximally linearly independent and each $\gamma_i$ maps to an element of $\bigoplus_{*=1,N_V+1,\ldots, n+1} H_*(V,\partial V;\mathbb{Z})$ under the natural map
\begin{equation}\label{map:ttozero}
 \xymatrix{
QH_1(V,\partial V;\mathbb{Q}) \ar[r] & 
\bigoplus_{*=1,N_V+1,\ldots, n+1}H_*(V,\partial V; \mathbb{Q}).\\
 & \bigoplus_{*=1,N_V+1,\ldots, n+1} H_*(V,\partial V;\mathbb{Z}) \ar@{^{(}->}[u]}
\end{equation}
The quantum products on this basis of $QH_1(V,\partial V;\mathbb{Q})$ can be expressed by
\begin{equation}
 \gamma_i*\gamma_j=\sum_{k=0,\ldots,d} \mu_{ijk} \gamma_k,
\end{equation}
with $\mu_{ijk}\in \mathbb{Z}$.
Let $f=\sum_{ijk} \mu_{ijk} \gamma_i^*\otimes \gamma_j^* \otimes \gamma_k^* 
\in QH_1(V,\partial V;\mathbb{Q})^{*\otimes 3}$.

\begin{definition}
Call $f$ the tensor associated to the quantum product on $QH_1(V,\partial V;\mathbb{Q})$. 
Call $A=(\mu_{ijk})_{ijk}\in \mathbb{Z}^{(d+1)^3}$ the hypermatrix associated 
to the quantum product on $QH_1(V,\partial V;\mathbb{Q})$ and the basis $\mathcal{B}$.
\end{definition}

\begin{definition} 
Define $\Delta_{QH_1(V,\partial V;\mathbb{Q})}:=Det(A)$.
Equivalently, $\Delta_{QH_1(V,\partial V;\mathbb{Q})}:=\Delta_{X}(f)$,
where $X$ is the image of the Segre embedding 
$$S:\mathbb{P}(QH_1(V,\partial V;\mathbb{C})^*)^{\times 3}\hookrightarrow 
\mathbb{P}(QH_1(V,\partial V;\mathbb{C})^{*\otimes 3}).$$
\end{definition}

\begin{remark}
The hyperdeterminant, which is the $X$-discriminant of the Segre variety (see Appendix~\ref{appendix:hyperdets}) is defined uniquely up to sign if we require it to have integral coefficients and to be irreducible over $\Z$. 
In particular, the hyperdeterminant of the hypermatrix $A$ is also defined uniquely up to sign. 
\end{remark}

\begin{lemma}
The hyperdeterminant of format $(d+1)\times (d+1)\times (d+1)$ evaluated on 
$A:=(\mu_{ijk})_{ijk}$ is an invariant (up to sign) of $QH_1(V,\partial V;\mathbb{Q})$.
\end{lemma}

\begin{proof}
Let $\beta=\{\gamma_0,\cdots,\gamma_d\}$ and 
$\tilde{\beta}=\{\tilde{\gamma}_0,\cdots,\tilde{\gamma}_d\}$ be two $\mathbb{Z}$-bases of 
$QH_1(V,\partial V;\mathbb{Q})$ and $M=(\mu_{ijk})$ and $\tilde{M}=(\tilde{\mu}_{ijk})$ 
the corresponding tensors defined by
\begin{equation*}
 \tilde{\gamma}_i*\tilde{\gamma}_j=\sum_k \mu_{ijk} \tilde{\gamma}_k
\end{equation*}
\begin{equation*}
 \tilde{\gamma}_i*\tilde{\gamma}_j=\sum_k \tilde{\mu}_{ijk} \tilde{\gamma}_k;
\end{equation*}
A base change between two $\Z$-bases $\tilde{\beta}$ and $\beta$ is described by a matrix in $T\in GL(d+1,\mathbb{Z})$,
i.e. $T=T^{\beta}_{\tilde{\beta}}=((\tilde{\gamma_0})^{\beta}\cdots 
(\tilde{\gamma_d})^{\beta})$ and $det(T)=\pm 1$.
Clearly
\begin{equation*}
\begin{array}{ccc}
(\tilde{\gamma_i})^{\beta}&=&T (\gamma_i)^{\beta}\\
(\gamma_i)^{\tilde{\beta}}&=&T^{-1} (\tilde{\gamma_i})^{\tilde{\beta}}.\\
\end{array}
\end{equation*}
We use the relation
\begin{equation*}
\gamma_i^{\tilde{\beta}}*\gamma_j^{\tilde{\beta}}=(\gamma_i*\gamma_j)^{
\tilde{\beta}}.
\end{equation*}
The left-hand side is:
\begin{equation*}
 \begin{array}{ccc}
  \gamma_i^{\tilde{\beta}}*\gamma_j^{\tilde{\beta}}&=&(T^{-1}_{li}e_i)^T 
\tilde{M}_{lmk}(T^{-1}_{mj}e_j)(\tilde{\gamma}_k)^{\tilde{\beta}}.
 \end{array}
\end{equation*}
The right-hand side gives
\begin{equation*}
 \begin{array}{ccc}
(\gamma_i*\gamma_j)^{\tilde{\beta}} &=&(e_i^TM_{ijk}e_j(\gamma_k)^{\beta})^{
\tilde{\beta}}\\
&=& e_i^T M_{ijk}e_j (\gamma_k)^{\tilde{\beta}}\\
&=& e_i^T M_{ijk}e_j T^{-1}_{kl}(\tilde{\gamma}_l)^{\tilde{\beta}}.\\
\end{array}
\end{equation*}
This implies 
\begin{equation*}
 {T^{T}}^{-1}_{li}\tilde{M}_{lmk}T^{-1}_{mj}=M_{ijr}T^{-1}_{rk},
\end{equation*}
and therefore
\begin{equation*}
 \tilde{M}_{lmr}=T^{T}_{il}M_{ijr}T_{jm}T^{-1}_{rk}.
\end{equation*}
Hence, a change of the basis corresponds to an action by the element $(T^T,T,T^{-1})$.
Proposition~\ref{prop:SL-invariance} completes the proof.
\end{proof}

\subsection{The ring $QH_0(L;\mathbb{Q})$}

There is a similar notion for closed Lagrangians.
\begin{assumption}\label{assum:on_closed_Lagrangian}
Suppose $L$ is a Lagrangian and
 \begin{enumerate}
  \item $L$ is closed and monotone with minimal Maslov number $N_{L}$ and such that 
$N_{L}|n$, where $n=dim(L)$.
  \item $L$ is oriented. Moreover, we assume that $L$ is spin with a fixed spin structure.
\end{enumerate}
\end{assumption}

\begin{remark}
 If $V$ is a cobordism fulfiling assumption~\ref{assum:for_discriminant_cobordism} then its ends automatically satisfy assumption~\ref{assum:on_closed_Lagrangian}.
\end{remark}

For a Lagrangian $L$ as in~\ref{assum:on_closed_Lagrangian}, let $\{\xi_0,\ldots, \xi_{r}\}$ be a $\mathbb{Z}$-basis of $QH_0(L;\mathbb{Q})$.
Suppose
$$\xi_i* \xi_j= \sum_k \lambda_{ijk} \xi_k $$
and set
$$g:=\sum_k \lambda_{ijk} \xi_i^*\otimes\xi_j^*\otimes\xi_k^* \text{ and } B:=(\lambda_{ijk})_{ijk}.$$

\begin{definition}\label{def:disc_closed_Lagrangian}
Define $\Delta_{L}:=\Delta_{QH_0(L;\mathbb{Q})}=Det(B)$, or equivalently $\Delta_{L}:=\Delta_X(g)$,
where $X$ is the image of the Segre embedding 
$$S:\mathbb{P}(QH_0(L;\mathbb{C})^*)^{\times 3}\hookrightarrow \mathbb{P}(QH_0(L;\mathbb{C})^{*\otimes 3}).$$
\end{definition}

If we assume moreover 
\begin{assumption}\label{assum:rank2}
 $$\rank(QH_0(L;\mathbb{Q}))=2.$$
\end{assumption}
Then $\Delta_{QH_0(L;\mathbb{Q})}$ coincides with the definition of the discriminant given by Biran and Membrez in~\cite{BM15}.
\begin{lemma}
 The definition of the discriminant in~\ref{def:disc_closed_Lagrangian} of a Lagrangian 
subspace fulfilling assumptions~\ref{assum:on_closed_Lagrangian} and~\ref{assum:rank2} is equivalent to Definition~\ref{def:discriminant_old}. (See also~\cite{BM15}.)
\end{lemma}

\begin{proof}
We used Schl\"aflis method. (See Appendix~\ref{appendix:hyperdets} and chapter 14.4 in~\cite{GKZ94}.) Suppose that $\xi_0=e$ and $\xi_1$ is the lift of a point.
The singular locus $\nabla_{sing}\subset \nabla$ consists of $3\times3$ matrices with rank $\leq 1$. This has codimension $4$ and hence, $c(2,2)=4>3$.
By Theorem~\ref{thm:SchlaefliRatio} the hyperdiscriminant of the format $2\times2\times2$ evaluated on $f$ is the 
discriminant of the determinant of the matrix
\begin{equation*}
\left(
 \begin{array}{ccc}
  \xi_0* \xi_0 & \xi_0* \xi_1\\
  \xi_1 * \xi_0 & \xi_1 * \xi_1\\
 \end{array}\right).
\end{equation*}
Suppose that $\xi_1* \xi_1=\sigma \xi_1 + \tau \xi_0$.
Then this matrix is
\begin{equation*}
\left(
 \begin{array}{ccc}
  \xi_0 & \xi_1 \\
  \xi_1 & \sigma \xi_1 + \tau \xi_0 \\
 \end{array}\right),
\end{equation*}
and hence, the determinant is $-\xi_1^2+\sigma \xi_1 + \tau \xi_0 $ which has 
discriminant $\sigma^2+4\tau$. 
\end{proof}

\begin{lemma}
 Let $L$ be a Lagrangian as in~\ref{assum:on_closed_Lagrangian} with $\rank(QH_0(L;\mathbb{Q}))=2$ and $f$ the tensor associated to the quantum product 
on $QH_0(L;\mathbb{Q})$. Let $e\in QH_0(L;\mathbb{Q})$ denote the unit and $X$ be the corresponding Segre 
variety. 
Then the following are equivalent
\begin{enumerate}
 \item $\{f=0\}$ is tangent to $X$ at some point 
$p\in \{e\otimes e\otimes e\neq0\}$.
 \item $\Delta_L=0$ and 
$\frac{\partial \Delta_x}{\partial a_{000}}(f)=0$.
 \item There exists $x\in QH_0(L)$ such that $x*x=0$.
\end{enumerate}
\end{lemma}

\begin{proof}
\item[1. $\Rightarrow$ 2.]
If 
$f=\sum_{i_1,i_2,i_3=0,1}a_{i_1i_2i_3}\gamma_{i_1}^*\otimes\gamma_{i_2}^*\otimes\gamma_{i_3}^* $
Theorem~\ref{preliminaries-thm:smooth-point-of-the-hyperplane-section} implies that 
$\{f=0\}$ is tangent to $X$ at $(\frac{\partial \Delta_X}{\partial 
a_{i_1i_2i_3}}(f):\ldots:\frac{\partial \Delta_X}{\partial 
a_{i_1i_2i_3}}(f)).$
Then $\{f=0\}$ is tangent to $X$ at a point 
$p\in \{e\otimes e\otimes e\neq0\}$ if and only if $\frac{\partial 
\Delta_X}{\partial a_{000}}(f)=0$ and $\Delta_L=0$.
\item[ 2. $\Rightarrow$ 3.] 
Using the formula for the hyperdeterminant of format 
$2\times2\times 2$ described in Proposition~\ref{prop:det222}
\begin{equation*}
 \begin{split}
   \frac{\partial\Delta_X}{\partial a_{000}}(f)=2 
a_{000}a_{111}^2-2a_{001}a_{110}a_{111}-2a_{010}a_{101}a_{111}-2a_{011}a_{100}a_{111}+4a_{
011}a_{101}a_{110}\\
=2\sigma^2+4\tau. \\
 \end{split}
\end{equation*}
Now $\frac{\partial \Delta_X}{\partial a_{000}}(f)=0$ and $\Delta_L=0$ implies that 
$\sigma=\tau=0$. Hence, with $x:=\gamma_1$ it follows $x*x=0$.
\item[ 3. $\Rightarrow$ 1.]
If $x*x=0$ then $f=x^*\otimes e^* \otimes x^* + e^*\otimes x^* \otimes x^* + 
e^*\otimes e^*\otimes e^*$ and $\Delta_L=0$. One can see that $\{f=0\}$ is 
not tangent to $X$ at a point in 
$\{e \otimes e \otimes e\neq 0\}$. 
\end{proof}

\subsection{Discriminants of the ends of a cobordism}
Let $V:L_0 \rightarrow (L_1, \cdots, L_r)$ be a Lagrangian cobordism, with ends $L_0, \ldots, L_r$.
Throughout this section assume that $V$ fulfils~\ref{assum:for_discriminant_cobordism} (and thus its ends fulfil~\ref{assum:on_closed_Lagrangian}) and moreover  $\rank(QH_0(L_i;\mathbb{Q}))=2$ for all $i=0,\ldots,r$.
Denote by 
$$\delta:QH_1(V,\partial V;\mathbb{Q}) \rightarrow QH_0(\partial V;\mathbb{Q})\cong 
\oplus_{i=0,\ldots,r} QH_0(L_i;\mathbb{Q})$$
the boundary map and by 
$$\pi_i:\oplus_{i=0,\ldots,r} QH_0(L_i;\mathbb{Q})\rightarrow QH_0(L_i;\mathbb{Q})$$
the projection to the $i$-th end.
Let $\{ e_{L_i},p_i\}$ be a basis of $QH_0(L_i;\Q)$, where $e_{L_i}$ is the unit and $p_i$ 
is the lift of a point. 
In particular, if $p_i^2=\sigma_i p_i + \tau_i e_{L_i}$, then the 
discriminant of $L_i$ is $\Delta_{L_i}=\sigma_i^2+4\tau_i$.

\begin{lemma}\label{lem:lifts_of_path_basis}
Assume that $V$ is a cobordism as above.
If $\rank(QH_1(V,\partial V);\mathbb{Q})=r+1$,
then there exists a $\mathbb{Z}$-basis $\{e, \gamma_1, \ldots, \gamma_r \}$ 
of $QH_1(V,\partial V;\mathbb{Q})$ with the properties:
\begin{enumerate}
  \item $e$ is the unit of $QH_*(V,\partial V;\mathbb{Q})$.
  \item $\pi_i\circ\delta(\gamma_i)= p_i$ and $\pi_0\circ \delta(\gamma_i)=p_0$ for $i=1,\ldots,r$.
  \item If $i\neq j$, $i, j\in \{1,\cdots, r\}$ then $\pi_i\circ\delta(\gamma_j)= 0$.
\end{enumerate}
\end{lemma}

\begin{proof}
Consider the sequence 
$$\xymatrix{\ldots \ar[r] & QH_1(V,\partial V;\Q)\ar[r]^-{\delta} & QH_0(\partial V;\Q) 
\ar[r]^-{i_*} & QH_0(V;\Q)\ar[r] & \ldots}$$
If we show that $i_*(p_i)=i_*(p_0)$ for all $i$ then $p_0-p_i\in ker(i_*)=im(\delta)$. Choose 
$\alpha_i$ such that $\delta(\alpha_i)=p_0-p_i$ and the lemma is proven.
Choose a Morse function on $(V,\partial V)$ as in Definition~\ref{def:special_MF}, where 
the underlying Morse functions $f_i$ on $L_i$ may be chosen to be perfect. 
We also adopt the notation from Definition~\ref{def:special_MF}.
Let $\xi_i$ be the unique critical point of index zero of $f_i$. 
Then the class $p_i$ is represented by $\xi_i$, ie. $[\xi_i]=p_i$.
Let $x_{i,\epsilon}:=(t_{i,\epsilon},a_i,\xi_i)$ be the corresponding critical points of $f$ 
for $\epsilon=0,1$. Let $\partial_{L_i}$ denote the differential of the Morse complex 
corresponding to $f_i$ and $\partial_{(V,\partial V)}$ and $\partial_{V}$ the differential to the Morse 
complexes corresponding to $f$ and $f|_U$. 
Clearly, $\partial_{(V,\partial V)}(x_{i,1})=x_{i,0}$ and $\partial_V(x_{i,0})=0$. Moreover, 
$\delta(x_i)=\xi_i$ and $i(\xi_i)=x_{i,0}$. 
Since $i(\xi_i)=x_{i,0}$ is a cycle and has index zero it is either null-homologous or it lifts the class of a point in $QH_0(V;\Q)$ (under the augmentation). 
Suppose there exists a pearly 
trajectory from some critical point $y$ to $x_{i,0}$. Then 
$$0=|x_{i,0}|=|y|-1+\mu(\lambda),$$ which implies $\mu(\lambda)=0$ and $|y|=1$. We 
therefore only have to consider the classical part of the differential, which is the 
Morse differential. Since we already know from singular homology that $[x_{i,0}]=p\neq 
0$ is the class of a point, we conclude that $i_*(p_i)=[i(\xi_i)]=[i(\xi_0)]=i_*(p_0)$ 
and the lemma follows.
\end{proof}

We know from Lemma~\ref{lem:j_and_prod} that the image $I:=im(j_*)$ of the map $j_*$ in the exact sequence
$$\xymatrix{\ldots \ar[r] & QH_1(V;\Q) \ar[r]^{j_*} & QH_1(V,\partial V;\Q)\ar[r]^-{\delta} & QH_0(\partial V;\Q) 
\ar[r] &  \ldots}$$
is an ideal in $QH_1(V,\partial V)$. The map $\delta$ descends to the quotient
$$\overline{\delta}:QH_1(V,\partial V;\Q)/I \rightarrow QH_0(\partial V;\Q)$$
and is injective.

\begin{lemma}\label{lem:lifts_of_path_basis_quotient}
Assume that $V$ is a cobordism as above.
Then there exists a $\mathbb{Z}$-basis $\{e, \gamma_1, \ldots, \gamma_r \}$ 
of $QH_1(V,\partial V;\Q)/I$ with the properties:
\begin{enumerate}
  \item $e$ is the unit of $QH_*(V,\partial V;\Q)$.
  \item $\pi_i\circ\overline{\delta}(\gamma_i)= p_i$ and $\pi_0\circ \overline{\delta}(\gamma_i)=p_0$ for $i=1,\ldots,r$.
  \item If $i\neq j$, $i, j\in \{1,\cdots, r\}$ then $\pi_i\circ\overline{\delta}(\gamma_j)= 0$.
\end{enumerate}
\end{lemma}

\begin{proof}
 Same as the proof of~\ref{lem:lifts_of_path_basis}.
\end{proof}

\begin{example}
 Suppose $L_1$ and $L_2$ are Lagrangian spheres intersecting at one point and $V$ the 
cobordism obtained from Lagrangian surgery.
Take $\gamma_1:=\frac{1}{-3}(([L_1]+2[L_2])\times R )*e_{(V,\partial V)}$ and 
$\gamma_2:=\frac{1}{-3}(([L_2]+2[L_1])\times R )*e_{(V,\partial V)}$.
\end{example}

\begin{example}
A similar construction as in the previous example works also if $V$ is the cobordism 
obtained from Lagrangian surgery of many Lagrangian spheres. However, if there are more than three spheres, none of which has non-trivial intersection product with more than one other sphere. Then all discriminants vanish.
See Proposition~\ref{prop:special-intersection-graph-gromov-witten}.
\end{example}

Let $\{\gamma_e:=e, \gamma_1, \ldots, \gamma_r \}$ a basis 
of $QH_1(V,\partial V;\Q)/I$ as above. 
Define the vector space $W:= \Span_{\C}\{\gamma_e:=e, \gamma_1, \ldots, \gamma_r \}$ and 
$W_i:=\Span_{\C}\{\gamma_e, \gamma_i \}$.
We fix some notations. 
Suppose that 
\begin{equation}
 \gamma_i*\gamma_j=\sum_{k=e,1,\ldots,r} \mu_{ijk} \gamma_k 
\end{equation}
and set $A=(\mu_{ijk})$.
Let $f=\sum_{i,j,k\in\{e,1,\ldots,n\}} \mu_{ijk}\gamma^i\otimes \gamma^j\otimes 
\gamma^k\subset W^{*\otimes 3}$ be the tensor defined by the quantum 
product.
Similarly,  $f_{W_i}=\sum_{i,j,k\in\{e,i\}} \mu_{ijk}\gamma^i\otimes \gamma^j\otimes 
\gamma^k\subset W_i^{*\otimes 3}$.
Let $X_i$ denotes the Segre variety $\mathbb{P}(W_i)^{\times 3}\subset \mathbb{P}(W_i^{\otimes 3})$.

\begin{lemma}\label{lem:isombasistoQH}
The vector spaces $W_i$ and $QH_0(L_i;\Q)$ are isomorphic as rings, i.e. the product is preserved.
In particular $\Delta{L_i}=\Delta_{X_i}(f_{W_i})$.
\end{lemma}

\begin{proof}
 Fix an $i\in\{1, \cdots ,r \}$. For every $j\in\{0,\ldots,r\}$ apply the map $\pi_j\circ \overline{\delta}$ to the equation
 \begin{equation*}
 \gamma_i* \gamma_i=\sum_{k=e,1,\ldots,r} \mu_{iik} \gamma_k .
\end{equation*}
We get
$$p_i^2=\mu_{iii}p_i +\mu_{iie} e_{L_i}$$
and for $j\neq i,0$
\begin{equation}\label{eq:proj_vanish}
 0=\mu_{iij}p_j +\mu_{iie} e_{L_j}.
\end{equation}
This shows that $\mu_{iij}=0$ for all $i\neq j$ with $i,j\in \{1,\cdots, r\}$. (And if $r\geq 2$ then also $\mu_{iie}$ vanishes for all $i$.) In particular $$\pi_i\circ\overline{\delta}(\gamma_i^2)=\pi_i\circ\overline{\delta}(\gamma_i)^2$$ 
and
$\xymatrix{
  \pi\circ \overline{\delta}_i|_{W_i}:W_i \ar[r]& QH_0(L_i;\Q)}$
is an isomorphism  of algebras.
\end{proof}

In the following we prove Theorem~\ref{thmB}.
\begin{theorem}\label{thm:equal_discr_rank2}
Let $V$ be a be a spin, monotone Lagrangian cobordism in a Lefschetz fibration fulfiling assumption~\ref{assum:for_discriminant_cobordism} and such that $\rank(QH_0(L_i);\mathbb{Q})=2$ for all its ends $L_i$.
Then
$$\Delta_{L_i}=\Delta_{L_j} \text{ for all } i, j.$$
If in addition $r\geq 2$, then this number is a square.
\end{theorem}

\begin{proof}
We know from Lemma~\ref{lem:isombasistoQH} that 
$$\gamma_i^2=\mu_{iii}\gamma_i+\mu_{iie}\gamma_e.$$
Applying $\pi_0\circ\overline{\delta}$ and $\pi_i\circ\overline{\delta}$ yields
$$p_0^2=\mu_{iii}p_0+\mu_{iie}e_{L_0}$$
$$p_i^2=\mu_{iii}p_i+\mu_{iie}e_{L_i},$$
which implies $\Delta_{L_0}=\Delta_{L_i}$ for all $i$.

If $r\geq2$ then~\ref{eq:proj_vanish} proves that $\mu_{iie}=0$ for all $i$ and thus the discriminants are squares. 
\end{proof}

\begin{theorem}\label{thm:delta_equal_general}
 Let $V:L_1\to L_2$ be an elementary spin and monotone Lagrangian cobordism in a Lefschetz fibration. Assume that $\pi_1\circ 
\overline{\delta}$ and $\pi_2\circ \overline{\delta}$ are surjective and that $ker(\pi_1\circ 
\overline{\delta})=ker(\pi_2\circ \overline{\delta})$. 
Then
$$\Delta_{L_1}=\Delta_{L_2}.$$
\end{theorem}

\begin{proof}
Let $\{x_e:=e_{L_1},x_1,\ldots,x_2\}$ be a $\mathbb{Z}$-basis of $QH_0(L_1;\Q)$ and let 
$\alpha_i\in QH_1(V,\partial V;\Q)$ such that $\pi_1\circ\overline{\delta}(\alpha_i)=x_i$ for all $i$. 
Since $ker(\pi_1\circ \overline{\delta})=ker(\pi_2\circ \overline{\delta})$ and $\pi_2\circ \overline{\delta}$ 
is surjective, an easy calculation shows that $\{e_{L_2},\pi_2\circ \overline{\delta}(\alpha_i)\}$ forms a 
$\Z$-basis of $QH_0(L_2;\Q)$. 
Using the fact that $\overline{\delta}$ is multiplicative the theorem follows directly 
from the definition of the hyperdeterminants.
\end{proof}

\subsubsection{If the product respects a split structure}

Let $R^*$ be an associative algebra over $\mathbb{C}$, i.e. a ring over the field 
$\mathbb{C}$, such that it is also a vector space over $\mathbb{C}$. 
Assume further that $R^*=\mathbb{C} 
e^*\oplus V^*\oplus W^* $, where $e$ is the unit in $R$.
A particularly nice situation would be to assume that the product on $R^*$ has the 
following property:
\begin{assumption}\label{properties_for_splitting}
 \hspace{2em}
 \begin{enumerate}
 \item $\forall v, v' \in \mathbb{C}e^*\oplus V^*$ we have $v*v'\in \mathbb{C}e^*\oplus 
V^*$
 \item $\forall w, w' \in \mathbb{C}e^* \oplus W^*$ we have $w*w'\in \mathbb{C}e^*\oplus 
W^*$
 \item $\forall v\in V^*$ and $\forall w \in W^*$ the 
product vanishes, i.e. $v*w=0$.
\end{enumerate}
\end{assumption}
Let $f$ denote the tensor in $R^{\otimes 3}=(V\oplus W\oplus \mathbb{C}e)^{\otimes 3}$ 
representing the 
product structure in $R^*$ and let $f_V:=f|_{(V\oplus \mathbb{C}e)^{\otimes 3}}$ and 
$f_W:=f|_{(W\oplus \mathbb{C}e)^{\otimes 3}}$ be the restrictions. 
Let $X$, $X_V$ and $X_W$ denote the corresponding Segre-varieties.
Denote $\mathbb{P}:=\mathbb{P}(R^{*\otimes 3})$ and similarly 
$\mathbb{P}_V:=\mathbb{P}((V\oplus \mathbb{C}e)^{*\otimes 3})$ and 
$\mathbb{P}_W:=\mathbb{P}((W\oplus \mathbb{C}e)^{*\otimes 3})$.
Let $L$ denote the subspace of $R^{*\otimes 3}$ spanned by elements of the form $x\otimes y\otimes z$ with $x,y,z$ not all contained in either $V^*\oplus \mathbb{C}e^{*} \text{ or } W^*\oplus \mathbb{C}e^{*}$.

\begin{theorem}\label{thm:split_ring_structure}
Under the assumptions~\ref{properties_for_splitting}:
If the $X$-discriminants $\Delta_{X_V}(f_V)$ or 
$\Delta_{X_W}(f_W)$ vanishe, then also $\Delta_X(f)$ vanishes.
\end{theorem}

\begin{remark}
 Notice that assumption~\ref{properties_for_splitting} is equivalent to the assumption 
that $f\in (\mathbb{P}/L)^*$.
\end{remark}

\begin{proof}[Proof of Theorem~\ref{thm:split_ring_structure}]
 Suppose $\Delta_{X_V}(f_V)=0$. Then $\{f_V=0\}$ is tangent to $X_V$ at some point 
$p=x^1\otimes x^2\otimes x^3\in (\mathbb{C}\cdot e^*\oplus 
V^*)$. Let $A$ be the hypermatrix corresponding to the 
tensor $f_V$. By remark~\ref{rmk:equivalent_properties_to_deltazero}
this is equivalent to saying that $A(x^1\otimes x^2\otimes (\mathbb{C}\cdot e^*\oplus 
V^*))=A(x^1\otimes (\mathbb{C}\cdot e^*\oplus 
V^*)\otimes x^3)=A((\mathbb{C}\cdot e^*\oplus 
V^*)\otimes x^2\otimes x^3)=0$. 
By the properties~\ref{properties_for_splitting} this implies that also 
$A(x^1\otimes x^2\otimes (R^*))=A(x^1\otimes (R^*)\otimes x^3)=A((R^*)\otimes x^2\otimes 
x^3)=0$, which proves that $\{f=0\}$ is tangent to $X$ at $p$.
\end{proof}

%

\subsection{Examples}

\begin{lemma}[\cite{BM15}]\label{lem:relation_quantum_product_and_GW}
Let $L$ be a Lagrangian submanifold satisfying condition $(1)$ - $(3)$ of 
Assumption~\ref{assum:on_closed_Lagrangian}.
Let $c\in H_n(M;\mathbb{Z})$ be a class satisfying $\xi:=\sharp(c\cdot [L])\neq 0$. 
 Then the augmentation $\epsilon_L(c*c*e_L)$ is given by
$$\frac{1}{\xi^2}\sum_{A\in H_2(M), \langle c_1,A \rangle =\frac{n}{2}} GW_{A,3}(c,c,[L]).$$
\end{lemma}

Another useful fact about the quantum product of Lagrangian spheres is described in the 
following Lemma.
The idea of working with elements of the form $[L]*e_{L}$ was introduced 
in~\cite{BM15}. We will make use of such elements and similar constructions in 
what follows.

\begin{lemma}[\cite{BM15}]\label{lem:prodzero}
 Let $L_1$ and $L_2$ be two Lagrangian spheres of even dimension.
If $[L_1]\cdot[L_2]=0$ then ${L_1}*{L_2}=0$.
\end{lemma}

\begin{proof}
 Suppose $[L_1]\cdot[L_2]=0$.
Let $\overline{p_i}$ be a lift of a point in $L_i$ under the augmentation and $e_{L_i}$ 
the unit. The elements $\overline{p_i}$ and $e_{L_i}$ form a basis of $QH_0(L_i;\Q)$ 
and therefore we can write $[L_1]*e_{L_2}=\alpha_1 \overline{p_2} +\beta_1 e_{L_2}$ 
as well as $[L_2]*e_{L_1}=\alpha_2 \overline{p_1} +\beta_2 e_{L_1}$.
Applying the augmentation to both sides yields
\begin{equation*}
\begin{array}{ccc}
[L_1]\cdot[L_2]&=&\alpha_1 p_2 \\
\end{array}
\end{equation*}
\begin{equation*}
\begin{array}{ccc}
[L_2]\cdot[L_1]&=&\alpha_2 p_1 \\
\end{array}
\end{equation*}
hencec $\alpha_1=\alpha_2=0$ and $[L_1]*e_{L_2}=\beta_1 e_{L_2}$ 
and $[L_2]*e_{L_1}=\beta_2 e_{L_1}$.
Now we take the inclusions $i_{L_1}$ and $i_{L_2}$ of $[L_2]*e_{L_1}$ and 
$[L_1]*e_{L_2}$ respectively, which shows that 
\begin{equation*}
 \beta_1 [L_2]=[L_1]*[L_2]=[L_2]*[L_1]=\beta_2 [L_1].
\end{equation*}
Apply $[L_2]\cdot$ to both sides of the equation, which gives
$\beta_1 [L_2]\cdot [L_2]=\beta_2 [L_2]\cdot [L_1]=0.$
This shows that $\beta_1=0$. Switching the role of $L_1$ and $L_2$ proves that $\beta_2=0$.
\end{proof}

These lemmas motivate the study of properties of the three-point Gromov-Witten invariant 
associated to three Lagrangian spheres.

\begin{proposition}\label{prop:properties-of-gw}
Let $L_1$, $L_2$ and $L_3$ be Lagrangian spheres of dimension $n$, where $n$ is even. 
\begin{enumerate}
  \item $\sum_{A, \langle c_1,A \rangle=n/2} GW_{A,3}([L_i],[L_i],[L_i])==0$ $\forall i=1,2,3$.
  \item Suppose that $L_1$ and $L_2$ intersect transversally at one point. Then 
  $$\sum_{A, \langle c_1,A \rangle =n/2} GW_{A,3}([L_1],[L_1],[L_2])=
\varepsilon\sum_{A, \langle c_1,A \rangle =n/2} GW_{A,3}([L_1],[L_2],[L_2]),$$
where $$\varepsilon:=(-1)^{\frac{n(n+1)}{2}}=\begin{cases} -1 & n=2 \mod 4\\
                                         1 & n=0\mod 4
                                        \end{cases}.$$
  \item Assume that 
$L_1\cdot L_2=L_2\cdot L_3=1$ and $L_1\cdot L_3=0$. Then
$$\sum_{A, \langle c_1,A \rangle =n/2} GW_{A,3}([L_1],[L_2],[L_2])= -\sum_{A, \langle c_1,A \rangle =n/2} GW_{A,3}([L_3],[L_2],[L_2]).$$
\item Assume again that $L_1$ and $L_3$ do not intersect. Then 
$$\sum_{A, \langle c_1,A \rangle =n/2} GW_{A,3}([L_1],[L_3],\cdot)=\sum_{A, \langle c_1,A \rangle =n/2} GW_{A,3}([L_3],[L_1],\cdot)=0.$$
\end{enumerate}
\end{proposition}

The proof of the proposition involves Dehn twists around spheres. For this we recall the Picard Lefschetz formula. 
The self-interesection number of a Lagrangian $n$-sphere is $(-1)^{\frac{n(n-1)}{2}}(1+(-1)^n)$.

\begin{proposition}[Picard-Lefschetz formula]
 Let $a$ be a class in the $n$-dimensional homology group of the ambient manifold $H_n(M)$. Let $\phi$ denote the Dehn twist around the Lagrangian sphere $\Delta\cong S^n$. The map induced in homology is 
 $$\phi_*(a)= a+(-1)^{\frac{(n+1)(n+2)}{2}}(a\cdot [\Delta])[\Delta].$$
\end{proposition}

For details see for example~\cite{Lef24} or~\cite{AGZV88}.

\begin{proof}[Proof of Proposition~\ref{prop:properties-of-gw}]
The Picard-Lefschetz formula for the Dehn twist around a Lagrangian sphere $\Delta$ becomes:
$\phi_*(a)=a-\varepsilon(a\cdot [\Delta])[\Delta].$
Moreover, $[\Delta]\cdot [\Delta]=(-1)^{\frac{n(n-1)}{2}}2$ and $\phi([\Delta])=-[\Delta]$.
\begin{enumerate}
 \item To prove the first identity take $\phi$ to be the Dehn twist around $L_i$. Then
 Let $\phi$ be the Dehn twist around the sphere $L_2$. Then
\begin{eqnarray*}
 \phi_*([L_j])=[L_j]-\varepsilon[L_i]\\
 \phi_*([L_i])=-[L_i].\\
\end{eqnarray*}
\begin{equation*}
\begin{split}
   &\sum_{A, C_1(A)=n/2} GW_{A,3}([L_i],[L_i],[L_i])\\
  &= \sum_{\phi^*A, C_1(\phi*A)=n/2} GW_{\phi^*A,3}(\phi^*[L_i],\phi^*[L_i],\phi^*[L_i])\\
&=   \sum_{A, C_1(A)=n/2} GW_{A,3}(-[L_i],-[L_i],-[L_i])\\
&=-  \sum_{A, C_1(A)=n/2} GW_{A,3}([L_i],[L_i],[L_i]),
\end{split}
\end{equation*}
and hence, $  \sum_{A, C_1(A)=n/2} GW_{A,3}([L_i],[L_i],[L_i])=0$.
\item Now let $\phi$ be the Dehn around $L_j$.
Then we have
\begin{equation*}
\begin{split}
   &\sum_{A} GW_{A,3}([L_i],[L_i],[L_j])\\
   &= \sum_{\phi^*A, C_1(\phi*A)=n/2} GW_{\phi^*A,3}(\phi^*[L_i],\phi^*[L_i],\phi^*[L_j])\\
&=   \sum_{A} GW_{A,3}([L_i]-\varepsilon [L_j],[L_i]-\varepsilon[L_j],-[L_j])\\
&=   \sum_{A} GW_{A,3}([L_i],[L_i],-[L_j])-\varepsilon\sum_{A}GW_{A,3}( 
[L_j],[L_i],-[L_j])\\
&-\varepsilon\sum_{A}GW_{A,3}([L_i], [L_j],-[L_j])+\sum_{A}GW_{A,3}( [L_j], [L_j],-[L_j]).
\end{split}
\end{equation*}
The last summand is zero by the calculation above.
Hence,
\begin{equation*}
\begin{array}{ccl}
  2 \sum_{A} GW_{A,3}([L_i],[L_i],[L_j])&=& 2\varepsilon 
GW_{\phi^*A,3}([L_i], [L_j],[L_j]),\\
\end{array}
\end{equation*}
and thus also
\begin{equation*}
\begin{array}{ccl}
 \sum_{A} GW_{A,3}([L_i],[L_i],[L_j])&=& \varepsilon GW_{\phi^*A,3}([L_i], [L_j],[L_j]),\\
\end{array}
\end{equation*}
\item Using Lagrangian surgery, we see that there exists a Lagrangian 
sphere in the class $[L_1+L_2+L_3]$. As $[L_2]\cdot[L_1+L_2+L_3]=0$ 
Lemma~\ref{lem:prodzero} implies that $[L_2]*[L_1+L_2+L_3]=0$. By the previous identity 
we know 
\begin{eqnarray*}
 0=\sum_{A, \langle c_1,A \rangle =n/2} GW_{A,3}([L_1+L_2+L_3],[L_2],[L_2])\\
= \sum_{A, \langle c_1,A \rangle =n/2} GW_{A,3}([L_1],[L_2],[L_2])+\sum_{A, \langle c_1,A \rangle =n/2} 
GW_{A,3}([L_2],[L_2],[L_2])\\
+\sum_{A, \langle c_1,A \rangle =n/2} GW_{A,3}([L_3],[L_2],[L_2])\\
= \sum_{A, \langle c_1,A \rangle =n/2} GW_{A,3}([L_1],[L_2],[L_2]) +\sum_{A, \langle c_1,A \rangle =n/2} 
GW_{A,3}([L_3],[L_2],[L_2]),\\
\end{eqnarray*}
where the last equality follows from the first identity.
\item As a consequence of Lemma~\ref{lem:relation_quantum_product_and_GW} and ~\ref{lem:prodzero} we see 
$$\sum_{A, \langle c_1,A \rangle =n/2} GW_{A,3}([L_1],[L_3],\cdot)=\sum_{A, \langle c_1,A \rangle =n/2} GW_{A,3}([L_3,[L_1],\cdot)=0.$$
\end{enumerate}
\end{proof}
 
\begin{theorem}\label{thm:discriminant_GW}
 Let $L_1$ and $L_2$ be even dimensional Lagrangian spheres as above, intersecting transversally at one point.
Then
\begin{equation}
  \Delta_{L_1}=\Delta_{L_2}= (\sum_{A} GW_{A,3}(L_1,L_1,L_2) )^2=(\sum_{A} GW_{A,3}(L_2,L_2,L_1) )^2
\end{equation}
\end{theorem}

\begin{proof}[Proof of theorem~\ref{thm:discriminant_GW}]
Let us fix the constant 
$$k:=\sum_{A} 
GW_{A,3}(L_1,L_1,L_2)=\varepsilon\sum_{A} GW_{A,3}(L_2,L_2,L_1),$$
where $\varepsilon=\pm1$, depending on the value of $n\mod 4$.
The element
$$x_1:=[L_2]*e_{L_1}$$ is a lift of a point as $\varepsilon_{L_1}([L_2]*e_{L_1})=[L_2]\cdot[L_1]=1$.
Squaring $x_1$ yields:
$$[L_2]*[L_2]*e_{L_1}=\sigma_1 x_1 +\tau_2 e_{L_1}.$$
Applying the augmentation to both sides shows that $k=\sigma_1$ and $\tau_1=0$.
Similarly, $x_2:=[L_1]*e_{L_2}$ lifts a point of $L_2$ and a similar computation shows 
$\sigma_2=\varepsilon k$ and $\tau_2=0$.
\end{proof}

\begin{proposition}\label{prop:special-intersection-graph-gromov-witten}
 Let $L_1$ $L_2$, $L_3$ and $L_4$ be three Lagrangian spheres of even dimension and suppose that their intersection graph has one of the following two forms:
 \begin{figure}[H]
 \centering
\includegraphics[scale=0.3]{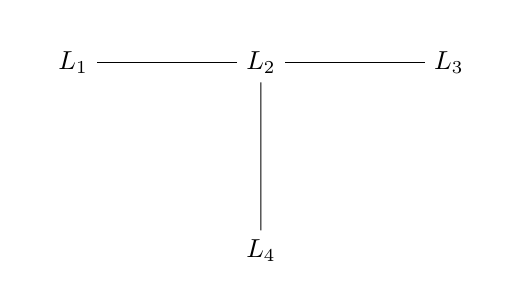}
\caption{Intersection graph of Lagrangian spheres}\label{fig:T-intersection}
\end{figure}
 \begin{figure}[H]
 \centering
\includegraphics[scale=0.3]{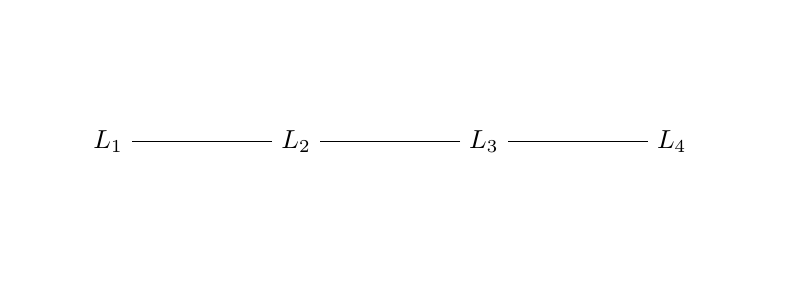}
\caption{Intersection graph of Lagrangian spheres}\label{fig:A1_2}
\end{figure}
Then
$$\Delta_{L_1}=\cdots=\Delta_{L_4}=0.$$
\end{proposition}

\begin{proof}
If the intersection graph is as in figure~\ref{fig:T-intersection} then we know by Proposition~\ref{prop:properties-of-gw} that
\begin{equation*}
  \sum_{A, \langle c_1,A \rangle =n/2} GW_{A,3}([L_i],[L_i],[L_2]) = -\sum_{A, \langle c_1,A \rangle =n/2}GW_{A,3}([L_j],[L_j],[L_2])
\end{equation*}
for every $i\neq j$ with $i,j\in \{1,3,4\}$.
Since $i$ and $j$ can take three values, this implies that
$\sum_{A, \langle c_1,A \rangle =n/2} GW_{A,3}([L_i],[L_i],[L_2])=0$ for every $i$. By Theorem~\ref{thm:discriminant_GW} the statement follows.

If the intersection graph looks like in figure~\ref{fig:A1_2} we can use Lagrangian surgery to find Lagrangian spheres in the classes $[L_3]+[L_4]$ and $[L_3]-[L_4]$. We get the graph
\begin{figure}[H]
 \centering
\includegraphics[scale=0.3]{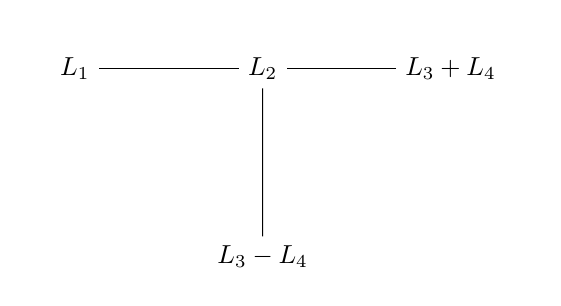}\label{fig:T-intersection2}
\end{figure}
and the rest follows as for the previous case.
\end{proof}

\paragraph{Surgery of three spheres}
Let $L_1$, $L_2$ and $L_3$ be three Lagrangian spheres of even dimension in $M^{2n}$ and suppose their 
intersection graph looks like 
\begin{figure}[H]
 \centering
\includegraphics[scale=0.3]{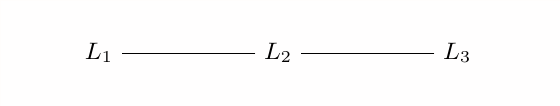}\label{fig:3spheres}
\end{figure}
Let $A$ be the matrix
$$
  \begin{bmatrix}
    \varepsilon2 & 1 & 0 \\
    1 & \varepsilon2 & 1 \\
    0 & 1 & \varepsilon2 \\
  \end{bmatrix},
$$ where $\varepsilon=(-1)^{n(n-1)/2}$ and which encodes the intersection product of the Lagrangian spheres. Then 
$$A^{-1}=
  \frac{-\varepsilon}{4}\begin{bmatrix}
    3 & -\varepsilon2 & 1 \\
    -\varepsilon2 & 4 & -\varepsilon2 \\
    1 & -\varepsilon2 & 3 \\
  \end{bmatrix}.
$$
We define the three elements 
$$\begin{array}{ccc}
   \gamma_1 &:=& (\square \times \varepsilon[3L_1-\varepsilon2L_2+L_3]))*e_{(V,\partial V)}\\
   \gamma_2 &:=& (\square \times\varepsilon[-\varepsilon2L_1+4L_2-\varepsilon2L_3]))*e_{(V,\partial V)}\\
   \gamma_3 &:=& (\square \times\varepsilon[L_1-\varepsilon2L_2+3L_3]))*e_{(V,\partial V)},\\
  \end{array}
$$
where $\square$ was the set obtained from $[-R-\epsilon,R+\epsilon]\times[c,c]$ by smoothening its boundary. (See section~\ref{sec:QH_trivial_LF}.)
Together with $\gamma_e:=e_{(V,\partial V)}$ they form a basis of $QH_1(V,\partial V;\Q)$.
Let 
$$\begin{array}{ccc}
   \overline{p_1} &:=& \frac{-1}{4} \pi_1\circ \delta(\gamma_1)\\
   \overline{p_2} &:=& \frac{-1}{4} \pi_2\circ \delta(\gamma_2)\\
   \overline{p_3} &:=& \frac{-1}{4} \pi_3\circ \delta(\gamma_3).\\
  \end{array}
$$
The elements $\overline{p_i}\in QH_0(L_i;\Q)$ are lifts of points in $H_0(L_i;\mathbb{Z})$.
Notice that $\epsilon_{L_i}\circ \pi_i\circ \delta 
(\gamma_j)=[A(j,1)L_1+A(j,2)L_2+A(j,3)L_3]\cdot[L_j]$ is the classical part of 
$\pi_i\circ 
\delta(\gamma_j)$.
Hence, it follows from the definition of the $\gamma_i$ that 
\begin{center}
\begin{table}[H]
 \begin{tabular}{clccc}
&\multicolumn{1}{c|}{}&$i=1$&$i=2$&$i=3$\\ \toprule
&\multicolumn{1}{c|}{$\pi_1\circ \delta(\gamma_i)$}&$-4 
\overline{p_1}$&$\beta_{21}e_{L_1}$&$\beta_{31}e_{L_1}$\\
&\multicolumn{1}{c|}{$\pi_2\circ \delta(\gamma_i)$}&$\beta_{12}e_{L_2}$&$-4 
\overline{p_2}$&$\beta_{32}e_{L_2}$\\
&\multicolumn{1}{c|}{$\pi_3\circ 
\delta(\gamma_i)$}&$\beta_{13}e_{L_3}$&$\beta_{23}e_{L_3}$&$-4 \overline{p_3}$\\
\end{tabular}
\end{table}
\end{center}
for some $\beta_{ij}\in \mathbb{Z}$.
We will compute the $\beta_{ij}$ in terms of Gromov-Witten invariants.
For this we proceed by applying the map $\pi_i\circ \delta$ to $\gamma_j$, then we take 
the quantum product with $[L_i]\in QH_n(M;\Q)$ from the left. 
The augmentation of $[L_i]*\pi_i\circ \delta (\gamma_j)$ is given by the Gromov-Witten 
invariant $\sum GW(L_i,A(j,1)L_1+A(j,2)L_2+A(j,3)L_3,L_j)$. (See Lemma~\ref{lem:relation_quantum_product_and_GW}.)

Let $\alpha:=\sum_{A, \langle c_1,A \rangle =n/2} GW_{A,3}([L_1],[L_1],[L_2])$. Using
Lemma~\ref{lem:relation_quantum_product_and_GW} and the relations described in~\ref{prop:properties-of-gw} a careful examination yields,

\begin{center}
\begin{table}[H]
 \begin{tabular}{clccc}
&\multicolumn{1}{c|}{}&$i=1$&$i=2$&$i=3$\\ \toprule
&\multicolumn{1}{c|}{$\pi_1\circ \delta(\gamma_i)$}&$-4 
\overline{p_1}$&$-2\alpha e_{L_1}$&$-\alpha e_{L_1}$\\
&\multicolumn{1}{c|}{$\pi_2\circ \delta(\gamma_i)$}&$\alpha e_{L_2}$&$-4 
\overline{p_2}$&$-\alpha e_{L_2}$\\
&\multicolumn{1}{c|}{$\pi_3\circ 
\delta(\gamma_i)$}&$\alpha e_{L_3}$&$2\alpha e_{L_3}$&$-4 \overline{p_3}$\\
\end{tabular}
\end{table}
\end{center}
We define a new basis of $QH_1(V,\partial V;\Q)$ as follows
$$\begin{array}{ccc}
   x_0&:=&e_{(V,\partial V)}\\
   x_1&:=& \gamma_1 -\alpha e_{(V,\partial V)}\\

   x_2&:=& \gamma_2 +2\alpha e_{(V,\partial V)}\\
   x_3&:=& \gamma_3 +\alpha e_{(V,\partial V)}\\
  \end{array}
$$
and we see that 
\begin{center}
\begin{table}[H]
 \begin{tabular}{clccc}
&\multicolumn{1}{c|}{}&$i=1$&$i=2$&$i=3$\\ \toprule
&\multicolumn{1}{c|}{$\pi_1\circ \delta(x_i)$}&$-4 
\overline{p_1}$&$0$&$0$\\
&\multicolumn{1}{c|}{$\pi_2\circ \delta(x_i)$}&$0$&$-4 
\overline{p_2}$&$0$\\
&\multicolumn{1}{c|}{$\pi_3\circ 
\delta(x_i)$}&$0$&$4\alpha e_{L_3}$&$-4 \overline{p_3}$\\
\end{tabular}
\end{table}
\end{center}
Consider the product 
$x_1*x_2=\mu_{121}x_1+\mu_{122}x_2+\mu_{123}x_3+\mu_{120}x_0$.
Applying $\pi_1\circ \delta$ gives
\begin{eqnarray*}
  0=\mu_{121}(-4\overline{p_1})+\mu_{120}e_{L_1},
\end{eqnarray*}
as $\pi_1\circ\delta(x_2)=0$ and thus also $\pi_1\circ\delta(x_1*
x_2)=\pi_1\circ\delta(x_1)*\pi_1\circ\delta(x_2)=0$.
Similarly, by applying $\pi_2\circ\delta$ and $\pi_3\circ\delta$ respectively we see
    \begin{eqnarray*}
      0=\mu_{122}(-4\overline{p_2})\\
      0=\mu_{122}4\alpha e_{L_3}+\mu_{123}(-4\overline{p_3}).\\
    \end{eqnarray*}
Solving these equations gives $\mu_{12i}=0$ for all $i$, or equivalently $x_1*x_2=0$.
A similar computation shows that $x_2*x_1=0$, $x_3*x_1=0$ and $x_1*x_3=0$.

Consider now 
$x_1*x_1=\mu_{111}x_1+\mu_{112}x_2+\mu_{113}x_3+\mu_{110}e_{L_1}$.
Applying $\pi_1\circ\delta$ gives
\begin{eqnarray*}
  (-4)^2\overline{p_1}^2=\mu_{111}(-4\overline{p_1})+\mu_{110}e_{L_1}.
\end{eqnarray*}
As $\overline{p_1}\in QH_0(L_1)$ is the lift of a point, we can write 
$\overline{p_1}^2=\sigma_1\overline{p_1}+\tau_1 e_{L_1}$ and thus
\begin{eqnarray*}
 16(\sigma_1\overline{p_1}+\tau_1 
e_{L_1})=\mu_{111}(4\overline{p_1})+\mu_{110}e_{L_1}.
\end{eqnarray*}
Applying $\pi_2\circ \delta$ and $\pi_3\circ \delta$ gives two other equations, from which 
we conclude that $\mu_{111}=-4\sigma_1$ and $16\tau_1=\mu_{110}=0$,
$\mu_{112}=\mu_{113}=0$. Hence
\begin{eqnarray*}
 x_1*x_1=-4\sigma_1 x_1.
\end{eqnarray*}
A similar procedure for $x_2*x_2$ and $x_3*x_3$ gives
\begin{eqnarray*}
 x_2*x_2=-4\sigma_2 x_2\\
 x_3*x_3=-4\sigma_3 x_3.\\
\end{eqnarray*}
An easy calculation in a similar spirit as above shows that 
\begin{eqnarray*}
 x_2*x_3=x_3*x_2=4\alpha x_3\\
\end{eqnarray*}
Now it is easy to check that $x_1\otimes x_2\otimes x_3\in \ker(A)$, where $A$ is the 
hypermatrix corresponding to this quantum structure.
Therefore the discriminant $\Delta_V$ of this cobordism vanishes.
However, we know from~\cite{BM15} and Theorem~\ref{thm:equal_discr_rank2} (or~\ref{thm:discriminant_GW})
that $\Delta_{L_1}=\Delta_{L_2}=\Delta_{L_3}$. In particular 
$-\sigma_1=\sigma_2=-\sigma_3=:\sigma$.
Notice that $\sigma$ does not need to be zero and hence, also the discriminants of the 
ends of the cobordism $V$ need not vanish.


\section{Quantum invariants of Lagrangians in quadrics}

\subsection{Real Lefschetz fibrations of quadrics}\label{sec:Lefschetz_quadric}
Parts of this section follows the construction in section 6.5 in~\cite{BC15} and it generalizes example 6.5.3 in~\cite{BC15}. We borrow some of their notation.

Given a Lefschetz pencil of complex quadric hypersurfaces in $\mathbb{C}P^{n+1}$ one can 
associate to it a real Lefschetz fibration (i.e. a Lefschetz fibration 
endowed with a real structure).
Consider $\mathbb{C} P^{n+1}$ and the very ample line bundle 
$\mathscr{L}=\mathcal{O}_{\mathbb{C} P^{n+1}}(2)$, $pr:\mathscr{L}\to \mathbb{C}P^{n+1}$ 
and endow it with the real structure $c_{\mathscr{L}}$ induced from complex conjugation.
(More precisely, for an open subset $U\subset \mathbb{C}P^{n+1}$ the real structure on 
$\mathcal{O}_{\mathbb{C} P^{n+1}}(2)|_{U}$ is given by $c_{\mathscr{L}}:(z,\lambda)\mapsto 
(\overline{z},\overline{\lambda})$.)
Let $H^0(\mathscr{L})$ be the space of holomorphic sections of 
$\mathscr{L}$. The real structure on $\mathscr{L}$ induces a real structure $c_H$ on 
$H^0(\mathscr{L})$ (defined by $c_H(s):=c_{\mathscr{L}}\circ s \circ c_{\mathbb{C}P^{n+1}} $ for a section $s\in 
H^0(\mathscr{L})$ and where $c_{\mathbb{C}P^{n+1}}$ is the complex conjugation). This descends to a real structure on 
$\mathbf{P}^*:=\mathbb{P}(H^0(\mathscr{L}))$ and 
$\mathbf{P}:=\mathbb{P}^*(H^0(\mathscr{L}))$. By abuse of notation we continue to denote 
these real structures by $c_H$.
Let $\mathbf{P}^*_{\mathbb{R}}\subset \mathbf{P}^*$ and 
$\mathbf{P}_{\mathbb{R}}\subset \mathbf{P}$ be the corresponding fixed point loci of $c_H$.
Using the coordinates $[x_0:\ldots: x_{n+1}]$ on $\mathbb{C}P^{n+1}$, $H^0(\mathscr{L})$ 
can be identified with the space of quadratic homogeneous polynomials 
$\lambda(\underline{X})$ in $\underline{X}=(X_0,\ldots,X_{n+1})$.
Let $$ \lambda(\underline{X})=\sum_{0\leq i \leq j \leq {n+1}}a_{i,j}X_iX_j$$ be a 
section of $H^0(\mathscr{L})$.
The choice of a basis $\{X_iX_j\}_{i\leq j}$ yields an identification 
$\mathbf{P}=\mathbb{P}(H^0(\mathscr{L}))\cong \mathbb{C}P^{N}$, where 
$N=\frac{(n+3)(n+2)}{2}$. 
The projective embedding 
$$\mathbb{C}P^{n+1}\hookrightarrow \mathbb{C}P^{N-1}$$
is given by 
$$[X_0:\ldots:X_{n+1}]\mapsto [X_0^2:X_0X_1:\ldots:X_iX_j:\ldots:X_{n+1}^2], \quad i\leq 
j.$$
A polynomial $\lambda\in H^0(\mathscr{L})$ 
corresponds to the hyperplane section
$$\Sigma^{(\lambda)}:=\{[X_0:\ldots:X_{n+1}]\in \mathbb{C}P^{n+1} | \lambda(X_0,\ldots, 
X_{n+1})=0\}.$$

Recall that the discriminant locus $\Delta(\mathscr{L})$ of $\mathbb{C}P^{n+1}$, which is 
the dual variety of $\mathbb{C}P^{n+1}$, can be characterized by the determinant of the 
matrix associated to the bilinear from $\lambda$.
More precisely, $\lambda(\underline{X})=\sum_{0\leq i\leq j\leq {n+1}} a_{i,j} X_iX_j$ 
belongs to $\Delta(\mathscr{L})$ if and only if
\begin{equation}\label{eq:det_quadric}
\det \left(
\begin{matrix}
    2a_{00} & a_{01} & a_{02} & \dots  & a_{0(n+1)} \\
    a_{10} & 2a_{11} & a_{12} & \dots  & a_{1(n+1)} \\
    \vdots & \vdots & \vdots & \ddots & \vdots \\
    a_{(n+1)0} & a_{(n+1)1} & a_{(n+1)2} & \dots  & 2a_{(n+1)(n+1)}
\end{matrix}
\right)=0.
\end{equation}
Moreover, it belongs to the smooth strata of $\Delta(\mathscr{L})$ if and only if the 
$(n+2)\times(n+2)$ matrix in~(\ref{eq:det_quadric}) has rank $n+1$ (i.e. one rank less 
than the maximal 
rank).
Endow $\Sigma^{(\lambda)}$ with the symplectic structure induced by the symplectic 
structure $\omega_{FS}$ of $\mathbb{C}P^{n+1}$ and denote it by $\omega_{\lambda}$.
Let $Q\subset \mathbb{C}P^{n+1}$ be the smooth complex $n$-dimensional quadric
$$Q:=\{Z\in \mathbb{C}P^{n+1} | Z_0^2+\ldots +Z_{n+1}^2=0\},$$ 
where $n\geq 2$.
Endowed $Q$ with the symplectic structure $\omega_Q$ induced from $\mathbb{C}P^{n+1}$.
Since the complement of $\Delta(\mathscr{L})$ is path-connected, the non-singular 
varieties $(\Sigma^{(\lambda)}, \omega_{\lambda})$ are all symplectomorphic and hence, 
they are all symplectomorphic to $(Q,\omega_Q)$ (via symplectic parallel transport).
Let $\Delta_{\mathbb{R}}(\mathscr{L}):=\Delta(\mathscr{L})\cap 
\mathbf{P}^*_{\mathbb{R}}$ be the real part of the discriminant locus.
For every $\lambda\in \mathbf{P}_{\mathbb{R}}^*\setminus 
\Delta_{\mathbb{R}}(\mathscr{L})$ the manifold $\Sigma^{(\lambda)}$ has a 
real structure $c_{\Sigma^{\lambda}}$, which is induced by the real structure on 
$\mathbb{C}P^{n+1}$. It 
is not hard to show that the real part $\Sigma_{\mathbb{R}}^{(\lambda)}$ is a Lagrangian 
and that its diffeomorphism type depends only on the connected component of 
$\mathbf{P}_{\mathbb{R}}^*\setminus \Delta_{\mathbb{R}}(\mathscr{L})$ that contains 
$\lambda$.

\begin{lemma}
Assume that $n$ is even. There are exactly $n/2+1$ connected components of $\mathbf{P}_{\mathbb{R}}^*\setminus 
\Delta_{\mathbb{R}}(\mathscr{L})$ and two real quadrics lie in the same component if and 
only if they have the same signature. The hyperplane sections corresponding to 
the polynomials $\lambda_i:=X_0^2+\ldots +X_{i}^2-X_{i+1}^2-\ldots -X_{(n+1)}^2$ for 
$i=-1,0\ldots, n/2$ each lie in distinct connected components of 
$\mathbf{P}_{\mathbb{R}}^*\setminus \Delta_{\mathbb{R}}(\mathscr{L})$.
The real parts of these components are either empty or Lagrangians and have the following diffeomorphism types:
\begin{enumerate}
 \item $\Sigma_{\mathbb{R}}^{(\lambda_i)}\cong \emptyset$ if 
$i=-1$,
 \item $\Sigma_{\mathbb{R}}^{(\lambda_i)}\cong S^n$ if $i=0$ and
 \item $\Sigma_{\mathbb{R}}^{(\lambda_i)}\cong S^i\times S^{n-i}/\mathbb{Z}_2$ if $1\leq 
i\leq n/2$, where $\mathbb{Z}_2$ acts on both factors of $S^i\times S^{n-i}$ by the antipodal map. 
\end{enumerate}
\end{lemma}

\begin{proof}
Let $\eta \in \mathbf{P}_{\mathbb{R}}^*\setminus \Delta_{\mathbb{R}}(\mathscr{L})$ be 
an arbitrary, non-degenerate, real hyperplane section and let $M$ be the matrix associated to $\eta$.
Suppose $M$ has signature $(k,n+2-k)$. It is easy to show that there exists a path in $Sym(n+2;\mathbb{R})\cap GL(n+2,\mathbb{R})$ (i.e. symmetric, invertible matrices) from $M$ to 
$diag(1,\ldots,1,-1,\ldots,-1)$, the diagonal matrix with the first $k$ diagonal elements equal to $1$ and the other diagonal entries equal to $-1$.

Notice that $(Q,\omega_Q)=(\Sigma_{\lambda_{n+1}},\omega_{\lambda_{n+1}})$.
Let $\phi_{(\lambda_i)}:(\Sigma^{(\lambda_i)},\omega_{\lambda_i})\to 
(Q,\omega_Q)$ be a symplectomorphism and let $K_i$ be the Lagrangian in $Q$ corresponding 
to $\Sigma_{\mathbb{R}}^{(\lambda_i)}$ in $\Sigma^{(\lambda_i)}$, i.e. 
$\phi(\Sigma_{\mathbb{R}}^{(\lambda_i)})=K_i$.
Then, 
\begin{equation*}
 \begin{array}{cccc}
  K_i &\cong& \{Z\in Q| Z_0,\ldots,Z_i\in \mathbb{R}, Z_{i+1},\ldots, 
Z_{n+1}\in i\mathbb{R}\}\\
&\cong& \{X\in \mathbb{R}P^{n+1}| X_0^2+\ldots+ X_i^2=X_{i+1}^2+ 
\ldots + X_{n+1}^2\}\\
&\cong& \begin{cases}
         \Sigma_{\mathbb{R}}^{(\lambda_i)}\cong S^i\times 
S^{n-i}/\mathbb{Z}_2 & 1\leq i\leq n/2\\
S^n & i=0\\
\emptyset & i=-1
        \end{cases}
 \end{array}
\end{equation*}

\end{proof}

The main result of this section is Theorem~\ref{thmC} from the introduction.

\begin{theorem}\label{thm:discs_real_parts}
\hspace{2em}
\begin{enumerate}
 \item If $n=2 \mod 4$ the discriminants of the Lagrangians 
$\Sigma_{\mathbb{R}}^{(\lambda_i)}$ are $4$ for all $0\leq i\leq n$.
 \item If $n=0 \mod 4$ the discriminants of the Lagrangians 
$\Sigma_{\mathbb{R}}^{(\lambda_i)}$ are $-4$ for $0\leq i\leq 1$ or $ n-1\leq i \leq n$.
\end{enumerate}
\end{theorem}

\subsection{Topology of the real parts}

We summarize results from~\cite{McC01} that compute the homology groups of fibrations.

\begin{definition}
 A local system  on a topological space $X$ is a locally constant sheaf $\mathcal{H}$ of discrete
abelian groups on $X$. 
If $X$ is path-connected then all fibers $\mathcal{H}_x$ are isomorphic to a discrete abelian 
group $H$. In this case a local system is equivalent to giving a homomorphism 
$\pi_1(X,x)\to Aut(H)$. 
\end{definition}

\begin{theorem}\label{thm:Leray-Serre_spectral}
 Let $G$ be an abelian group. Given a fibration, $\xymatrix{F \ar@{^{(}->}[r] & E 
\ar[r]^{\pi} & B}$, where $B$ is 
path-connected and $F$ is connected, there is a first quadrant spectral sequence, 
$\{E^r_{*,*},d^r\}$, converging to $H_*(E;G)$, with 
$$E^2_{p,q}\cong H_p(B;\mathcal{H}_q(F;G)),$$
the homology of the space $B$ with local coefficients in the homology $\mathcal{H}_q(F;G)$ of the fibre of 
$\pi$. 
\end{theorem}

If the action of $\pi_1(B)$ on the homology of the fibre $H_*(F;G)$ is trivial, also called a simple local system, this can 
be simplified to:

\begin{theorem}\label{thm:Leray-Serre_spectral_simple}
 Let $G$ be an abelian group. Given a fibration,  $\xymatrix{F \ar@{^{(}->}[r] & E 
\ar[r]^{\pi} & B}$, where $B$ is 
path-connected, $F$ is connected, and the system of local coefficients on $B$ determined 
by the fibre is simple, there is a spectral sequence, $\{E^r_{*,*},d^r\}$, with 
$$E^2_{p,q}\cong H_p(B;H_q(F;G)),$$
and converging to $H_*(E;G)$.
\end{theorem}

For explicit computations it is useful to use equivariant homology groups. Let $\Pi$ be a 
group acting on a space $X$. Then $\Pi$ also acts on the chains $C_*(X)$ by sending a 
chain $\sigma$ to the composition $\xi\circ \sigma$ for $\xi \in \Pi$ and $C_*(X)$ 
becomes a left-module over the group ring $\mathbb{Z}[\Pi]$. We can turn $C_*(X)$ into a right-module by defining the action
$$C_*(X) \times \mathbb{Z}[\Pi] \to C_*(X): (c,\xi)\to c\cdot \xi:=\xi^{-1}\cdot c.$$
Moreover, it is easy to see that this descends to $H_*(X)$. 
Now suppose that $G$ is an abelian group on which $\Pi$ acts on the left. 
Consider the chain complex $(C_*(X)\otimes G,d_C\otimes Id_{G})$ 
and quotient it by the subgroup $Q(X,G)$ generated by all elements of the form 
$c\xi\otimes a-c\otimes \xi a=\xi^{-1}c\otimes a-c\otimes \xi a$ for $a\in G,\xi\in \Pi, c\in C_*(X)$. Denote the 
complex obtained in this way by $C_*(X)\otimes_{\Z[\Pi]}G$. In other words tensoring over $\mathbb{Z}[\Pi]$ corresponds to modding out the action of $\Pi$, since 
$$\xi \cdot (c\otimes a) =  c \xi^{-1}\otimes \xi a= c\otimes a.$$
Since the boundary operator of $C_*(X)$ maps $Q(X,G)$ to itself 
this induces an endomorphism $\delta$ on $C_*(X)\otimes_{\Pi}G$ which turns 
the latter into a chain complex. (I.e. $\delta\circ\delta=0$.)
The homology groups of this chain complex
$$E_q(X;G):=H_q(C_*(X)\otimes_{\Pi} G)$$ 
are called the \emph{equivariant homology groups} of $X$ with coefficients in the 
$\Pi$-module $G$.

Now let $\xymatrix{F \ar@{^{(}->}[r] & E 
\ar[r]^{\pi} & X}$ be a fibration, where $X$ is 
path-connected, $F$ is connected, and the system of local coefficients $\mathcal{H}$ on 
$X$ determined by the fibre. I.e. $\mathcal{H}_x:=H_q(\pi^{-1}(x);\mathbb{Z})$ for 
every $x\in X$.
Choose $\Pi:=\pi_1(X;x_0)$ and let $H:=H_q(\pi^{-1}(x_0);\mathbb{Z})$. 
Let $\tilde{X}\to X$ be the universal cover. Clearly $\Pi$ acts on $\tilde{X}$ by deck 
tranformations and moreover, it also acts on the local system $\mathcal{H}$ on $X$. 
The equivariant homology of the covering space $\tilde{X}$ is isomorphic to the homology of $X$ with local coefficients. 

\begin{theorem}[Eilenberg]\label{thm:equivariant=local_coeff}
The  homology groups $H_*(X; \mathcal{H})$ with respect to the system 
$\mathcal{H}=H_q(F;G)$ of local coefficients are isomorphic to the equivariant homology 
groups $E_*(\tilde{X};H)$, where $H=H_q(F_0;G)$ is the 
homology group of the fiber $F_0$ over the base point $x_0\in X$. 
\end{theorem}

For a proof and more explanations on local coefficients and equivariant 
homology groups see~\cite{Whi12}.

\begin{remark}
 If we choose a $CW$-structure on $\tilde{X}$ that is invariant under deck 
transformations we can instead compute the equivariant homology of the cellular chain 
complex. The homology groups obtained in this way are isomorphic to $H_*(X; \mathcal{H})$.
\end{remark}

\subsubsection{The singular homologies}

Consider the Serre fibration
$$S^{n-i}\to (S^i\times S^{n-i})/\mathbb{Z}_2\to \mathbb{R}P^i$$
with fiber $F\cong S^{n-i}$.
We may assume w.l.o.g that $i\leq n-i$.

\begin{lemma}\label{lem:HforK_i}
Suppose that $n$ is even and $i\leq n-i$. Denote by $\mathcal{H}_{n-i}$ the local 
system on $X$ defined by $\mathcal{H}_{n-i,x}:=H_{n-i}(F_x;\mathbb{Z})$.
The Lagrangians $K_i\cong (S^i\times S^{n-i})/\mathbb{Z}_2$ have the homology 
groups:
\begin{equation}
 H_*(K_i;\mathbb{Z})\cong \begin{cases}  

H_*(\mathbb{R}P^{i};\mathbb{Z})\oplus H_{*-(n-i)}(\mathbb{R}P^i;\mathbb{Z}) & i \text{ 
and } n-i \text{ are odd}\\
H_*(\mathbb{R}P^{i};\mathbb{Z})\oplus H_{*-(n-i)}(\mathbb{R}P^i;\mathcal{H}_{n-i}) & 
i \text{ and } n-i \text{ are even},
\end{cases}
\end{equation}
where 
$$H_{*-(n-i)}(\mathbb{R}P^i;\mathcal{H}_{n-i}) = \begin{cases} \mathbb{Z}_2 & n-i\leq 
* 
<n \text{ and $*$ is even}\\
\mathbb{Z} & *=n\\
0 & \text{ otherwise}
\end{cases}$$
is the homology of $\mathbb{R}P^i$ with local coefficients in the homology of the fibre.
\end{lemma}

\begin{proof}
We abbreviate $X:=(S^i\times S^{n-i})/\mathbb{Z}_2$.
\begin{enumerate}
 \item Assume first that $i$ and $n-i$ are both odd. 
A generator of the group $\pi_1(\mathbb{R}P^i)\cong \mathbb{Z}_2$ acts on the fiber 
$F\cong S^{n-i}$ as the antipodal map. Since $n-i$ is odd the antipodal is homotopic to 
the identity and the action of $\pi_1(\mathbb{R}P^i)\cong \mathbb{Z}_2$ on
$H_*(F;\mathbb{Z})$ is trivial.
By Theorem \ref{thm:Leray-Serre_spectral_simple} there exists a spectral sequence with
$$E^2_{p,q}=H_p(\mathbb{R}P^i;H_q(F,\mathbb{Z}))$$
and that converges to $H_*(X;\mathbb{Z})$.
Since 
$$H_q(S^{n-i};\mathbb{Z})\cong \begin{cases} \mathbb{Z}& q=0,n-i\\
                                0 & \text{otherwise}
                               \end{cases}$$
it follows that
$$E^2_{p,q}=\begin{cases} H_p(\mathbb{R}P^i) & q= 0, n-i \\
             0 & \text{otherwise}
            \end{cases}.$$
Recall that $d_2:E^2_{p,q}\to E^2_{p-2,q+1}$. We claim that $d_2$ vanishes. 
Suppose first that $n-i=1$. The fact that $n$ is even and $i\leq n-i$ implies $i=1$. Notice that 
$d_2$ could be nonzero if and only if $q=0$. 
Hence, suffices to look at $d_2:E^2_{p,0}\to E^2_{p-2,1}$, which is 
$d_2:H_p(\mathbb{R}P^1;\mathbb{Z})\to H_{p-2}(\mathbb{R}P^1;\mathbb{Z})$. Clearly this 
vanishes. All higher differentials 
vanish, since the degree shift for the index $p$ is bigger than the dimension of 
$\mathbb{R}P^1$.
If $n-i\neq 1$, at least one of $E^2_{p,q}$ and $E^2_{p-2,q+1}$ is zero and hence, $d_2$ 
vanishes.
For the same reason all differentials $d_k$ vanish for $k\leq n-i$.
In particular, $E^2_{p,q}=E^3_{p,q}=\ldots=E^{n-i+1}_{p,q}$ and on the $n-i+1$-th page
\begin{equation*}\label{eq:d_n-i+1}
 d_{n-i+1}:E^{n-i+1}_{p,q}\to E^{n-i+1}_{p-n+i-1,q+n-i}.
\end{equation*}
Since $E^{n-i+1}_{p,q}=E^2_{p,q}=H_p(\mathbb{R}P^i;H_q(F;\mathbb{Z}))$ we  need to treat the case $q=0$. 
Hence
\begin{equation*}
 d_{n-i+1}:H_p(\mathbb{R}P^i;\mathbb{Z})\to H_{p-n+i-1}(\mathbb{R}P^i;\mathbb{Z}).
\end{equation*}
and this is zero if either $p> i$ or $p< n-i+1$.
As $i\leq n-i$ the latter happens for all $p$. 
All higher differentials vanish, since the degree shift for the index $p$ is bigger than 
the dimension of $\mathbb{R}P^{i}$.

It follows that the spectral sequence collapses on the second page and 
\begin{equation*}
 \begin{array}{cll}
  H_*(X,\mathbb{Z})&=&\bigoplus_{p+q=*}E_{p,q}^2\\
 &=&E_{*,0}^2\oplus E_{*-(n-i),n-i}^2\\
 &=& H_*(\mathbb{R}P^i;\mathbb{Z})\oplus H_{*-(n-i)}(\mathbb{R}P^i;\mathbb{Z}))\\
 \end{array}
\end{equation*}
\item Now we assume that $i$ and $n-i$ are both even. Then the action of 
$\pi_2(\mathbb{R}P^i)$ is trivial on $H_0(F;\mathbb{Z})$ and $-Id$ on 
$H_{n-i}(F;\mathbb{Z})$.
By Theorem~\ref{thm:Leray-Serre_spectral} there exists a spectral 
sequence with 
$$E_{p,q}^2\cong H_p(\mathbb{R}P^i;\mathcal{H}_q),$$
where $\mathcal{H}_q$ is the local system $H_q(F_x;\mathbb{Z})$ and $F_x\cong 
S^{n-i}$ is the fiber over $x$.
Since the action is trivial on $H_0(F_x;\mathbb{Z})$ we have 
$H_p(\mathbb{R}P^i;\mathcal{H}_0)\cong H_p(\mathbb{R}P^i;\mathbb{Z})$.
For he computation of $H_p(\mathbb{R}P^i;\mathcal{H}_{n-i})$ 
we apply Eilenberg's Theorem~\ref{thm:equivariant=local_coeff}.
The universal cover of $\mathbb{R}P^i$ is $S^i$.
We start with a cell complex on $S^i$, which has two cells in each dimension, say $E^k_+$ 
and $E^k_-$. More precisely,
$$E^k_+:=\{(x_1,\ldots, x_{k+1},0,\ldots,0)\in S^i| x_{k+1}\geq 0\}$$
$$E^k_-:=\{(x_1,\ldots, x_{k+1},0,\ldots,0)\in S^i| x_{k+1}\leq 0\}.$$
Moreover, we orient these cells in such a way that $A(E^k_+)=E^k_-$, where $A$ denotes 
the antipodal map of $S^i$.
We then have
$$\partial^{cell}_{2k}(E^{2k}_+)=E^{2k-1}_++E^{2k-1}_-=\partial^{cell}_{2k}(E^{2k}_-)$$
and
$$\partial^{cell}_{2k+1}(E^{2k+1}_+)=E^{2k}_+-E^{2k}_-=-\partial^{cell}_{2k+1}(E^{2k+1}_-).$$
Now $C_*^{cell}(S^i)$ is 
$$\xymatrix{
0 \ar[r] & \mathbb{Z}^2 \ar[r]^{\partial^{cell}_{i}} & 
\mathbb{Z}^2 \ar[r]^{\partial^{cell}_{i-1}}& \ldots \ar[r]^{\partial^{cell}_2} & 
\mathbb{Z}^2 
\ar[r]^{\partial^{cell}_1} & \mathbb{Z}^2 \ar[r] & 0
}$$
and we tensor this with $\otimes_{\pi_1(\mathbb{R}P^i)} H_{n-i}(F;\mathbb{Z})$.
Let $e\in H_{n-i}(S^{n-i};\mathbb{Z})$ be the fundamental class.
The action of a generator $g\in\pi_1(\mathbb{R}P^i)$ on the cells is the antipodal 
map, i.e.
$$g\cdot E^k_+=E^k_- \text{ and } g\cdot E^k_-=E^k_+.$$
On $H_q(F;\mathbb{Z})$ the generator $g\in\pi_1(\mathbb{R}P^i)$ acts by the map induced 
from the antipodal map. Since $F\simeq S^{n-i}$ and $n-i$ is even, this map is equal to 
$-Id$. 
Thus
$$g\cdot (E^k_+\otimes e)=g\cdot E^k_+\otimes g \cdot e=E^k_+\otimes g\cdot e=E^k_-\otimes (-e)=-E^k_-\otimes e$$
and similarly
$$g\cdot (E^k_-\otimes e)=-E^k_+\otimes e.$$
The chain complex $C_*^{cell}(S^i)\otimes_{\pi_1(\mathbb{R}P^i)} 
H_{n-i}(S^{n-i};\mathbb{Z})$ has one cell $[E^k_+\otimes e]=-[E^k_-\otimes e]$ in each 
dimension:
$$\xymatrix{
0 \ar[r] & \mathbb{Z} \ar[r]^{d_i} & \mathbb{Z} \ar[r]^{d_{i-1}}& 
\ldots \ar[r]^{d_2} & \mathbb{Z} \ar[r]^{d_1} & \mathbb{Z} \ar[r] & 0
},$$
where $d_k=\partial^{cell}_k\otimes Id$.
Hence, 
$$d_{2k}([E^{2k}_+\otimes e])=([E^{2k-1}_+]+[E^{2k-1}_-])\otimes e=0$$
and
$$d_{2k+1}([E^{2k+1}_+\otimes e])=([E^{2k}_+]-[E^{2k}_-])\otimes e=2[E^{2k}_+\otimes e].$$
We conclude that 
$$H_{*-(n-i)}(\mathbb{R}P^i;\mathcal{H}_{n-i}) = \begin{cases} \mathbb{Z}_2 & 
n-i\leq * 
<n \text{ and $*$ is even}\\
\mathbb{Z} & *=n\\
0 & \text{ otherwise}.
\end{cases}$$
\end{enumerate}

\end{proof}

\subsubsection{Orientability}

\begin{corollary}
 If $n$ is even $K_i$ is orientable for all $i$.
\end{corollary}

\begin{corollary}
 The homology groups with coefficients in $\mathbb{Z}_2$ are
\begin{equation}
 H_*(K_i;\mathbb{Z}_2)\cong H_*(\mathbb{R}P^{i};\mathbb{Z}_2)\oplus 
H_{*-(n-i)}(\mathbb{R}P^i;\mathbb{Z}_2).
\end{equation}
\end{corollary}

\begin{proof}
 This follows from~\ref{lem:HforK_i} with the universal coeffient theorem. Alternatively, we could use Theorem~\ref{thm:Leray-Serre_spectral_simple}
 with coefficients in $\mathbb{Z}_2$. With $\mathbb{Z}_2$ coefficients the action of $\pi_1(\mathbb{R}P^{n-i})$ is always trivial on 
the fiber.
\end{proof}

\begin{lemma}
 If $n$ is odd, the Lagrangians $K_i$ in the quadric $Q\subset \mathbb{C}P^{n+1}$ are 
non-orientable for $2\leq i\leq n-2$.
\end{lemma}

\begin{proof}
 This follows directly from a similar computation as in the proof of 
Lemma~\ref{lem:HforK_i}.
\end{proof}

\subsection{The cobordisms}\label{sec:cobordism}

For a fixed $i$ between $-1$ and $n+1$ there exists a cobordism between 
$\Sigma_{\mathbb{R}}^{\lambda_i}$ and 
$\Sigma_{\mathbb{R}}^{\lambda_{i+1}}$ inside the Lefschetz fibration described in~\ref{sec:Lefschetz_quadric}.
Let $\ell \subset 
\mathbf{P}^*$  be the pencil that passes through $\lambda_i$ and $\lambda_{i+1}$. 
\begin{remark}
This pencil as well as some of the subsequent notations clearly depend on $i$. 
For the sake of simplicity however and readability of this text, we choose to omit the 
index $i$ from the notation.
In the end of this paragraph, we will reintroduce the index $i$ in the notation of the so 
obtained Lefschetz fibration and the corresponding cobordism.
\end{remark}
The pencil $\ell$ can be parametrized by
$$\mathbb{C}P^1 \ni [t_0,t_1]\mapsto \lambda_{[t_0:t_1]}:=t_0 
\lambda_i +t_1 \lambda_{i+1}.$$
It is easy to see that $\ell\cap \Delta(\mathscr{L})= 
\{\lambda_{[1:1]},\lambda_{[1:-1]}\}$,
where $\lambda_{[1:1]}$ is a smooth point of $\Delta(\mathscr 
{L})$ and the intersection is transverse. The point $\lambda_{[1:-1]}$ does not lie in 
the smooth strata of $\Delta(\mathscr {L})$.
Recall that $\lambda_i$ and $\lambda_{i+1}$ lie inside $\mathbf{P}_{\mathbb{R}}^*$.
Thus also $\ell\subset \mathbf{P}_{\mathbb{R}}^*$ and it is invariant under the 
anti-holomorphic involution.
Choose a disk $D\subset \ell$ that contains the point 
$\lambda_{[1:-1]}$, but not the points $\lambda_{[1:1]}$, $\lambda_{[1:0]}=\lambda_i$ and 
$\lambda_{[0:1]}=\lambda_{i+1}$.
Then $\ell\setminus D\cong \mathbb{C}$. Fix an orientation 
preserving diffeomorphism
$\beta:\ell\setminus D \to \mathbb{C},$
such that $\beta(\lambda_{[1:1]})=(0,0)$, 
$\beta(\lambda_{[1:0]})=(-1,0)$ and 
$\beta(\lambda_{[0:1]})=(1,0)$. 
Set $E:=\{(\lambda,z)|\lambda\in \ell\setminus D, z\in 
\Sigma^{(\lambda)})$.
This yields a Lefschetz fibration
$$\pi:E\to \mathbb{C}.$$
By Proposition 6.5.1 in~\cite{BC15} the real part of 
$E$ is a cobordism between $\Sigma_{\mathbb{R}}^{\lambda_i}$ 
and $\Sigma_{\mathbb{R}}^{\lambda_{i+1}}$. Let us denote this cobordism by 
$V^{(\lambda_i,\lambda_{i+1})}$. If $i=-1$ or $i=n+1$ this cobordism has only 
one end and it is contractible.

Let $\Sigma^{sing}$ denote the real part of the singular fiber of $E$.  
Let $C_-$ be the intersection of the vanishing cycle in $\Sigma^{(\lambda_i)}$ 
with the Lagrangian $\Sigma_{\mathbb{R}}^{(\lambda_i)}$. Similarly, $C_+$ is the 
intersection of the vanishing cycle of $\Sigma^{(\lambda_{i+1})}$ with the Lagrangian
$\Sigma_{\mathbb{R}}^{(\lambda_{i+1})}$.
Under the symplectomorphism 
$\phi_{(\lambda_i)}:(\Sigma^{(\lambda_i)},\omega_{\lambda_i})\to 
(Q,\omega_Q)$ the cycle $C_-$ is mapped to the set $S^i\times 
\{(1,0,\ldots,0)\}\subset (S^i\times S^{n-i})/\mathbb{Z}_2$.
Similarly, $C_+$ is mapped to $\{(0,\ldots,0,1)\}\times 
S^{n-i-1}\subset (S^{i+1}\times S^{n-i-1})/\mathbb{Z}_2$ by the symplectomorphism 
$\phi_{(\lambda_{i+1})}$.
Moreover, $(\Sigma^{(\lambda_i)},C_-)$ and $(\Sigma^{(\lambda_{i+1})},C_+)$ are good pairs and 
$\tilde{H}_*(\Sigma^{sing})\cong 
H_*(\Sigma^{(\lambda_i)},C_-)$ as well as $\tilde{H}_*(\Sigma^{sing})\cong H_*(\Sigma^{(\lambda_{i+1})},C_+)$.
Thus 
$$\tilde{H}_*(\Sigma^{sing})\cong H_*((S^i\times S^{n-i})/\mathbb{Z}_2,S^i\times \{pt\}) 
$$
or 
$$\tilde{H}_*(\Sigma^{sing})\cong H_*((S^{i+1}\times S^{n-i-1}/\mathbb{Z}_2,\{pt\}\times 
S^{n-i-1}) 
.$$

\begin{lemma}
 The cobordism $V^{(\lambda_i,\lambda_{i+1})}$ has the same homotopy type as 
$\Sigma^{sing}$. It is orientable if $n$ is even and in particular $H_{n+1}(V,\partial V)\cong\mathbb{Z}$.
\end{lemma}

\begin{proof}
If $i=-1$ or $i=n$ the cobordism is contractible and the map $H_{n+1}(V,\partial V)\to H_n(\partial V)$ in the long exact sequence of the pair $(V,\partial V)$ becomes an isomorphism. We have seen that the end $\partial V\cong S^n$ and it follows that $V$ is orientable.

In any case, the cobordism $V$ deformation retracts onto $\Sigma^{sing}$. The long exact sequence of the pair $(\Sigma_{\lambda_i},C_-)$ for $i\in 
\{0, \ldots, n-2\}$ in degree $n$ is 
$$\xymatrix{ 0 \ar[r] & H_n(\Sigma_{\lambda_i}) \ar[r]^{j_*} & H_n(\Sigma_{\lambda_i}, C_-) \ar[r] & H_{n-1}(C_-).} $$
If $i\neq n-1$ and since $H_n(\Sigma_{\lambda_i})\cong \mathbb{Z}$, this implies that $\mathbb{Z}\cong H_n(\Sigma_{\lambda_i}, 
C_-)\cong \tilde{H}_n(\Sigma^{sing})\cong \tilde{H}_n(V)$.
If $i=n-1$ do the same calculation using $(\Sigma^{(\lambda_{i+1})}, C_+)\cong(S^n,S^0)$
Since $V$ is a manifold with boundary of dimension $n+1$ the top homology group 
$H_{n+1}(V,\partial V)$ is either $0$ or $\mathbb{Z}$. Consider the long exact sequence 
of the pair $(V,\partial V)$:
$$\xymatrix{ 0 \ar[r] & H_{n+1}(V,\partial V) \ar[r] & H_n(\partial V) 
\ar[r] \ar[d]^{\cong} & H_n(V)\ar[r] \ar[d]^{\cong} & \cdots\\
& & H_n(\Sigma^{(\lambda_0)})\oplus H_n(\Sigma^{(\lambda_1)}) \ar[d]^{\cong}& \mathbb{Z}\\
& & \mathbb{Z}\oplus \mathbb{Z} & &}$$
It follows $H_{n+1}(V,\partial V)\cong \mathbb{Z}$. 
\end{proof}

\subsection{The spin structures}

In this subsection all cohomology groups are taken with coefficients in $\mathbb{Z}_2$.
The manifolds $K_i\simeq S^i\times S^{n-i}/\mathbb{Z}_2$ can be expressed as the total 
space of a fibration 
$$S^i\hookrightarrow S^i\times S^{n-i}/\mathbb{Z}_2\rightarrow \mathbb{R}P^{n-i}.$$ 
In fact, we have a commutative diagram of the form 
\begin{equation}\label{diag:manyfibrations}
\xymatrix{ &  S^i\times S^{n-i}/\mathbb{Z}_2 \ar[d]^{q} \ar@/_/[ldd]_{\pi_1} 
\ar@/^/[rdd]^{\pi_2} & \\
& \mathbb{R}P^i\times \mathbb{R}P^{n-i} \ar[ld]^{pr_1} \ar[rd]_{pr_2} & \\
\mathbb{R}P^i &  & \mathbb{R}P^{n-i}}
\end{equation}
and in cohomology:
\begin{equation}\label{diag:manyfibrations_cohomology}
\xymatrix{ &  H^*(S^i\times S^{n-i}/\mathbb{Z}_2) & \\
& H^*(\mathbb{R}P^i\times \mathbb{R}P^{n-i}) \ar[u]^{q^*} & \\
H^*(\mathbb{R}P^i) \ar@/^/[uur]^{{\pi_1}^*} \ar[ur]_{{pr_1}^*}&  & 
H^*(\mathbb{R}P^{n-i}) \ar@/_/[luu]_{{\pi_2}^*} \ar[lu]^{{pr_2}^*}.}
\end{equation}


\begin{proposition}\label{prop:Lagrangians_are_spin}
 \hspace{2em}
 \begin{enumerate}
  \item If $n=2 \mod 4$ then the Lagrangians $K_i$ are spin for all $0\leq i \leq n$.
  \item If $n=0 \mod 4$ then the Lagrangian $K_i$ is spin if and only if $i \in 
\{0,1,n-1,n\}$.
\end{enumerate}
\end{proposition}

\begin{proof}
We will prove the proposition in three steps. First we prove that for arbitrary $n$ 
the Lagrangians $K_i$ are spin if $i\in \{0,1,n-1,n\}$.
Then we will consider the Lagrangians $K_i$ with $2\leq i \leq n-2$ and prove that 
they are spin if and only if $n=2\mod 4$. The case $S^2\times S^2/\mathbb{Z}_2$ has to be 
considered separately. 
\begin{enumerate}
 \item Clearly $K_0$ and $K_{n}$ are diffeomorphic to $S^n$ and if $i\in \{1,n-1\}$ then $K_i$ is diffeomorphic to $S^1\times S^{n-1}$.
\item Now let us assume that $2\leq i \leq n-2$. We fix an $i$ and abbreviate 
$X:=S^i\times S^{n-i}/\mathbb{Z}_2$.
\begin{enumerate}
\item Assume now that $n>4$. Then either $i$ or $n-i$ is greater than $2$, and we may 
assume w.l.o.g. that $i>2$. 
The Gysin sequence (see~\ref{thm:Gysin_sequence}) of the fibration $S^i\hookrightarrow 
X\rightarrow \mathbb{R}P^{n-i}$ is
$$\scalebox{0.9}{
\xymatrix{ \ldots \ar[r] &H^{1-i}(\mathbb{R}P^{n-i})\ar[r]^-{\cup e} 
&H^2(\mathbb{R}P^{n-i}) \ar[r]^-{\pi_2^*} & H^2(X) \ar[r] & 
H^{2-i}(\mathbb{R}P^{n-i}) \ar[r]& \ldots}
}$$
Since $i>2$ this yields an isomorphism
$$\xymatrix{ \mathbb{Z}_2\cong H^2(\mathbb{R}P^{n-i})\ar[r]^-{\pi_2^*}_-{\cong} & H^2(X)}.$$
Diagram~(\ref{diag:manyfibrations_cohomology}) forces $q^*:H^2(\mathbb{R}P^{i}\times \mathbb{R}P^{n-i})\to H^2(X)$ to be an isomorphism. 
Hence, $w_2(X)=0$ if and only if $w_2(\mathbb{R}P^{i}\times \mathbb{R}P^{n-i})=0$.
Moreover, $w_1(\mathbb{R}P^j)=0 \iff \text{j odd}$ and $w_2(\mathbb{R}P^j)=0\iff 
j=3 \mod 4$. (See for example~\cite{GH14}.)
It is known (see for example~\cite{MilS16}) that the second 
Stiefel-Whitney class 
of $\mathbb{R}P^i\times\mathbb{R}P^{n-i}$ is $w_2(\mathbb{R}P^i)\otimes1 + 1\otimes 
w_2(\mathbb{R}P^{n-i}) + w_1(\mathbb{R}P^i)\otimes w_1(\mathbb{R}P^{n-i})$.
Combining this we find that $w_2(\mathbb{R}P^i\times\mathbb{R}P^{n-i})=0 \iff i+(n-i)=2 
\mod 4 \iff n=2\mod 4$.
\item We need to treat the case $n=4$ and $i=n-i=2$ separately.
The Gysin sequence of the fibration $\xymatrix{S^0 \ar@{^{(}->}[r] & X \ar[r]^-{q} & 
\mathbb{R}P^2\times \mathbb{R}P^2}$ gives:
$$\xymatrix{ \cdots \ar[r] & H^4(X) \ar[r] & H^4(\mathbb{R}P^2\times \mathbb{R}P^2) \ar[r]^-{\cup e} & H^5(\mathbb{R}P^2\times \mathbb{R}P^2)\ar[r] & \cdots \\}.$$
Since $H^5(\mathbb{R}P^2\times \mathbb{R}P^2)=0$ and $H^4(\mathbb{R}P^2\times \mathbb{R}P^2)\cong \mathbb{Z}_2 $ and also $H^4(X)\cong 
\mathbb{Z}_2$ this implies that $q^*: H^4(\mathbb{R}P^2\times \mathbb{R}P^2) \to H^4(X)$ 
is an isomorphism. 
Now consider the fibrations $\xymatrix{S^2 \ar@{^{(}->}[r] & X \ar[r]^-{\pi_i} & 
\mathbb{R}P^2}$ for $i=1,2$. The Gysin sequence yields
$$\xymatrix{0 \ar[r] & H^2(\mathbb{R}P^2) \ar[r]^-{\pi_{i}^*} & 
H^2(X) \ar[r] & H^{0}(\mathbb{R}P^2) \ar[r] & 0},$$
and in particular $\pi_i^*: H^2(\mathbb{R}P^2)\to H^2(X)$ is injective. 
Moreover, we know that the two maps 
$$\xymatrix{H^2(\mathbb{R}P^2) \ar[r]^-{pr_1^*} & H^2(\mathbb{R}P^2\times \mathbb{R}P^2)}\text{ and }
\xymatrix{H^2(\mathbb{R}P^2) \ar[r]^-{pr_2^*} & H^2(\mathbb{R}P^2\times \mathbb{R}P^2)}$$
are also injective by the K\"unneth Theorem.
I.e. we have a commutative diagram 
\begin{equation}\label{diag:manyfibrations_cohom}
\xymatrix{ &  H^2(X)  & \\
& H^2(\mathbb{R}P^2\times \mathbb{R}P^2) \ar[u]_{q^*}^{\cong}  & \\
H^2(\mathbb{R}P^2) \ar@{^{(}->}[ru]_{pr_1^*} \ar@{^{(}->}@/^/[ruu]^{\pi_1^*}&  & 
H^2(\mathbb{R}P^2) \ar@{_{(}->}[lu]^{pr_2^*} \ar@{_{(}->}@/_/[luu]_{\pi_2^*}.}
\end{equation}
Let $x$ denote the generator of the cohomology of the first $\mathbb{R}P^2$ factor, i.e.  
$H^*(\mathbb{R}P^2)\cong \mathbb{Z}_2[x]/(x^3)$, $|x|=1$. Similarly, let $y$ be the 
generator of the cohomology of the second $\mathbb{R}P^2$ factor, i.e.  
$H^*(\mathbb{R}P^2)\cong \mathbb{Z}_2[y]/(y^3)$, $|y|=1$.
Denote $\alpha:=pr_1^*(x)$ and $\beta:=pr_2^*(y)$. They are generators for the cohomology of $\mathbb{R}P^2\times 
\mathbb{R}P^2$. Then $\alpha^2\cup \alpha^2=\beta^2 
\cup\beta^2=0$ and $\alpha^2\cup \beta^2 =PD(pt)\in H^4(\mathbb{R}P^2\times 
\mathbb{R}P^2)$ is the 
unit. 
Set also $a:=\pi_1^*(x)$ and $b:=\pi_2^*(y)$. 
We see that 
\begin{equation}
 \begin{array}{ccccc}
  a^2\cup a^2 &=& q^*(0) &=& 0\\
  b^2\cup b^2 &=& q^*(0) &=& 0\\
  a^2\cup b^2 &=& q^*(PD(pt)) &=& PD(pt).\\
 \end{array}
\end{equation}

The following criterion is useful.

\begin{lemma}[\cite{GSS99}]
Let $X$ be a compact, oriented topological $4$-manifold and let $[X]$ denote the 
fundamental class in $H_4(X,\partial X;\mathbb{Z})$.
Let 
$$\begin{array}{rccc}
   Q_X:& H^2(X,\partial X;\mathbb{Z})\otimes H^2(X,\partial X; \mathbb{Z})&\to & 
\mathbb{Z}\\
	& (a,b) & \mapsto & \langle a \cup b,[X]\rangle
  \end{array}$$
denote the intersection form of $X$.
Then Wu's formula reads
$$\langle w_2(X),a\rangle =Q_X(PD(a),PD(a))\mod 2$$
for all $a\in H^2(X,\partial X)$ and where $PD(-)$ denotes the Poincar\'e dual. 
In particular $X$ is spin if and only if $Q_X$ is even.
\end{lemma}

In our case $Q_X$ corresponds to the matrix 
$\begin{pmatrix}
    0 & 1 \\
    1 & 0 \\
\end{pmatrix}$, which is not even and hence, $X$ is not spin.
\item If $n\leq 3$ then necessarily $i\in\{0,1,n-1,n\}$ and all $K_i$'s are spin by case 
$1$.
\end{enumerate}
\end{enumerate}
\end{proof}

\begin{proposition}\label{prop:cobordism_spin}
 \hspace{2em}
 \begin{enumerate}
  \item If $n=2 \mod 4$ then the cobordism $V^{(\lambda_i,\lambda_{i+1})}$ is spin for 
all $0\leq i \leq n$.
  \item If $n=0 \mod 4$ then the cobordism $V^{(\lambda_i,\lambda_{i+1})}$ is spin if and 
only if $i=0$ or $i=n-1$.
\end{enumerate}
\end{proposition}

Relative Stiefel-Whitney classes were discussed in~\cite{Ker57} and~\cite{Ste51}.
Let $B$ be a CW-complex and $E\to B$ a vector bundle. We associate to it the bundle 
$V_{n-1}(\mathbb{R}^n)\to V_{n-1}(E)\to B$ with fibre the Stiefel manifold 
$V_{n-1}(\mathbb{R}^n)$ of $n-1$-frames in 
$\mathbb{R}^n$.
Suppose $L$ is a subcomplex of $B$ and $E|_L$ is trivializable over $L$. Fix a 
trivialization $\theta$, which can also be seen as a section of $V_{n-1}(E)$. Then a 
representative of the Stiefel manifold may be obtained as follows.
Choose a cross-section $\theta$ of $V_{n-1}(E)$ over $L$ and extend it stepwise to 
$L\cup B^{(0)}$ and $L\cup B^{(1)}$. By definition this representative takes value zero 
on each $2$-simplex of $L$ and hence, defines a relative homology class $w_2(\theta)\in 
H^2(B,L;\mathbb{Z}_2)$ that depends on $\theta$ but not on the choice of extension to 
$L\cup B^{(2)}$ (see~\cite{Ste51}). In fact Steenrod shows that if $\theta'$ is another 
section then $w_2(\theta)-w_2(\theta')$ lies in the image of the connecting homomorphism 
$\delta$ of the sequence
$$\xymatrix{\cdots &H^2(B;\mathbb{Z}_2) \ar[l] & H^2(B,L;\mathbb{Z}_2)\ar[l]_{\phi^*} & 
H^1(L;\mathbb{Z}_2)\ar[l]_-{\delta} & \cdots \ar[l].}$$
Let us call the class $w_2(\theta)$ the relative Stiefel-Whitney class of $B \mod L$ 
associated to $\theta$. In particular $\phi^*(w_2(\theta))=w_2(E)$ is the usual second 
Stiefel-Whitney class of $E$ and it is independent of the choice of $\theta$.
Also relative Stiefel-Whiney classes fulfil naturality. The next lemma is a special case 
of Theorem 35.7 in~\cite{Ste51}.
\begin{lemma}[Steenrod]\label{lem:naturality_SWclasses}
 Let $f:(B',L')\to (B,L)$ be a map with $f(L')\subset L$ and $E\to B$ a vector bundle.
Then 
$$w_2(f^*E)=f^*(w_2(E)).$$
\end{lemma}

\begin{proof}[Proof of~\ref{prop:cobordism_spin}]
Let $\pi:E\to C$ be the Lefschetz fibration with real part 
$V^{(\lambda_i,\lambda_{i+1})}$. The set $C_-:=Z_-\cap \Sigma^{(\lambda_i)}$, is 
diffeomorphic to $S^{i}$.
Let $T^i$ denote the corresponding thimble of $Z_-$. The vector bundle $TV|_{C_-}$ is trivializable 
as it extends to $T^i\cap V\cong B^{i+1}$.
Hence, $$w_i(TV|_{C_-})=0 \text{ for all i.}$$
 \begin{enumerate}
  \item 
If $n=2 \mod 4$ the Lagrangian $\Sigma_{\mathbb{R}}^{(\lambda_i)}$ is spin by 
Proposition~\ref{prop:Lagrangians_are_spin} and thus its Stiefel-Whitney class $w_2(\Sigma_{\mathbb{R}}^{(\lambda_i)})$ vanishes. 
Since the normal bundle of $\Sigma_{\mathbb{R}}^{(\lambda_i)}$ in 
$V^{(\lambda_i,\lambda_{i+1})}$ is 
trivial also $w_2(TV|_{\Sigma_{\mathbb{R}}^{(\lambda_i)}})=0$.
Choose a trivialization $\theta$ of $TV|_{C_-}$.
By~\ref{lem:naturality_SWclasses} and the discussion above
\begin{equation}\label{eq:w_2}
j^*(w_2(TV|_{\Sigma_{\mathbb{R}}^{(\lambda_i)}},TV|_{C_-}))=w_2(TV|_{\Sigma_{\mathbb{R}}^{
(\lambda_i) } } )=0 ,
\end{equation}
where $j^*:H^2(\Sigma_{\mathbb{R}}^{(\lambda_i)},C_-)\to 
H^2(\Sigma_{\mathbb{R}}^{(\lambda_i)})$ is induced by the 
inclusion $j:(\Sigma_{\mathbb{R}}^{(\lambda_i)},\emptyset) \hookrightarrow 
(\Sigma_{\mathbb{R}}^{(\lambda_i)},C_-)$.
The map $j^*$ is part of the long exact sequence
$$\xymatrix{ \cdots \ar[r] & H^1(C_-) \ar[r]^-{\delta} & 
H^2(\Sigma_{\mathbb{R}}^{(\lambda_i)},C_-) 
\ar[r]^-{j^*} & H^2(\Sigma_{\mathbb{R}}^{(\lambda_i)}) \ar[r] &\cdots }$$
If $i\neq 1$ then $H^1(C_-)=0$ and $j^*$ is injective, which implies 
$w_2(TV|_{\Sigma_{\mathbb{R}}^{(\lambda_i)}},TV|_{C_-})=0$ by~(\ref{eq:w_2}). 
As the normal bundle of $\Sigma_{\mathbb{R}}^{(\lambda_i)}$ is trivial also
$w_2(\Sigma_{\mathbb{R}}^{(\lambda_i)}, C_-)=0$.

Let $\Sigma_{\mathbb{R}}^{sing}$ denote  the real part of the singular fiber of $E$. Let $x_0\in \Sigma_{\mathbb{R}}^{sing}$ be the 
singular point.
Let $F:(V,V\cap T)\times I\to (V,V\cap T)$ be the obvious 
deformation retraction of $(V,V\cap T)$ onto $(\Sigma_{\mathbb{R}}^{sing},x_0)$. 
In particular $F(\cdot, 1)=k\circ r:(V,V\cap T)\to (V,V\cap T)$, where $r:(V,V\cap T)\to 
(\Sigma_{\mathbb{R}}^{sing},x_0) $ is a retraction and $k:(\Sigma^{sing}_{\mathbb{R}},x_0)\hookrightarrow (V,V\cap T)$ the inclusion.
Define the map 
$f:=r|_{(\Sigma_{\mathbb{R}}^{(\lambda_i)},C_-)}:(\Sigma_{\mathbb{R}}^{(\lambda_i)},
C_-)\to 
(\Sigma_{\mathbb{R}}^{sing},x_0)$.
The maps $i$ and $k\circ f$ in the diagram
$$\xymatrix{ (\Sigma_{\mathbb{R}}^{(\lambda_i)},C_-) \ar@{^{(}->}@/^/[r]^{i} \ar[d]^{f} & 
(V,V\cap T)\\
(\Sigma_{\mathbb{R}}^{sing},x_0) \ar@{^{(}->}@/^/[ru]^{k}}$$
are homotopic (via $F|_{(\Sigma_{\mathbb{R}}^{(\lambda_i)},C_-)}$).
As $f$ and $k$ induce isomorphism in homology, so does the inclusion 
$i:(\Sigma_{\mathbb{R}}^{(\lambda_i)},C_-)\hookrightarrow (V,T\cap V)$:
$$\xymatrix{i^*:H^*(V,T\cap V)\ar[r]^{\cong} & 
H^*(\Sigma_{\mathbb{R}}^{(\lambda_i)},C_-)}.$$
By naturality $0=w_2(\Sigma_{\mathbb{R}}^{(\lambda_i)},C_-)=i^*(w_2(V,T\cap V))$ and 
since $i^*$ is an isomorphism $w_2(V,T\cap V)=0$. 
Now, by definition $w_2(V)=0$ and hence, $V$ is spin.

If $i=1$ then since $n\neq 3$ by assumption we can repeat the same argument for the 
pair $(\Sigma_{\mathbb{R}}^{(\lambda_{i+1})},C_+)$, with $C_+\cong S^{n-i-1}$ and conclude 
analogously 
that $V^{(\lambda_i,\lambda_{i+1})}$ is spin.

\item A spin structure on a cobordism induces a spin 
structure on its boundaries. Hence, by~\ref{prop:Lagrangians_are_spin} the cobordism 
$V^{(\lambda_i,\lambda_{i+1})}$ is not spin if $n=0 \mod 4$ and $i\neq 0, n-1$.
If $i\in\{0,n-1\}$ apply the methods from part 1 to show that 
$V^{(\lambda_i,\lambda_{i+1})}$ is spin.
\end{enumerate}
\end{proof}

\subsection{Regularity of the standard complex structure}

By the standard complex structure on $Q$ we mean the complex structure induced by the 
standard complex structure $J_0$ on $\mathbb{C}P^{n+1}$.
Endow $\mathbb{C}P^{n+1}$ with the Fubini-Study K\"ahler form 
$\omega_{FS}$, then also $\Sigma^{(\lambda)}$ is K\"ahler with the induced complex 
structure and symplectic form. With respect to this K\"ahler structure 
$\mathbb{R}P^{n+1}$ is a totally real submanifold of $\mathbb{C}P^{n+1}$ and 
similar for $\Sigma^{(\lambda)}_{\mathbb{R}}\subset \Sigma^{(\lambda)}$.
We briefly discuss the regularity of the linearized $\overline{\partial}$-operator.
The aim of this section is to prove:
\begin{proposition}\label{prop:D_u_is_onto}
 For every $\lambda\in \mathbf{P}_{\mathbb{R}}\setminus 
\Delta_{\mathbb{R}}(\mathscr{L})$ and every $J_0$-holomorphic disk 
$$u:(D^2,\partial D^2)\to (\Sigma^{(\lambda)},\Sigma_{\mathbb{R}}^{(\lambda)}),$$
the operator $D_u:=D_u(\overline{\partial} _{J_0})$ is onto. 
\end{proposition}

We will describe a method of reflecting the holomorphic disk using the real 
structure $c_{\Sigma^{(\lambda)}}$,which then yields a holomorphic sphere.
This analysis is a modification of those in~\cite{Oh93b}.
Recall the equation of the holomophic disk $u$ in holomorphic coordinates $z=s+it$:

\begin{equation*}\label{eq:hol_disk}
\begin{array}{ccc}
 \overline{\partial}_{J_0}u= \partial_s u + J_0(u)\partial_t u = 0 &&\\
 u(z)\in \Sigma^{(\lambda)}_{\mathbb{R}}\subset \mathbb{R}P^{n+1} & & \forall 
z\in \partial D^2\\
 \int |du|_{J_0}^2dvol < \infty &&\\
\end{array}
\end{equation*}

Now $u$ is smooth by elliptic regularity. Consider the vector field $\xi\in W^{k,2}$ 
along $u$ with $\xi(z)\in T_{u(z)}\Sigma^{(\lambda)}_{\mathbb{R}}$. 
Let $\tau \mapsto u_{\tau}$ be a curve with $u_0=u$ and 
$\frac{d}{d\tau}|_{\tau=0}u_{\tau}=\xi$. 
The linearization of the 
$\overline{\partial}_{J_0}$-operator is given by 
$$D_u\overline{\partial}_{J_0}:=\nabla_{\tau}|_{\tau=0}\overline{\partial}_{J_0}(u_{\tau}
).$$
We get 
\begin{equation}\label{eq:delbar_J_0}
\begin{array}{ccl}
 \nabla_{\tau}|_{\tau=0}\overline{\partial}_{J_0}(u_{\tau}) &=& 
\nabla_{\tau}|_{\tau=0}(\partial_su_{\tau}+J_0(u_{\tau})\partial_t u_{\tau})\\
	    &=& 
\nabla_{\tau}\partial_s u_{\tau}|_{\tau=0}+\nabla_{\tau}J_0(u_{\tau})|_{\tau=0}\partial_t 
u +J_0(u)\nabla_{\tau}\partial_t u_{\tau}|_{\tau=0}\\
	    &=& \nabla_s\partial_{\tau}u_{\tau}|_{\tau=0}+J_0(u)\nabla_t 
\partial_{\tau}u_{\tau}|_{\tau=0}\\
	    &=& \nabla_s \xi +J_0(u)\nabla_t \xi,
\end{array}
\end{equation}
where we used the fact that $\nabla J_0=0$ characterizes the 
K\"ahler case (see~\cite{McS12}). Moreover, the identity 
$\nabla_{s}\partial_{\tau}=\nabla_{\tau}\partial_s$ holds since 
$[\partial_{\tau},\partial_s]=0$ and similar 
$\nabla_{t}\partial_{\tau}=\nabla_{\tau}\partial_t$.
This proves that 
$$D_u\overline{\partial}_{J_0}=\nabla_s+J_0(u)\nabla_t.$$
Let $D_u\overline{\partial}_{J_0}^*$ denote the adjoint of 
$D_u\overline{\partial}_{J_0}$, i.e. the operator defined by the equation
$$\langle \xi, D_u\overline{\partial}_{J_0}^*\eta\rangle_2 =\langle 
D_u\overline{\partial}_{J_0}\xi, \eta\rangle_2.$$
Here we use the $L^2$-norm $\langle\cdot,\cdot \rangle_2$ and $\xi,\eta\in W^{k,2}$.
To show that $coker D_u\overline{\partial}_{J_0}$ is trivial and the range of 
$D_u\overline{\partial}_{J_0}$ is closed it suffices to prove instead that $ker 
D_u\overline{\partial}_{J_0}^*$ is trivial.
We compute 
\begin{equation*}
 \begin{array}{ccl}
  \langle D_u\overline{\partial}_{J_0}\xi,\eta\rangle_2 &=& \int_{D^2} \langle 
\nabla_s\xi+J_0(u)\nabla_t\xi,\eta\rangle\\
      &=& \int_{D^2}\langle \nabla_s\xi,\eta\rangle+\langle J_0(u)\nabla_t\xi, 
\eta\rangle.
 \end{array}
\end{equation*}
Integration by parts yields
\begin{equation*}
 \begin{array}{ccl}
  &=&-\int_{D^2}\langle \xi, \nabla_s \eta \rangle +\int_{\partial D^2}\langle \xi(z), 
\eta(z)\rangle Re(z) d\gamma\\
  & & -\int w(\xi,\nabla_t\eta) + \int_{\partial D^2} w(\xi(z),\eta(z))Im(z) d\gamma\\
  &=& \int_{D^2}\langle \xi, -\nabla _s\eta+J_0(u)\nabla_t\eta\rangle\\
  & & +\int_{\partial D^2}\langle\xi(x),(Re(z)\eta(z)-Im(z)J_0(u)\eta(z)\rangle d\gamma\\
  &=& \int_{D^2}\langle \xi, -\nabla_s\eta+J_0(u)\nabla_t\eta\rangle+\int_{\partial 
D^2}\langle \xi(z),\overline{z}\eta(z)\rangle d\gamma,
 \end{array}
\end{equation*}
where $d\gamma$ is the length measure of $\partial D^2$.
The adjoint problem is
\begin{equation}\label{eq:adjoint}
 \begin{array}{cc}
  (-\nabla_s+J_0(u)\nabla_t)\eta=0 & \forall \eta\in W^{k,2}\\
  \overline{z}\eta(z)\in (T_{u(z)}\Sigma^{(\lambda)}_{\mathbb{R}})^{\perp} & \forall 
z\in\partial D^2,
 \end{array}
\end{equation}
and where $()^{\perp}$ denotes the orthogonal complement with respect to the Fubini-Study 
metric.
Eliptic regularity ensures again that $\eta\in coker(D_u\overline{\partial}_{J_0})$ is 
smooth. After reflecting the disk by the anti-holomorphic involution we obtain a 
holomorphic sphere 
$$v:\mathbb{C}P^1\to \Sigma^{(\lambda)}$$
and a smooth vector field $X\in W^{1,2}$:
$$X:\mathbb{C}P^1 \to T\Sigma^{(\lambda)},$$
such that $\pi(X(z))=v(z)$ $\forall z\in \mathbb{C}P^1$, where $\pi:T\mathbb{C}P^1\to 
\mathbb{C}P^1$ is the projection.
Equation~\ref{eq:adjoint} becomes
\begin{equation}\label{eq:adjoint_reflected}
 -\nabla_sX+J_0\nabla_tX=0.
\end{equation}
Grothendieck proved (see~\cite{Gro57}) that any holomorphic bundle over $\mathbb{C}P^1$ 
is isomorphic to a direct sum of holomorphic line bundles. Moreover, this splitting is 
unique up to permutation of the summands. Thus we have

$$v^*T\Sigma^{(\lambda)}\cong L_1\oplus L_2\oplus \cdots \oplus L_n.$$

\begin{lemma}[\cite{McS12} Lemma 3.3.2]\label{lem:regularity_critirion}
 Assume that $J$ is an integrable almost complex structure on a symplectic manifold 
$(M,\omega)$ and let $v:\mathbb{C}P^1\to M$ be a $J$-holomorphic sphere. If every summand 
of $v^*TM$ has Chern number $c_1\geq -1$, then $D_v$ is onto.
\end{lemma}

\begin{definition}
\begin{enumerate}
 \item A line bundle on a complete algebraic variety is said to be numerically effective (or nef) if 
$$\langle c_1(L), C\rangle \geq 0$$
for every complete curve $C\subset X$.
\item A vector bundle $E \to X$ on a complete algebraic variety is called numerically 
effective (or nef) if $\mathcal{O}_{\mathbf{P}(E)}$ is nef.
\end{enumerate}
\end{definition}

We are going to show that quadrics have nef tangent bundles. Since pullback bundles as 
well as quotient bundles of nef bundles are nef, this also implies that the line 
bundles $L_i\to \mathbb{C}P^1$ are nef. However, for line bundles over $\mathbb{C}P^1$ 
being nef is equivalent to having non-negative Chern number. Hence:

\begin{proposition}\label{prop:D_v_is_onto}
For every $J_0$-holomorphic sphere $$v:\mathbb{C}P^1\to 
\Sigma^{(\lambda)},$$
the operator $D_v$ is onto.
\end{proposition}

\subsubsection{Homogeneous manifolds have nef tangent bundle}

These subsections relate to a discussion in~\cite{JMP12} and~\cite{MOC+14}.
Let $G$ be a connected algebraic group over $\mathbb{C}$, and $\mathfrak{g}:=T_{e}G$ its 
Lie algebra.

\begin{definition}
 A $G$-variety is an algebraic variety $X$ with a $G$-action 
$$G\times X\to X: (g,x)\mapsto g\cdot x,$$
which is a morphism of algebraic varieties.
\end{definition}

Let $\mu_x:G\to X: g\mapsto gx$ be the orbit map. An element of the Lie algebra 
$\mathfrak{g}$ defines a vector field on $X$ by
\begin{equation*}
 \begin{array}{rlcc}
  D_{e}\mu_x:&T_{e}G &\to & T_xX\\
  & \xi & \mapsto & D_{e}\mu_x(\xi).
 \end{array}
\end{equation*}
Thus, the Lie algebra $\mathfrak{g}$ of a $G$-variety $X$ acts on $X$ by vector fields. 
If $X$ is smooth, let $\mathcal{T}_X$ denote the tangent sheaf, and we have a 
homomorphism of Lie algebras
\begin{equation*}
 \begin{array}{rlcc}
  \Theta:&\mathfrak{g} &\to& \Gamma(X,\mathcal{T}_X)\\
  & \xi &\mapsto & D_e\mu_{(\cdot)}(\xi)
 \end{array}
\end{equation*}
and, at the level of sheaves,
\begin{equation*}
 \begin{array}{rlcc}
ev:&\mathcal{O}_X\otimes \mathfrak{g} &\to& \mathcal{T}_X\\
  & ((x, (f|_U)_U), \xi) &\to & (x, f_U(x)D_e\mu_x(\xi)).
 \end{array}
\end{equation*}

\begin{definition}
 A $G$-variety $X$ is called homogeneous if the action of $G$ on $X$ is transitive.
\end{definition}

For every point $x\in X$, let $G_x:=Stab_G(x)$ denote the stabilizer of $x$ in $G$.
$G_x$ is closed, since $G_x$ is the fiber of the orbit map $\mu_x:G\to X:g\to gx$.
Suppose that $X$ is homogeneous. The orbit map factors through $G/G_x$ and we get an 
isomorphism of $G$-varieties with base points: $(G/G_x,eG_x)\cong (X,x)$.
Moreover, if $X$ is homogeneous then the map $ev$ is surjective. 

Let $Aut^0(X)$ denote the identity component of $Aut(X)$.
The next theorem identifies the Lie algebra of $Aut^0(X)$.

\begin{theorem}[Ramanujam~\cite{Ram64}]\label{thm:ramanu}
 If $X$ is a complete, irreducible algebraic variety, then $Aut^0(X)$ is a connected 
algebraic group 
with Lie algebra $\Gamma(X,\mathcal{T}_X)$.
\end{theorem}

\begin{corollary}\label{col:globally_gen_nef}
 Let $X$ be a complete, irreducible variety. Then $X$ is homogeneous if and only if 
$\mathcal{T}_X$ is 
generated by global sections, i.e. if $ev: \mathcal{O}_X\otimes 
\Gamma(X,\mathcal{T}_X)\to\mathcal{T}_X$ is surjective.
In particular if $X$ is a smooth homogeneous complex manifold its tangent bundle is nef.
\end{corollary}

\begin{proof}
We have already noticed that for a homogeneous space the homomorphism $ev$ is 
surjective. 

For the other direction, let $G:=Aut^0(X)$ and from theorem~\ref{thm:ramanu} we know that 
the Lie algebra $\mathfrak{g}$ is isomorphic to $\Gamma(X,\mathcal{T}_X)$. 
The tangent bundle $\mathcal{T}_X$ of an algebraic variety is $G$-linearized in 
the sense that $D_xg:T_xX\to T_{gx}X$ is linear and the diagram
$$\xymatrix{T_xX \ar[r]^-{D_xg} \ar[d]^{p} & T_{gx}X \ar[d]^{p}\\
X\ar[r]^-{g}&X}$$
is commutative for all $x\in X$ and $g\in G$.
Notice that $G$ itself is a $G$-variety and moreover it is homogenous. Let 
$$\begin{array}{rlcc}
   \sigma_g:& G &\to & G\\
    & h &\mapsto &hg
  \end{array}$$
denote the orbit map. We also have
$$\begin{array}{rlcc}
   D_e\sigma_g:& T_eG &\to & T_gG\\
    & \xi &\mapsto &D_e\sigma_g(\xi).
  \end{array}$$
Notice that $(ev)_x$ is surjective if and only if $D_e\mu_x$ is surjective, which holds by assumption.
Since $\mu_x\circ \sigma_g=\mu_{gx}$ we have
$$D_e(\mu_x\circ \sigma_g)=D_g\mu_x\circ D_e \sigma_g.$$
The equivariance of the map $\mu_x:g\mapsto gx$ and the fact that $G$ itself is 
homogeneous shows that $D\mu_x$ is everywhere surjective.
In other words, $\mu_x$ is a submersion, and therefore $Im(\mu_x)=G\cdot x$ is open in $X$.
Since $X$ is a variety and the orbit $G\cdot x$ is open in $X$ for all $x\in X$, we 
conclude that $X=G\cdot x$, i.e. $X$ is homogeneous.
\end{proof}

\subsubsection{Proof of Proposition~\ref{prop:D_v_is_onto} and 
Proposition~\ref{prop:D_u_is_onto}}

\begin{proof}
 Define the orthogonal group of the quadratic form $\lambda$ by
$$O(\lambda):=\{A\in Mat(n+2,\mathbb{C})|\lambda(Ax)=\lambda(x)\},$$
where $Mat(n+2,\mathbb{C})$ denotes the $(n+2)\times (n+2)$ matrices with coefficients in 
$\mathbb{C}$. 
Clearly this group acts on the smooth quadric $\Sigma^{(\lambda)}$ and the action is 
transitive. By Corollary~\ref{col:globally_gen_nef} $T\Sigma^{(\lambda)}$ is 
nef and thus also $v^*T\Sigma^{(\lambda)}\cong L_1\oplus\cdots\oplus L_n$.
Since $L_i\subset v^*T\Sigma^{(\lambda)}$ is a quotient bundle $\forall i$, it is nef and 
thus has Chern number $c_1\geq 0$.
By Lemma~\ref{lem:regularity_critirion} $D_v$ is onto, which proves 
Proposition~\ref{prop:D_v_is_onto}. 

Now assume that $0\neq \eta\in coker (D_u)$, i.e. $\langle D_u\xi, \eta\rangle=0 \forall 
\xi$. Then the reflection $X$ of $\eta$ is non-trivial and satisfies 
equation~\ref{eq:adjoint_reflected}, which means $X\in Ker D_v^*$. 
Hence, $Coker D_v$ is non-trivial, which contradicts Proposition~\ref{prop:D_v_is_onto}.
\end{proof}

\subsection{The quantum homology groups}

\paragraph{Wideness and the quantum homologies of the Lagrangians}

Throughout this section we assume that $n=2\mod 4$. It is known that the minimal Chern 
number of the quadric $Q^n\subset \mathbb{C}P^{n+1}$ is $n$. 
(As the minimal Chern number of $\mathbb{C}P^{n+1}$ is $n+2$ and the normal bundle of $Q$ in $\mathbb{C}P^{n+1}$ is $\mathcal{O}(2)$.)
In particular the minimal Maslov number of the Lagrangians $K_i$ is also $n$ for $i\neq 0,n$. The Lagrangians $K_0$ and $K_n$ have minimal Maslov number $2n$.
\begin{proposition}
The Lagrangians $K_i$ are wide for all possible choices of spin 
structures.
\end{proposition}

\begin{proof}
\begin{enumerate}
  \item If $i\in \{0,n\}$ then $K_i\cong S^n$ which is wide.
  \item Assume now that $2\leq i \leq n/2$ and w.l.o.g. $i\leq n-i$.
If $i$ and $n-i$ are odd then by~(\ref{lem:HforK_i})
$$H_1(K_i;\mathbb{Z})\cong H_1(\mathbb{R}P^{i};\mathbb{Z})\oplus 
H_{1-(n-i)}(\mathbb{R}P^i;\mathbb{Z})\cong 
H_1(\mathbb{R}P^1;\mathbb{Z})\cong\mathbb{Z}_2.$$
Similarly,
$$H_{n-1}(K_i;\mathbb{Z})\cong H_{n-1}(\mathbb{R}P^i;\mathbb{Z})\oplus 
H_{n-1-(n-i)}(\mathbb{R}P^i;\mathbb{Z}))\cong 0,$$
as $n-1-(n-i)=i-1$ is even and $n-1\geq i$. 
Moreover, $H_0(K_i;\mathbb{Z})\cong H_n(K_i;\mathbb{Z})\cong \mathbb{Z}$.
Consider the spectral sequence that comes from the degree filtration of the pearl complex.
(For more details on this spectral sequence see section 6.1 in~\cite{BC07}.)
The first page is the singular homology.
Since the minimal Maslov number is $n$ the differentials on the first page have degree 
$n-1$. 
Clearly the differentials $d_1:H_*(K_i;\mathbb{Z})\to H_{*+n-1}(K_i;\mathbb{Z})$ are all 
zero for $*\geq 2$ or $*\leq -1$.
The other differentials are 
$$d_1:H_0(K_i;\mathbb{Z})\to H_{n-1}(K_i;\mathbb{Z}) \text{ and } d_1:H_1(K_i;\mathbb{Z})\to H_n(K_i;\mathbb{Z}),$$
which also vanish since $H_{n-1}(K_i;\mathbb{Z})=0$ and because there exists no 
non-trivial homomorphism $\mathbb{Z}_2\to \mathbb{Z}$.

If $i$ and $n-i$ are both even then by~(\ref{lem:HforK_i})
$$H_1(K_i;\mathbb{Z})\cong H_1(\mathbb{R}P^{i};\mathbb{Z})\oplus 
H_{1-(n-i)}(\mathbb{R}P^i;\mathcal{H}_{n-i})\cong 
H_1(\mathbb{R}P^1;\mathbb{Z})\cong\mathbb{Z}_2$$
since $1-(n-i)<0$ if $2\leq i\leq n-i$. 
Moreover
$$H_{n-1}(K_i;\mathbb{Z})\cong H_{n-1}(\mathbb{R}P^i;\mathbb{Z})\oplus 
H_{n-1-(n-i)}(\mathbb{R}P^i;\mathcal{H}_{n-i}))\cong 0,$$
as $i< n-1$ and $n-1-(n-i)$ is odd.
As before $d_1:H_*(K_i;\mathbb{Z})\to H_{*+n-1}(K_i;\mathbb{Z})$ vanish for $*\geq 2$ or 
$*\leq -1$.
If $*=0$ or $1$ we have 
$$d_1:H_0(K_i;\mathbb{Z})\to H_{n-1}(K_i;\mathbb{Z}) \text{ and } d_1:H_1(K_i;\mathbb{Z})\to H_n(K_i;\mathcal{H}_{n-i}),$$
which also vanish since $H_{n-1}(K_i;\mathbb{Z})=0$ and because there exists no 
non-trivial homomorphism $\mathbb{Z}_2\to \mathbb{Z}$.
\item The case $i=1$ is more involved. 
Assume first that $n\geq 4$ and in particular $n-i>1$.
First notice that $K_1\cong S^1\times S^{n-1}/\mathbb{Z}_2$ and
$$H_*(K_1)\cong H_*(\mathbb{R}P^1;\mathbb{Z})\otimes 
H_{*-(n-1)}(\mathbb{R}P^1;\mathbb{Z})\cong \begin{cases} \mathbb{Z} & *=0,1,n-1,n\\
                                            0 & \text{otherwise.}
                                           \end{cases}
$$
As before we look at the spectral sequence that comes from the degree filtration of the pearl 
complex. The first page is the singular homology and $d_1:H_*(K_1)\to H_{*+n-1}(K_1)$ 
vanishes for $*\leq -1$ or $*\geq2$.
For the case $*=1,2$, the differential is a homomorphism $\mathbb{Z}\to \mathbb{Z}$.
We will show that the pearly trajectories in the pearl complex always come in pairs that 
cancel each other. The sign of the trajectories depends on the spin structure of $K_1$.

\begin{lemma}[\cite{Cho04} Theorem 6.4]\label{lem:moduli_orientation_change}
Let $L\subset M$ be a Lagrangian submanifold and $w:(D^2,\partial D^2)\to (M,L)$ 
a holomorphic disk in the class $B\subset H_2(M,L)$.
Let $\mathcal{M}(L;B)$ denote the moduli space of such discs.
A spin structure on $L$ induces an orientation on $\mathcal{M}(L;B)$.
The orientation of the moduli space 
$\mathcal{M}(L;B)$ changes if we change the homotopy class of trivialization of 
$w|_{\partial D^2}^*(TL)$.
\end{lemma}

By Proposition~\ref{prop:D_u_is_onto} all disk with $\mu\neq 0$ are regular for the standard 
complex structure. 
Given a holomorphic disk $w:(D^2,\partial D^2)\to (Q,K_1)$, the holomorphic disk $\overline{w}:=c_Q\circ w\circ 
c_{\mathbb{C}}:(D^2,\partial D^2)\to (Q,K_1)$ lies in $\mathcal{M}(K_1;c_{Q*}(B))$. 
Clearly $w$ and $\overline{w}$ have the same Maslov index.
Moreover, it is easy to see that $\partial(image(w))=-\partial(image(\overline{w}))$. 
Thus a given spin structure on $L$ induces opposite homotopy classes of trivializations 
on $w|_{\partial D^2}^*(TL)$ and $\overline{w}|_{\partial D^2}^*(TL)$.
By~\ref{lem:moduli_orientation_change} the moduli spaces $\mathcal{M}(L;B)$  
and $\mathcal{M}(L;c_{Q*}(B))$ come with opposite orientations.
For $n\neq 2$ the homology groups of $K_i$ are: $H_1(K_i)\cong\mathbb{Z}$, 
$H_2(K_i)\cong 0$. The ones of $Q$ are: $H_1(Q)\cong 0$, $H_2(Q)\cong \mathbb{Z}$. (See~\cite{BC07}.) We have the short exact sequence:
$$\xymatrix{ 0 \ar[r] & H_2(Q) \ar[r]^{j_*}  & H_2(Q,K_1) 
\ar[r]^{\delta}  & H_1(K_1) \ar[r]  & 0}.$$
Clearly, $\delta([w])=\delta(-[\overline{w}])$ and therefore $w+\overline{w}$ represents the 
class of a sphere in $H_2(Q)$ and $[w]+[\overline{w}]\in 
ker{\delta}=im{j_*}$.
Since the holomorphic sphere $\overline{w}+w$ has Chern class 
$c_1(\overline{w}+w)=\mu(\overline{w})=\mu(w)\neq 0$, it cannot be contractible in $Q$ 
and hence, represents a non-zero class in $H_2(Q)$. As $j_*$ is injective, the class $[w]+[\overline{w}]$ is 
non-zero in $H_2(Q,K_1)$. 

For every pearly trajectory between two critical points passing through a holomophic disk 
in the class $B$ we will find exactly one pearly trajectory between the same critical 
points passing through a disk in the class $c_{Q*}(B)$ and such that they contribute to 
the differential with opposite signs.
This shows that also 
$$d_1:H_0(K_1)\to H_{n-1}(K_1) \text{ and } d_1:H_1(K_1)\to H_n(K_1)$$ vanish.

The proof for the case $n=2$ is the same, except that now $K_1\cong S^1\times S^1$ and
$$H_*(K_1)\cong \begin{cases} \mathbb{Z} & *=0,2\\
			      \mathbb{Z}\oplus \mathbb{Z} & *=1\\
				0 & \text{otherwise.}
                                           \end{cases}
$$
Again $H_3(Q)=H_1(Q)=0$ and it is easy to see that the inclusion $\mathbb{R}P^1\times \R P^1 \hookrightarrow \C P^1 \times \C P^1$ induces the zero map in homology. We get an exact sequence
$$
\scalebox{0.9}{
\xymatrix{ 0 \ar[r] & H_3(Q,K_1) \ar[r]^-{\cong}  & H_2(K_1) 
\ar[r]^-{0}  & H_2(Q) \ar[r]^-{j_*} & H_2(Q,K_1) 
\ar[r]^-{\delta}  & H_1(K_1) \ar[r]  & 0}
}$$
This shows that the map $j_*$ is again injective and the same reasoning as above ends the 
proof. 
\end{enumerate}
\end{proof}

\subsection{Proof of Theorem~\ref{thm:discs_real_parts}}

\begin{proof}
 In~\cite{BM15} Biran and Membrez showed that the discriminant of a Lagrangian 
sphere $S^n\subset Q$ has discriminant $(-1)^{\frac{n(n-1)}{2}+1}4$. 
Theorem~\ref{thm:equal_discr_rank2} and the construction of the cobordisms above imply 
the result.
\end{proof}

\begin{appendix}

\section{Filtered chain complexes and spectral sequences}

This subsection follows~\cite{McC01}, where the author learned about 
spectral sequences.
We begin with some definitions.
\begin{definition}
 A graded $R$-module $M$ is an indexed family $M=(M_p)_{p\in\mathbb{Z}}$ of $R$-modules.
\end{definition}

\begin{definition}
 A bigraded $R$-module $M$ is an indexed family $M=(M_{p,q})_{(p,q)\in\mathbb{Z}\times 
\mathbb{Z}}$ of $R$-modules.
\end{definition}

\begin{definition}
 A filtration $F_*$ on an $R$-module $A$ is a family of submodules $\{F_pA\}$ for $p\in 
\mathbb{Z}$ so that 
$$\ldots \subset F_{p-1}A\subset F_pA\subset F_{p+1}A\subset \ldots \subset A \text{ 
(increasing filtration).}$$ 
\end{definition}
Collapsing these filtered modules gives associated graded modules, denoted by $E_*^0(A)$ 
and defined by $E^0_p(A):=F_pA/F_{p-1}A$. 
\begin{remark}
 The construction of the associated graded module to a filtered module is 
straightforward. However, if the filtration is not bounded from below and above it is in 
general not easy to reconstruct the filtration from the associated graded module. 
Moreover, even if one can reconstruct such a filtration, it is in general not canonical.
See~\cite{McC01} for more details.
\end{remark}

Now suppose that we are given a graded $R$-module $H_*$ that is also filtered. Then we 
can define a filtration on each degree by $F_pH_n:=F_pH_*\cap H_n$ and set 
$$E^0_{p,q}(H_*,F):=F_pH_{p+q}/F_{p-1}H_{p+q},$$ 
which is a bigraded module.

\begin{definition}
 A (homological) spectral sequence is a sequence of differential bigraded $R$-modules, 
that is, for $r=1,2,\ldots$ and for $p$ and $q\geq 0$, we have an $R$-module $E_{p,q}^r$. 
Furthermore, each bigraded module, $E_{*,*}^r$ is equipped with a linear mapping 
$$d_r:E_{*,*}^r\to E_{*,*}^r,$$
which is a differential ($d_r\circ d_r=0$), of bidegree $(-r,+r-1)$,
$$d_r:E_{p,q}^r\to E_{p-r,q+r-1}^r.$$
Finally, for all $r\geq 1$, $E_{*,*}^{r+1}\cong H(E_{*,*}^r,d_r)$, that is,
$$E_{p,q}^{r+1}=\frac{\ker d_r:E_{p,q}^r\to E_{p-r,q+r-1}^r}{\im d_r:E_{p+r,q-r+1}^r\to 
E_{p,q}^r}.$$
\end{definition}

There is a notion of convergence of a spectral sequence.

\begin{definition}
 A spectral sequence $\{E_{p,q}^r,d_r\}$ is said to converge to a graded 
$R$-module $H_*$, if there exists a filtration $F$ of $H_*$ such that 
$$E_{p,q}^{\infty}\cong E_{p,q}^0(H_*,F),$$
where $E_{*,*}^{\infty}$ is the limit term of the spectral sequence. 
\end{definition}

We discuss an important family of examples in which spectral sequences arise and a 
sufficient property for them to converge.

\begin{definition}
 An $R$-module $A$ is a filtered differential graded module if
\begin{enumerate}
 \item $A$ is a direct sum of submodules, $A=\bigoplus_{n=0}^{\infty}A_n$.
 \item There is an $R$-linear mapping, $d:A_n\to A_{n-1}$ for all $n$, of degree $-1$, 
satisfying $d\circ d=0$.
 \item $A$ has a filtration $F$ and the differential $d$ respects the filtration, i.e. 
$d:F_pA\to F_pA$.
\end{enumerate}
\end{definition}

Because of point $(iii)$ in the previous definition, the filtration of $A$ induces a 
filtration on the homology $H(A,d):=\ker d/ \im d$, i.e.
$$F_pH(A,d)=\im(\xymatrix{H(F_pA,d)\ar[r]^-{H(inclusion)} & H(A,d))}.$$

The next theorem describes a sufficient condition for a special family of spectral 
sequences to converge. For a proof see~\cite{McC01} for example.

\begin{theorem}\label{thm:spectral_seq_from_filtered_complex}
 Each filtered differential graded module $(A,d,F_*)$ determines a spectral sequence, 
$\{E_{*,*}^r,d_r\}$, $r\in \mathbb{Z}$, with $d_r$ of bidegree $(-r,r-1)$ 
$$E_{p,q}^1\cong H_{p+q}(F_pA/F_{p-1}A).$$
Furthermore, suppose that the filtration is bounded, i.e. there exists $s=s(n)$ and 
$t=t(n)$, such that
$$\{0\}= F_{s-1}A_n\subset F_sA_n\subset F_{s+1}A_n\subset \ldots \subset 
F_{t-1}A_n\subset F_tA_n=A_n,$$
for all $n$.
Then the spectral sequence converges to $H(A,d)$, i.e. 
$$E_{p,q}^{\infty}\cong F_pH_{p+q}(A,d)/F_{p-1}H_{p+q}(A,d).$$
\end{theorem}


\section{Spin structures}

\begin{definition}\label{def:spin_structure}
A spin structure of an oriented vector bundle $E$ over a manifold $X$ is a homotopy 
class of a trivialization of $E$ over the 1-skeleton of $X$ which can be extended to the 
2-skeleton of $X$.
\end{definition}
\begin{definition}\label{def:spin_manifold}
 A spin manifold is an oriented manifold with a spin structure on its tangent 
bundle.
\end{definition}

For a $n$-dimensional vector space $W$ let $V_k(W)$ be the Stiefel manifold, which is the set 
of orthonormal $k$-frames of $W$ with respect to some Riemannian metric. To a vector bundle $\mathbb{R}^n\to E\to B$ 
we can associate the fiber bundle 
$$V_k(\mathbb{R}^n)\to V_k(E)\to B.$$
A section of $V_k(E)\to B$ is the same thing as $k$ orthonormal sections of $E\to B$.
It is known that $\pi_{j}(V_k(\mathbb{R}^n))=0$ for all $j< n-k$. 
This implies the existence of a $k$-frame of $E$ on the $n-k$-skeleton.
Certainly a $k$-frame exist on $B^0$. Suppose it exists on $B^j$. Extending it to 
$B^{j+1}$ is possible if the maps $S^j\to V_k(\mathbb{R}^n)$, coming from the section 
on the $j$-skeleton, extend to disks $D^{j+1}\to V_k(\mathbb{R}^n)$. This depends on 
the homotopy class of the map $S^j\to V_k(\mathbb{R}^n)$, which lies in 
$\pi_{j}(V_k(\mathbb{R}^n))$.

The following is a well-known result from obstruction theory.
\begin{theorem}
 A $k$-frame of $E$ over the $n-k$-skeleton $B^{n-k}$ can be extended to the 
$n-k+1$-skeleton $B^{n-k+1}$ if and only if the $n-k+1$-th Stiefel Whitney class 
$w_{n-k+1}$ is zero. 
\end{theorem}
In particular, we have 
\begin{theorem}
 An manifold is orientable if and only if its first Stiefel-Whitney class vanishes.
\end{theorem}
Notice that if $E\to B$ is orientable then a section of $V_{n-1}(E)\to B$ uniquely 
defines a trivialization and vice versa. 
From this it is also not hard to see that
\begin{theorem}\label{thm:spin_stiefel12}
 $E\to B$ is spin if and only if $w_1(E)=w_2(E)=0$. 
\end{theorem}
For a nice elaboration of the subject and a proof of Theorem~\ref{thm:spin_stiefel12} see~\cite{Bos18}.


\section{Projective duality, discriminants and hyperdeterminants}\label{appendix:hyperdets}

This section summarizes some of the definitions and results given in~\cite{GKZ94}, 
~\cite{Tev03} and~\cite{JLLZF14} and basics can be found in~\cite{Harr13} and~\cite{Hart13}.

\section{Projective duality}

Let $V$ be a $n+1$-dimensional complex vector space and $\mathbb{P}(V):=V\setminus 
\{0\} / \sim$ its projectivization, where $x\sim y \text{ iff } \exists 0 \neq \lambda\in \mathbb{C} \text{ 
s.t. } x=\lambda y$.
The element $[x_0:\ldots:x_n]$ in $\mathbb{P}(V)$ denotes the equivalence classes of the point 
$(x_0,\ldots, x_n)\in V$.
Let $\{0\}\neq U\subset V$ be a non-trivial linear subspace of $V$. 
Then the projectivization $\mathbb{P}(U)$ of $U$ is a projective subspace of 
$\mathbb{P}(V)$.
Let $V^*$ be the dual vector space of $V$. 
We define the dual projective space $\mathbb{P}(V)^*:=\mathbb{P}(V^*)$,which can be thought of as the set of 
hyperplanes in $\mathbb{P}(V)$.
A point $h\in\mathbb{P}(V)^*$ corresponds to a hyperplane in $\mathbb{P}(V)$.
Conversely, to a point $p\in \mathbb{P}(V)$ we associate a hyperplane $h_p\in 
\mathbb{P}(V)^*$, namely the space of all hyperplanes in $\mathbb{P}(V)$ passing 
through $p$.
This gives a natural identification $\mathbb{P}(V)^{**}\cong \mathbb{P}(V)$.
Let $L$ be a projective subspace.
The dual projective subspace $L^*\subset\mathbb{P}(V)^*$ is parametrized by all hyperplanes in $\mathbb{P}(V)$ containing 
$L$.

For convenience, we will from now on write $\mathbb{P}$ or $\mathbb{P}^*$ instead of 
$\mathbb{P}(V)$ or $\mathbb{P}^*(V)$ if there is no 
ambiguity.
If we wish to indicate the dimension of the projective space we write $\mathbb{P}^n$ or 
${\mathbb{P}^*}^{n}$ respectively.

\begin{definition}
 Let $X\subset \mathbb{P}^n$ be a projective variety, possibly with singular points.
Let $I(X):=\{f\in \mathbb{C}[x_0,\ldots, x_n] \mid f \text{ is homogeneous and } f(p)=0 
\quad
\forall p\in X\}$ denote the homogeneous ideal of $X$.
The homogeneous coordinate ring $A(X)$ of $X$ is $A(X):=\mathbb{C}[x_0,\ldots, 
x_n]/I(X)$. 
For a point $x\in X$ the maximal ideal at $x$ is given by $m_x:=\{f\in 
A(X)|f(x)=0\}$.
The Zariski tangent space $\mathbb{T}_x(X)$ of $X$ at $x$ is the dual vector space of 
$m_x/m_x^2$, i.e.
$$\mathbb{T}_x(X):=(m_x/m_x^2)^*.$$
\end{definition}

In the affine setting:

\begin{definition}
 Let $Y\in \mathbb{C}^n$ be an affine variety. 
We define the affine coordinate ring of $Y$ by $A(Y):=\mathbb{C}[x_1,\ldots,x_n]/I(Y)$,
where $I(Y)$ is the ideal of $Y$, given by $I(Y):=\{f\in 
\mathbb{C}[x_1,\ldots,x_n] \mid f(y)=0 \quad \forall y\in Y\}$.
The maximal ideal of $Y$ at a point $y\in Y$ is
$m_y:=\{f\in A(Y) \mid f(y)=0\}$.
The Zariski tangent space of $Y$ at $y$ is
$$\mathbb{T}_y (Y):=(m_y/m_y^2)^*.$$
\end{definition}

In the affine case the Zariski tangent space can also be expressed in local coordinates.
$$\mathbb{T}_y(Y)=\{v\in \mathbb{T}_y(\mathbb{C}^n)\cong \mathbb{C}^n \mid df(v)=0 
\quad \forall f\in 
I(X)\}.$$

Following properties of the Zariski tangent spaces are well-known.

\begin{lemma}
 \begin{enumerate}
  \item $\dim\mathbb{T}_x(X)\geq \dim(X)$, where equality holds if and only if $x$ is a 
smooth point.
  \item For any morphism $\phi:X\rightarrow Y$, there is a natural induced 
$\mathbb{C}$-linear map $T_x(\phi):\mathbb{T}_x(X)\rightarrow \mathbb{T}_{\phi(x)}(Y)$.
 \end{enumerate}

\end{lemma}

For smooth points $x\in X$ there is another notion of tangent space.

\begin{definition}
 Let $X\subset \mathbb{P}^n(V)$ be a smooth irreducible projective variety and $x\in X$.
Let $C(X):=\{l \mid l\in V \text{ a $1$-dim vector subspace s.t. } 
\mathbb{P}(l)\subset 
X\}$ be the cone of $X$ in $V$.
For a smooth point $x\in X\subset \mathbb{P}^n(V)$ and for any non-zero 
vector $v\in V$ in the class 
$x$ the point $v$ is a smooth point of $C(X)$.
We define the embedded projective tangent space of $X$ at $x$ to be the projectivization 
of the 
tangent space of $C(X)$ at any point $v$.
I.e. 
$$\hat{T}_xX:=\mathbb{P}(T_v C(X)).$$
This is well-defined, since for $v\neq w$ both lying in the class $x$ we have $\mathbb{P}(T_v C(X))=\mathbb{P}(T_w C(X))$.
\end{definition}

\begin{definition}
 Let $X\subset \mathbb{P}^n(V)$ be a smooth irreducible projective variety and $x\in X$ 
a smooth point. We say that a hyperplane $H\subset \mathbb{P}(V)$ is tangent to $X$ at 
$x$ if it contains the tangent space $\hat{T}_xX$.
\end{definition}

\begin{definition}
 Let $X$ be a smooth projective algebraic variety in $\mathbb{P}(V)$.
We define its dual variety $X^{\vee}$ as the set of all 
hyperplanes $H\in \mathbb{P}(V)$ which are tangent to $X$ at some point $x$.
If $X$ is not smooth, we define the dual variety $X^{\vee}$ to be the 
Zariski closure of the union of all hyperplanes $H\in \mathbb{P}(V)$ which are tangent to 
$X$ at some smooth point $x$.
\end{definition}

We can reformulate the above definition.
A hyperplane $H$ is tangent to $X$ if and only if the intersection $H\cap X$ 
is singular.
Notice also that, even if $X$ is smooth, its dual variety $X^{\vee}$ could be singular. 
It actually turns out, that $X^{\vee}$ is singular in many cases.

\begin{definition}
 Let $X_{sm}$ be the smooth locus of a projective variety $X\subset \mathbb{P}^n$. 
Let $I^0_X\subset \mathbb{P}\times \mathbb{P}^*$ be the set of pairs $(x,H)$, where $x\in 
X$ and $H \in X^{\vee}$ a hypersurface tangent to $X$ at $x$. 
The conormal variety $I_X$ of $X$ is the Zariski closure of $I^0_X$.
\end{definition}

\begin{proposition}\label{prop:irreducible}
 If $X$ is irreducible then $X^{\vee}$ is irreducible.
\end{proposition}
\begin{proof}
Consider the projection
\begin{equation}
\xymatrix{
&I_X^0\ar[r]^{pr_1} &X_{sm}
},
\end{equation}
which defines $I_X^0$ as a bundle over $X_{sm}$. The fibers are 
$n-\dim(X)-1$-dimensional projective subspaces of $\mathbb{P}^*$.
More precisely, the fiber of $I_X^0$ at a point $x\in X$ is the projective subspace 
$\mathbb{P}/T^{\vee}_xX$ of dimension $n-dimX-1$.
As $X$ is irreducible also $I_X^0$ and $I_X$ are irreducible varieties and their 
dimension is $n-1$.
The image of the projection
\begin{equation}
\xymatrix{
&I_X\ar[r]^{pr_2} &\mathbb{P}^*
}
\end{equation} is$X^{\vee}$. Hence, also $X^{\vee}$ is irreducible.
\end{proof}

\subsection{General discriminants}

For convenience, we assume that $\mathbb{P}=\mathbb{P}(V^*)$ for some complex vector 
space $V$. Let $X\subset \mathbb{P}$ be a projective variety and $X^{\vee}$ its dual in 
$\mathbb{P}^*=\mathbb{P}(V)$.
Analysing the proof of Proposition~\ref{prop:irreducible} we see that, since 
$\dim I_X=n-1$, we would in most cases except that $X^{\vee}$ is a hypersurface.

\begin{definition}
 The defect $\defe(X)$ of a projective algebraic variety $X\in \mathbb{P}$ is defined 
to be the number $\codim_{\mathbb{P}^*} X^{\vee}-1$.
I.e. $\defe(X)$ is zero if and only if $X^{\vee}$ is a hypersurface.
\end{definition}
 
Suppose that $X^{\vee}$ is a hypersurface in $\mathbb{P}^*$.

\begin{definition}
Suppose $\codim X^{\vee}=1$.
The defining polynomial of $X^{\vee}$ in $\mathbb{P}(V)$ is called the $X$-discriminant and is denoted by $\Delta_X$.
It is defined uniquely up to multiplication by a non-zero constant.
If $\defe(X)>0$ or equivalently $\codim X^{\vee}>1$ we set $\Delta_X=1$.
\end{definition}

\subsection{The reflexivity theorem and some applications}

The next observation is a well-known result for projective duality.
\begin{theorem}[Reflexivity Theorem]\label{thm:reflex}
 \hspace{2em}
 \begin{enumerate}
  \item Let $X\subset \mathbb{P}^n$ be an irreducible projective variety. Then
$${X^{\vee}}^{\vee}=X.$$
  \item More precisely, for any smooth points $z\in X$ and $H\in X^{\vee}$ we have the 
following relation:
$H$ is tangent to $X$ at $z$ if and only if $z$ is tangent to $X^{\vee}$ at $H$.
 \end{enumerate}
\end{theorem}

In order to prove this theorem we will recall some facts about normal and conormal 
bundles.

\begin{definition}
 Let $X$ be a smooth algebraic variety and $TX:=\bigcup_{x\in X}T_xX$ its 
tangent bundle.
For a smooth subvariety $Y\subset X$ we have that $TY\subset TX|_Y$ is a subbundle.
We define the normal bundle of $Y$ in $X$ by
$$N_XY:=TX|_Y/TY.$$
Similarly, the conormal bundle of $Y$ in $X$ is defined by
$$N_X^*Y:=T^*X|_Y/T^*Y.$$
\end{definition}

Recall that $I_X:=\overline{\{(x,H)\mid H \text{ tangent to } X \text{ 
at } x\}}\subset \mathbb{P}\times \mathbb{P}^*$ is the conormal variety and 
$$X^{\vee}=pr_2(I_X)$$ and $$pr_1:I_X^0\rightarrow X_{sm}$$ is a projective bundle.
The choice of a hyperplane $H$ tangent to $X$ at a smooth point $x$ is 
equivalent to the choice of a hyperplane $H\subset T_x\mathbb{P}^n$ containing $T_xX$.
Hence, we have the identification
$$I_X^0=\mathbb{P}(N^*_{X_{sm}}\mathbb{P}).$$
We therefore have the following reformulation of the Reflexivity Theorem:

\begin{theorem}[Reflexivity Theorem']
 $$I_X=I_{X^*}$$
\end{theorem}

For the proof of the Reflexivity theorem it is convenient to work with vector spaces 
instead of projective spaces.
Let $C:=C(X)\subset V$ be the cone of $X$ and similarly $C^*:=C(X^{\vee})\subset V^*$ the 
cone of $X^{\vee}$.
We define $\Lag(C):=\overline{N^*_{C_{sm}}V}$ the closure of the conormal bundle of 
$C_{sm}$ in the cotangent bundle $T^*V$ of $V$.
We identify the cotangent bundles $T^*V$ as well as $T^*V^*$ with the space $V\times 
V^*$. 

Let $M$ be a smooth algebraic variety and $K\subset M$ be any irreducible subvariety.
Then $T^*M$ carries a canonical symplectic structure. Consider 
$\Lag(K):=\overline{N^*_{K_{sm}}M}$.
The following result is well-known. 
\begin{theorem}\label{thm:Lagrangian}
\hspace{2em}
\begin{enumerate}
  \item $\Lag(K)$ is a Lagrangian subvariety of $T^*M$.
  \item Any canonical Lagrangian subvariety of $T^*M$ has the form $\Lag(K)$ for some 
irreducible 
subvariety $K\subset M$.
\end{enumerate}
\end{theorem}

See~\cite{Tev03} for a proof.
With this we can prove Theorem~\ref{thm:reflex}.
\begin{proof}[Proof of the Reflexivity Theorem']
 By the above discussion it suffices to show $\Lag(C)=\Lag(C^*)$.
The identification $T^*V=V\times V^*=T^*V^*$ takes the canonical symplectic structure of 
$T^*V$ to minus the canonical symplectic structure of $T^*V^*$.
Then $\Lag(C)$ as a subvariety of $T^*V^*$ is also a canonical Lagrangian.
By theorem~\ref{thm:Lagrangian} $\Lag(C)\subset T V^*$ is the Lagrangian $\Lag(Z)$ 
of some irreducible subvariety $Z\subset V^*$.
But $pr_2(\Lag(Z))=pr_2(\Lag(C))=C^*$.
Hence, $Z=C^*$ and $\Lag(C)=\Lag(C^*)$.
\end{proof}

In what follows we list some important implications of the Reflexivity theorem.
%
%
\begin{definition}\label{def:ruled}
 A projective subvariety $X\subset \mathbb{P}$ is ruled in projective space of dimension 
$r$ if for any point $x\in X$ there exists an $r$-dimensional projective subspace $L$, 
such that $x\in L\subset X$.
\end{definition}

\begin{remark}
 It is enough to show that the condition in Definition~\ref{def:ruled} is satisfied in a 
Zariski open dense subset of $X$.
\end{remark}

The following result describes the dual variety $X^*$ in the case that it is not a 
hypersurface, i.e. if $\defe(X)>0$.

\begin{theorem}
 Suppose that $\defe(X)=r>0$. Then
$X$ is ruled in projective space of dimension $r$.
\end{theorem}

\begin{proof}
By theorem~\ref{thm:reflex} we can switch the roles of $X$ and $X^*$.
Suppose that $\codim(X)=r>1$. We want to show that $X^*$ is ruled in projective space of 
dimension $r$.
Let $H$ be tangent to $X$ at $x$, then by definition $\hat{T}_xX\subset H$.
The space of all hyperplanes in $\mathbb{P}$ tangent to $X$ at $x$ can be identified with 
the 
space $\mathbb{P}/\hat{T}_xX$, which is a projective subspace of $\mathbb{P}$ of 
dimension $r$.
\end{proof}

\subsection{The hyperdeterminant}

\paragraph{Geometric interpretation}

Consider the complex vector spaces $V_j^*=\mathbb{C}^{k_j+1}$ and their projectivizations 
$\mathbb{P}^{k_j}:=\mathbb{P}(V_j^*)$ for $j=1,\cdots, r$.
The product $\mathbb{P}^{k_1}\times\cdots \times \mathbb{P}^{k_r}$ can be embedded into 
$\mathbb{P}^{(k_1+1)(k_2+1)\cdots (k_r+1)-1}:=\mathbb{P}(V_1^*\otimes \cdots \otimes 
V_r^*)$ via the Segre 
embedding:
\begin{equation*}
 \begin{array}{ccc}
  S:\mathbb{P}^{k_1}\times\cdots \times \mathbb{P}^{k_r} &\rightarrow& 
\mathbb{P}^{(k_1+1)(k_2+1)\cdots (k_r+1)-1}\\
 {[}x^1{]}, \ldots , {[}x^r{]} &\rightarrow & 
{[}x^1\otimes \ldots \otimes x^r{]}\\
 \end{array}
\end{equation*}
or in coordinates
\begin{equation*}
 \begin{array}{ccc}
  S:\mathbb{P}^{k_1}\times\cdots \times \mathbb{P}^{k_r} &\rightarrow& 
\mathbb{P}^{(k_1+1)(k_2+1)\cdots (k_r+1)-1}\\
 {[}x^1_{0}: \ldots :x^1_{k_1}{]}, \ldots , {[}x^r_0: \ldots :x^r_{k_r}{]} &\rightarrow & 
{[}\ldots : 
x^1_{i_1} \cdots  x^r_{i_r}: \ldots{]}\\
 \end{array}
\end{equation*}
For simplicity we denote $\mathbb{P}:=\mathbb{P}^{(k_1+1)(k_2+1)\cdots (k_r+1)-1}$. 
Set $X:= Im(S) \subset \mathbb{P}$ a closed 
algebraic subvariety. Notice that $X$ is the set of all decomposable elements 
in $\mathbb{P}(V_1^*\otimes \cdots\otimes V_r^*)$.
Suppose $X$ is irreducible in $\mathbb{P}$, then Proposition~\ref{prop:irreducible} implies that its 
dual $X^{\vee}$ is irreducible in $\mathbb{P}^*$.

\begin{definition}
The hyperdeterminant $\Det$ of the format $(k_1+1)\times \cdots \times (k_r+1)$ is the 
$X$-discriminant of $X:= Im(S) \subset \mathbb{P}^{(k_1+1)(k_2+1)\cdots (k_r+1)-1}$.
I.e.
\begin{equation}
 \Det_{(k_1+1)\times \cdots \times (k_r+1)}:=\Delta_X.
\end{equation}
It is a homogeneous polynomial on $V_1\otimes \ldots \otimes V_r$ and moreover it is defined uniquely up to sign if we require it to have integral coefficients and to be irreducible over $\mathbb{Z}$.
\end{definition}

\begin{theorem}
 The hyperdeterminant is a non-trivial polynomial (i.e. $X^{\vee}$ is a hypersurface) if and only if 
$$k_i\leq \sum_{j\neq i} k_j,$$ for all $j\in\{1,\ldots,r\}$.
\end{theorem}
The proof of this theorem is beyond the scope of this thesis. The reader is referred to~\cite{GKZ94}.

To a hypermatrix $A=(a_{i_1,\cdots, i_r})\in \mathbb{C}^{(k_1+1)(k_2+1)\cdots (k_r+1)}$ 
we can associate a linear form 
\begin{equation}
\begin{array}{ccc}
  A:V_1^*\otimes \cdots \otimes V_r^* & \rightarrow & \mathbb{C}.\\
\end{array}
\end{equation}
 We abuse notation and denote this linear form again by $A$. It is induced by 
the multilinear form
\begin{equation}
\begin{array}{ccc}
  V_1^*\times \cdots \times V_r^* & \rightarrow & \mathbb{C},\\
(\gamma^1,\cdots,\gamma^r)&\mapsto & \sum_{i_1,\cdots,i_r} a_{i_1,\cdots,i_r} 
\gamma^1_{i_1}\dots \gamma^r_{i_r}.
\end{array}
\end{equation}
Note that $A\in (V_1^*\otimes \cdots \otimes V_r^*)^*=V_1\otimes \cdots \otimes 
V_r$.
Recall that $X\subset \mathbb{P}(V_1^*\otimes \cdots\otimes V_r^*)$ and 
$X^{\vee}\subset \mathbb{P}(V_1\otimes \cdots\otimes V_r)$.
Therefore, the hyperdeterminant is a polynomial on $V_1\otimes \cdots\otimes V_r$.

\begin{definition}
 We call $\Det(A):=\Delta_X(A)$ the hyperdeterminant of the hypermatrix $A\in V_1\otimes 
\cdots\otimes V_r$.
\end{definition}

\begin{definition}
 A slice of $A\in V_1\otimes \cdots\otimes V_r$ in the $j$-th direction is 
the hypermatrix $(a_{i_1,\cdots,i_{j-1},l,i_{j+1}, i_r})\in 
V_1\otimes \cdots \otimes \hat{V_j}\otimes  \cdots \otimes V_r$ for some fixed 
$l\in\{1,\cdots, 
k_j+1\}$.
\end{definition}

Let $M\in GL(V_j)$ be an invertible matrix of dimension $(k_j+1)\times (k_j+1)$.
Consider the action 
\begin{equation}
 \begin{array}{ccc}
  GL(V_j)\times V_1\otimes \ldots \otimes V_r & \rightarrow & V_1\otimes \ldots 
\otimes V_r\\
 (M,A) & \mapsto & M*_j A,
 \end{array}
\end{equation}
where 
$$M*_j A:=(\sum_{l=1, \ldots, r_j+1} 
m_{l,i_j}\cdot a_{i_1,\ldots,i_{j-1},l,i_{j+1},\ldots,i_r}).$$
More generally we define the action 
\begin{equation}
 \begin{array}{ccc}
  (GL(V_1)\times \cdots \times Gl(V_r))\times V_1\otimes \cdots \otimes V_r & 
\rightarrow & V_1\otimes \cdots 
\otimes V_r\\
 ((M_1,\cdots, M_r),A) & \mapsto & (M_1,\cdots, M_r)*A,
 \end{array}
\end{equation}
where $$(M_1,\cdots, M_r)*A:=M_1*_1 M_2 *_2 \cdots M_r*_r A.$$

\begin{proposition}\label{prop:SL-invariance}
The hyperdeterminant is invariant under the action of 
the group $SL(V_1)\times \cdots \times SL(V_r)$.
It is relatively invariant under the action of the group 
$GL(V_1)\times \cdots\times GL(V_r)$ in the following sense.
Let $A\in V_1\otimes\ldots\otimes V_r$ and $(M_1,\ldots,M_r)\in GL(V_1)\times 
\cdots\times GL(V_r)$. Then:
\begin{equation}
 Det((M_1,\ldots,M_r)* A)= 
det(M_1)^{\frac{d}{k_1+1}}\cdots det(M_r)^{\frac{d}{k_r+1}} Det_{(k_1+1)\times \cdots 
\times (k_r+1)}(A),
\end{equation}
where $d$ is the degree of $Det_{(k_1+1)\times \cdots \times (k_r+1)}$.
\end{proposition}

\begin{proof}
 Note that the Segre embedding is canonical, i.e. independent 
of the choice of the basis. 
It follows from the definition that the hyperdeterminant is
invariant under the action of the group $SL(V_1)\times \cdots\times SL(V_r)$.
Moreover, it is relatively invariant under the action of $GL(V_1)\times \cdots\times 
GL(V_r)$, i.e. invariant up to multiplication by some non-zero constant.
We see that for some $M,N\in GL(V_j)$ we have 
$$M*_j(N*_j A)=(MN)*_jA,$$
where $MN$ denotes the usual matrix multiplication.
Moreover, if $\det(M)=1$ we have that $
Det(M*_jA)= Det(A)$.
It follows that $Det(M*_jA)=det(M)^p\Det(A)$ for some integer $p\in \mathbb{Z}$.
By definition the hyperdeterminant is homogeneous on each slice.
The degree has to be the same on parallel slices. 
Since there are $(k_j+1)$ parallel slices, the hyperdeterminant is homogeneous of degree 
$\frac{d}{k_j+1}$.
Thus
$$Det(M_j*_jA)=det(M_j)^{\frac{d}{k_j+1}} Det(A)$$ and hence, also
$$Det((M_1,\cdots, 
M_r)*A)=det(M_1)^{\frac{d}{k_1+1}}\cdots det(M_r)^{\frac{d}{k_r+1}} 
Det_{(k_1+1)\times \cdots \times (k_r+1)}(A).$$
\end{proof}

\begin{corollary}
\hspace{2em}
 \begin{enumerate}
\item $\frac{d}{k_i+1}$ is an integer.
\item Swapping to parallel slices in the $i$th direction changes the hyperdeterminant by 
$(-1)^{\frac{d}{k_i+1}}$.
\item A hypermatrix with linearly dependent slices has vanishing hyperdeterminant.
 \end{enumerate}
\end{corollary}

For the case of a hypermatrix of format $2\times 2\times 2$ there is a simple formula for 
the discriminant.

\begin{proposition}\label{prop:det222}
 The hyperdeterminant of a hypermatrix $A=(a_{ijk})$ with $i,j,k=0,1$ is given by
\begin{equation}
 \begin{matrix}\label{det222}
Det(A) &=& (a_{000}^2a_{111}^2+a_{001}^2a_{110}^2+a_{010}^2a_{101}^2+a_{011}^2a_{100}^2)\\
 
&&-2(a_{000}a_{001}a_{110}a_{111}+a_{000}a_{010}a_{101}a_{111}+a_{000}a_{011}a_{100}a_{111
}\\
&& 
+a_{001}a_{010}a_{101}a_{110}+a_{001}a_{011}a_{110}a_{100}+a_{010}a_{011}a_{101}a_{100})\\
&&+4(a_{000}a_{011}a_{101}a_{110}+a_{001}a_{010}a_{100}a_{111}).\\
 \end{matrix}
\end{equation}
\end{proposition}

See~\cite{GKZ94} for a proof.

\paragraph{Algebraic interpretation}

\begin{definition}
 A hypermatrix $A\in \mathbb{C}^{(k_1+1)\times \cdots \times (k_r+1)}$ is called 
degenerate if there exists a non-zero vector $(x^1, \cdots, x^r)\in 
\mathbb{P}^{k_1}\times\cdots \times \mathbb{P}^{k_r}$ such that
\begin{equation}
A(x^1,\cdots,x^{j-1},(\mathbb{C}^{k_j+1})^*,x^{j+1},\cdots,x^r)=0 \text{ } \forall j=1,\cdots, r.
\end{equation}
The set of all non-zero vectors $(x^1,\cdots,x^r)\in 
\mathbb{P}^{k_1}\times\cdots \times \mathbb{P}^{k_r}$ with this property is the 
kernel $K(A)$ of $A$.
\end{definition}

Notice that $X$ is the set of all decomposable elements in $\mathbb{P}$.
The tangent spaces of $X\subset \mathbb{P}(V_1^*\otimes \cdots, \otimes V_k^*)$ at a 
point $x=[(x^1\otimes \cdots \otimes x^r)]$ is described in the following proposition.
We leave the verification to the reader.

\begin{proposition}
 Let $x=[(x^1\otimes \cdots \otimes x^r)]$ be a point in $X$. 
\begin{enumerate}
 \item The tangent space of $X$ at $x$ in $\mathbb{P}^{(k_1+1)\cdots(k_r+1)-1}$ is the 
span of the $r$ linear subspaces
$\mathbb{P}^{k_j}_x$, which are defined as the projectivization of $x^1\otimes \cdots 
\otimes x^{j-1}\otimes V_j^*\otimes x^{j+1}\otimes \cdots\otimes x^r$.
I.e. 
$$T_xX=\mathbb{P}(\bigoplus_{j=1,\cdots, r} x^1\otimes \cdots \otimes x^{j-1}\otimes 
V_j^* 
\otimes x^{j+1}\otimes \cdots\otimes x^r).$$
\item $X\cap T_x X=\cup_j \mathbb{P}^{k_j}_x$.
\item Any linear space in $X$ passing through $x$ is contained in one of the linear 
subspaces $\mathbb{P}^{k_j}_x$.
\end{enumerate}
\end{proposition}

Here is another corollary of the Reflexivity theorem.

\begin{theorem}\label{preliminaries-thm:smooth-point-of-the-hyperplane-section}
 Let $X\subset \mathbb{P}^*$ be smooth and $X^{\vee}\subset \mathbb{P}$ a 
hypersurface. Let $z_1,\ldots,z_n$ be homogeneous coordinates of $\mathbb{P}^{n*}$ and 
$a_1,\ldots, a_n$ the dual homogeneous coordinates on $\mathbb{P}^n$.
Suppose that $f=(a_1:\ldots:a_n)$ is a smooth point of $X^{\vee}$.
Then the hyperplane section $\{f=0\}\cap X$ has a unique singular point with coordinates
$$(\frac{\partial \Delta_X}{\partial a_1}(f):\ldots:\frac{\partial \Delta_X}{\partial 
a_n}(f)).$$
\end{theorem}

\begin{proof}
 Let $H\in \mathbb{P}^n$ be the hyperplane corresponding to $f$. By the Reflexivity 
theorem~\ref{thm:reflex} it is tangent to $X$ at $z$ if and only if the hyperplane 
corresponding to $z$ is tangent to $X^{\vee}$ at $H$. Since $X^{\vee}$ is a hypersurface 
and smooth at $f$ such a point $z$ is unique with coordinates $(\frac{\partial 
\Delta_X}{\partial 
a_1}(f):\ldots:\frac{\partial \Delta_X}{\partial a_n}(f))$.
\end{proof}

\begin{corollary}\label{cor:smoothptofdegmatrices}
 Let $A_0=(a_{i_1,\ldots,i_r})$ be a degenerate matrix that is a smooth point of the 
variety of degenerate matrices. Then the kernel $K(A_0)$ consists of a unique point 
$x=(x^1,\ldots,x^r)\in \mathbb{P}^{k_1}\times\cdots \times \mathbb{P}^{k_r}$.
Let $S$ denote the Segre embedding.
The coordinates of $S(x)=x^1\otimes \cdots \otimes x^r$ are up to normalization given by
$$x^1_{i_1}\dots x^r_{i_r}=\frac{\partial \Delta(A)}{\partial 
a_{i_1,\ldots,i_r}}|_{A=A_0}$$ for all $i_1,\ldots,i_r$.
\end{corollary}

\begin{remark}\label{rmk:equivalent_properties_to_deltazero}
Summarizing the properties from the preceding chapter: for a hypermatrix 
$A\in  \mathbb{C}^{(k_1+1)\times \cdots \times (k_r+1)}$ we have
\begin{equation*}
\begin{array}{ccc}
  Det(A)=\Delta_X(A)=0 &\iff&\\
 \text{ hyperplane section } \{A=0\}\cap X \text{ is singular} &\iff&\\
  \{A=0\} \text{ belongs to } X^{\vee} &\iff&\\
  \exists x\in X : A(x)=\frac{\partial A}{\partial x_i}(x)=0 \forall i\\
\end{array}.
\end{equation*}
\end{remark}

\subsection{Schl\"aflis technique of computation}

There is a useful technique to compute the hyperdeterminant for some hypermatrices of low 
degrees.
This method is due to~\cite{Sch52}.

Given a hypermatrix $A=(a_{i_1,\ldots,i_r})\in \mathbb{C}^{(k_1+1)\cdots (k_r+1)}$, 
consider the linear operator
\begin{equation*}
 \begin{array}{ccc}
  f_A:\mathbb{C}^{k_1+1}&\rightarrow & \mathbb{C}^{(k_2+1)\cdots (k_r+1)}\\
x &\mapsto & \sum_{i_1=0}^{k_1} a_{i_1,\ldots,i_r} x_{i_1}.\\
 \end{array}
\end{equation*}
This is a family of $r$-dimensional matrices linearly depending on $x$.
The function $F_A(x):=Det(f_A)(x)$ is a homogeneous polynomial in the $x_{i_1}$. Hence, we 
can compute its discriminant $\Delta(F_A)$.
(Here $\Delta$ is just the usual discriminant of a quadratic polynomial, which is also the same as the discriminant of the Segre variety
$S(\mathbb{P}^1\times \mathbb{P}^1)\subset \mathbb{P}^3$.)
In particular, $\Delta(F_A)$ can be viewed as a homogeneous polynomial in the matrix 
entries of $A$ which is zero if and only if there exists an $x$ such that $F_A(x)$ is 
degenerate.

The first crucial observation is the following theorem

\begin{theorem}
 The polynomial $\Delta(F_A)$ is divisible by the hyperdeterminant $Det(A)$ of $A$.
\end{theorem}

\begin{remark}
 Note that $\Delta(F_A)$ could be identically zero.
\end{remark}

\begin{proof}
Suppose that $Det(A)$ is zero, i.e. the hypermatrix $A$ is degenerate. 
Then $K(A)$ is non empty. Let $(x^1\otimes \ldots \otimes x^r)\in K(A)$.
We want to show that $F_A$ vanishes at $x^1$ with all its first derivatives.
We use the notation $$B(x)={(b_{i_2\ldots i_r})}_{i_2,\ldots,i_r}=f_A(x) \text{ , } 
c_{i_2,\ldots,i_{r}}={\frac{\partial 
Det}{\partial b_{i_2,\ldots,i_{r}}}}|_{B}.$$
Then $$F_A(x)=Det_{(k_2+1)\times\ldots \times (k_r+1)}(B)(x)$$ and
$$\frac{\partial F_A(x)}{\partial x_{i_1}}|_{x=x^1}=\sum_{i_2,\ldots,i_r} 
a_{i_1,i_2,\ldots,i_r}c_{i_2,\ldots,i_r}$$ for $i_1=1,\ldots,k_1$.
Since $(x^1\otimes \ldots \otimes x^r)\in K(A)$ also $(x^2\otimes \ldots \otimes x^r)\in K(A_0)$, where $A_0=B(x^1)$. Hence, $F_A(x^1)=0$.
If the $c_{i_2,\ldots,i_r}$ are all zero 
then $\frac{\partial F_A(x)}{\partial x_{i_1}}|_{x=x^1}=0$, which implies $\Delta(F_A)=0$  
and the theorem is proven.
Suppose now that the $c_{i_2,\ldots,i_r}$ are not all zero. 
Then the matrix $B$ is a smooth point of the variety of degenerate matrices. 
By Corollary~\ref{cor:smoothptofdegmatrices} $c_{i_2,\ldots,i_r}=x^2_{i_2}\cdots 
x^r_{i_r}$, and we get 
$$\frac{\partial F_A(x)}{\partial x_{i_1}}|_{x=x^1}=\sum_{i_2,\ldots,i_r} 
a_{i_1,i_2,\ldots,i_r}x^2_{i_2}\cdots x^r_{i_r}.$$ 
Since $(x^1,\ldots,x^r)\in K(A)$ this shows that 
$$\frac{\partial F_A(x)}{\partial x_{i_1}}|_{x=x^1}=\sum_{i_2,\ldots,i_r} 
a_{i_1,i_2,\ldots,i_r}x^2_{i_2}\cdots x^r_{i_r}=A((1,\ldots,1)\otimes 
x^2\otimes\ldots\otimes x^r)=0.$$
\end{proof}

The above theorem can be refined.
Let $\nabla=\nabla(k_2,\ldots,k_r)$ be the variety of degenerate hypermatrices of format 
$(k_2+1)\times\ldots\times (k_r+1)$.
The projectivization of $\nabla$ is by definition the dual variety $X^{\vee}$ 
of the Segre variety $\mathbb{P}^{k_2}\times \ldots \times\mathbb{P}^{k_r}\subset 
\mathbb{P}^{(k_2+1)\cdots(k_r+1)-1}$.
Furthermore, let $\nabla_{sing}$ denote the variety of singular points of $\nabla$ and 
$X^{\vee}_{sing}$ its projectivization.
Let $c(k_2,\ldots,k_r)$ be the minimal codimension of all irreducible parts of $\nabla_{sing}$.
Then there is a more precise description of the ratio of $\Delta(F_A)/Det(A)$.

\begin{theorem}\label{thm:SchlaefliRatio}
 The ratio $G(A):=\Delta(F_A)/Det(A)$ has the following form:
 \begin{enumerate}
  \item if $k_1+1<c(k_2,\ldots,k_r)$ then $G$ is a non-zero constant
  \item if $k_1+1=c(k_2,\ldots,k_r)$ then 
  $$G(A)=\prod R_Z^{m_z}(Im(f_A)),$$
  where $Z$ ranges over the irreducible components of $\nabla$, 
  $R_Z$ is the chow form of $Z$ and $m_Z$ are some multiplicities.
  \item if $k_1+1>c(k_2,\ldots,k_r)$ then $G(A)$ and hence, also $\Delta(F_A)$ are 
identically zero.
\end{enumerate}
\end{theorem}

The first step to prove the theorem is the next lemma.
\begin{lemma}\label{lem:notdivisiblebyDetDet}
 If $\Delta(F_A):\mathbb{C}^{(k_1+1)}\rightarrow 
\mathbb{C}$ is not identically zero, then it is not divisible by $Det(A)^2$.
\end{lemma}

\begin{proof}
Recall that the dual variety $X^{\vee}$ is the union of the projective spaces $P_x$, 
$x\in X$ where $P_x$ is the set of hyperplanes tangent to $X$ at $x$.
The codimension of $P_x$ is $\dim(X)+1=k_1+\ldots + k_r+1$.
Suppose that $\Delta(F_A)$ is not identically zero.
The vanishing of $\Delta(F_A)$ means that $F_A$ vanishes at some points with 
all its derivatives.
In other words, $Im(f_A)=f_A(\mathbb{C}^{k_1+1})$ is tangent to $\nabla$ at some non-zero 
point.
Suppose that $Det(A)^2$ divides $\Delta(F_A)$ and that $\Delta(F_A)$ vanishes at some 
point $x$ 
but is not identically zero.
For every one-parameter algebraic family $A_t$ of $r$-dimensional hypermatrices with 
$A_0=A$ and such that $Im(f_{A_0})$ is tangent to $\nabla$, the function $t\mapsto 
\Delta(F_{A_t})$ is divisible by $t^2$.
Indeed, if $\Delta(F_A)=0$, then either $Det(f_A)(x)=0$ for all $x$ or there exists an 
$x_0$ such that $F_A(x_0)$ has a multiple root. 
If $Det(f_A)(x)$ vanishes for all $x$ then also $Det(A)$ vanishes and thus 
$Det(A_t)=o(t)$ and $\Delta(F_A)$ is divisible by $t^2$. 
The function $F_A(x_0)$ having a multiple root on the other hand means that $Im(f_A)$ is 
tangent to $\nabla^{sing}$ at some non-zero point. Hence, the tangent is of order two, 
which again implies that $\Delta(F_A)$ is divisible by $t^2$.
We will show that this is impossible for a suitable choice of a generic family $A_t$.

Let $B\in\nabla$ and $\xi:=[B]\in X^{\vee}$ .
We may assume that $\xi$ is contained in exactly one $P_x$.
If this was not the case for generic points $\xi\in X^{\vee}$ then $X^{\vee}$ would not 
be a hypersurface.
Now consider the space $Z$ of $k_0$-dimensional projective subspaces of $P$ that are 
tangent to $X^{\vee}$ at $\xi$.
Since $k_0\le k_1+\ldots +k_r$ we know that a dense open subset of $Z$ is formed by 
projective subspaces that meet $P_x$ only at the point $\xi$.
Let $L\in Z$ be such a generic point.
Hence, $L$ has a simple tangency with $X^{\vee}$.
Therefor we can find an algebraic family $A_t$ of $r$-dimensional hypermatrices 
with $\mathbb{P}(Im(F_{A_0}))=L$.
The simple tangency condition of $L$ implies that $t\mapsto \Delta(F_{A_t})$ has a simple 
zero at $t=0$.
This contradicts the fact that $\Delta(F_{A_t})$ is divisible by $t^2$.
\end{proof}

We finish this section with the proof of Theorem~\ref{thm:SchlaefliRatio}.

\begin{proof}[Proof of Theorem~\ref{thm:SchlaefliRatio}]
 We have seen in the proof of Lemma~\ref{lem:notdivisiblebyDetDet} that $\Delta(F_A)=0$ 
if and only if either $Det(A)=0$ or $f_A(x_0)\in \nabla^{sing}$ for some non-zero $x_0\in 
\mathbb{C}^{k_1+1}$. 
Let $W$ be the variety of matrices $A$ such that $Im(f_A)$ intersects $\nabla_{sing}$ at 
a non-zero point. By Lemma~\ref{lem:notdivisiblebyDetDet} we see that the ratio $G(A)$ 
vanishes if and only if $A\in W$. 
Notice that $codim(W)>1$ if $ k_1+1<c(k_2,\dots,k_r)$, $codim(W)=1$ if $ 
k_1+1=c(k_2,\dots,k_r)$ and if $ k_1+1>c(k_2,\dots,k_r)$ then $W$ coincides with the 
whole matrix space. The Theorem follows from the definition of the Chow form.
\end{proof}

\subsection{Duality, discriminants and projections}\label{subsec:projections}

Let $\mathbb{P}$ be an $n$-dimensional projective space and $K\subset \mathbb{P}$ a 
$k$-dimensional subspace, where $k>0$.
We can identify the quotient space $\mathbb{P}/K$ with the set of $(k+1)$-dimensional 
projective subspaces in $\mathbb{P}$ containing $K$.
Suppose $H$ is a $(n-k-1)$-dimensional projective subspace in $\mathbb{P}$ not 
intersecting $K$.
Then every $(k+1)$-dimensional plane containing $K$ intersects $H$ in exactly one point. 
Hence, we can identify $H$ with $\mathbb{P}/K$.
Moreover, the dual space $(\mathbb{P}/K)^*$ in $\mathbb{P}^*$ is identified with the 
space 
of hyperplanes 
in $\mathbb{P}$ containing $K$.
We use the following notation. 
The projection from $K$ is the map
\begin{equation*}
 \pi_K:\mathbb{P}\setminus K \rightarrow \mathbb{P}/K,
\end{equation*}
defined by sending a point $x\in \mathbb{P}\setminus K$ to the $(k+1)$-dimensional 
projective 
subspace spanned by $x$ and $K$.

\begin{remark}\label{rmk:proj}
 Notice that if $K\subset K'$ are projective subspace of a projective space $\mathbb{P}$ 
then 
$\mathbb{P}/K'\subset \mathbb{P}/K$ and $K'/K$ is a subset of $\mathbb{P}/K$ such that 
$(\mathbb{P}/K)/(K'/K)=\mathbb{P}/K'$ and
\begin{equation}
 \pi_{K'/K}\circ \pi_K=\pi_{K'}
\end{equation}
\end{remark}

The following proposition and its proof can be found in~\cite{GKZ94} and~\cite{Tev03}. 
\begin{proposition}\label{prop:pi_K}
 Let $X\subset \mathbb{P}$ be an algebraic subvariety not intersecting the projective 
subspace $K$ 
and such that $\dim(X)<\dim(\mathbb{P}/K)$.
Then
\begin{equation}
 \pi_K(X)^{\vee}\subset K^{\vee}\cap X^{\vee}.
\end{equation}
Moreover, equality hold if $\pi_K:X\rightarrow \pi_K(X)$ is an isomorphism of algebraic 
varieties.
\end{proposition}

\begin{proof}
 Let $H$ be a hyperplane tangent to $\pi_K(X)$ at a point $\pi_K(p)$. 
 By definition $H$ is a hyperplane in $\mathbb{P}/K$ and therefore it can be regarded as 
a 
hyperplane in $\mathbb{P}$ containing $K$.
 Moreover, as a hyperplane in $\mathbb{P}$ it is tangent to $X$ at $p$.
 This proves the first part of Proposition~\ref{prop:pi_K}.
 
 Now suppose that $\pi_K: X \rightarrow \pi_K(X)$ is an isomorphism of algebraic 
varieties.
Let $X_0^{\vee}$ denote the set of hyperplanes tangent to $X$ at a smooth point $p$ of 
$X$. 
Similarly, $\pi_K(X)_0^{\vee}$ denotes the set of hyperplanes tangent to $\pi_K(X)$ at a 
smooth point.
The isomorphism $\pi_K: X \rightarrow \pi_K(X)$ induces an isomorphism of the smooth loci.
Let $H$ be a hyperplane tangent to $X$ at a smooth point $p$, i.e. $H\in 
X_0^{\vee}$.
Then $H\cap K^{\vee}$ is tangent to $\pi_K(p)$, since $K^{\vee}\cong \mathbb{P}/K$.
Hence, $X_0^{\vee}\cap K^{\vee}=  \pi_K(X)_0^{\vee}$.
Clearly $\overline{X_0^{\vee}}=X^{\vee}$ and 
$\overline{\pi_K(X)_0^{\vee}}=\pi_K(X)^{\vee}$.
Hence, it is left to show that also $\overline{X_0^{\vee}\cap 
K^{\vee}}=\overline{X_0^{\vee}}\cap 
K^{\vee}$.

\begin{lemma}
 Let $H(t)$ be a $1$-parameter family of hyperplanes depending smoothly on $t$, tangent 
to $X$ at smooth points $x(t)$ for every $t\neq 0$ and such that $H(0)$ contains $K$. In 
other words $H(0)\in \overline{X_0^{\vee}}\cap K^{\vee}$. 
Then there exists another $1$-parameter family $H'(t)$, such that 
\begin{enumerate}
 \item $H'(0)=H(0)$
 \item $H'(t)$ is a hyperplane tangent to $X$ at $x(t)$ and containing $K$. 
In other words, $H'(0)\in \overline{X_0^{\vee}\cap 
K^{\vee}}$.
\end{enumerate}
\end{lemma}

\begin{proof}
Let $x(0)\in X$ be the point $lim_{t\to 0} x(t)$. 
This point exists, since $\mathbb{P}$ is compact, but it could  be a singular 
point of $X$.
Let $\hat{T}_{x(0)}(X)$ denote the embedded tangent space of $X$ at $x(0)$.
Then it is well-known that $\pi_K$ induces an isomorphism of tangent spaces.
It follows that $\hat{T}_{x(0)}(X)$ is a projective subspace of 
$\mathbb{P}$ not intersecting $K$.
For $t\neq0$ the Zariski tangent spaces $T(t):=\hat{T}_{x(t)}(X)$ form an 
analytic $1$-parameter family of projective subspaces of dimension $\dim (X)$ and they do not intersect $K$. 
The limit projective subspace $T(0)$ of the family $T(t)$ is contained in the 
tangent space $\hat{T}_{x(0)}$ and hence, does also not intersect $K$. 
By assumption the hyperplanes $H(t)$ contain $T(t)$ for every $t\neq 0$ and since the 
family is analytic it also holds that $T(0)\subset H(0)$.
Now consider the projective subspaces $<K,T(t)>$ for $t\neq0$. 
They form an analytic family of projective subspaces of dimension $\dim(K)+\dim(X)+1$. 
Clearly for $t=0$ we have that $<K,T(0)>\subset H(0)$.
We can therefore construct an analytic family $H'(t)$ of hyperplanes s.t. $ H'(0)=H(0)$ 
and $<K,T(t)>\subset H'(t)$.
This proves the lemma.
\end{proof}
By the above lemma we conclude the proof of Proposition~\ref{prop:pi_K}.
\end{proof}

The next result gives a relation between the discriminants of $X$ and $\pi_K(X)$.
It is a consequence of Proposition~\ref{prop:pi_K}.
Before stating it, we fix some notations.
Let 
\begin{equation*}
 0 \rightarrow U\rightarrow V \rightarrow W \rightarrow 0
\end{equation*}
be an exact sequence of vector spaces. 
Let $\mathbb{P}=\mathbb{P}(V^*)$ and $K=\mathbb{P}(W^*)$. Then we also have 
$\mathbb{P}/K=\mathbb{P}(U^*)$ and the projection $\pi_K$ 
is induced by the map $V^*\rightarrow U^*$.

\begin{corollary}\label{cor:projectiondisc}
 Let $X\subset \mathbb{P}$ be a projective subvariety not intersecting $K$ and such that 
$\dim(X)<\dim(\mathbb{P}/K)$.
Then $\Delta_{\pi_K(X)}$ is a factor of the restriction of $\Delta_X$ to $U$.
Moreover, if $\pi_K:X\rightarrow \pi_K(X)$ is an isomorphism of algebraic 
varieties then $\Delta_X|_U=\Delta_{\pi_K(X)}$.
\end{corollary}

Let $X$ be a projective variety embedded into a projective space $\mathbb{P}$ and 
$\mathbb{P}$ is embedded 
as a projective subspace into a projective space $\mathbb{M}$.
Then $X$ can also be considered as a subvariety of $\mathbb{M}$. 
Let $(X^{\vee})_\mathbb{P}$ denote the dual variety of $X$ in $\mathbb{P}^*$ and 
$(X^{\vee})_\mathbb{M}$ the dual 
variety of $X$ in $\mathbb{M}^*$.
Notice that $\mathbb{P}^*$ can be identified with 
$\mathbb{M}^*/(\mathbb{P}^{\vee})_\mathbb{M}$, and we have a projection 
\begin{equation*}
 \pi:\mathbb{M}^*\setminus(\mathbb{P}^{\vee})_\mathbb{M} \rightarrow \mathbb{P}^*
\end{equation*}
with has as centre the space $\mathbb{P}^{\vee}$.
Let $Z\subset \mathbb{P}^*$ be a subvariety. The cone over $Z$ with apex 
$(\mathbb{P}^{\vee})_\mathbb{M}$ is defined 
to be the union of $(\mathbb{P}^{\vee})$ and the fibers $\pi^{-1}(z)$ for all $z\in Z$.
Consequently we have the following results.
\begin{corollary}
 The variety $(X^{\vee})_\mathbb{M}$ is the cone over $(X^{\vee})_\mathbb{P}$ with apex 
$(\mathbb{P}^{\vee})_\mathbb{M}$.
\end{corollary}

Given a vector space surjection $\pi:E\rightarrow V$ and let 
$\mathbb{P}:=\mathbb{P}(V^*)$ 
and $\mathbb{M}:=\mathbb{P}(E^*)$.
We have an embedding $j:\mathbb{P} \hookrightarrow \mathbb{M}$ and for every subvariety 
$X\subset \mathbb{P}$ we can also regard 
$X$ as a subvariety $j(X)\subset \mathbb{M}$.
Then the discriminants $\Delta_X$ and $\Delta_{j(X)}$ are polynomial functions on $V$ 
and $E$ respectively. Moreover,
\begin{corollary}\label{cor:Deltajx}
 The discriminants $\Delta_X$ and $\Delta_{j(X)}$ satisfy for every 
$f\in E$:
\begin{equation*}
 \Delta_{j(X)}(f)=\Delta_X(\pi(f)).
\end{equation*}
\end{corollary}

\end{appendix}

\bibliographystyle{alpha}
\bibliography{main_arXiv.bbl}

 \end{document}